\documentclass[10pt,reqno]{amsart}

\calclayout

\usepackage[british]{babel}
\usepackage[T1]{fontenc}         
\usepackage[utf8]{inputenc} 
\usepackage{lmodern} 

\usepackage{amsthm,amsmath,amsfonts,amssymb}

\usepackage[all]{nowidow}

\usepackage[numbers]{natbib}

\usepackage[colorlinks,citecolor=blue,urlcolor=blue,allcolors=blue]{hyperref}
\usepackage{graphicx}
\usepackage[british]{babel}
\usepackage{mathtools} 
\mathtoolsset{showonlyrefs} 
\usepackage{booktabs,longtable}
\usepackage{multirow}
\usepackage{mathrsfs}
\usepackage[super]{nth}
\usepackage[plain,noend,ruled]{algorithm2e}
\usepackage[noend]{algpseudocode}
\usepackage{todonotes}
\usepackage{thmtools}
\usepackage{thm-restate}
\usepackage{caption}
\theoremstyle{theorem}
\newtheorem{proposition}{Proposition}[section]
\newtheorem{theorem}[proposition]{Theorem}
\newtheorem{lemma}[proposition]{Lemma}

\newtheorem{appendixlemma}[proposition]{Lemma}

\theoremstyle{definition}
\newtheorem{definition}[proposition]{Definition}

\newtheorem{notation}[proposition]{Notation}
\newtheorem*{notation*}{Notation}

\newtheorem{example}[proposition]{Example}

\newtheorem*{motivation*}{Motivation}

\newtheorem{remark}[proposition]{Remark}

\usepackage{bm}
\usepackage{bbm}
\usepackage{enumerate}
\usepackage{mathtools}
\usepackage{soul}
\usepackage{todonotes}
\newtheorem*{theorem*}{Theorem}
\newtheorem*{lemma*}{Lemma}
\newtheorem*{proposition*}{Proposition}

\theoremstyle{plain} 
\newcommand{\thistheoremname}{}
\newtheorem{genericthm}[definition]{\thistheoremname}

\usepackage{mathrsfs} 

\DeclareSymbolFont{extraup}{U}{zavm}{m}{n}
\DeclareMathSymbol{\vardiamond}{\mathalpha}{extraup}{87}

\usepackage{accents}

\usepackage{stackengine}
\stackMath
\newcommand\tsup[2][2]{%
 \def\useanchorwidth{T}%
  \ifnum#1>1%
    \stackon[-.5pt]{\tsup[\numexpr#1-1\relax]{#2}}{\scriptscriptstyle\sim}%
  \else%
    \stackon[.5pt]{#2}{\scriptscriptstyle\sim}%
  \fi%
}

\makeatletter
\providecommand*{\sha}{%
  \mathbin{\mathpalette\sha@{}}%
}
\newcommand*{\sha@}[2]{%
  \sbox0{$#1\vcenter{}$}%
  \kern .15\ht0 
  \rlap{\vrule height .25\ht0 depth 0pt width 2.5\ht0}%
  \raise.1\ht0\hbox to 2.5\ht0{%
    \vrule height 1.75\ht0 depth -.1\ht0 width .17\ht0 %
    \hfill
    \vrule height 1.75\ht0 depth -.1\ht0 width .17\ht0 %
    \hfill
    \vrule height 1.75\ht0 depth -.1\ht0 width .17\ht0 %
  }%
  \kern .15\ht0 
}
\makeatother

\newcommand\independent{\protect\mathpalette{\protect\independenT}{\perp}}
\def\independenT#1#2{\mathrel{\rlap{$#1#2$}\mkern2mu{#1#2}}}

\usepackage{stmaryrd}

\usepackage[makeroom]{cancel} 

\makeatletter
\DeclareRobustCommand{\cev}[1]{%
  {\mathpalette\do@cev{#1}}%
}
\newcommand{\do@cev}[2]{%
  \vbox{\offinterlineskip
    \sbox\z@{$\m@th#1 x$}%
    \ialign{##\cr
      \hidewidth\reflectbox{$\m@th#1\vec{}\mkern4mu$}\hidewidth\cr
      \noalign{\kern-\ht\z@}
      $\m@th#1#2$\cr
    }%
  }%
}
\makeatother

\makeatother

\makeatletter
\newcommand{\oset}[3][-1.75ex]{%
  \mathrel{\mathop{#3}\limits^{
    \vbox to#1{\kern-2\ex@
    \hbox{$\scriptstyle#2$}\vss}}}}
\makeatother
\makeatletter
\newcommand{\osettext}[3][-1.15ex]{%
  \mathrel{\mathop{#3}\limits^{
    \vbox to#1{\kern-2\ex@
    \hbox{$\scriptstyle#2$}\vss}}}}
\makeatother

\newcommand{\vertiii}[1]{{\left\vert\kern-0.25ex\left\vert\kern-0.25ex\left\vert #1 
    \right\vert\kern-0.25ex\right\vert\kern-0.25ex\right\vert}}
\newcommand{\bvertiii}[1]{{\big\vert\kern-0.25ex\big\vert\kern-0.25ex\big\vert #1 
    \big\vert\kern-0.25ex\big\vert\kern-0.25ex\big\vert}}
\newcommand{\onevar}[1]{{\left\vert\kern-0.25ex\left\vert #1
    \right\vert\kern-0.25ex\right\vert_{1\text{-}\mathrm{var}}}}

\usepackage{tikz}
\usepackage[super]{nth}

\newcommand{\coloneqq}{\mathrel{\mathop:}=}
\newcommand{\eqqcolon}{\mathrel{=\hspace{-2.75pt}\mathop:}}

\newcommand{\supp}{\mathrm{supp}}

\DeclareMathOperator*{\argmin}{\mathrm{arg\,min}}

\newcommand\restr[2]{{
  \left.\kern-\nulldelimiterspace 
  #1 
  \vphantom{\big|} 
  \right|_{#2} 
  }}

\usepackage{enumitem}
\usepackage{accents}
\usepackage{float}
\usepackage{subfig}

\newcommand{\R}{\mathbb{R}}
\newcommand{\C}{\mathbb{C}}
\newcommand{\I}{\mathbb{I}}
\newcommand{\N}{\mathbb{N}}

\newcommand{\p}{\mathfrak{p}}

\newcommand{\E}{\mathbb{E}}

\newcommand{\sI}{\mathscr{I}}
\newcommand{\fI}{\mathfrak{I}}

\newcommand{\fD}{\mathfrak{D}}
\newcommand{\fS}{\mathfrak{S}}

\newcommand{\GL}{\operatorname{GL}_d}

\newcommand{\M}{\operatorname{M}_d}
\newcommand{\ds}{\diamondsuit}
\newcommand{\sig}{\mathfrak{sig}}

\newcommand{\fp}{\mathfrak{p}}

\usepackage{xcolor}

\makeatletter
\newcommand\niton{\mathrel{\m@th\mathpalette\canc@l\owns}}
\newcommand\canc@l[2]{{\ooalign{$\hfil#1/\mkern1mu\hfil$\crcr$#1#2$}}}
\makeatother

\usepackage{thmtools, thm-restate}
\usepackage[multiple]{footmisc}

\mathtoolsset{showonlyrefs}

\usepackage{blindtext}

\usepackage{multicol}

\makeatletter
\providecommand*{\shuffle}{%
  \mathbin{\mathpalette\shuffle@{}}%
}
\newcommand*{\shuffle@}[2]{%
  \sbox0{$#1\vcenter{}$}%
  \kern .15\ht0 
  \rlap{\vrule height .25\ht0 depth 0pt width 2.5\ht0}%
  \raise.1\ht0\hbox to 2.5\ht0{%
    \vrule height 1.75\ht0 depth -.1\ht0 width .17\ht0 %
    \hfill
    \vrule height 1.75\ht0 depth -.1\ht0 width .17\ht0 %
    \hfill
    \vrule height 1.75\ht0 depth -.1\ht0 width .17\ht0 %
  }%
  \kern .15\ht0 
}
\makeatother

\newcommand{\bdiam}{\tikz[baseline={([yshift=-.9ex]current bounding box.center)}]{\node[fill=black,rotate=45,inner sep=.1ex, text height=0.75ex, text width=0.75ex] {};%
\node[ font=\color{white}] (wi) {};}}

\makeatletter
\newcommand{\leqnomode}{\tagsleft@true\let\veqno\@@leqno}
\newcommand{\reqnomode}{\tagsleft@false\let\veqno\@@eqno}
\makeatother

\usepackage{thmtools, thm-restate}
\usepackage[multiple]{footmisc}

\usepackage{amsaddr}

\usepackage{faktor}

\usepackage{empheq}

\usepackage{varwidth}

\usepackage{etoc}
\usepackage{apptools}
\AtAppendix{}

\makeatletter
\renewcommand{\tocsection}[3]{%
  \indentlabel{\@ifnotempty{#2}{\bfseries\ignorespaces#1 #2\quad}}\bfseries#3}
\renewcommand{\tocsubsection}[3]{%
  \indentlabel{\@ifnotempty{#2}{\ignorespaces#1 #2\quad}}#3}

\newcommand\@dotsep{4.5}
\def\@tocline#1#2#3#4#5#6#7{\relax
  \ifnum #1>\c@tocdepth 
  \else
    \par \addpenalty\@secpenalty\addvspace{#2}%
    \begingroup \hyphenpenalty\@M
    \@ifempty{#4}{%
      \@tempdima\csname r@tocindent\number#1\endcsname\relax
    }{%
      \@tempdima#4\relax
    }%
    \parindent\z@ \leftskip#3\relax \advance\leftskip\@tempdima\relax
    \rightskip\@pnumwidth plus1em \parfillskip-\@pnumwidth
    #5\leavevmode\hskip-\@tempdima{#6}\nobreak
    \leaders\hbox{$\m@th\mkern \@dotsep mu\hbox{.}\mkern \@dotsep mu$}\hfill
    \nobreak
    \hbox to\@pnumwidth{\@tocpagenum{\ifnum#1=1\bfseries\fi#7}}\par
    \nobreak
    \endgroup
  \fi}
\AtBeginDocument{%
\expandafter\renewcommand\csname r@tocindent0\endcsname{0pt}
}
\def\l@subsection{\@tocline{2}{0pt}{2.5pc}{5pc}{}}
\makeatother

\begin{document}
\title[Robust ICA]{A Robustness Analysis of Blind Source Separation}  


\author{\vspace{-2em}\footnotesize Alexander Schell}
\address[A1]{\vspace{-0.75em}Department of Statistics, Columbia University} 

\begin{abstract}
Blind source separation (BSS) aims to recover an unobserved signal $S$ from its mixture $X=f(S)$ under the condition that the effecting transformation $f$ is invertible but unknown. As this is a basic problem with many practical applications, a fundamental issue is to understand how the solutions to this problem behave when their supporting statistical prior assumptions are violated. In the classical context of linear mixtures, we present a general framework for analysing such violations and quantifying their impact on the blind recovery of $S$ from $X$. Modelling $S$ as a multidimensional stochastic process, we introduce an informative topology on the space of possible causes underlying a mixture $X$, and show that the behaviour of a generic BSS-solution in response to general deviations from its defining structural assumptions can be profitably analysed in the form of explicit continuity guarantees with respect to this topology. This allows for a flexible and convenient quantification of general model uncertainty scenarios and amounts to the first comprehensive robustness framework for BSS. Our approach is entirely constructive, and we demonstrate its utility with novel theoretical guarantees for a number of statistical applications. 
\end{abstract}

\thanks{\href{https://mathscinet.ams.org/mathscinet/msc/msc2020.html}{\color{black}{\textit{MSC2020 subject classification:}}} Primary 62H25, 62G35; secondary 62H05, 60L10, 62M86.\newline 
}
\email[A1]{\href{mailto:alexander.schell@columbia.edu}{\texttt{alexander.schell@columbia.edu}}} 

\vspace*{-2em}
\maketitle

\vspace{-1em}

\section{Introduction}
\noindent
The Problem of Blind Source Separation concerns the (local) inversion of an unknown function given only the image under this function of some also unobserved argument: 
\begin{equation}\label{BSS-Robust-BlindInversion}
\text{Provided only \ $X=f(S)$ \ for \, $f$ and $S$ \, unknown, \ invert \ $X$ \, for \, $S$.}
\end{equation}This problem is central to a great variety of applications ranging from uses in medical imaging, neuroscience and biology \citep{aziz2016fuzzy,delorme2004eeglab,lee2003application,makeig2011erp} over finance and astronomy \citep{back1997first, nuzillard2000blind} to physics and engineering \citep{akutsu2020application,cvejic2007improving,takahata2012unsupervised} to name but a few, witness also the large number of practical and theoretical approaches that have been proposed to address it; see for instance \cite{HBS,HKO,MNT} for an overview. Since the generic formulation \eqref{BSS-Robust-BlindInversion} of the BSS-problem is highly underdetermined, all of these methods have to operate under some additional conditions on $f$ and $S$---so-called `identifia- bility \mbox{assumptions'---to guarantee that a meaningful blind inversion $X\mapsto S$ is possible at all.}\\[-0.5em] 

\noindent
The by far most successful ansatz to interpreting and solving \eqref{BSS-Robust-BlindInversion} to date assumes that the hidden relation $f$ is linear\footnote{\ While the linearity assumption on $f$ is essential for the vast majority of currently employed BSS-methods, cf.\ e.g.\ \citep{HBS,HKO,HYP}, this formerly ironclad constraint has recently been relaxed at the cost of stronger structural assumptions on $S$, see for instance \citep{halva2021disentangling,AUX,khemakhem20iVAE} and our recent paper \cite{sigNICA2021} and the references therein.} and the source $S$ is modelled as a multidimensional random variable with statistically independent components. This ansatz defines a framework commonly referred to as Independent Component Analysis (ICA), which is based on probabilistic insights independently obtained by Reiers{\o}l, Skitovich, and Darmois \citep{reiersol1950, SKI, DAR} that have been influentially formulated as an optimisation problem by Comon \citep{COM}. Since its surprisingly recent \cite{HKO} conception, ICA has undergone various adaptations and extensions, most of which involve modifications and analyses of its underlying optimisation procedures or variations of the source model in terms of different (identifiability-sufficient) `non-degeneracy' assumptions on its component distributions, cf.\ e.g.\ \citep{HBS,HKO,MNT} and the references therein.\\[-0.5em] 

\noindent
A comparatively under-explored yet both obvious and fundamental question, especially given how extensively ICA-methods are used in their considerable range of applications, is whether and to what extent the solutions to \eqref{BSS-Robust-BlindInversion} persist when their underlying identifiability assumptions on $f$ and $S$ are violated. The scope and relevance of this question and its implications are illustrated by the following classic example from signal processing:  

\begin{example}[A More Realistic Cocktail Party]\label{example:cocktailparty}
Suppose you are monitoring a crowded room in which there are $d\in\N$ people, say $(S^1, \ldots, S^d)\eqqcolon S$, each of whom is engaged in conversation. Suppose further that each of their utterances is potentially relevant and informative to you, and that your goal is therefore to obtain a (consensual) record of what each speaker $S^i$ has said. To this end, you have installed $d$ voice recorders at different places in the room. However, instead of speaking directly into your microphones, the speakers are so engaged that their speeches overlap acoustically, so that instead of a single clear voice per recorder, each microphone records a mixture $X^i = f_i(S^1,\ldots, S^d)$ of the $d$ different voices. Will you still be able to recover the individual contribution of each speaker from your recordings? 

In the first instance at least, it seems plausible (due to the linearity of sound propagation) to model each of the recordings as an unknown linear superposition of the speakers' voices, i.e.\ to assume that $X\equiv(X^1,\cdots,X^d) = f(S)$ for some $f=(f_1,\cdots,f_d) : \R^d\rightarrow \R^d$ linear. In addition, although perhaps a little harder to justify in this context, it is surely of considerable mathematical convenience to further assume that the voices sought are\footnote{\ $\ldots$ sampled from some sound-describing stochastic processes $S^i=(S^i_t)_{t\geq 0}$ in $\R$, which in turn are $\ldots$} statistically independent. From these angles, it is then compelling to try to recover the desired individual speech signals according to \eqref{BSS-Robust-BlindInversion} by performing an ICA on the recorded mixtures $X$. 

While intuitively this is certainly a plausible idea, some limiting considerations about the validity of this approach and its underlying structural assumptions quickly arise: 
\begin{itemize}
\item due to the interactive nature of human conversation, it is almost certain that some or all of the speech signals -- for example, those of speakers engaged in a joint discussion -- will not actually be statistically independent of each other, at least not entirely;
\item owing to the spatial distance between the microphones and various phenomena affecting the transmission and propagation of sound, such as acoustic dispersion or attenuation, the actual relation between the recorded mixtures and the speaker's individual voices is often effectively non-linear, at least mildly so; 
\item the voice signals themselves are typically not iid in time but naturally exhibit a great deal of intertemporal statistical complexity, and both the physical propagation of these signals and their recording are rarely perfectly unperturbed but typically subject to various forms of noise and technical corruption, such as ambient noise or asynchronous recording at inhomogeneous sampling rates between different microphones.  
\end{itemize}  
In view of these complications which make the true mixing constellation deviate from the simplifying idealisation of linearity and independence, will an ICA of the recorded mixtures still result in a meaningful reconstruction of the individual vocal expressions that, if probably not exact, is at least still `good enough' to identify the main message that each speaker has communicated? Or will the violations of the underlying model assumptions affect the ICA performance so severely that the returned signals are distorted beyond intelligibility, or the anticipated signal separation fails completely? Is the transition between these outcomes continuous, and can the accuracy of the returned signals be meaningfully quantified against the `amount of distortion' that causes the true scenario to deviate from its idealised identifiability assumptions?  \hfill $\bdiam$     
\end{example} 
\noindent
Formally speaking, each of the above items represents a potential violation of the structural prior assumptions about source and mixing that underlie the vast majority of ICA procedures, including all of the most popular such methods to date, cf.\ \citep{HBS,HKO,MNT}. Referring to the pair $(S,f)$ in \eqref{BSS-Robust-BlindInversion} as `the cause' of $X$, the central question then becomes how and to what extent such violations of the prior assumptions about the cause of an observation affect the algorithmic consistency of an ICA procedure, and whether this consistency can be meaningfully quantified against general distortions of its supporting identifiability assumptions.\\[-0.5em]  

\noindent
Although according investigations on the robustness of BSS and ICA solutions have been repeatedly called for \cite{HRS}, questions of this sort have received surprisingly little rigorous attention in the literature. To the best of our knowledge, there is currently no robustness theory for BSS or ICA that comprehensively analyses the above questions, and available results are limited to either proofs of statistical consistency\footnote{\ Consistency results can be seen as `asymptotic (partial) robustness wrt.\ empirical approximation' since they establish the convergence of solutions along certain sequences of sample-based estimators (but not necessarily along all convergent sequences, which would characterise robustness under a sequential topology).}, e.g.\ \citep{chen2005consistent,chen2006efficient,samarov2004nonparametric}, or partial sensitivity analyses of numerical subroutines \citep{afsari2008, cai2019perturbation, shi2015}, or are asymptotic \cite{chen2005consistent}, or confined to highly specific BSS-models \cite{behr2018multiscale,kervazo2021robust} or to only very particular or (semi-)parametric noise or dependence models \cite{belkin2013,kawanabe2005,pfister2019,SAM}, and even these few results are mostly restricted to the `time-independent' special case where $X$ is modelled as an iid sequence of random vectors.\\[-0.5em]

\noindent
In this paper, we aim to fill these gaps by providing a flexible and comprehensive robustness theory for the blind source separation of (iid and) time-dependent observables. Modelling the source $S$ in \eqref{BSS-Robust-BlindInversion} as a discrete- or continuous-time stochastic process in $\R^d$ and working under the general paradigm of ICA, our main contributions are as follows:\\[-0.5em]

\noindent
Building on a concise analysis of the abstract BSS problem \eqref{BSS-Robust-BlindInversion}, we propose a general and flexible topological notion of statistical robustness\footnote{\ Our use of the word `robust' is in the sense of Huber and Ronchetti \cite{huber2009}, as specified in Section \ref{appendix:sect:BSSformal}.} for Blind Source Separation and Independent Component Analysis that covers a wide range of relevant infringement scenarios (Section \ref{appendix:sect:BSSformal}).\newline
The idea is to describe robustness as continuity with respect to an informative ICA-tailored topology on the space of hidden causes underlying a BSS-observable. In Section \ref{chap:robustICA:sect:product_topology}, we construct such an ICA-topology from a natural and coarse premetric topology on the space of laws of stochastic processes (Section \ref{chap:robustICA:sect:premetric}), where our premetric is derived from a graded family of path-space functionals pertaining to the signature map from rough paths theory (Section \ref{sect:signalsconvergence}). Combined with our general formalisation of robustness, this ICA-topology provides a flexible and model-free way of capturing violations of idealised identifiability assumptions and relating them, in a finite---i.e.\ `non-asymptotic'---and conveniently quantifiable way, to their resulting impact on the accuracy of an ICA-inversion. This is demonstrated in Section \ref{chap:robustICA:sect:auxiliaries}, where we present a generic blind-inversion map that performs \eqref{BSS-Robust-BlindInversion} under ICA-typical identifiability assumptions while being the first to achieve this with unified stability guarantees for very general deviations from its underlying consistency assumptions. Conceptually, this inversion map has the property of being continuous with respect to the aforementioned ICA-topology, and its robustness guarantees are due to informative and entirely explicit moduli of continuity with respect to said topology (Theorems \ref{thm:robust_ica} and \ref{thm:robustness}). The significance of our approach for applications is illustrated in a series of statistical use cases in Section \ref{sect:applications}, including the demonstration of provable robustness for the practical infringement scenarios from Example \ref{example:cocktailparty}.     

\vfill
To keep our exposition compact, most of our proofs and some technical auxiliary results are delegated to the \hyperref[sect:preliminariesnotation]{appendix}. 
\vfill

\newpage
\noindent
We assume familiarity with Section \ref{sect:preliminariesnotation} and its notation, and adopt the following convention. 

\begin{restatable}[Identifying Signals With Their Laws]{convention}{convA}\label{convention:signals-distributions}
From Section \ref{sect:bss} onwards, we will tacitly identify a given signal $\bm{Y}$ in $\R^d$ (a random variable on $\mathcal{C}_d$) with its law $Y\equiv\mathbb{P}_{\bm{Y}} \coloneqq\mathbb{P}\circ\bm{Y}^{-1}$ (a prob.\ measure on $\mathcal{C}_d$). This justifies the prevalent abuse of notation $\bm{Y}\,[\cong Y]\in\mathcal{M}_1(\mathcal{C}_d)$.
\end{restatable} 
\noindent
(Equivalent to considering two [$\mathcal{C}_d$-valued] random variables `equal' if they are equal in distribution.) \mbox{Transformations of laws are understood as push-forwards, i.e.\ in the sense of \eqref{transformation}.} 

\section{The Problem of Blind Source Separation}
\label{appendix:sect:BSSformal} 
\noindent
A well-founded stability analysis of blind source separation requires a conceptionally precise formulation of the general problem statement \eqref{BSS-Robust-BlindInversion} that we gave at the beginning. Taking $\bm{S}=(\bm{S}_t)_{t\in\I}$ as a stochastic process in $\R^d$, $\I\subset\R$ compact, this section provides such a formulation: Following a rigorous definition of the inverse problem \eqref{BSS-Robust-BlindInversion} and its associated concept of solution (Section \ref{sect:bss}), we anticipate a general notion for the stability of blind inversion and turn this into a natural and exact definition of robustness in the context of ICA (Section \ref{sect:robustica-basic}).\\[-0.75em]  

The notions introduced in this section, which at first reading may seem somewhat abstract to the general reader, are illustrated in Figure \ref{fig:BSStriple} and further motivated in Section \ref{sect:bssformal:add_motivation}.   

\noindent 
\subsection{Blind Source Separation}\label{sect:bss} The problem of blind source separation is an inverse problem with the objective of recovering an unobserved signal $\bm{S}$ from its mixture $\bm{X}=f(\bm{S})$ under the condition that the effecting transform $f$ is invertible but unknown. The underlying relationship between $\bm{X}$ and $\bm{S}$ is expressed by the identity 
\begin{equation}\label{appendix:sect:BSSformal:eq1}
\bm{X}_t \, = \, f(\bm{S}_t) \quad \text{ for all } \ t\in\I,
\end{equation}
which is well-defined iff the domain of $f$ contains the spatial support $D_{\bm{S}}$ of $\bm{S}$, see \eqref{def:spatial_support:eq1}. Assuming that $f$ and its inverse are both continuous, we have $\overline{f(D_S)} = D_{X}$ (cf.\ \eqref{spatsup:law}) and call such
\begin{equation}\label{appendix:sect:BSSformal:eq1.1}
\text{$(X,S,f)$ \ \ a \emph{BSS-triple}} 
\end{equation}   
with \emph{observable $X$, source $S$}, and \emph{mixing transformation $f$}. In partial summary of the above, note that the `causal part' $(S,f)$ of any such triple is an element of the space 
\begin{equation}\label{causalspace}
\mathfrak{C}\coloneqq\left\{(\tilde{S}, \tilde{f})\in\mathcal{M}_1\times Z^Z \, \left.\right| \, \tilde{f}\in C^{0,0}(D_{\tilde{S}})\right\} 
\end{equation}
where we set $Z\coloneqq\R^d$ and employed Convention \ref{convention:signals-distributions}.\footnote{\ Recall that by this convention we identify a signal $\tilde{S}$ with its law $\mathbb{P}_{\tilde{S}}\in\mathcal{M}_1$. Note further that in \eqref{causalspace} we made the slight abuse of notation: $\tilde{f}\in C^{0,0}(D_{\tilde{S}}) \ :\Leftrightarrow \ \left.\tilde{f}\right|_{D_{\tilde{S}}}\!\!\in C^{0,0}(D_{\tilde{S}})$, to not clutter our exposition.} A few more definitions will be useful.

\subsubsection{The Problem of Blind Source Separation}Our blindness to the right-hand side of \eqref{appendix:sect:BSSformal:eq1} imposes an asymmetry between the observable $X$ and its cause $(S,f)$: If we are only given $X$, bar any information on $S$ or $f$, then we cannot distinguish the `original' cause $(S,f)$ underlying $X$ from any other pair in the set $\{(\tilde{S},\tilde{f}) \mid (X, \tilde{S},\tilde{f}) \text{ BSS-triple}\}\subseteq\mathfrak{C}$ of alternative causes. In this case, the original $S$ cannot be recovered from $X$ any better than finding an element of 
\begin{equation}
[S]_X \coloneqq \big\{\tilde{S} \in \mathcal{M}_1 \ \big| \ \exists\, \tilde{f}\in C^{0,0}(D_{\tilde{S}}) \, : \, \text{$(X,\tilde{S}, \tilde{f})$ is BSS-triple} \big\} = \operatorname{ev}_X\!\big(C^{0,0}(D_X)\big).
\end{equation} 

\noindent 
The BSS paradigm ameliorates this by upgrading our information from (nothing but) $X$ to:\vspace{-0.25em}
\begin{equation}\label{appendix:sect:BSSformal:intext:eq:identconds}
\text{$X$ \quad and \quad $\sI$ \quad and \quad the \emph{inclusion} \ $(S,f)\stackrel{!}{\bm{\in}} \mathscr{I}$}
\end{equation}
for some a-priori given non-empty subset $\mathscr{I}$ of $\mathfrak{C}$, called an \emph{identifiability assumption} (IA), which describes our prior knowledge on the original cause $(S,f)$ of $X$. The larger the set $\sI$, the less we know about $(S,f)$: if $\sI=\{(S,f)\}$ then we know the original cause exactly, and if $\sI=\mathfrak{C}$ then we are `completely blind' about the original cause.\\[-0.5em]

\noindent
The upgrade of the data from $X$ to $(X,\sI)$ reduces the asymmetry between $X$ and $(S,f)$: Under \eqref{appendix:sect:BSSformal:intext:eq:identconds}, the source of a triple $(X,S,f)$ can now be recovered up to within the smaller class   
\begin{equation}\label{appendix:sect:BSSformal:intext:eq:maxsol2}
\langle X\rangle_{\!\mathscr{I}}\coloneqq \big\{\tilde{S}\equiv\tilde{g}(X) \ \big| \ \tilde{g}\in C^{0,0}(D_X)\ \text{and} \ \big(\tilde{g}(X),\tilde{g}^{-1}\big)\in\mathscr{I}\big\}\ \subseteq\ [S]_X.
\end{equation}   
In other words: each element in $\langle X\rangle_{\!\mathscr{I}}$ is a best-approximation of $S$ given \eqref{appendix:sect:BSSformal:intext:eq:identconds}, and any two elements in $\langle X\rangle_{\!\mathscr{I}}$ are `indistinguishable'\footnote{\ More precisely, this is captured by the equivalence relation $\left[\tilde{S} \, \sim_{(X,\mathscr{I})} \check{S} \, : \Leftrightarrow \ \tilde{S}, \check{S} \, \in \, \langle X\rangle_{\!\mathscr{I}}\right]$ on $[S]_X$.} from $S$ given $(X,\sI)$. Let us agree on terminology. 

\begin{definition}[BSS-triple on $\sI$]\label{def:BSS-triple}
Let $\sI\subseteq\mathfrak{C}$. A BSS-triple $(X,S,f)$ such that $(S,f)\in\sI$ is called \emph{a BSS-triple on $\mathscr{I}$}, in symbols: $(X,S,f)_{\!\mathscr{I}}$. The associated class $\langle X\rangle_{\!\mathscr{I}}$ in \eqref{appendix:sect:BSSformal:intext:eq:maxsol2} is the \emph{maximal solution} from $X$ given $\mathscr{I}$, and its elements are the \emph{quasi sources} of $X$ given $\mathscr{I}$.
\end{definition}  

A schematic illustration of these and the following definitions is given in Figure \ref{fig:BSStriple}, while in Example \ref{example:cocktailparty_contd} the notions of this section are fleshed out in the cocktail party context of p.\ \pageref{example:cocktailparty}.
\begin{remark}\label{rem:triples-and-bounds} 
\begin{enumerate}[font=\upshape, label=(\roman*)]
\item The maximal solution \eqref{appendix:sect:BSSformal:intext:eq:maxsol2} of a BSS-triple $(X,S,f)_{\mathscr{I}}$ is maximal wrt.\ set inclusion and (hence) unique by construction. 
\item If the identifiability assumption $\mathscr{I}$ in \eqref{appendix:sect:BSSformal:intext:eq:identconds} is given as a product set in $\mathcal{M}_1\times C^{0,0}(\R^d)$, say $\mathscr{I} = \mathscr{S}\times\mathfrak{T}$, then the maximal solution $\langle X\rangle\equiv\langle X\rangle_{\mathscr{S}\times\mathfrak{T}}$ from $X$ given $\mathscr{I}$ reads
\begin{equation}\label{appendix:sect:BSSformal:intext:eq:maxsol_prodset}
\langle X\rangle = \operatorname{ev}_X\!\left(\mathfrak{T}\right)\cap\mathscr{S}
\end{equation} 
where $\operatorname{ev}_X(\mathfrak{T})$ is the image of the evaluation map $\operatorname{ev}_X : \mathfrak{T}\ni g\mapsto g(X)$, cf.\ \eqref{transformation}. 
\item\label{rem:triples-and-bounds:it3} The [size of the] set $\langle X\rangle_{\!\mathscr{I}}$ depends on $\mathscr{I}$. In particular: the weaker the IA $\sI$ (i.e.\ the larger the set $\mathscr{I}$), the larger is the class $\langle X\rangle_{\!\mathscr{I}}$ and hence the `further away' the set of all quasi sources will be from the original source $S \cong\{S\}$. Accordingly, any $S$-containing subset $[S]_\circ \, \subseteq \, \langle X\rangle_{\!\mathscr{I}}\,[\subseteq [S]^\circ, \text{ resp.}]$ defines an upper [resp.\ lower] bound on the `accuracy' to which $S$ can be recovered from $X$.  
\end{enumerate}  
\end{remark} 

\noindent 
Central to the blind inversion task $X \mapsto S$ is the \emph{accuracy} to which it is to be achieved --- that is, what deviations [of the attained quasi sources] from the original source $S$ are considered optimal or most tolerable? The answer is typically defined by the BSS-practitioner or their application context, and it translates into the constraint\footnote{\ So that each element $\tilde{S}\in\lceil S\rceil\cap\langle X\rangle_{\!\sI}$ is an \emph{optimal solution} to the such-constrained BSS problem.} $\langle X\rangle_{\!\sI}\subseteq\lceil S\rceil$ for some a-priori given `optimal' superset $\lceil S\rceil$ of $S$ (the `accuracy bound'). 
The main difficulty, then, is to identify an IA $\sI$ for which this inclusion holds, i.e.\ for which the set $\langle X\rangle_{\!\sI}$ is identical to the optimum superset $[S]_\circ\coloneqq\lceil S\rceil\cap\langle X\rangle_{\!\sI}$ of $S$. In case of non-optimality, the deviation between $\langle X\rangle_{\!\sI}$ and $[S]_\circ$ can be controlled by deriving a lower bound $[S]^\circ$ a posteriori to the choice of $\sI$.\\[-0.75em]

The above immediately suggests the following notions. 

\begin{definition}[Blind Inversion]\label{def:blindinversion}Given a map $\lceil\,\cdot\,\rceil : \mathcal{M}_1\rightarrow 2^{\mathcal{M}_1}$, an IA $\mathscr{I}\subseteq\mathfrak{C}$ is called
\begin{equation}\label{def:blindinversion:eq1}
\text{\emph{$\lceil\,\cdot\,\rceil$-sufficient} \ \ if \ \ $\langle X\rangle_{\!\mathscr{I}}\subseteq\lceil S\rceil$ \ for each BSS-triple $(X,S,f)_{\!\sI}$.} 
\end{equation}
Further, for any $\sI\subseteq\mathfrak{C}$ we call \emph{inversion map on $\sI$} any function $\Phi : \mathcal{M}_1\rightarrow 2^{\mathcal{M}_1}$ such that: 
\begin{equation}\label{def:blindinversion:eq2}
\emptyset\,\neq\,\Phi(X)\,\subseteq\,\langle X\rangle_{\!\sI}, \quad \text{for each BSS-triple } (X,S,f)_{\!\sI}.
\end{equation}
Finally, a map $\varphi : \mathcal{M}_1\rightarrow 2^{\mathcal{M}_1}$ shall be called \emph{faithful} if $S\in\varphi(S)$ for each $S\in\mathcal{M}_1$.
\end{definition} 
The task of blind inversion now seeks to achieve the following: For a given accuracy bound $\lceil\,\cdot\,\rceil$, can we find an identifiability assumption $\sI$ (the weaker the better) that is strong enough to guarantee that for any cause $(S,f)\in\sI$, the source $S$ can be recovered from its mixture \eqref{appendix:sect:BSSformal:eq1} up to an accuracy of $\lceil\,\cdot\,\rceil$, i.e.\ up to within an element of $\lceil S\rceil$? And if so, is there an explicitly computable algorithm that performs this inversion?\\[-0.5em]

\begin{figure*} 
\centering
\begin{tabular*}{\textwidth}{@{\extracolsep{\fill}}c c cc}
&\hspace{-2.5em}\includegraphics[width=.5\textwidth]{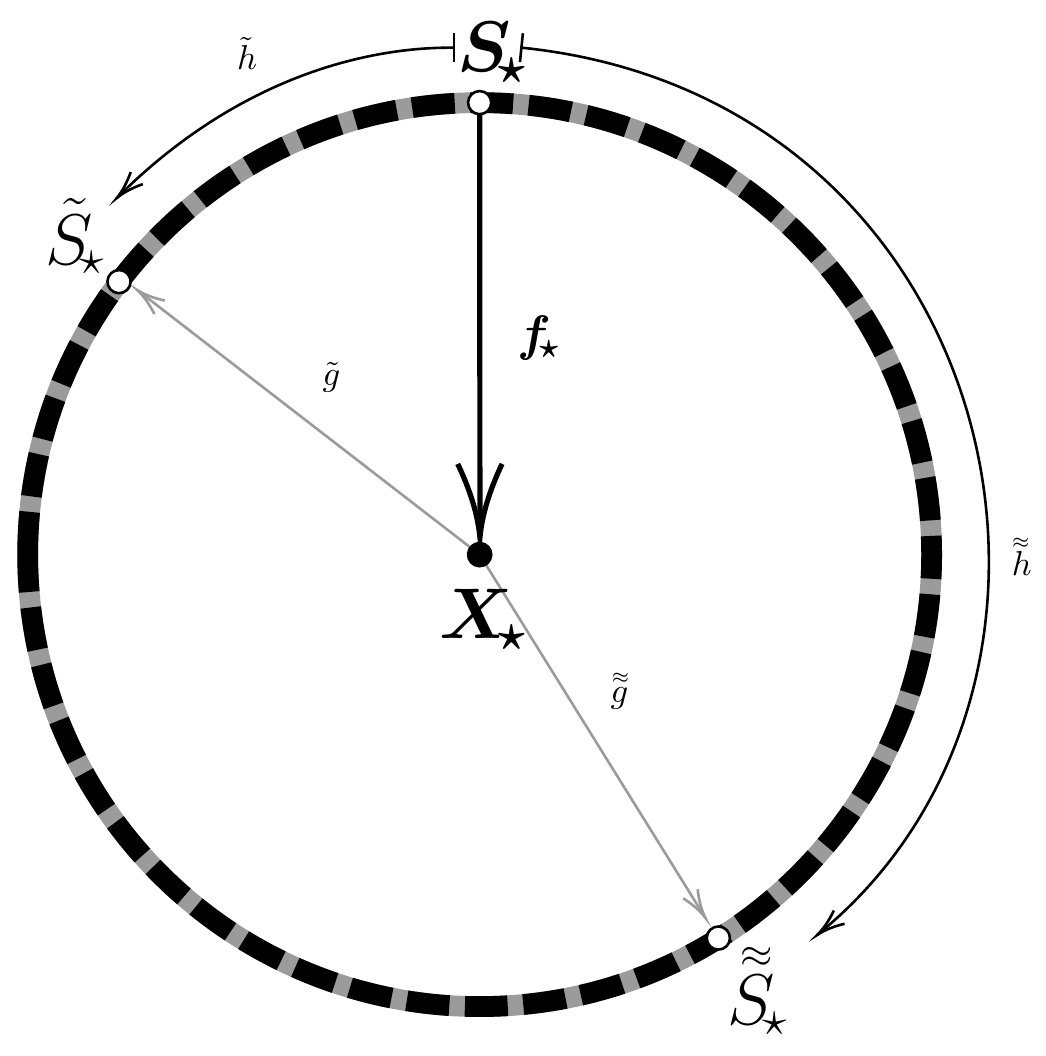} &
\hspace{2.5em}\includegraphics[width=.5\textwidth]{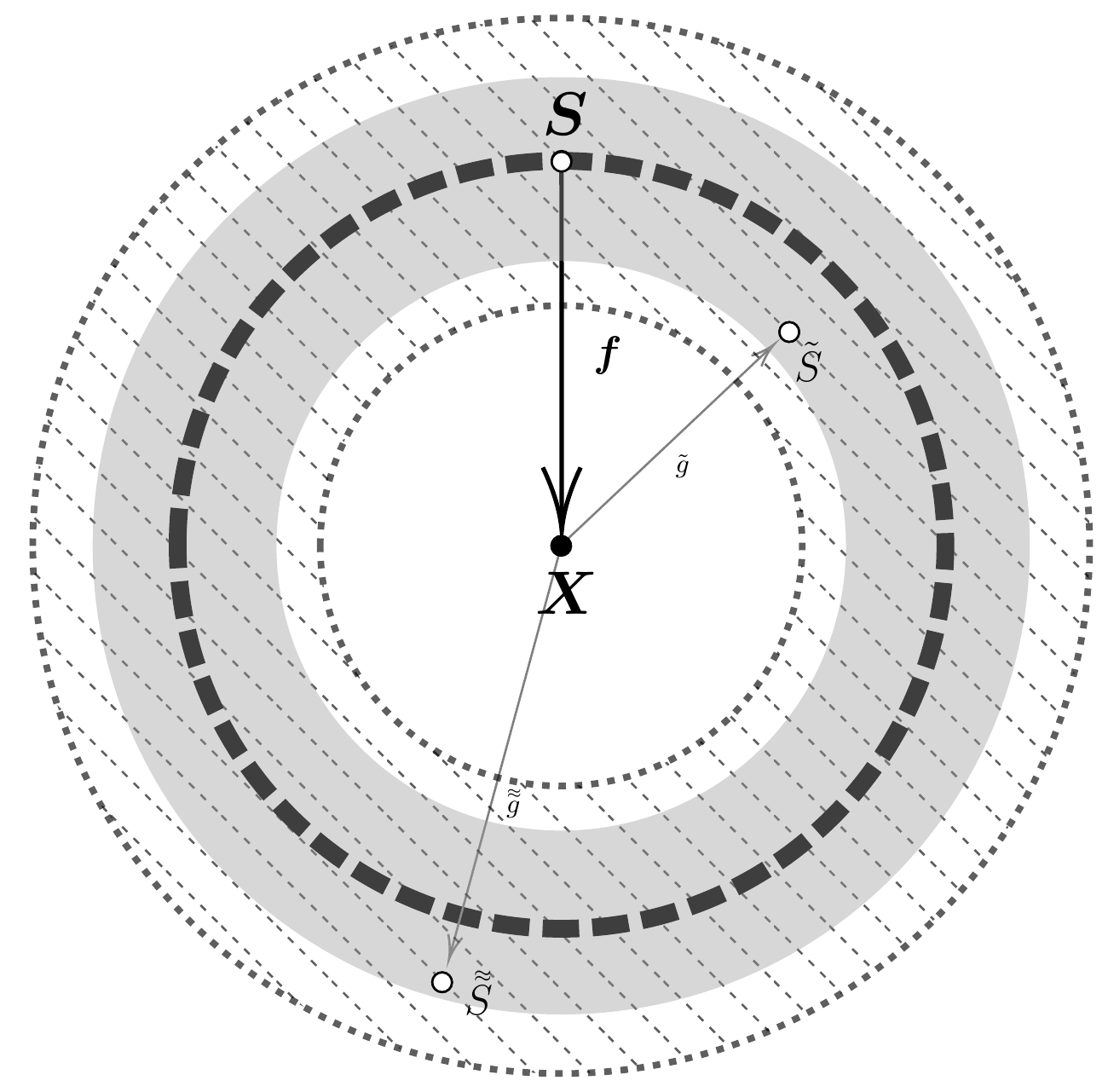} &\\[0.5em]
&a.\, $[S_\star]_\circ = \langle X_\star\rangle_{\!\mathscr{I}_\ast} \!\!= \Phi(X_\star)$ & b.\, $[S]_\circ\subsetneq \langle X\rangle_{\!\mathscr{I}}, \ \Phi(X) \neq [S]_\circ$&
\end{tabular*}  
\vspace{1em}
\caption{\emph{Schematic illustration of two BSS-triples $(X_\star,S_\star,f_\star)_{\!\sI_\ast}, (X, S, f)_{\!\sI}$ and their respective maximal solution. Each point in the plane represents a signal, i.e.\ an element of $\mathcal{M}_1$, and the coloured areas visualise subsets of $\mathcal{M}_1$ related to the observables $X_\star, X$ or their sources $S_\star, S$. Depicted in this way are, for some accuracy-bound $\lceil\,\cdot\,\rceil$: the sets of accurate quasi sources $[S_\star]_\circ\coloneqq\langle X_\star\rangle_{\!\sI_\ast}\cap\lceil S_\star\rceil$ and $[S]_\circ\coloneqq\langle X\rangle_{\!\sI}\cap\lceil S\rceil$ \emph{(}dashed black circles\emph{)}, the maximal solution $\langle X\rangle_{\!\sI}$ \emph{(}striped annulus with pointed boundary\emph{)}, as well as the outputs $\Phi(X_\star)$ and $\Phi(X)$ of some inversion map $\Phi\equiv\Phi_{\!\sI_\ast}$ on $\sI_\ast$ \emph{(}grey annuluses\emph{)}. The IA $\sI_\ast$ supporting $X_\star$ is $\lceil\,\cdot\,\rceil$-sufficient, which implies that the sets $\langle X_\star\rangle_{\!\sI_\ast}=\Phi(X_\star)$ (i.e., the set of all $\sI_\ast$-informed best guesses on $S_\star$ given $X_\star$) and $[S_\star]_\circ$ coincide, see panel a. In particular, each quasi source $\tilde{S}_\star\equiv\tilde{g}(X_\star)$ in $\langle X_\star\rangle_{\!\sI_\ast}$ is an `accurate' guess for $S_\star$ insofar as the residual $\tilde{h}=\tilde{g}\circ f_\star$ that relates $S_\star$ to $\tilde{S}_\star$ is in `minimum distance to identity' in a sense determined\protect\footnotemark by $\lceil\,\cdot\,\rceil$. In panel b., the IA $\sI$ underlying $X$ is too weak to be $\lceil\,\cdot\,\rceil$-sufficient and so the maximal solution $\langle X\rangle_{\!\sI}$ is a proper superset of the optimum $[S]_\circ$. The \emph{robustness} of $(\sI_\ast,\Phi)$ asserts that the inversion property of the map $\Phi$ is resilient under deviations from $\sI_\ast$ of the causes (in $\mathfrak{C}$) underlying its input signal. This amounts to a `continuous transition from a.\ to b.', as captured by \eqref{BSS:robust_prelim1}: if, in some general and appropriate sense, the causes $(S,f)\in\sI$ and $(S_\star,f_\star)\in\sI_\ast$ are `not too different', can we quantitatively guarantee that $\Phi(X)$ and $\lceil S\rceil$ are not too far apart either?}}   
\label{fig:BSStriple}    

\end{figure*}

In formal terms, the (general) \textbf{Problem of Blind Source Separation} reads as follows:\\[-0.25em]   

\noindent\makebox[\textwidth][c]{%
\fbox{
\noindent
\begin{minipage}{0.9\textwidth}
\leqnomode
\begin{equation}\label{appendix:sect:BSSformal:BSS_reformulation}\tag{$\star$}
\begin{gathered}
\text{Given a faithful map $\lceil\,\cdot\,\rceil : \mathcal{M}_1\rightarrow 2^{\mathcal{M}_1}$, determine:}\\
\begin{aligned}
&\text{(a) a [weak] IA $\mathscr{I}\subseteq\mathfrak{C}$ that is $\lceil\,\cdot\,\rceil$-sufficient;}\\
&\text{(b) a computable inversion map $\Phi$ on $\sI$.}
\end{aligned}\\[0.25em] 
\end{gathered}
\end{equation}
\end{minipage}}
}{\ }\\[0.5em]
     
\footnotetext{\ For instance, the popular choice \eqref{example:cocktailparty_contd:eq1} of $\lceil S\rceil$ as the monomial orbit $\M\cdot S$ mandates $\tilde{h}\in\M$.}Any pair $(\sI, \Phi)$ that satisfies (a)$\,\&\,$(b) is called a \emph{solution} to the BSS-problem for $\lceil\,\cdot\,\rceil$.

\begin{remark}As it stands, the above problem \eqref{appendix:sect:BSSformal:BSS_reformulation} in fact comprises a whole family of inverse problems, even for a fixed $\lceil\,\cdot\,\rceil$, since the $\sI$-qualifying adjective `weak' in $\eqref{appendix:sect:BSSformal:BSS_reformulation}_{(a)}$ is not explicitly\footnote{\ E.g., while the choice $(\lceil S\rceil, \sI) = \big(\{S\}, \{(S, f)\}\big)$ is trivial, the choice $(\lceil S\rceil, \sI) = \big([S]_X, \mathfrak{C}\big)$ is meaningless; solutions to $\eqref{appendix:sect:BSSformal:BSS_reformulation}_{(a)}$ that are of practical or theoretical relevance will lie strictly between these two extremes.} specified. This ambiguity can be avoided by demanding $\sI$ to be the `weakest' (i.e.\ globally maximal wrt.\ set inclusion) $\lceil\,\cdot\,\rceil$-sufficient IA in $\mathfrak{C}$, which would turn $\eqref{appendix:sect:BSSformal:BSS_reformulation}_{(a)}$ into the stronger task of characterising $\lceil\,\cdot\,\rceil$-accurate invertibility. We refrain from this stronger formulation here in order to better capture the BSS-problem in the form in which it is currently discussed and approached in the literature, cf.\ e.g.\ \citep{HBS,HKO,MNT} and the references therein. \hfill $\bdiam$   
\end{remark}

\begin{remark}The BSS-problem \eqref{appendix:sect:BSSformal:BSS_reformulation} consists of an abstract part, (a), which theoretically ensures the [sufficiently accurate] blind invertibility of an observable back to its source, and a more procedural part, (b), which asks for an algorithm to perform this inversion in (both) theory and statistical practice. The documented solutions to \eqref{appendix:sect:BSSformal:BSS_reformulation} usually address the tasks (a) and (b) simultaneously, typically by first proving the $\lceil\,\cdot\,\rceil$-sufficiency \eqref{def:blindinversion:eq1} of an IA constructively and then deriving an explicit inversion map \eqref{def:blindinversion:eq2} as a direct consequence of the proof strategy employed, cf.\ \cite{COM,HBS,HKO,MNT}. We also follow this general scheme in Section \ref{chap:robustICA:sect:auxiliaries}. \hfill $\bdiam$ 
\end{remark}

\noindent
In addition to the classical formulation \eqref{appendix:sect:BSSformal:BSS_reformulation}, we propose to (provisionally) say that a solution
\begin{equation}\label{BSS:robust_prelim1}
\begin{gathered}
(\sI, \Phi) \ \text{is \textbf{robust}} \quad \text{if,} \quad \text{in some appropriate sense,} \quad\text{the maps }\\
(\tilde{S},\tilde{f})\,\mapsto\,\Phi(\tilde{f}(\tilde{S})) \ \text{ and } \ (\tilde{S},\tilde{f})\,\mapsto\,\mathrm{dist}_{\varrho\,}\!\!\left(\Phi(\tilde{f}(\tilde{S})),\big\lceil\tilde{S}\big\rceil\right) \quad \text{are continuous \ on $\sI$,} 
\end{gathered} 
\end{equation}   
where as usual a map $\Phi$ is called \emph{continuous on $\mathscr{I}$} (or \emph{$\sI$-continuous}) if $\Phi$ is pointwise continuous on $\mathscr{I}$; compare to \citep{huber2009} for the general notion of robustness that we follow here. 

The above uses the `asymmetric distance'
\begin{equation}\label{BSS:robust_prelim1:eq1}
\mathrm{dist}_{\varrho\,}\!\!\left(\mathcal{A}, \mathcal{B}\right) \coloneqq \sup\nolimits_{a\in\mathcal{A}}\mathrm{d}_{\varrho\,}\!(a,\mathcal{B}) \quad\text{with}\quad \mathrm{d}_{\varrho\,}\!(a,\mathcal{B})\coloneqq\inf\nolimits_{b\in\mathcal{B}}\varrho(a,b) 
\end{equation}
where $\varrho : \mathcal{M}_1^{\times 2}\rightarrow\R_+$, with $\varrho\circ(\mathrm{id}\times\mathrm{id})=0$, is some appropriate distance function between the outputs of $\Phi$ and the desired signals in $\lceil\,\cdot\,\rceil$; specific choices of $\varrho$ will depend on $\lceil\,\cdot\,\rceil$.

\begin{remark}
From a purely conceptional angle, the (anticipated) notion of robustness \eqref{BSS:robust_prelim1} adds to the conventional analysis of problem \eqref{appendix:sect:BSSformal:BSS_reformulation} the classical aspects of (numerical and analytical) well-conditionedness and stability, which play a central role throughout many areas of pure and applied mathematics, see e.g.\ \citep{hadamard1902,higham2002,huber2009}, \citep[p.~139]{arendturban2018}, or \citep[Sect.\ IV.12.3]{gowers2008princeton}.
\end{remark}  

\subsection{Independent Component Analysis}\label{sect:robustica-basic}
As indicated above, the preliminary description of BSS-robustness \eqref{BSS:robust_prelim1} is still only an informal desideratum rather than a precise definition; for the latter, both $\varrho$ and a suitable topology (`$\sI$-continuity') need to be further specified. This gap is closed in the present subsection, which provides an according completion of \eqref{BSS:robust_prelim1} for the most widely received special instance of \eqref{appendix:sect:BSSformal:BSS_reformulation} known as Independent Component Analysis (ICA).\\[-0.5em] 

\noindent
To prepare, denote by $\lceil\,\cdot\,\rceil_{\mathfrak{m}} : \mathcal{M}_1\rightarrow 2^{\mathcal{M}_1}$ the function that sends a signal $S$ to its orbit $\lceil S\rceil_{\mathfrak{m}}\coloneqq \mathrm{M}_d\cdot S$ under the monomial group (action) $\mathrm{M}_d\coloneqq\{\Lambda P\mid \Lambda\in\tilde{\Delta}_d, P\in\mathrm{P}_d\}$ on $\mathcal{M}_1$, cf.\ \eqref{example:cocktailparty_contd:eq1}. Let similarly $\tilde{\mathrm{M}}_d\coloneqq\{h : u \mapsto Mu + b\mid M\in\M, \, b\in\R^d\}$ be the subgroup [of $C^{1,1}(\R^d)$] of affinely monomial transformations, with $\lceil S\rceil_{\tilde{\mathfrak{m}}}\coloneqq\tilde{\mathrm{M}}_d\cdot S$. Write also $\mathfrak{R}\coloneqq\{g : \R^d\rightarrow\R^d\,|\, g\text{ is Borel}\}$ for the set of Borel measurable transformations on $\R^d$.\\[-0.75em]          

\noindent
Below, see \eqref{def:icainversion:eq1}, defines a refinement of \eqref{def:blindinversion:eq2}. Note for this that if $\sI\!\!\subseteq\!\mathfrak{C}$ is $\lceil\,\cdot\,\rceil_{\mathfrak{m}}$-sufficient, then for any inversion map $\Phi$ on $\sI$ there is $\hat{\Phi} : \mathcal{M}_1 \rightarrow 2^{\mathfrak{R}}$, $X\mapsto\{B\in\mathfrak{R}\mid B\cdot X\in\Phi(X)\}$, with 
\begin{equation}\label{lem:solmap-matrixvalued:eq1}
\Phi(X) \, = \, \hat{\Phi}(X)\cdot X, \quad \text{for each BSS-triple } (X,S,f)_{\!\sI}.
\end{equation}
Among the class of inverse problems subsumed under \eqref{appendix:sect:BSSformal:BSS_reformulation}, the following variant is by far most established one to date, cf.\ for instance \cite{HBS, HRS, HKO, MNT}.      
\begin{definition}[Independent Component Analysis]\!\!\!\!\label{appendix:sect:BSSformal:def:ICA}
A set \!$\sI\!\subseteq\!\mathfrak{C}$ is called an \emph{ICA-condition} if
\begin{equation}\label{appendix:sect:BSSformal:def:ICA:eq1}
\begin{gathered}
\mathscr{I} = \mathscr{S}\times\GL \quad\text{for some}\quad\mathscr{S}\supseteq\hat{\mathscr{S}}_\bot \quad \text{and} \quad \mathscr{I} \ \text{ is \ $\lceil\,\cdot\,\rceil_{\mathfrak{m}}$-sufficient}; 
\end{gathered} 
\end{equation}
a BSS-triple $(X,S,f)_{\!\mathscr{I}}$ is called an \emph{ICA-triple} (\emph{on $\sI$}) if $\sI$ is an ICA-condition, and in this case we call an \emph{ICA-inversion on $\sI$} any function 
\begin{equation}\label{def:icainversion:eq1}
\hat{\Phi} : \mathcal{M}_1\rightarrow2^\mathfrak{R} \quad\text{such that:}\quad \emptyset\,\neq\,\hat{\Phi}(f(S))\circ f\,\subseteq\, \tilde{\mathrm{M}}_d, \quad \text{for each } (S,f)\in\sI.
\end{equation}The Problem of \emph{Independent Component Analysis (ICA)} is \eqref{appendix:sect:BSSformal:BSS_reformulation} for $\lceil\,\cdot\,\rceil\equiv\lceil\,\cdot\,\rceil_{\tilde{\mathfrak{m}}}$ and with $\sI$ and $\Phi$ of the form \eqref{appendix:sect:BSSformal:def:ICA:eq1} and \eqref{def:icainversion:eq1}; call an according pair $(\sI,\hat{\Phi})$ a \emph{solution} to the ICA problem.  
\end{definition}    
\begin{remark}In other words, an ICA-triple $(X,S,f)_{\!\mathscr{I}}$ is a BSS-triple where the source $S$ is to be recovered up to some permutation and scaling (and perhaps some constant offset) of its original components, and where the prior information $\sI\ni(S,f)$ for this consists of a condition on the source (namely $\mathscr{S}$, which describes the regularity of each source component and their statistical independence) and a linearity condition on the mixing ($\GL$), which are imposed separately by \eqref{appendix:sect:BSSformal:intext:eq:identconds}. Note also that if $\hat{\Phi}$ is an ICA-inversion of the form \eqref{def:icainversion:eq1} then $X\mapsto\hat{\Phi}(X)\cdot X$ is an inversion map on $\sI$ in the sense of Definition \ref{def:blindinversion}, as follows from \eqref{def:blindinversion:eq2} and \eqref{appendix:sect:BSSformal:intext:eq:maxsol_prodset} and provided that $\tilde{\mathrm{M}}_d\cdot\mathscr{S}\subseteq\mathscr{S}$. \hfill $\bdiam$ 
\end{remark} 
\noindent
The specificity \eqref{appendix:sect:BSSformal:def:ICA:eq1} and \eqref{def:icainversion:eq1} of ICA allows a natural analytical interpretation of the generic robustness notion \eqref{BSS:robust_prelim1}. For this, let $(\sI, \hat{\Phi})$ be any fixed ICA-solution, set $\Phi : \bm{X} \mapsto \hat{\Phi}(X)\cdot \bm{X}$, and let us quantify the (deviations from the) accuracy condition \eqref{def:icainversion:eq1} via the `deviance function'
\begin{equation}\label{def:deviance}
\begin{aligned}
\partial_{\hat{\Phi}}:\,\mathfrak{C}\rightarrow\R_+ \quad\text{ given by }\quad \partial_{\hat{\Phi}}(S,f)&\coloneqq\sup\nolimits_{B\in\hat{\Phi}(f(S))}\inf\nolimits_{(M,v)\in\mathrm{M}_d\times\R^d}\varrho_B((S,f),(M,v))\\ 
\text{for}\quad \varrho_{B}((S,f),(M,v))&\coloneqq \sup\nolimits_{u\in D_S^{(v)}}\!\!\tfrac{|B\circ f(u-v) - Mu|}{|Mu|}\mathbbm{1}_{\!\times}\!{\scriptstyle(Mu)},  
\end{aligned} 
\end{equation}
with $\mathbbm{1}_\times\!\coloneqq\mathbbm{1}_{\R^d\setminus\{0\}}$ and $D_S^{(v)}\coloneqq v + D_S$. To read inequality \eqref{lem:robustness-inversionstability:eq1} below, recall notation \eqref{BSS:robust_prelim1:eq1}.
  
\begin{lemma}\label{lem:robustness-inversionstability}
The map $\partial_{\hat{\Phi}}$ vanishes on $\sI$, i.e.\ $\restr{\partial_{\hat{\Phi}}}{\sI}=0$, and for any function $f$ and every stochastic process $\bm{S}=(\bm{S}_t)$ with $(S,f)\in\mathfrak{C}$ we have that, with probability one, 
\begin{equation}\label{lem:robustness-inversionstability:eq1}
\mathrm{dist}_{\tilde{\varrho}}\!\left(\Phi(f(\bm{S})),\lceil \bm{S}\rceil_{\tilde{\mathfrak{m}}}\right)\,\leq\,\partial_{\hat{\Phi}}(S,f) \quad\text{for}\quad \tilde{\varrho}(\bm{X},\bm{Y})\coloneqq \sup\nolimits_{t\in\I}\tfrac{|\bm{X}_t - \bm{Y}_t|}{|\bm{Y}_t|}\mathbbm{1}_{\!\times}{\scriptstyle\!(\bm{Y}_t)}.
\end{equation}  
\end{lemma}    

With the relative error between the output of $\Phi$ and the desired accuracy $\lceil\,\cdot\,\rceil_{\tilde{\mathfrak{m}}}$ controlled by $\partial_{\hat{\Phi}}$, we can now provide the anticipated formalisation of \eqref{BSS:robust_prelim1} for the special case of ICA. Note that the following is a natural combination of the BSS-specific stability requirements \eqref{BSS:robust_prelim1} with the classical notion of statistical robustness as set out in \citep[Sects.\ 1.2--1.4, 2.6]{huber2009}.     

\begin{definition}[Robust ICA]\label{def:robustICA} 
A solution $(\sI, \hat{\Phi})$ to the ICA-problem is said to be \emph{robust} if 
\begin{equation}\label{def:robustICA:eq1}
\partial_{\hat{\Phi}}:\mathfrak{C}\rightarrow\R_+ \quad\text{is continuous on } \sI
\end{equation} 
with respect to a sufficiently coarse topology on the causal space $\mathfrak{C}$. \mbox{Here, a topology $\tau_{\mathfrak{C}}$ on $\mathfrak{C}$ is}
\begin{equation}\label{rem:suffcoarse:eq1}
\text{\emph{sufficiently coarse} \, if the $(\tau_{\mathfrak{C}}, \tilde{\pi}_1)$-induced final topology on $\mathcal{M}_1$ is coarser than $\tau^\star$},   
\end{equation}
where $\tilde{\pi}_1 : \mathfrak{C}\rightarrow\mathcal{M}_1$, $(\tilde{S}, \tilde{f})\mapsto \tilde{S}$, is the canonical projection from a cause to its signal and $\tau^\star$ is the topology of weak convergence\footnote{\ See for instance \citep[Definition 8.1.1]{BOG} and \citep[Section 7]{BIL}.} on $\mathcal{M}_1(\mathcal{C}_d)$.
\end{definition}
In addition to the qualitative assertion \eqref{def:robustICA:eq1}, it is often desirable to also have an explicit \emph{quantitative analysis} of this robustness (cf.\ e.g.\ \citep[Sect.\ 1.3f.]{huber2009}). Provided that the underlying topology on $\mathfrak{C}$ is informative enough, this can be achieved by complementing \eqref{def:robustICA:eq1} with a modulus of continuity of $\partial_{\hat{\Phi}}$ at $\sI$, as done in Section \ref{chap:robustICA:sect:auxiliaries} below.    
 
\begin{remark}[`Sufficiently Coarse']\label{rem:coarseness}Intuitively speaking, any topology on $\mathfrak{C}$ for which condition \eqref{def:robustICA:eq1} holds is a model for how `sensitively' the $\lceil\,\cdot\,\rceil_{\tilde{\mathfrak{m}}}$-accuracy of $\Phi$ reacts to violations of the assumption $\sI$: the coarser such a topology [i.e., the fewer open sets it contains], the less sensitively (``the more robustly'') does $\Phi$ vary under general perturbations of $\sI$. Since the continuity requirement \eqref{def:robustICA:eq1} is redundant if understood wrt.\ the discrete, i.e.\ the finest, topology on $\mathfrak{C}$, a certain level of coarseness of the \eqref{def:robustICA:eq1}-supporting topology on $\mathfrak{C}$ is clearly required. However, there is no immediate or canonical choice for such a topology on $\mathfrak{C}$, so setting a threshold for `how coarse is coarse enough' is notoriously a matter of judgement and, as with other uses of the robustness concept \citep{hampel1971, huber2009}, typically depends on the specific (application) context in which the `stability of $\Phi$ under violations of $\sI$' is of interest; a one-size-fits-all topology -- or universal level of coarseness -- to optimally capture all conceivable violations is generally difficult to pinpoint. However, using classical theory as a rough guide, see e.g.\ \citep[Sections 1.3, 2.6]{huber2009}, we can hope to capture `most' relevant violations by the above specification \eqref{rem:suffcoarse:eq1}. We emphasise, however, that \eqref{rem:suffcoarse:eq1} is merely a reference point for orientation and should not be taken as a necessary or strictly binding constraint on \eqref{def:robustICA:eq1}; we do not restrict Definition \ref{def:robustICA} to topologies of this kind only. In fact, in the following sections, \eqref{def:robustICA:eq1} will be established for a topology $\tau_{\mathfrak{C}}$ on $\mathfrak{C}$ which, while satisfying \eqref{rem:suffcoarse:eq1} in the setting of most applications, is generally slightly finer than \eqref{rem:suffcoarse:eq1} but still sufficiently coarse to informatively capture (the `convergence to exactness' of) many IA-violations that are of relevance in theory and practice, cf.\ e.g.\ Example \ref{example:cocktailparty} and Section \ref{sect:applications}. \hfill $\bdiam$     
\end{remark}  

With the operable formulation \eqref{appendix:sect:BSSformal:BSS_reformulation} of the BSS-problem \eqref{BSS-Robust-BlindInversion} -- approached via ICA \eqref{appendix:sect:BSSformal:def:ICA:eq1} -- and its associated notion of robustness \eqref{BSS:robust_prelim1} at hand, the overarching goal for the remainder of this paper is to bring this framework to life and fruition: In Section \ref{chap:robustICA:sect:product_topology}, we will provide a natural, easy-to-handle and informative (premetric) coarse topology on $\mathfrak{C}$ that supports the notion of stability \eqref{BSS:robust_prelim1} for a generically derived ICA-solution $(\sI,\hat{\Phi})$ and allows the resulting BSS-robustness \eqref{def:robustICA:eq1} of this solution to be thoroughly quantified via explicit moduli of continuity, see Theorems \ref{thm:robust_ica} and \ref{thm:robustness}. A number of application-relevant corollaries to this robustness, presented in Section \ref{sect:applications}, illustrate the practical significance of the theory.\\[-0.5em] 

Fundamental to the announced topology on $\mathfrak{C}$, and thus to our above-mentioned robustness program as a whole, is a distinguished set of `coordinates' to describe the signals in $\mathcal{M}_1$. This leads to a structured and explicitly computable quantification of $\mathcal{M}_1$ which, alongside the topologies that this coordinatisation both captures and imposes, is studied in the next section.               

\section{Coordinatisable Signals and Their Convergences}\label{sect:signalsconvergence}
\noindent
We introduce coordinates on $\mathcal{M}_1$ in the form of moment-like statistics which are defined by iteratively integrating the sample paths of a signal against themselves. For this to be well-defined, the realisations of a signal are required to be of a certain smoothness. This is taken into account by certain `regularity classes' within $\mathcal{C}_d$ in which the samples of a coordinatisable signal are assumed to be contained. An expressive and convenient\footnote{\ \ldots though certainly not the largest possible regularity class, see Remark \ref{rem:roughpath-extensions} \ldots} such class is the subspace  
\begin{equation}\label{notation:absolutelycontinuous}
\mathcal{C}^1\equiv\mathcal{C}^1_d \,\coloneqq\, \left\{x : [0,1]\rightarrow\R^d \ \middle|\ x \text{ is absolutely continuous\footnotemark }\right\}
\end{equation}\footnotetext{\ To recall the definition of absolute continuity, see e.g.\ \citep[Definition 1.17]{FVI}.}of $\mathcal{C}_d$, and the signals whose samples are contained in this class are the elements of the subspace             
\begin{equation}\label{notation:absolutelycontinuous2}
\mathfrak{D}\coloneqq\{\mu\in\mathcal{M}_1(\mathcal{C}_d)\mid \mu(\mathcal{C}^1)=1\}
\end{equation} 
of $\mathcal{M}_1$, i.e.\ the set of all Borel probability measures on $(\mathcal{C}_d, \|\cdot\|_\infty)$ supported on $\mathcal{C}^1$. Writing $C^{1,1}(\R^d)\eqqcolon C^{1,1}$ for the group of all $C^1$-diffeomorphisms on $\R^d$, note that the space $\fD$ is left invariant under the $C^{1,1}$-action \eqref{transformation} on $\mathcal{M}_1(\mathcal{C}_d)$ (thanks to the chain rule, cf.\ \eqref{lem:absolutelycontinuous:eq1}).\\[-0.75em]

\noindent
In Section \ref{subsect:coordinatisable} we establish a few important structural properties of the space $\mathcal{C}^1$, which is also seen to contain any signals that are observed in discrete time. We are then ready to introduce the announced signal coordinates on $\mathfrak{D}$ in Section \ref{subsect:sigmoments}, and in Sections \ref{subsect:convergencetypes} and \ref{sect:convergence_sigmoments} will study how sensitively some natural topologies on $\mathcal{M}_1$ are captured by these coordinates.      

\subsection{Properties of the Sample Space}\label{subsect:coordinatisable}The path space $\mathcal{C}^1$ has some topological regularity properties that make it a convenient choice as the sample space of our signals. We note, however, that the main ideas of this paper can readily be extended to signals with much less regular sample paths, see Remark \ref{rem:roughpath-extensions} and the references therein at the end of this subsection. 

The content of the following lemma is well-known. 

\begin{lemma}\label{lem:absolutelycontinuous}
The above space $\mathcal{C}^1$ of absolutely continuous paths in $\R^d$ can be written as
\begin{equation}\label{lem:absolutelycontinuous:eq1}
\mathcal{C}^1 = \left\{x\in\mathcal{C}_d \ \middle|\ \exists\,!\, \dot{x}\in L^1([0,1];\R^d)\,:\, x = x_0 + \int_0^\cdot\!\dot{x}_s\,\mathrm{d}s\right\};
\end{equation}
this space is a Borel subset of $(\mathcal{C}_d, \|\cdot\|_\infty)$ and a separable Banach space with respect to the 1-variation norm
\begin{equation}\label{lem:absolutelycontinuous:eq2}
\|x\|_{1\text{-}\mathrm{var}} \,=\, |x_0| \, + \, \|\dot{x}\|_{L^1}.    
\end{equation}   
\end{lemma} 
\begin{proof}
The identity \eqref{lem:absolutelycontinuous:eq1} holds by \citep[Propositions 1.31$\,\&\,$1.32]{FVI}. In fact, \citep[Prop.\ 1.31]{FVI} asserts that for $\mathfrak{X} \coloneqq \R^d\times L^1([0,1];\R^d)$ and $\mathfrak{Y}\coloneqq\mathcal{C}_d$, the map $f: \mathfrak{X}\rightarrow \mathfrak{Y}$ given by $f(c,v)\coloneqq c + \int_0^\cdot\! v_s\,\mathrm{d}s$ is a Banach space isomorphism (which also proves \eqref{lem:absolutelycontinuous:eq2}). From this, \citep[Theorem 15.1]{kechris1995} implies that the image $f(\mathfrak{X})=\mathcal{C}^1$ is a Borel subset of $(\mathcal{C}^1, \|\cdot\|_\infty)$, i.e.\ that $\mathcal{C}^1\in\mathcal{B}(\mathcal{C}_d)$. That $(\mathcal{C}^1, \|\cdot\|_{^\text{-}\mathrm{var}})$ is separable and Banach is stated as \citep[Corollary 1.35]{FVI}.    
\end{proof} 

\begin{remark}\label{rem:pvariation}
Recall that the norm $\onevar{\,\cdot\,}$ in fact dominates a whole family $(\|\cdot\|_{p\text{-}\mathrm{var}})_{p\in[1,\infty]}$ of $p$-variation norms on $\mathcal{C}^1$ which, in generalisation of \eqref{lem:absolutelycontinuous:eq2} (cf.\ \citep[Section 5.1]{FVI}), are given by 
\begin{equation}\label{sect:pweak-convergence:eq2}
\|x\|_{p\text{-}\mathrm{var}}= |x_0| + \left[\sup_{(t_\nu)}\sum_{\nu}|x_{t_{\nu+1}} - x_{t_\nu}|^p\right]^{\!1/p} \quad \text{ for } \ \ 1\leq p < \infty 
\end{equation}  
and $\|\cdot\|_{\infty\text{-}\mathrm{var}}\coloneqq\|\cdot\|_\infty$, where the above supremum runs over all (finite) dissections $(t_\nu)$ of $[0,1]$.    
Since these norms become weaker as $p$ increases, that is $\|\cdot\|_{p\text{-}\mathrm{var}}\geq \|\cdot\|_{q\text{-}\mathrm{var}}$ for each $1\leq p\leq q\leq\infty$ (e.g.\ \citep[Prop.\ 5.3]{FVI} and \citep[Lem.\ 1.6.4.]{FLO}), we have the inclusions of subspaces
\begin{equation}\label{sect:pweak-convergence:eq3}
C_b\big(\mathcal{C}^1;\|\cdot\|_{1\text{-}\mathrm{var}}\big) \supseteq C_b\big(\mathcal{C}^1;\|\cdot\|_{p\text{-}\mathrm{var}}\big) \supseteq C_b\big(\mathcal{C}^1;\|\cdot\|_{q\text{-}\mathrm{var}}\big) \quad \text{ for }\ \ 1\leq p \leq q \leq\infty. 
\end{equation} 
\end{remark}
 
While the uniform topology on $\mathcal{C}^1$ is weaker than its 1-variation topology, it is useful to note that both induce the same measurable structure on $\mathcal{C}^1$. This follows essentially by definition $\left.\eqref{sect:pweak-convergence:eq2}\right|_{p=1}$ of the $1$-variation norm and the above fact that $\mathcal{C}^1$ is separable wrt.\ $\onevar{\,\cdot\,}$.

\begin{lemma}\label{lem:sigma-algebras}
The spaces $(\mathcal{C}^1, \|\cdot\|_\infty)$ and $(\mathcal{C}^1, \|\cdot\|_{1\text{-}\mathrm{var}})$ have the same Borel $\sigma$-algebra.
\end{lemma}
\begin{proof}
This essentially follows from the identity \eqref{sect:pweak-convergence:eq2} for the $1$-variation norm and the fact that $\mathcal{C}^1$ is separable wrt.\ $\onevar{\,\cdot\,}$ (Lemma \ref{lem:absolutelycontinuous}), see Appendix \ref{pf:lem:sigma-algebras}.
\end{proof} 

Sometimes the observable $\bm{X}=(\bm{X}_t)$ in \eqref{appendix:sect:BSSformal:eq1} is modelled a priori as a random vector or a discrete time series in $\R^d$ rather than as a continuous-time stochastic process. Such discrete-time models are still special cases of our formalism, as per the following identifications.
\begin{remark}[Discrete Signals]\label{rem:discrete-case}
For $Z\coloneqq\R^d$ and $\hat{\mathcal{C}}_{0}\coloneqq\{x\in\mathcal{C}_{(0,1)}\mid x_0=0\}$, consider maps
\begin{equation}\label{rem:discrete-case:eq1}
\hat{\iota}_1 \, : \, \R^d\ni u \,\stackrel{\sim}{\longmapsto}\, (t\cdot u)_{t\in[0,1]}\in\hat{\mathcal{C}}_0 \quad\text{and}\quad \hat{\iota}_n\,:\,Z^{\times n} \stackrel{\sim}{\longrightarrow} \hat{\mathcal{C}}_{\mathcal{E}_n} \quad (n\geq 2),
\end{equation}
where for a dissection $\mathcal{I}=(t_\nu \mid 0=t_0 < \ldots < t_{n-1}=1)$ of $[0,1]$ we set the subspace (of $\mathcal{C}^1$) 
\begin{equation}\label{rem:discrete-case:eq2} 
\hat{\mathcal{C}}_\mathcal{I}\coloneqq\left\{x=(x_t)_{t\in[0,1]} \ \middle| \ x_t= x_{t_\nu} + \tfrac{t - t_\nu}{t_{\nu+1}-t_\nu}(x_{t_{\nu+1}} - x_{t_\nu}) \text{ for } t\in[t_\nu, t_{\nu+1}), \, \nu\in[n-1]_0\right\}
\end{equation} 
(the space of all $\mathcal{I}$-piecewise linear paths from $[0,1]$ to $\R^d$) and denote by $\hat{\iota}_n$ the map that sends a tuple $(z_\nu)\in Z^{\times n}$ to its piecewise-linear interpolation along $\mathcal{E}_n\coloneqq\big((\nu-1)/(n-1)\mid\nu\in[n]\big)$ (the equidistant dissection of $[0,1]$; cf.\ Remark \ref{rem:pwl-interpolation}). The maps in \eqref{rem:discrete-case:eq1} are then identifications between their domain and co-domain spaces as they are (isometric for $n=1$) Banach isomorphisms  with inverses $\hat{\pi}_n\coloneqq\hat{\iota}_n^{-1}$ given by $\hat{\pi}_1\coloneqq\pi_1$ resp.\ $\hat{\pi}_n=\pi_{\mathcal{E}_n}$ for $n\geq 2$.  
For any `discrete signal' $\upsilon\in\mathcal{M}_1(Z^{\times n})$, its continuous-time identification is the pushforward $\hat{\upsilon}\coloneqq(\hat{\iota}_1)_\ast\upsilon\in\widehat{\fD}_1:\cong\{\mu\in\mathcal{M}_1(\mathcal{C}^1)\mid \mu(\hat{C}_0)=1\}$ (if $n=1$) resp.\  
\begin{equation}
\hat{\upsilon} \coloneqq (\hat{\iota}_n)_\ast\upsilon \,\in\, \widehat{\fD}_{\mathcal{E}_n}\coloneqq\!\big\{\mu\in\mathcal{M}_1(\mathcal{C}^1)\ \big|\ \mu(\hat{\mathcal{C}}_{\mathcal{E}_n})=1\big\} \quad (\text{for } n\geq 2). 
\end{equation} 
For the (`timeless') case $n=1$, the signature moments \eqref{def:coordinates:eq1} of $\hat{\upsilon}$ coincide with the classical [multivariate] moments of $\upsilon$ up to the factor $1/m!$. \hfill $\bdiam$   
\end{remark}  

\subsection{Signature Moments}\label{subsect:sigmoments}Let us now define the announced coordinates of a signal $\mu$ in $\fD$.\\[-0.5em]

The key idea behind these coordinates is that each sample $x\in\mathcal{C}^1$ of the signal $\mu$ admits a concise non-local (in-time) description in terms of the iterated-integral statistics     
\begin{equation}\label{subsect:sigmoments:eq1.1}
x=\big(x^1_t, \cdots, x^d_t\big)_{t\in[0,1]} \ \longmapsto\ \int_{0\leq t_1\leq\cdots\leq t_m\leq 1}\!\mathrm{d}x^{i_1}_{t_1}\cdots\mathrm{d}x ^{i_m}_{t_m}\,\eqqcolon\,\sig_{i_1\cdots i_m}\!(x) 
\end{equation}
for $i_1,\ldots, i_m\in[d]$, $m\in\N$, and $\sig_\emptyset(x)\coloneqq 1$; since every $x\in\mathcal{C}^1$ is of bounded variation, the integrals \eqref{subsect:sigmoments:eq1.1} exist in the sense of Lebesgue-Stieltjes. These numbers are the so-called \emph{signature coefficients} of $x$, and they combine to a multi-indexed family $\sig(x)\coloneqq\big(\sig_{w}(x)\mid w\in[d]^\star\big)$ which is known as the \emph{signature} of $x$. (Notice the formal similarity of $\sig_{i_1\cdots i_m}$ with the moment map $x \mapsto x_{0,1}^{i_1}\cdots x_{0,1}^{i_m}$, \mbox{motivating us to call \eqref{subsect:sigmoments:eq1.1} the `noncommutative moments' of $x$.)}\\[-0.5em] 

It has been shown in \citep{HLY} that the signature in fact characterises\footnote{\label{footnote:treelikeequivalence}\ \ldots up to an inessential (``equality up to a null-set''-like) equivalence relation on $\mathcal{C}^1$ \ldots} its path, and it does so in quite an informative way, see e.g.\ \cite{bonnier2020adapted, CHL, chevyrev2022signature, chevyrev2018persistence, crisan2013robust, KOB, perez2018signature} among many others. This justifies viewing $\sig(\,\cdot\,)$ as a (nonlinear) `coordinate system' on $\mathcal{C}^1$, where each `coordinate vector' $\sig(x)$ determines the position of a point $x\in\mathcal{C}^1$ uniquely\textsuperscript{\ref{footnote:treelikeequivalence}}. This signature-based description \eqref{subsect:sigmoments:eq1.1} of a path can then be naturally transferred from $\mathcal{C}^1$ to the signals in $\fD$ by duality (testing the measures in $\fD$ against the ordered family of nonlinear functionals \eqref{subsect:sigmoments:eq1.1}).\\[-0.5em] 

\hfill The resulting `dual coordinates' on $\fD$ are introduced in the definition below. 
 
\begin{lemma}\label{cor:signature-measurable}
The map $(\mathcal{C}^1, \|\cdot\|_\infty)\ni x\mapsto \sig_w(x)\in\R$ is Borel-measurable for each $w\in[d]^\star$. 
\end{lemma}
\begin{proof}
By Lem.\ \ref{lem:sigma-algebras} and the fact (e.g.\ \citep[Thm.\ 3.1.3]{LYQ}) that $\sig_w$ is continuous wrt.\ $\onevar{\,\cdot\,}$.  
\end{proof}              
\begin{definition}[Signature Moments]\label{def:coredinates}
Given a signal $\mu$ in $\mathfrak{D}$, the numbers (in $\bar{\R}$)
\begin{equation}\label{def:coordinates:eq1}
\langle\mu\rangle_{i_1\ldots i_m} \,\coloneqq\, \int_{\mathcal{C}_d}\!\mathfrak{sig}_{i_1\ldots i_m}(x)\,\mu(\mathrm{d}x) \qquad (i_1, \ldots, i_m\in[d], \ m\in\N)
\end{equation}are called the \emph{signature moments of $\mu$}. We refer to the \emph{coredinates of $\mu$} as the matrices 
\begin{equation}\label{def:coordinates:eq2}
[\mu]_{0} \coloneqq \big(\langle\mu\rangle_{ij}\big)_{\!i,j\in[d]} \quad\text{ and }\quad [\mu]_{\nu}\coloneqq \big(\langle\mu\rangle_{ij\nu}\big)_{\!i,j\in[d]} \quad (\nu=1,\ldots, d),    
\end{equation} 
and denote by $[\mu]_{\mathfrak{c}}\coloneqq([\mu]_0, [\mu]_1, \cdots, [\mu]_d)$ their ordered collection. We say that $[\mu]_{\mathfrak{c}}$ \emph{exists}, in symbols: $[\mu]_{\mathfrak{c}}<\infty$, if all the signature moments in \eqref{def:coordinates:eq2} exist in $\R$.   
\end{definition}
\hfill (The word `coredinates' is a portmanteau of `core coordinates'.)\\[-0.5em]     

\noindent
Due to the above-mentioned characteristicness of the \eqref{subsect:sigmoments:eq1.1}-based coordinatisation of $\mathcal{C}^1$, the signature moments \eqref{def:coordinates:eq1}, if contained in $\R$, can characterise a signal uniquely, see \citep{CHO}. Moreover, as a result of algebraic interrelations within the family $(\sig_w\,|\,w\in[d]^\star)$, the moments \eqref{def:coordinates:eq1} also efficiently encode the pushforward-action on $\fD$ (cf.\ \eqref{transformation}) induced by any path-transformation $f : \mathcal{C}^1 \rightarrow \mathcal{C}^1$ (e.g.\ \citep[Thm.\ 5.1]{chevyrev2018persistence}). In fact, the coredinates \eqref{def:coordinates:eq2} already capture the \emph{linear} action  
\begin{equation}\label{subsect:sigmoments:linact}
\R^{d\times d}\ni A\,\longmapsto\, A\cdot\mu\equiv A_\ast\mu\coloneqq\mu\circ A^{-1}
\end{equation} 
sufficiently: Flattening the third-order tensors $\big(\langle\mu\rangle_{ijk}\big)$ to the matrices $[\mu]_\nu$ in \eqref{def:coordinates:eq2} yields a family of affine equivariant statistics (see below); this equivariance is a simple yet extremely useful algebraic property of $[\mu]_{\mathfrak{c}}$, and when combined with mild assumptions on $\mu$ can be used to study the action \eqref{subsect:sigmoments:linact} for a wide class of (non-IC) signals and under nonlinear perturbations. 

\begin{lemma}[Affine Equivariance]\label{lem:linequiv}
Let $\mu\in\fD$ and define $f\equiv f(u)\coloneqq A\cdot u + b$ on $\R^d$ for some $A=(a_{ij})\in\R^{d\times d}$ and $b=(b_i)\in\R^d$. Then for any $i_1, \ldots, i_m\in[d]$ and $m\in\N$, 
\begin{equation}\label{lem:linequiv:eq1}
\big\langle f_\ast\mu\big\rangle_{i_1\ldots i_m} \, = \, \sum_{j_1, \ldots, j_m=1}^d a_{i_1 j_1}\!\cdots\, a_{i_m j_m}\langle\mu\rangle_{j_1\ldots j_m}.  
\end{equation} 
\end{lemma} 
\begin{proof} 
Clear by the multilinearity of iterated Stieltjes integration, but see Appendix \ref{pf:lem:linequiv}.  
\end{proof}
The signature statistics \eqref{def:coordinates:eq2} share their equivariance property \eqref{lem:linequiv:eq1} with many classical signal statistics of comparable order, such as covariance matrices or classical multivariate moments or cumulants, which is what underlies the latter's utility in many classical ICA approaches \cite{HBS, HKO, MNT}. Hence, and if preferred, one may simply use these classical statistics in lieu of \eqref{def:coordinates:eq1}, in which case the topological robustness approach developed in this paper remains equally valid, mutatis mutandis, upon replacement of \eqref{def:coordinates:eq2} with any of their (affine equivariant) classical counterparts; cf.\ also Rem.\ \ref{rem:discrete-case}. This leads us to the following remark.  
\begin{remark}[Why Signatures?]In stark contrast to the classical (`commutative') moment-based statistics for a stochastic process, the coredinates \eqref{def:coordinates:eq2} of a signal $\mu$ are `global' [in time] statistics of that signal, i.e.\ functions that exhaust (the law of) the \emph{whole} signal and not (as for classical statistics) just the finite-dimensional distributions of the signal probed at a small number of preselected time-points. This globality of the signal statistics \eqref{def:coordinates:eq1} allows them to capture non-local statistical effects (such as time-dependent noise or estimation errors due to asynchronously sampled channel data, cf.\ Section \ref{sect:applications}) much more naturally and efficiently than with classical (`fixed-$t$') moment statistics \cite{bonnier2020adapted, bonnier2019signature, CHL, CHO, chevyrev2018persistence, PVA}. Moreover, the derived statistics \eqref{def:coordinates:eq2} do in general\footnote{\ An exception to this are the coredinate matrices of mean-stationary product signals and (due to \eqref{prop1:eq2}) their linear transformations, which do coincide with classical moment statistics; see Lemma \ref{lemma1}.} also carry more information per matrix than classical statistics of the same dimension \cite{bonnier2019signature}, which is reflected in the fact that the matrices \eqref{def:coordinates:eq2} are generally asymmetric and thus contain about twice as much information as same-dimensional ICA statistics obtained from commutative moments, which are typically all symmetric. It is these aspects in particular, which will be developed and leveraged in the following, that make the signal coordinates \eqref{def:coordinates:eq1} and their derived statistics \eqref{def:coordinates:eq2} a very natural and informative quantification of statistical dependencies within and between multidimensional signals and their (non)linear transformations. Owing to these qualities, we will see that these coordinates provide an excellent basis for constructing the desired robustness topology for \eqref{def:robustICA:eq1}. \hfill $\bdiam$
\end{remark}

\subsection{Convergence of Signals}\label{subsect:convergencetypes}We formalise how different signals can be topologically related to each other by introducing a natural notion of $p$-variation-graded weak convergence on $\fD$. The resulting signal topology is at most gradually stronger than the topology of classical weak-convergence on $\fD$ and is in fact equal to the latter in the context of many applications.\\[-0.5em]

\noindent
To begin, let us recall\footnote{\ See for instance \citep[Definition 8.1.1]{BOG} and \citep[Section 7]{BIL}.} that a net of signals $(\mu_\alpha)$ in $\fD$ is said to converge \emph{weakly} to $\mu\in\fD$ if 
\begin{equation}\label{sect:convergence:classical}
\lim_{\alpha}\mu_\alpha(\tilde{f}) = \mu(\tilde{f}) \ \text{ for each }\tilde{f}\in C_{\mathrm{b}}(\mathcal{C}^1; \|\cdot\|_\infty), \quad\text{where}\quad \mu_\alpha(\tilde{f})\coloneqq\int_{\mathcal{C}^1}\tilde{f}\,\mathrm{d}\mu_\alpha
\end{equation}
and $C_b\big(\mathcal{C}^1;\|\cdot\|_\infty\big)$ denotes the space of bounded $\|\cdot\|_\infty$-continuous functions $\tilde{f} : \mathcal{C}^1\rightarrow\R$.\\[-0.5em] 

As can be seen from its definition above, the notion of weak convergence is an inherently topological concept: While different topologies on $\mathcal{C}^1$ may induce the same Borel $\sigma$-algebra, the set of (bounded) continuous functions they provide, and hence the `modes of weak convergence' that they induce on $\fD$, may well differ.\\[-1em]

\noindent
Lemma \ref{lem:sigma-algebras} pertains to this, as it implies that the signal spaces $\mathcal{M}_1(\mathcal{C}^1)\equiv\mathcal{M}_1\big((\mathcal{C}^1, \|\cdot\|_\infty)\big)\cong\fD$ and $\mathcal{M}_1\big((\mathcal{C}^1, \onevar{\,\cdot\,})\big)$ coincide. This allows us to integrate each measure $\mu$ in $\mathcal{M}_1\big(\mathcal{C}^1\big)$ against any element of the space $C_b\big(\mathcal{C}^1;\onevar{\,\cdot\,}\big)$ of bounded $\onevar{\,\cdot\,}$-continuous functions $\tilde{f} : \mathcal{C}^1\rightarrow\R$.\\[-0.5em]

\noindent
As before, cf.\ \citep[Def.\ 8.1.1]{BOG}, this gives rise to a notion of pointwise (``weak'') convergence 
\begin{equation}\label{sect:pweak-convergence:eq1}
\mu_\alpha\xrightharpoonup{1\text{-}\mathrm{var}}\mu \quad :\,\Longleftrightarrow \quad \left[\lim_{\alpha}\mu_\alpha(\tilde{f}) = \mu(\tilde{f}), \ \ \forall\, \tilde{f}\in C_b\big(\mathcal{C}^1;\onevar{\,\cdot\,}\big)\right] 
\end{equation}   
for a net $(\mu_\alpha)\subset\mathcal{M}_1(\mathcal{C}^1)$ and $\mu\in\mathcal{M}_1(\mathcal{C}^1)$. Thanks to the inclusions \eqref{sect:pweak-convergence:eq3}, this convergence can be systematically generalised to the following weaker notions of convergence.     

\begin{definition}[$p$-Weak Topology]\label{def:weakp-convergence} 
Let $p\in[1,\infty)$. A net $(\mu_\alpha)\subset\fD$ will be called \emph{$p$-weakly convergent} to a signal $\mu\in\fD$, in symbols: $\mu_\alpha\xrightharpoonup{p\text{-}\mathrm{var}}\mu$, if 
\begin{equation}\label{sect:pweak-convergence:eq4}
\lim_{\alpha}\mu_\alpha(\tilde{f}) = \mu(\tilde{f}) \quad \text{ for each } \ \tilde{f}\in C_b\big(\mathcal{C}^1;\|\cdot\|_{p\text{-}\mathrm{var}}\big).\vspace{-0.5em} 
\end{equation} 
The convergence of nets \eqref{sect:pweak-convergence:eq4} canonically\footnote{\ See for instance the section on convergence classes in \citep[Chapter 2 (pp.\ 73)]{kelley1961}. See also from there, cf.\ \citep[Theorem 2.9]{kelley1961}, that the notion of converging nets (as defined by the \eqref{sect:pweak-convergence:eq1}-induced $p$-weak topology) is exactly the same as the notion of convergence described by \eqref{sect:pweak-convergence:eq1}.} defines a unique topology on $\fD$ which we call the \emph{$p$-weak topology}, and any $q$-weak topology for $1\leq q < 2$ will be called a \emph{weak'-topology}. 
\end{definition} 
In other words, \eqref{sect:pweak-convergence:eq4} is precisely the weak convergence of $(\mu_\alpha)$ to $\mu$ wrt.\ the $p$-variation topology on $\mathcal{C}^1$ (cf.\ \citep[Section 8.1]{BOG}). The $p$-indexed convergences \eqref{sect:pweak-convergence:eq4} stand in natural hierarchy to each other: For any $1\leq p \leq q < \infty$, the $p$-weak topology is stronger (due to \eqref{sect:pweak-convergence:eq3}) than the $q$-weak topology which, in turn, is stronger than the classical (i.e.\ \eqref{sect:convergence:classical}-induced) weak topology on $\mathcal{M}_1(\mathcal{C}^1)$ but only gradually so, as per the following remark. 

\begin{remark}\label{rem:pweak-convergence}The norms $\|\cdot\|_{\mathrm{p}\text{-}\mathrm{var}}$ $(1 < p < \infty)$ and $\|\cdot\|_\infty$ can be related via a classical interpolation inequality (e.g.\ \citep[Proposition 5.5]{FVI}) which, for every $x\in\mathcal{C}_d$, implies that   
\begin{equation}\label{rem:pweak-convergence:eq1}
\|x\|_\infty \,\leq\, \|x\|_{p\text{-}\mathrm{var}} \,\leq\, 2\|x\|^{\frac{1}{p}}_{1\text{-}\mathrm{var}}\!\cdot\|x\|_\infty^{1 - \frac{1}{p}}  \quad \text{ for each } p\in(1,\infty)
\end{equation}
and hence that the norms $\|\cdot\|_{p\text{-}\mathrm{var}}$ and $\|\cdot\|_\infty$ are asymptotically equivalent in the limit $p\rightarrow\infty$. In particular, the convergence-defining sets of test functions $C_b\big(\mathcal{C}^1;\|\cdot\|_{p\text{-}\mathrm{var}}\big)$ and $C_b\big(\mathcal{C}^1;\|\cdot\|_\infty\big)$ are `asymptotically coincidental' as $p\rightarrow\infty$, and in this sense we may regard the convergence topologies defined by \eqref{sect:pweak-convergence:eq4} and \eqref{sect:convergence:classical}, respectively, as `essentially equivalent' for $p$ large. \hfill $\bdiam$    
\end{remark} 
The a priori gap (in strength) between the $p$-weak topologies and the classical [i.e.\ \eqref{sect:convergence:classical}-induced] weak topology on $\fD$ is mainly a theoretical one with no relevance for many practical applications. This is because for many practically relevant signal spaces the $p$-weak topologies and the classical weak topology coincide, as the following example shows.    

\begin{proposition}\label{prop:boundeddomains}
For any $p\in[1,\infty)$ and $K>0$, consider the `compactified' signal spaces
\begin{equation}\label{prop:boundeddomains:eq1}
\fD^p_K\coloneqq\big\{\mu\in\mathcal{M}_1(\mathcal{C}^1)\ \big| \ \mu(\mathcal{C}_{p,K}^1)=1\big\} \quad\text{for}\quad \mathcal{C}_{p,K}^1\coloneqq\big\{x\in\mathcal{C}^1 \ \big| \ \|x\|_{p\text{-}\mathrm{var}}\leq K\big\}, 
\end{equation}
and for any dissection $\mathcal{I}$ of $[0,1]$ consider the signals
\begin{equation}\label{prop:boundeddomains:eq2}
\widehat{\fD}_\mathcal{I}\coloneqq\big\{\mu\in\mathcal{M}_1(\mathcal{C}^1)\ \big| \ \mu(\hat{\mathcal{C}}_\mathcal{I})=1\big\} \ \text{ supported on } \ \hat{\mathcal{C}}_\mathcal{I},  
\end{equation} 
the space of $\mathcal{I}$-piecewise linear paths from \eqref{rem:discrete-case:eq2}. On these spaces, the classical weak topology is equivalent to: the $q$-weak topology on $\fD^p_K$ for any $q>p$, and to the $1$-weak topology on $\widehat{\fD}_\mathcal{I}$. 
\end{proposition} 
\begin{proof}
This essentially follows from a generalisation of \eqref{rem:pweak-convergence:eq1}, see Appendix \ref{pf:prop:boundeddomains}.
\end{proof}
The above signal spaces \eqref{prop:boundeddomains:eq1} and \eqref{prop:boundeddomains:eq2} owe their practical relevance to the fact that real-world signal processing systems are always subject to capacity constraints such as limited data storage and finite resolution of the recorded observables, e.g.\ \citep[Sects.\ 5, 12]{damelin2012signals}. The former type of constraint leads to \eqref{prop:boundeddomains:eq1}, while the latter type of constraint mandates (finite horizon) time-discretizations of the signals as represented by \eqref{prop:boundeddomains:eq2}, cf.\ for instance \cite{olevskii2020}.                  

\subsection{Convergence of Signature Moments}\label{sect:convergence_sigmoments}This section shows that under a natural growth condition, the signature coordinates of a signal are continuous wrt.\ the above signal topologies.\\[-0.5em]

In order to combine the considerations of Sections \ref{subsect:coordinatisable} and \ref{subsect:convergencetypes} into an informative robustness topology for \eqref{def:robustICA:eq1}, we will need to link the convergence \eqref{sect:pweak-convergence:eq4} of signals to the convergence of their moments \eqref{def:coordinates:eq1}. As usual for integral statistics, this is done by imposing a growth condition.     
\begin{definition}\label{def:unifsigintegr}
A subset $\mathfrak{U}$ of $\fD$ is called \emph{uniformly signature integrable of order $m$} if 
\begin{equation}\label{def:unifsigintegr:eq1}
\inf_{a>0}\sup_{\mu\in\mathfrak{U}}\int_{\{|\sig_w|>a\}}\!|\sig_w|\,\mathrm{d}\mu \, = \, 0  \quad \text{ for each } w\in[d]^\star_{\leq m}. 
\end{equation} 
We call $\mathfrak{U}$ \emph{uniformly signature integrable} if $\mathfrak{U}$ is uniformly signature integrable of order 3. 
\end{definition} 
Put plainly, a set of signals is uniformly signature integrable (of order $m$) if its elements attain with uniformly lower probability those sample paths whose signature coefficients (up to order $m$) are very large. The growth assumption \eqref{def:unifsigintegr:eq1} ensures that on such spaces, the moment coordinates \eqref{def:coordinates:eq1} are continuous with respect to any weak'-topology $\left.\eqref{sect:pweak-convergence:eq4}\right|_{p\in[1,2)}$. 
\begin{lemma}\label{lem:unifsigintegr}
Let $\mathfrak{U}$ be uniformly signature integrable of order $m$, and $(\mu_\alpha)\subset\mathfrak{U}$ be any net such that $\mu_\alpha\rightarrow\mu\in\mathfrak{U}$ in a weak'-topology. Then $\langle\mu_\alpha\rangle_w\rightarrow\langle\mu\rangle_w$ for each $w\in[d]_{\leq m}^\star$.   
\end{lemma}
\begin{proof}
Since $\langle\tilde{\mu}\rangle_w = \int_{\mathcal{C}^1}\sig_w\,\mathrm{d}\tilde{\mu}$ and because $\sig_w$ is continuous wrt.\ the $p$-variation topology on $\mathcal{C}^1$ for any $p\in[1,2)$, e.g.\ \citep[Thm.\ 3.1.3]{LYQ}, the assertion follows from \citep[Lemma 8.4.3]{BOG}.   
\end{proof}
As an immediate\footnote{\label{footnote:signature-bound}\ Upon recalling that $|\sig_w(x)|\lesssim\|x\|_{p\text{-}\mathrm{var}}^{p|w|}$ for any $x\in\mathcal{C}^1$ and $p\in[1,2)$, e.g.\ \citep[Thm.\ 3.7 $\&$ Sect.\ 1.2.2]{FLO}; notice that said inequality also (and for all $\mu\in\mathcal{M}_1$) implies that: $\int_{\mathcal{C}_d}\!\|x\|^3_{p\text{-}\mathrm{var}}\,\mu(\mathrm{d}x)<\infty$ \ $\Longrightarrow$ \ $[\mu]_{\mathfrak{c}} < \infty$.} example, uniformly signature integrable spaces (of any order) include all subsets of $\fD$ whose elements have uniformly $\|\cdot\|_{p\text{-}\mathrm{var}}$-bounded support, cf.\ Proposition \ref{prop:boundeddomains}. Drawing on the classical notion of uniform integrability, these bounded-support examples of uniformly signature integrable spaces can be easily generalised as follows.  

\begin{proposition}\label{lem:unifsigintegr:sufficient}
Let $m\in\N$, and let $\mathfrak{U}$ be a subset of $\fD$ for which there is $q>1$ such that
\begin{equation}\label{lem:unifsigintegr:sufficient:eq1}
\sup_{\mu\in\mathfrak{U}}\int_{\mathcal{C}_d}\!\|x\|_{1\text{-}\mathrm{var}}^{mq}\,\mu(\mathrm{d}x) \, < \, \infty. 
\end{equation} 
Then $\mathfrak{U}$ is uniformly signature integrable of order $m$.  
\end{proposition}  
\begin{proof}
This follows by combination of a few standard inequalities, see Appendix \ref{pf:lem:unifsigintegr:sufficient}. 
\end{proof}              
\subsubsection{Integrability}\label{subsubsect:moment-integrability}An analysis of the robustness \eqref{def:robustICA:eq1} requires to capture how nonlinear transformations of $\R^d$ affect nearby source signals [via \eqref{transformation}]. For this, and only for this (see Rems.\ \ref{rem:partialcontinuity}, \ref{rem:toproof:thm:robustness} \ref{rem:toproof:thm:robustness:it2}), it will be opportune to sharpen the notion of locality in $\fD$ by enforcing the comparability of (the coredinates of) signals via uniform integrability \eqref{lem:unifsigintegr:sufficient:eq1} of their strong moments:\\[-1.5em] 

\begin{notation*}A set $\mathfrak{U}$ in $\mathfrak{D}$ is called \emph{core integrable} if it is uniformly signature integrable and
\begin{equation}\label{notation:unifsigintegr:eq1}
K_{\mathfrak{U}|\beta}\coloneqq\sup_{\mu\in\mathfrak{U}}\int_{\mathcal{C}_d}\!\|x\|_{1\text{-}\mathrm{var}}^{\beta}\,\mu(\mathrm{d}x)\,<\,\infty \quad\text{ for some } \ \beta\in(5/2,3).
\end{equation} 
For any fixed $\beta$ as in \eqref{notation:unifsigintegr:eq1}, denote $p_\beta\coloneqq(\beta-2)^{-1}$ and $\alpha=\alpha_\beta\coloneqq 1 - 1/p_\beta$ $[\in(0,\tfrac{1}{2})]$.\\[-0.75em]

\noindent
We say that a subset $\mathfrak{U}\subseteq\fD$ is \emph{core-integrable as per $(K_{\mathfrak{U}|\beta}, \beta, p, \alpha)$} if \eqref{notation:unifsigintegr:eq1} holds for some fixed $\beta\in(5/2,3)$ and associated constants $p=p_\beta$ and $\alpha=\alpha_\beta$ as defined above, and for any such tuple we further define $c_p\coloneqq(1-2^{1-2/p})^{-1}$ and $C_p\coloneqq 2^{\alpha+1} c_p(1 + K_{\mathfrak{U}|\beta})$. \hfill $\bdiam$
\end{notation*} 

As an immediate example, a subspace $\mathfrak{U}\subseteq\fD$ is core integrable if it satisfies \eqref{lem:unifsigintegr:sufficient:eq1} for $m=3$.

\noindent
Remark \ref{rem:unifsigintegr} provides a bit of context on the condition \eqref{lem:unifsigintegr:sufficient:eq1} and how it relates to the usual integrability conditions for classical moments of random vectors in $\R^d$.\\[-0.5em] 

\noindent
Our approach towards ICA-robustness (\eqref{BSS:robust_prelim1}, \eqref{def:robustICA:eq1}) will make use of the coredinates \eqref{def:coordinates:eq2} of a signal, which is why for the following we will focus on the `coredinatisable' space
\begin{equation}\label{sect:convergence_sigmoments:subspace}
\dot{\fD}\,\coloneqq\, \left\{\mu\in\fD \ \middle| \ [\mu]_{\mathfrak{c}} < \infty, \, \langle\mu\rangle_{ii}\neq 0 \ \, \text{for} \ i=1,\ldots,d\right\}. 
\end{equation}    
Note that this is a very large subspace of $\fD$: The required existence of a signal's coredinates includes (e.g.) the full $3$-Wasserstein space on $(\mathcal{C}^1, \|\cdot\|_{1\text{-}\mathrm{var}})$ (Footnote \ref{footnote:signature-bound}; cf.\ \citep[Sect.\ 2.1]{panaretos2020}), while the condition $\langle\mu\rangle_{ii}\neq 0$ of non-vanishing \nth{2}-order diagonal moments is merely a mild non-degeneracy assumption that can be imposed with essentially\footnote{\ Indeed: Since for each $\mu=(\mu^1, \ldots, \mu^d)\in\fD$ we have $\langle\mu\rangle_{ii}=\tfrac{1}{2}\E[(\mu_{0,1}^i)^2]$, the condition $\langle\mu\rangle_{ii}\neq 0$ asks for the increments $\mu_{0,1}^i\coloneqq\mu_1^i - \mu_0^i$ to have non-zero variance (see Sect.\ \ref{sect:operationsonmeasures} for notation). Practically this is no restriction since one can simply cut-off any (`non-degenerate', cf.\ Sect.\ \ref{sect:signals}) signal in $\fD$ to a (regular) interval over whose end-points its increment has positive variance, and then rescale this cut-off signal back to $[0,1]$.} no loss of generality.\\[-0.5em]          

We are now in a position to endow the signal space $\dot{\fD}$ with an explicit and well-structured topology which will help us to arrive at an insightful quantification of the robustness \eqref{def:robustICA:eq1}.       

\section{An ICA-Tailored Topology on Causal Space}\label{chap:robustICA:sect:product_topology}
\noindent
We describe an informative and explicitly computable premetric topology on causal space $\mathfrak{C}$ that relates to the coarseness requirement \eqref{rem:suffcoarse:eq1} while still being strong enough to support, in a conveniently quantifiable manner, the robustness \eqref{def:robustICA:eq1} for a generic ICA-solution $(\sI, \hat{\Phi})$, such as the one derived in Section \ref{chap:robustICA:sect:auxiliaries}. As the first of three main steps to this end, Section \ref{chap:robustICA:sect:premetric} shows how the signature moments from Definition \ref{def:coredinates} allow for a systematic, divergence-like quantification of the (statistical) `distance' between two signals in $\dot{\fD}$. The information thus encoded amounts to a topological landscape of statistical dependence within and between signals which is algebraically and analytically flexible and sensitive to both spatial (non)linear transformations as well as to intrinsic statistical variations, including those explored in Section \ref{subsect:convergencetypes}. This topological dependence profile is then linked to the ICA-specific identifiability structure \eqref{appendix:sect:BSSformal:def:ICA:eq1} in Section \ref{sect:orthogonal_signals}, where accurately identifiable source signals are charted as (belonging to) the set of global minima of this landscape. Finally, Section \ref{subsect:topology_identspace} makes good on the final announcement of Section \ref{sect:robustica-basic} by extending our premetric signal topology to a coarse topology on the whole causal space $\mathfrak{C}$, for which robustness \eqref{def:robustICA:eq1} will then be established in Section \ref{chap:robustICA:sect:auxiliaries}.       

\subsection{A Premetric Topology on Signal Space}\label{chap:robustICA:sect:premetric}From the countable infinitude of coordinates provided by the hierarchically graded signature moments \eqref{def:coordinates:eq1}, the lower-order arrangement \eqref{def:coordinates:eq2} captures those aspects of a signal which are most expressive of any linear action \eqref{subsect:sigmoments:linact} performed upon it. In addition, and still similar to covariance matrices or cumulants for random vectors, the coredinate matrices \eqref{def:coordinates:eq2} describe the (second- and third-order) statistical dependence between the components of a signal in terms of how `pronounced' their off-diagonal structure is relative to their diagonal, cf.\ Lemma \ref{lemma1}. This, combined with the inversion theory of Section \ref{chap:robustICA:subsect:signature_identifiability:algorithm}, will allow for an explicit topological coordinatisation of ICA-identifiable causes (Fig.\ \ref{fig:BSStriple} a.) and the causes close to them (Fig.\ \ref{fig:BSStriple} b.). Let $\dot{\fD}$ be the signal space from \eqref{sect:convergence_sigmoments:subspace}.\\[-0.5em] 

We propose to implement this charting by way of the following distance function. Recall that $\mu\equiv(\mu_t)\in\mathcal{M}_1$ is called mean-stationary if $\E\mu_t = \E\mu_0$ for each $t\in[0,1]$ (Sect.\ \ref{sect:operationsonmeasures}).       
 
\begin{definition}[ICA-Premetric]\label{def:delta}
For any two signals $\mu, \tilde{\mu}\in\dot{\fD}$ with coredinates \eqref{def:coordinates:eq2}, define 
\begin{gather}\label{def:delta:eq1}
\delta(\mu,\tilde{\mu})\,\coloneqq\,\sqrt{\alpha_{\mu,\tilde{\mu}} + \beta_{\mu,\tilde{\mu}}} \qquad \text{ for }\\
\alpha_{\mu,\tilde{\mu}} \coloneqq \sum_{i,j=1}^d\left(\frac{\langle\tilde{\mu}\rangle_{ij} - \langle\mu\rangle_{ij}}{\sqrt{\langle\mu\rangle_{ii}\langle\mu\rangle_{jj}}}\right)^{\!\!2} \quad \text{ and } \quad 
\beta_{\mu,\tilde{\mu}} \coloneqq \sum_{\nu=1}^d\sum_{i,j=1}^d\left(\frac{\langle\tilde{\mu}\rangle_{ij\nu} - \langle\mu\rangle_{ij\nu}}{\sqrt{\langle\mu\rangle_{ii}\langle\mu\rangle_{jj}\langle\mu\rangle_{\nu\nu}}}\right)^{\!\!2}. 
\end{gather} 
For a mean-stationary $\mu=(\mu^1, \ldots, \mu^d)\in\dot{\fD}$ we refer to the \emph{IC-defect} of $\mu$ as the number  
\begin{equation}\label{def:delta:eq2}
\delta_{\independent}(\mu) \, \coloneqq \, \delta(\mu^1\otimes\cdots\otimes\mu^d,\mu)\,;
\end{equation}
for a non mean-stationary signal, this defect is defined\footnote{\ Lemma \ref{lemma1} guarantees that the two definitions \eqref{def:delta:eq2} and  \eqref{rem:icdefect-nonstationary:eq1} are consistent (see Remark \ref{rem:icdefect-nonstationary}).} by equation \eqref{rem:icdefect-nonstationary:eq1} below.
\end{definition} 

For signals $\mu, \tilde{\mu}\in\dot{\fD}$ with coredinates $[\mu]_{\mathfrak{c}}, [\tilde{\mu}]_{\mathfrak{c}}$ as in \eqref{def:coordinates:eq2} and for the standardisation matrix $N_\mu\coloneqq\mathrm{ddiag}\big(\langle\mu\rangle_{11}, \cdots, \langle\mu\rangle_{dd}\big)^{1/2}$, notice that \eqref{def:delta:eq1} can be more compactly written as 
\begin{equation}\label{sect:coredinates:eq4}
\delta(\mu,\tilde{\mu}) \, = \, \left\|N_\mu^{-1}\cdot\!\left([\tilde{\mu}]_0-[\mu]_0, \frac{[\tilde{\mu}]_1-[\mu]_1}{\sqrt{\langle\mu\rangle_{11}}}, \ldots, \frac{[\tilde{\mu}]_d-[\mu]_d}{\sqrt{\langle\mu\rangle_{dd}}}\right)\!\cdot N_\mu^{-1}\right\|_{\mathcal{V}}  
\end{equation} 
where $\big\|(C_0,\cdots,C_d)\big\|_{\mathcal{V}}\coloneqq\sqrt{\sum_{\nu=0}^d{\|C_\nu\|^2}}$ is the $\ell_2$-norm on the space $\mathcal{V}\coloneqq \big(\R^{d\times d}\big)^{\oplus\,(d+1)}$ derived from the Frobenius norm $\|\cdot\|$ on $\R^{d\times d}$.\\[-0.5em]

\noindent
The map $(\mu, \tilde{\mu})\mapsto\delta(\mu,\tilde{\mu})$ defines a \emph{premetric} on $\dot{\fD}$ (cf.\ \citep[Def.\ 2.4.4]{arkhangelski1990}), which is to say that 
\begin{equation}\label{subsect:premetric:defeq}
\delta(\mu,\tilde{\mu})\geq 0 \quad\text{ and }\quad \delta(\mu,\mu)=0 
\end{equation}    
for any two $\mu, \tilde{\mu}\in\dot{\fD}$. This premetric canonically (e.g.\ \citep[Section 2]{bruno2017}) induces a topology $\tau_\delta$ on the signal space $\dot{\fD}$, of which we note the following basic facts.  
\begin{lemma}\label{lem:premetric_topofacts}
For the premetric topology $\tau_\delta$ on $\dot{\fD}$ induced by \eqref{def:delta:eq1}, the following holds.  
\begin{enumerate}[font=\upshape, label=(\roman*)]
\item\label{lem:premetric_topofacts:it1} For each subset $\mathcal{O}$ in $\dot{\fD}$, we have $\mathcal{O}\in\tau_\delta$ iff for each $\mu\in\mathcal{O}$ there is an $r>0$ such that  
\begin{equation}\label{lem:premetric_topofacts:eq1}
B_\delta(\mu,r)\coloneqq\big\{\tilde{\mu}\in\dot{\fD} \ \big| \ \delta(\mu,\tilde{\mu}) < r\big\},
\end{equation}
the $\mu$-centred $\delta$-ball of radius $r$, is contained in $\mathcal{O}$.  
\item\label{lem:premetric_topofacts:it2} The topological space $(\dot{\fD}, \tau_\delta)$ is sequential. 
\item\label{lem:premetric_topofacts:it3} For any metric space $(W, d_W)$, a map $\phi : (\dot{\fD}, \tau_\delta) \rightarrow W$ is continuous at $\mu\in\dot{\fD}$ iff: 
\begin{equation}\label{lem:premetric_topofacts:it3:eq1}
\forall\,\tilde{\varepsilon}>0 \,:\, \exists\,\tilde{\delta}>0\,:\,\forall\,\tilde{\mu}\in\dot{\fD}\,:\ \delta(\mu,\tilde{\mu})<\tilde{\delta} \ \Longrightarrow \ d_W(\phi(\mu),\phi(\tilde{\mu}))<\tilde{\varepsilon}. 
\end{equation} 
\end{enumerate}  
\end{lemma}
\begin{proof}
The first two points follow directly from \eqref{subsect:premetric:defeq} -- for reference, item \ref{lem:premetric_topofacts:it1} is reported in \citep[p.\ 23]{arkhangelski1990} while statement \ref{lem:premetric_topofacts:it2} is \citep[Proposition 2.4.9]{arkhangelski1990}. Item \ref{lem:premetric_topofacts:it3} follows from the statements \ref{lem:premetric_topofacts:it1} and \ref{lem:premetric_topofacts:it2} and \citep[Proposition 3.1.3]{arkhangelski1990}; or alternatively from \citep[Lemma 2.2$\,\&\,$Corollary 2.3]{bruno2017}. 
\end{proof}
\begin{remark}\label{rem:premetric_topofacts:continuity}
The $\varepsilon$-$\delta$ characterisation \eqref{lem:premetric_topofacts:it3:eq1} of (topological) continuity holds more generally if $\phi$ is a metric-space-valued map on any premetric space, see e.g.\ \citep[Lem.\ 2.2 $\,\&\,$ Cor.\ 2.3]{bruno2017}.  
\end{remark}

\begin{proposition}\label{prop:coarseness}
For any uniformly signature integrable subset $\tilde{\fD}$ of $\dot{\fD}$, the (subspace) topology $\tau_\delta$ on $\tilde{\fD}$ is coarser than any weak' topology on $\tilde{\fD}$.
\end{proposition}
\begin{proof}
Let $\tilde{\fD}\subseteq(\dot{\fD}, \tau_\delta)$ be uniformly signature integrable, and recall that the $\tau_\delta$-subspace topology on $\tilde{\fD}$ is induced by the restricted premetric $\restr{\delta}{\tilde{\fD}^{\times 2}}$. Fix any $1\leq q < 2$. The assertion then holds if for any net $(\mu_\alpha)$ in $\tilde{\fD}$ we have (recalling Definition \ref{def:weakp-convergence} for notation) that:
\begin{equation}\label{lem:premetric_topofacts:aux1}
\mu_\alpha \xrightharpoonup{q\text{-}\mathrm{var}} \mu\in\tilde{\fD} \quad \text{ implies } \quad \delta(\mu,\mu_\alpha) \rightarrow 0,
\end{equation}  
see also \citep[Lemma 2.1]{bruno2017}. By the definition \eqref{def:delta:eq1} of $\delta$, the desired implication \eqref{lem:premetric_topofacts:aux1} clearly holds if, for any $(\mu_\alpha)$ and $\mu$ as above and any $i_1, \ldots, i_m\in[d]$ for $m\leq 3$, we have that: $\mu_\alpha \xrightharpoonup{q\text{-}\mathrm{var}} \mu$ implies $\langle\mu_\alpha\rangle_{i_1\cdots i_m} \!\rightarrow \langle\mu\rangle_{i_1\cdots i_m}$. This last implication, however, is guaranteed by Lemma \ref{lem:unifsigintegr}, which concludes the proof.                 
\end{proof} 
The idea behind the \eqref{sect:coredinates:eq4}-induced topologization of $\mathcal{M}_1$ is closely related to the natural idea of maximum mean discrepancy (MMD) over the finite set of test functions $\mathcal{H}\coloneqq\{\sig_w(\cdot)\mid |w|\leq 3\}$ and combined with an appropriate standardization that absorbs the $\tilde{\mathrm{M}}_d$-inflicted (`redundant') degrees of freedom which are due to the $\lceil\,\cdot\,\rceil_{\tilde{\mathfrak{m}}}$-controlled indistinguishability of the true source (for the latter, recall Definition \ref{appendix:sect:BSSformal:def:ICA}).\\[-0.75em] 

While not required here, note that the premetric \eqref{sect:coredinates:eq4} can be extended to a full metric via the systematic addition of higher-order signature moments, as detailed in Remark \ref{rem:pretometric}.\\[-0.5em]          

Let us include the following addition to Definition \ref{def:delta}. 
\begin{remark}\label{rem:icdefect-nonstationary}
By Lemma \ref{lemma1}, the IC-defect of a mean-stationary signal $\mu\in\dot{\fD}$ reads
\begin{equation}\label{rem:icdefect-nonstationary:eq1}
\displayindent0pt
\displaywidth\textwidth
\delta_{\independent}(\mu) \,=\, \left\|N_\mu^{-1}\!\left([\mu]_0-\big(\delta_{ij}\langle\mu\rangle_{ij}\big), \frac{[\mu]_1-\big(\delta_{ij1}\langle\mu\rangle_{ij1}\big)}{\sqrt{\langle\mu\rangle_{11}}}, \ldots, \frac{[\mu]_d-\big(\delta_{ijd}\langle\mu\rangle_{ijd}\big)}{\sqrt{\langle\mu\rangle_{dd}}}\right)\!N_\mu^{-1}\right\|_{\mathcal{V}};
\end{equation}
for a signal $\mu\in\dot{\fD}$ that is not mean-stationary, we declare its IC-defect $\delta_{\independent}(\mu)$ to be defined by the right-hand side of \eqref{rem:icdefect-nonstationary:eq1}. Clearly $\mu\in\dot{\fD}$ is orthogonal iff $\delta_{\independent}(\mu)=0$, and $\delta_{\independent}$ is invariant under the $\M$-action on $\dot{\fD}$ as follows directly from \eqref{rem:icdefect-nonstationary:eq1} and Lemma \ref{lem:linequiv}. \hfill $\bdiam$  
\end{remark}
\noindent    
The IC-defect $\delta_{\independent}$, defined by \eqref{def:delta:eq1} and \eqref{rem:icdefect-nonstationary:eq1}, is an auxiliary function on $\dot{\fD}$ with which the blind identifiability \eqref{def:blindinversion:eq1} of a (non-)linearly transformed signal can be quantitatively controlled: Source signals for which this defect vanishes (Sect.\ \ref{sect:orthogonal_signals}) will be exactly identifiable, while blind inversion for a source will become increasingly inaccurate -- or even impossible -- as its IC-defect increases, see Section \ref{chap:robustICA:sect:auxiliaries}. In anticipation of this, the next lemma paves the way for an informative robustness analysis \eqref{def:robustICA:eq1} by establishing moduli of continuity for $\delta_{\independent}$ that relate to the above premetric topology on $\dot{\fD}$ and the uniform topology on the space of spatial mixing transformations, respectively. (As it is mainly technical, the proof is delegated to Appdx.\ \ref{pf:lem:premetric_facts2}.)\\[-0.75em]                   

\noindent
Here and in the sequel, the Jacobian $D_R$ of a map $R\in C^1(\R^d;\R^d)$ is normed via $\|D_R\|_\infty\coloneqq\sup_{x\in\R^d}\|D_R(x)\|_2$, where $\|\cdot\|_2$ is the Euclidean operator norm (`spectral norm') on $\R^{d\times d}$.\\[-0.75em]  

\hfill
(For the following, the notation of Subsection \ref{subsubsect:moment-integrability} is in use.)         

\begin{lemma}\label{lem:premetric_facts2}
Let $\hat{\mu}\in\dot{\fD}$, and $B_r\equiv B_\delta(\hat{\mu}, r)$ be the $\hat{\mu}$-centered $\delta$-ball of radius $r>0$ \emph{(}see \eqref{lem:premetric_topofacts:eq1}\emph{)}. Suppose that $r < \hat{r}\coloneqq\rho_0\wedge (\rho_0/\rho_1)$, with $\rho_0\coloneqq\min_{i\in[d]}\langle \hat{\mu}\rangle_{ii}$ and $\rho_1\coloneqq\max_{i\in[d]}\langle\hat{\mu}\rangle_{ii}$.  
\begin{enumerate}[font=\upshape, label=(\roman*)]
\item\label{lem:premetric_facts2:it1}There are constants $\mathfrak{m}_{\hat{\mu}|r}=\mathfrak{m}_{\hat{\mu}|r}(\hat{\mu},r), K_{\hat{\mu}|r} = K_{\hat{\mu}|r}(\hat{\mu}, r) \geq 0$, growing monotonously in $r$, such that for each $\mu\in B_r$ we have $\max\nolimits_{|w|=2,3}\big|\langle\mu\rangle_w - \langle\hat{\mu}\rangle_w\big|\leq\mathfrak{m}_{\hat{\mu}|r}\cdot r$ and  
\begin{equation}\label{lem:premetric_facts2:eq1}
\big|\delta_{\independent}(\mu) - \delta_{\independent}(\hat{\mu})\big|\,\leq\,K_{\hat{\mu}|r}\sqrt{r + r^2}\,; 
\end{equation} 
\item\label{lem:premetric_facts2:it2} if $B_\delta(\hat{\mu}, \rho_0)$ is core-integrable as per $(K_0, \beta, p, \alpha)$, then for each $R\in C^1(\R^d;\R^d)$ there exists a constant $\tilde{K}_{\hat{\mu}|r} = \tilde{K}_{\hat{\mu}|r}(\hat{\mu}, r, K_0, R)\geq 0$, growing monotonously in $r$, such that: for each $\mu\in B_r$ we have, for the auxiliary functions $\phi_k(\cdot)$ from Lemma \ref{lem:premetric_facts} \ref{lem:premetric_facts:it3}, 
\begin{equation}\label{lem:premetric_facts2:eq2}
\big|\delta_{\independent}((I+R)_\star\mu) - \delta_{\independent}(\mu)\big| \,\leq\, \tilde{K}_{\hat{\mu}|r}\cdot\phi_3(\|D_R\|_\infty)^{\tfrac{1}{2}}\big(\|D_R\|_\infty\|R\|_\infty\big)^{\!\tfrac{\alpha}{2}}
\end{equation} 
provided that $\rho_R\coloneqq C_p\phi_2(\|D_R\|_\infty)(\|D_R\|_\infty\|R\|_\infty)^\alpha < \rho_0 - \rho_1r$. 
\end{enumerate}
For any given $c, C>0$ with $c < \rho_0$, the constant $\tilde{K}_{\hat{\mu}|r}$ in \eqref{lem:premetric_facts2:eq2} can be chosen to apply uniformly -- that is, dependent only on $\hat{\mu}, K_0$ and $c, C$ but independent of $R$, $\mu$ and $r$ -- on the set 
\begin{equation}\label{lem:premetric_facts2:eq3}
\mathcal{R}^C_c\coloneqq\{(\mu,R)\in \dot{\fD}\times C^1(\R^d;\R^d)\mid \rho_1\delta(\hat{\mu},\mu) + \rho_R \leq c \, \text{ and } \, \rho_R\,\vee\,\|D_R\|_\infty\leq C\}.
\end{equation} 
All of the above constants are explicit.        
\end{lemma}

The following section describes the signals whose IC-defect $\delta_{\independent}$ is minimal. These will then constitute a (large) class $\sI_\ast$ of ICA-identifiable signals in $\fD$, [cf.\ \eqref{appendix:sect:BSSformal:intext:eq:identconds}, \eqref{def:icainversion:eq1} and] see Section \ref{chap:robustICA:sect:auxiliaries}. 
\subsection{Orthogonal Signals}\label{sect:orthogonal_signals}It is well-known that general signals are not identifiable from their blind linear mixtures, but that recovery becomes possible if the components of the source are mutually independent or, in the time-dependent case, at least pairwise uncorrelated. Here we show that the coredinates of such signals take a particularly simple algebraic form.            

\begin{definition}[Product Signal]
An element $\mu\equiv(\mu^1, \cdots, \mu^d)\in\fD$ is called is called a \emph{product signal} if it is the product of its marginals, that is if $\mu = \mu^1\otimes\cdots\otimes\mu^d$. 
\end{definition}
Clearly, a process $Y$ in $\R^d$ is a product signal iff its components $Y^1, \ldots, Y^d$ are mutually independent. The coredinates \eqref{def:coordinates:eq2} of a product signal, if mean-stationary, are all diagonal.    

\begin{lemma}\label{lemma1}
For $\mu\in\fD$ a mean-stationary product signal, we have   
\begin{equation}\label{lemma1:eq1}
\langle\mu\rangle_{ij} = \delta_{ij}\cdot\langle\mu\rangle_{ij} \quad\text{ and }\quad \langle\mu\rangle_{ijk} = \delta_{ijk}\cdot\langle\mu\rangle_{ijk} 
\end{equation}
where $\delta_{ij}$ and $\delta_{ijk}$ are the Kronecker deltas in dimension two and three, respectively. 
\end{lemma}   
\begin{proof}
Let $\mu=\mu^1\otimes\cdots\otimes\mu^d\in\fD$ be a mean stationary (signed) measure on the path space $\mathcal{BV}$. The assertion \eqref{lemma1:eq1} then follows from the more general statement that: 
\begin{equation}\label{lem1:pf:aux1}
\langle\mu\rangle_{i_1\ldots i_m} = \, 0 \qquad (m\in\N)
\end{equation}
for every multiindex $(i_1,\ldots,i_m)\in[d]^{\times m}$ of such kind that at least one of its entries $i_\nu$ appears exactly once, i.e.\ for every multiindex contained in the set  
\begin{equation}
J \, \coloneqq \, \big\{(i_1,\ldots,i_m) \in [d]^{\times m} \ \big|\ \exists\, \nu_0 \in [m]  \text{ : } i_{\nu_0}\neq i_\nu \ \text{for each} \ \nu\neq \nu_0 \big\}.
\end{equation} 
To prove \eqref{lem1:pf:aux1}, recall from Lemma \ref{lem:absolutelycontinuous} that each path $x\equiv(x^1_t, \ldots, x^d_t)_{t\in\I}\in\mathcal{C}^1$ admits an integrable derivative $\dot{x}\equiv(\dot{x}^i)$ almost everywhere (set $\dot{x}$ to zero everywhere else) with    
\begin{equation}\label{lem1:pf:aux3}
x^i_t \,=\, x^i_s + \int_s^t\!\dot{x}^i_r\,\mathrm{d}r \qquad (0\leq s\leq t\leq 1)
\end{equation}  
for all $i\in[d]$. Now by the mean-stationarity of $\mu$, we for each $i\in[d]$ have $0 = \E[\mu^i_{s,t}] = \int_{\mathcal{C}_d}\!x_{s,t}\,\mu^i(\mathrm{d}x) = \int_{\mathcal{C}_d}\!(x^i_t - x^i_s)\,\mu(\mathrm{d}x)$ for each $0\leq s\leq t\leq 1$ (see Section \ref{sect:operationsonmeasures} for notation), whence from \eqref{lem1:pf:aux3} and Fubini we find that $0 = \int_s^t\!\phi_\mu^i(r)\,\mathrm{d}r$ for $\phi_\mu^i(r)\coloneqq\int_{\mathcal{C}_d}\!\dot{x}^i_r\,\mu(\mathrm{d}x).$
Since $s,t\in[0,1]$ (with $s\leq t$) have been arbitrary, it follows that 
\begin{equation}\label{lem1:pf:aux6}
\phi^i_\mu \,=\, 0 \ \ \text{ a.e.\ \ on \ $\I$}, \qquad \text{for all \ } i\in[d].
\end{equation}Now since $\phi^i_\mu(r) = \int_{\mathcal{C}_1}\dot{z}_r\,\mu^i(\mathrm{d}z)$, observe from \eqref{def:coordinates:eq1} that for every $\bm{i}\equiv(i_1,\ldots,i_m)\in J$ we have 
\begin{equation}
\begin{aligned}
\langle\mu\rangle_{i_1\ldots i_m} \,&=\, \int_{\Delta_m}\int_{\mathcal{C}_d}\,\dot{x}^{i_1}_{t_1}\cdots\dot{x}^{i_m}_{t_m}\,\mu^1\otimes\cdots\otimes\mu^d(\mathrm{d}(x^i))\,\mathrm{d}^m\bm{t} \,=\, \int_{\Delta_m}\!\phi^{i_{\nu_0}}_\mu(t_{i_{\nu_0}})\cdot Q^{\bm{i}}_\mu(\bm{t})\,\mathrm{d}^m\bm{t}\,\eqqcolon\,\beta   
\end{aligned} 
\end{equation}
for the product $Q^{\bm{i}}_\mu(\bm{t})\coloneqq\prod_{i\in[d]\setminus\{i_{\nu_0}\}}\int_{\mathcal{C}_1}\!z^{\bm{q}_i(\bm{t})}\,\mu^i(\mathrm{d}z)$ with $z^{\bm{q}_i(\bm{t})}\coloneqq \prod\nolimits_{\nu\in[m]\,:\,i_\nu=i} \dot{z}^{i_\nu}_{t_\nu}$ $(z^i\in\mathcal{C}_1$) for each $i\in[d]$ (setting $z^{\bm{q}_i(\bm{t})}\coloneqq 1$ if the product is empty). Consequently, for $\bm{t}'\coloneqq(t_2,\ldots, t_m)$,
\begin{equation}
\begin{aligned}
\beta \,=\, \int_{\I^{d-1}}\!Q^{\bm{i}}_\mu(\bm{t}')\!\left[\int_\I\!\phi^1_\mu(t_1)\cdot\mathbbm{1}_{\Delta_m}(\bm{t})\,\mathrm{d}t_1\right]\!\mathrm{d}^{m-1}\bm{t}' \,=\, 0    \qquad (\text{by } \eqref{lem1:pf:aux6}),   
\end{aligned}
\end{equation} 
where for convenience we assumed $i_{\nu_0}=1$ wlog (otherwise permuting $\Delta_m\coloneqq\{(t_1, \cdots, t_m)\mid 0\leq t_1\leq t_2\leq\ldots\leq t_m\leq 1\}$ accordingly). This proves \eqref{lem1:pf:aux1}, as desired.     
\end{proof} 
\begin{definition}\label{def:orthogonal}A signal $\mu\in\fD$ for which \eqref{lemma1:eq1} holds will be called (\nth{3}-order) \emph{orthogonal}. 
\end{definition}
\begin{remark}Given the fact, verifiable by a direct computation, that the signature moments \eqref{def:coordinates:eq1} of a mean-stationary product signal $\mu$ coincide with its signature cumulants \citep{bonnier2019signature} up to order $m=3$, the assertion of Lemma \ref{lemma1} can also be derived from \citep[Theorem 1.2]{bonnier2019signature}.        
\end{remark}   
  
As outlined at the end of Section \ref{sect:robustica-basic}, the following subsection introduces an identifiability space $E$ that is endowed with a (coarse) extension of the above signature topology on $\dot{\fD}$.
\subsection{A Coarse Topology on the Space of Causes}\label{subsect:topology_identspace}In this section we extend the signal topology of Section \ref{chap:robustICA:sect:premetric} to an `identifiability' topology on the causal space $\mathfrak{C}$, based on which the robustness property \eqref{def:robustICA:eq1} can then be appropriately analysed.\\[-0.75em] 

\noindent
Exploiting the premetric structure $\tau_\delta$ from Section \ref{chap:robustICA:sect:premetric} and for technical convenience, said topology will be supported on `regular' subspaces $E$ of $\mathfrak{C}$ of the product form
\begin{equation}\label{subsect:topology_identspace:eq1}
E\,=\,\mathfrak{S}\times C^{1,1}(\R^d) \quad\text{ with }\quad \fS\subseteq\dot{\fD} \text{ core-integrable} 
\end{equation} 
as per $(K_{\fS},\beta,p,\alpha)\eqqcolon c_{\fS}$, see Subsection \ref{subsubsect:moment-integrability} for notation.\\[-0.5em]  

\noindent
The factor $\fS$ is endowed with the $\tau_\delta$-induced subspace topology $\tau'_\delta$ on $\fS$, i.e.\ $\tau_\delta'=\{\mathcal{O}\cap\fS\mid\mathcal{O}\in\tau_\delta\}$,\footnote{\ Clearly then, $\tau_\delta'$ is the topology induced by the restriction of $\delta$ to the subset $\fS\times\fS$, which is again a premetric on $\fS$. (We use the same symbol for this restriction and the original premetric \eqref{def:delta:eq1}, for convenience.)} while $C^{1,1}(\R^d)\eqqcolon C^{1,1}$ is assigned the premetric $\bar{d}:C^{1,1}\times C^{1,1}\rightarrow\mathbb{R}_+$ given by
\begin{equation}\label{intext:trafopremetric}
\bar{d}(\theta_1, \theta_2)\coloneqq \min\left(\big(\big\|\theta_1^{-1}\circ\theta_2 - \mathrm{id}\big\|_{\infty}+1\big)\!\cdot\!\|D(\theta_1^{-1}\circ\theta_2 - \mathrm{id})\big\|_\infty, \, 1 \,\right)
\end{equation}
which measures the residual deviation between two transformations, in adaptation of the classical (pseudo)norm $\min(\|\cdot\|_\infty + \|D(\,\cdot\,)\|_\infty, 1)$ on $C^1(\R^d;\R^d)$. Both $\delta$ and $\bar{d}$ combine to a premetric $\mathbbm{d} \, : \, E^{\times 2}\rightarrow\mathbb{R}_+$ on the subspace \eqref{subsect:topology_identspace:eq1}, which we define by           
\begin{equation}\label{subsect:topology_identspace:eq2}
\mathbbm{d}\big((\mu_1, \theta_1), (\mu_2, \theta_2)\big)\coloneqq\, \delta(\mu_1, \mu_2) \, + \, \bar{d}(\theta_1, \theta_2).   
\end{equation}  
Being the sum of the premetrics \eqref{def:delta:eq1} and \eqref{intext:trafopremetric}, the premetric\footnote{\ Since $(E, \mathbbm{d})$ is a premetric space, Lemma \ref{lem:premetric_topofacts} holds as stated upon replacing $(\dot{\fD}, \delta)$ with $(E,\mathbbm{d})$.} $\mathbbm{d}$ combines the marginal toplogies $(\fS, \delta)$ and $(C^{1,1}(\R^d), \bar{d})$ to a coarse joint topology $\tau_{\mathbbm{d}}$ on $E$. The lemma below shows how this topology relates to the `identifiability controlling' defect $\delta_{\independent}$ of a signal.\\[-0.75em] 

The premetric space $(E,\mathbbm{d})$ is the domain on which we will show and analyse the robustness \eqref{def:robustICA:eq1}. (Accordingly, our corresponding IA $\sI$ will then be a subset of $E$.) To stay consistent with the setup of Section \ref{sect:robustica-basic}, we may (trivially) complement $\tau_{\mathbbm{d}}$ into a topology on $\mathfrak{C}$ as follows. 

\begin{definition}[Causal Topology]\label{def:causaltopology}Given $E$ as in \eqref{subsect:topology_identspace:eq1} and $E^{\mathrm{c}}\coloneqq\mathfrak{C}\setminus E$, we call 
\begin{equation}\label{def:causaltopology:eq1}
\text{the disjoint union topology \ $\tau_{\mathfrak{C}}$ \ on } \ (E, \tau_{\mathbbm{d}}) \, \sqcup \, (E^{\mathrm{c}}, \tau_\emptyset)
\end{equation}
the ICA-supporting \emph{causal topology} on $\mathfrak{C}$. (Here, $\tau_\emptyset=\{\emptyset, E^{\mathrm{c}}\}$ is the trivial topology on $E^{\mathrm{c}}$.) 
\end{definition}
\begin{remark}\label{rem:causaltopology}
By definition, the topology $\tau_{\mathfrak{C}}$ on $\mathfrak{C}$ is simply the disjoint union $\tau_{\mathfrak{C}}=\tau_{\mathbbm{d}}\sqcup\tau_{\emptyset}$. Hence and by construction of \eqref{subsect:topology_identspace:eq2}, the topology $\tau_{\mathfrak{C}}$ relates to the referential coarseness \eqref{rem:suffcoarse:eq1} via Proposition \ref{prop:coarseness}. Consequently still and by Remark \ref{rem:premetric_topofacts:continuity}, a map $\phi : (\mathfrak{C}, \tau_{\mathfrak{C}})\rightarrow (W, d_W)$ [the latter pair a metric space] is seen to be continuous at a point $\mathfrak{c}_\star\in E \subseteq\mathfrak{C}$ if for every $\varepsilon>0$ there is a $\tilde{\delta}>0$ such that $d_W(\phi(\mathfrak{c}), \phi(\mathfrak{c}_\star))<\varepsilon$ for each $\mathfrak{c}\in\{\tilde{\mathfrak{c}}\in E \mid \mathbbm{d}(\mathfrak{c}_\star,\tilde{\mathfrak{c}})<\tilde{\delta}\}$. \hfill $\bdiam$ 
\end{remark}
      
\begin{lemma}\label{lem:robustness1}
For a fixed orthogonal signal $\mu_\star\in E$ and $A\in\GL$, consider the $\mathbbm{d}$-ball  
\begin{equation}\label{lem:robustness1:eq1}
\mathbb{B}_r^\star\equiv\mathbb{B}_r(\mu_\star, A)\coloneqq \left\{\mathfrak{p}\in E \ \middle| \ \mathbbm{d}\big((\mu_\star, A), \mathfrak{p}\big) < r\right\} 
\end{equation}
for $r>0$. Let further $\rho_0\coloneqq\min_{i\in[d]}\langle \mu_\star\rangle_{ii}$ and $\rho_1\coloneqq\max_{i\in[d]}\langle\mu_\star\rangle_{ii}$. Provided $r < r_0\coloneqq\big(\rho_0/(\rho_1 + 2C_p)\big)^{1/\alpha}$, there are then explicit constants $\hat{K}=\hat{K}(\mu_\star, C_p)\geq 1$ and $K_r=K_r(\mu_\star, r, C_p)$, increasing in $r$, such that for any given $(\mu, f)\in\mathbb{B}_r^\star$ there exists a unique $\tilde{\mu}\in\fD$ such that:   
\begin{equation}\label{lem:robustness1:eq2}
f_\ast\mu = A_\ast\tilde{\mu} \quad\text{ with }\quad \delta_{\independent}(\tilde{\mu})\,\leq\, K_r\cdot r^{\alpha/2} \quad\text{ and }\quad \big\|[\tilde{\mu}]_{\mathfrak{c}} - [\mu_\star]_{\mathfrak{c}}\big\|_{\mathcal{V}} \,\leq\, \hat{K}r^\alpha. 
\end{equation}   
\end{lemma}  
\begin{proof}
Fix any $(\mu, f)\in\mathbb{B}_r^\star$. Then $\delta(\mu_\star, \mu) + \bar{d}(A, f) < r$ by defs.\ \eqref{lem:robustness1:eq1} and \eqref{subsect:topology_identspace:eq2}. Further, $\bar{d}(A,f) \geq \|D_R\|_\infty\|R\|_\infty$ for $R\coloneqq A^{-1}\circ f - \mathrm{id}$, by \eqref{intext:trafopremetric} (note that $r_0<(\rho_0/\rho_1)^{1/\alpha}\leq 1$). Abbreviating $r_\delta\equiv\delta(\mu_\star, \mu)$ and $\rho_R\coloneqq C_p\phi_2(\|D_R\|_\infty)\big(\|D_R\|_\infty\|R\|_\infty\big)^{\!\alpha}$, we thus find that
\begin{equation}
r_\delta + \|D_R\|_\infty\|R\|_\infty <  r < r_0 \qquad\text{ and hence }\qquad \rho_1 r_\delta + \rho_R \,<\, c_0 r^\alpha \,<\, \rho_0
\end{equation}
for $c_0\coloneqq\rho_1 + 2C_p$ (note that $\|D_R\|_\infty\leq 1$ implies $\phi_2(\|D_R\|_\infty)\leq 2$ and $\phi_3(\|D_R\|_\infty)\leq 4$). Thus $(\mu, R)\in\mathcal{R}_c^C$, cf.\ \eqref{lem:premetric_facts2:eq3}, for $c\coloneqq c_0r^\alpha$ and $C\coloneqq\max\{2C_p, 1\}$. Hence by Lemma \ref{lem:premetric_facts2} \ref{lem:premetric_facts2:it2} there is a constant $\tilde{K}_r=\tilde{K}_r(\mu_\star, r, K_{\fD})\geq 0$, independent of $\mu$ and $r_\delta$ and $R$, such that      
\begin{equation}\label{lem:robustness1:aux2}
\big|\delta_{\independent}((I+R)_\ast\mu) - \delta_{\independent}(\mu)\big| \,\leq\, \tilde{K}_{r}\cdot r_R^{\alpha/2}
\end{equation}   
for $r_R\coloneqq\|D_R\|_\infty\|R\|_\infty$. Denoting $\tilde{\mu}\coloneqq(\mathrm{I} + R)_\ast\mu$, we also find that 
\begin{equation}
f_\ast\mu = [A(I+R)]_\ast\mu = A_\ast\tilde{\mu} \qquad \text{(by definition of $R$)}
\end{equation}  
and, by combination of \eqref{lem:robustness1:aux2} with Lemma \ref{lem:premetric_facts2} \ref{lem:premetric_facts2:it1}, using that $\delta_{\independent}(\mu_\star) = 0$ (cf.\ Lemma \ref{lemma1}),   
\begin{equation}\label{lem:robustness1:aux3}
\begin{aligned}
\big|\delta_{\independent}(\tilde{\mu})\big| \,\leq\, \big|\delta_{\independent}(\tilde{\mu}) - \delta_{\independent}(\mu)\big| + \big|\delta_{\independent}(\mu) - \delta_{\independent}(\mu_\star)\big| \,\leq\, K_r'\left(\sqrt{r_R^\alpha} + \sqrt{r_\delta}\right) \,\leq\, K_r\cdot r^{\alpha/2}
\end{aligned}
\end{equation}
for the constants $K_r'\coloneqq\max\{\tilde{K}_{r}, \sqrt{2}K_{\mu_\star\!|r}\}$, for $K_{\mu_\star\!|r}$ as in \eqref{lem:premetric_facts2:eq1}, and for $K_r\coloneqq 2K_r'$.  

As to the distance between the coredinates $[\tilde{\mu}]_{\mathfrak{c}}$ and $[\mu_\star]_{\mathfrak{c}}$, note that by \eqref{lem:premetric_facts:it3:eq1} and \eqref{lem:premetric_facts:it2:eq1}, 
\begin{equation}\label{lem:robustness1:aux4}
\begin{gathered}
\big\|[\tilde{\mu}]_{\mathfrak{c}} - [\mu_\star]_{\mathfrak{c}}\big\|_{\mathcal{V}} \,\leq\, \big\|[\tilde{\mu}]_{\mathfrak{c}} - [\mu]_{\mathfrak{c}}\big\|_{\mathcal{V}} + \big\|[\mu]_{\mathfrak{c}} - [\mu_\star]_{\mathfrak{c}}\big\|_{\mathcal{V}} \\
=\, \sqrt{\textstyle\sum_{|w|=2,3}\Delta_w(\tilde{\mu}, \mu)^2} + \sqrt{\textstyle\sum_{|w|=2,3}\Delta_w(\mu, \mu_\star)^2}\,\leq\, c_d\big[5C_p r_R^\alpha + \mathfrak{m}_{\mu_\star} r_\delta\big] \,\leq\, \hat{K}r^\alpha 
\end{gathered} 
\end{equation} 
where we denoted $\Delta_w(\mu^{(1)}, \mu^{(2)})\coloneqq |\langle\mu^{(1)}\rangle_w - \langle\mu^{(2)}\rangle_w|$ (so that the identity holds by definition of $\|\cdot\|_{\mathcal{V}}$) as well as $c_d\coloneqq(\sum_{|w|=2,3}1)^{1/2} = d\sqrt{d+1}$ and $\hat{K}\coloneqq c_d(5C_p + \mathfrak{m}_{\mu_\star})$.   
\end{proof} 
\noindent
Inherent in the definition \eqref{def:causaltopology:eq1} of the causal topology $\tau_{\mathfrak{C}}$ is an element of choice concerning the selection of $\fS$ in \eqref{subsect:topology_identspace:eq1}: That we resort to an \emph{integrable} subset $\fS$ of $\dot{\fD}$ is to ensure uniform comparability of different nonlinearly transformed signals under the residual metric \eqref{intext:trafopremetric}, cf.\  Section \ref{subsubsect:moment-integrability} and \eqref{lem:robustness1:aux2},\,\eqref{lem:robustness1:aux3}; the larger the enclosed threshold \eqref{notation:unifsigintegr:eq1}, the more `fine grained' will be the $\tau_{\mathfrak{C}}$-supported variation of  \eqref{def:robustICA:eq1} on $\mathfrak{C}$. The `necessity of choice' behind \eqref{subsect:topology_identspace:eq1} can be avoided by replacing \eqref{intext:trafopremetric} with a `non-incremental' (e.g., quotient) [pre]metric on $C^{1,1}$, or by restricting the latter space to a `tamed' subspace of transformations (so as to obtain a $\mu$-uniformly integrable bound on the RHS of \eqref{lem:premetric_facts:aux14}). No integrability is needed (i.e., $\fS = \dot{\fD}$ works out) if only variations in the source, and not in its transformation, are considered.           

\begin{remark}\label{rem:partialcontinuity}
Taking $\fS=\dot{\fD}$, the proof of Lemma \ref{lem:robustness1} shows that on the $A$-sections 
\begin{equation}
E_{A, r}\coloneqq\big(\dot{\fD}\times\{A\}\big)\cap \mathbb{B}_r(\mu_\star, A) = B_\delta(\mu_\star,r)\times\{A\}, 
\end{equation}
i.e.\ when the hidden relation $A$ between source and observable is linear and the same for all signals so that any inexactness in the observable-to-source inversion \eqref{def:robustICA:eq1} is due to source deviations (from $\mu_\star$) only, then no integrability assumptions on $\fS$ of the form \eqref{notation:unifsigintegr:eq1} are required and the estimates \eqref{lem:robustness1:eq2} improve as follows: 

{\ }\\[-0.5em] 
\noindent\makebox[\linewidth][c]{
\begin{minipage}{.9\textwidth}
\centering
\begin{lemma*}
\emph{Let $\mu_\star\in\dot{\fD}$ be orthogonal and $A\in\GL$, and let $0< r < \rho_0/\rho_1$. Then there are constants $\hat{L}=\hat{L}(\mu_\star)$ and $L_r=L_r(\mu_\star, r)$, the latter increasing in $r$, with
\begin{equation}\label{rem:partialcontinuity:eq2}
\delta_{\independent}(\mu)\,\leq\, L_r\sqrt{r} \quad\text{ and }\quad \big\|[\mu]_{\mathfrak{c}} - [\mu_\star]_{\mathfrak{c}}\big\|_{\mathcal{V}} \,\leq\, \hat{L}r \quad\text{ for each } \, \mu\in B_\delta(\mu_\star, r). 
\end{equation}} 
\end{lemma*}
\begin{proof}
Revisiting the proof of Lemma \ref{lem:robustness1} on $E_{A,r}$, we have $R\equiv 0$ and hence note \eqref{rem:partialcontinuity:eq2} for the constants $L_r\coloneqq \sqrt{2}K_{\mu_\star|r}$ via \eqref{lem:robustness1:aux3}, and $\hat{L}\coloneqq\mathfrak{m}_{\mu_\star}$ via \eqref{lem:robustness1:aux4}.      
\end{proof}
\end{minipage}}\\[0.5em]
The above simplification is useful in many applications, as the relation between source and observable can often be modelled as exactly linear but with the underlying source signal deviating from orthogonality; see for instance Section \ref{sect:applications} for a few examples. \hfill $\bdiam$          
\end{remark} 

\section{Robust Independent Component Analysis}\label{chap:robustICA:sect:auxiliaries}
\noindent
We introduce a quantifiably robust statistical procedure for recovering (the inverse of) a matrix from its action on an unobserved [non]orthogonal source signal. The general strategy for this is based on the classical idea of jointly diagonalising a set of derived equivariant matrix statistics, given in our case by the coredinates \eqref{def:coordinates:eq2} of the source signal (Section \ref{chap:robustICA:subsect:signature_identifiability:algorithm}). By linking our inversion procedure to the causal topology from Definition \ref{def:causaltopology}, we obtain both an identifiability theorem that provides explicit stability bounds for the recovery of signals that may deviate from orthogonality in any way (Theorem \ref{thm:robust_ica}) and, as our main result, a readily implementable new ICA-solution that is provably robust in the sense of Definition \ref{def:robustICA}. This solution comes with strong and general robustness guarantees that can be expressed in terms of explicitly computable and naturally interpretable moduli of continuity (Theorem \ref{thm:robustness}).          
  
\subsection{ICA-Inversion from Coredinates}\label{chap:robustICA:subsect:signature_identifiability:algorithm}Throughout this section, we follow the classical ICA-paradigm \eqref{appendix:sect:BSSformal:def:ICA:eq1} and consider two linearly related signals $\zeta$ and $\chi$ in $\dot{\fD}$, that is  
\begin{equation}\label{transformation_A}
\chi = A\cdot\zeta \quad\text{ with }\quad A\equiv(a_{ij})\equiv(a_1|\cdots|a_d) \ \in \ \R^{d\times d}
\end{equation}some fixed invertible $d\times d$ matrix with columns $a_i$. For instance, we can always think of 
\begin{equation}\label{measures_BSStriple}
\zeta=S \ \text{ and } \ \chi=X \quad \text{ for a BSS-triple \ $(X,S,A)$ \ as in \eqref{appendix:sect:BSSformal:eq1.1}},
\end{equation}    
with $X=\mathbb{P}_{\bm{X}}$ and $S=\mathbb{P}_{\bm{S}}$ for \eqref{appendix:sect:BSSformal:eq1}-related stochastic processes $\bm{S}=(\bm{S}_t)$ and $\bm{X}=(\bm{X}_t)$ in $\R^d$.\\[-0.75em] 

\noindent
The Problem of ICA (Definition \ref{appendix:sect:BSSformal:def:ICA}) is then to recover the inverse $A^{-1}$ from the input $\chi$, i.e.\ to find a map $\hat{\Phi}$ on $\dot{\fD}$ such that $\hat{\Phi}(\chi)\doteq A^{-1}$ up to some minimal ambiguity, see \eqref{def:icainversion:eq1}.\\[-0.75em]

\noindent
In principle, such a map $\hat{\Phi}$ may be expected to operate on certain `pieces of information' about $\chi$ that efficiently relate the action of $A$ to some relevant statistical properties of $\zeta$, cf.\ \eqref{appendix:sect:BSSformal:intext:eq:identconds}.\\[-0.75em] 

The following ensures that the coredinates $[\chi]_0, \ldots, [\chi]_d$ of $\chi$ can be used to this effect.    
\begin{proposition}\label{prop1}
Let $\chi$ and $\zeta$ be as in \eqref{transformation_A} with coredinates $[\chi]_{\mathfrak{c}}$ and $[\zeta]_{\mathfrak{c}}$, respectively. Provided that $[\zeta]_{0}$ is invertible, we have for the inverse $\theta\coloneqq A^{-1}$ that 
\begin{equation}\label{prop1:eq2} 
\theta\cdot[\chi]_{0}\cdot\theta^\intercal = [\zeta]_{0} \quad\text{ and }\quad \theta\cdot[\chi]_{\nu}\cdot\theta^\intercal = \sum_{\ell=1}^da_{\nu\ell}[\zeta]_{\ell}
\end{equation}for each $\nu\in[d]$, and for each $c\equiv(c_0, \underline{c})\in\R^{1+d}$ with $\underline{c}\equiv(\tilde{c}_1, \ldots, \tilde{c}_d)$ find that  
\begin{gather}\label{prop1:eq3}
\theta^{-1}\cdot[\chi]_\odot^c\cdot \theta \ = \ [\zeta]_\odot^{A|c} \qquad \text{for the matrices}\\\label{prop1:eq4}
[\chi]_\odot^c \,\coloneqq\, c_0[\chi]_{0}^{-1}\cdot\sum_{\nu=1}^d \tilde{c}_\nu[\chi]_{\nu} \quad \text{ and }
\quad [\zeta]_\odot^{A|c} \,\coloneqq\, c_0[\zeta]_{0}^{-1}\cdot \sum_{\ell=1}^d \langle a_\ell, \underline{c}\rangle_2[\zeta]_{\ell}.
\end{gather}  
If $\zeta$ is \nth{3}-order orthogonal, then $[\zeta]^{A|c}_\odot$ is diagonal. 
\end{proposition} 
\begin{proof}
Straightforward from Lemmas \ref{lem:linequiv}, \ref{lemma1} and the definitions \eqref{prop1:eq4}; see Appendix \ref{pf:prop1}. 
\end{proof}
The main takeaway from Proposition \ref{prop1} is how the congruences \eqref{prop1:eq2} between the `input' and the `target' statistics $[\chi]_\nu$ and $[\zeta]_\nu$ combine to the $A$-carried conjugacy \eqref{prop1:eq3} between the ($c$-parametrized families of) contracted statistics $[\chi]_\odot^c$ and $[\zeta]_\odot^{A|c}$. This contraction of the coredinates $[\chi]_{\mathfrak{c}}$, $[\zeta]_{\mathfrak{c}}$ opens up additional degrees of freedom, parametrised by $c$, which can be conveniently exploited, via \eqref{prop1:eq3}, to characterise $A^{-1}$ up to monomial ambiguity if the eigenspectrum of $[\zeta]_\odot^{A|c}$ is non-degenerate. This spectrum can in turn be controlled by the amount of statistical dependence between the components of $\zeta$, cf.\ Lemma \ref{lemma1} and \eqref{def:delta:eq2}, provided that these components satisfy some mild statistical non-degeneracy condition, see e.g.\ \eqref{robustica:s0}.\\[-0.75em]
 
The above strategy underlies the proof of Theorem \ref{thm:robust_ica} below. Its mathematical implementation, i.e.\ the construction of an ICA-inversion \eqref{def:icainversion:eq1} from Proposition \ref{prop1}, is next. 

\subsubsection{Auxiliary and Core Statistics}
For a given $\mu\in\fD$ we define $C_\mu\coloneqq\tfrac{1}{2}([\mu]_0 + [\mu]_0^\intercal)\in\R^{d\times d}$, which is symmetric and positive semidefinite. In the context of \eqref{transformation_A}, suppose that 
\begin{equation}\label{rem:choiceofR:eq1}
C_\chi \in \GL  \quad\text{ and } \quad R\equiv R_{(\chi)}\coloneqq C_{\chi}^{-1/2} \ \text{ is the square root of the inverse } C_\chi^{-1}.
\end{equation}
The matrix $R$ exists and is unique, e.g.\ \citep[Korollar 1.10]{kanzow2007}, and, by its definition, $R^2 = C_\chi^{-1}$ and hence $C_{R\cdot\chi} = R C_\chi R^\intercal = \mathrm{I}$. This latter identity marks the purpose of $R$, and its choice as an inverse square root \eqref{rem:choiceofR:eq1} is also known in the literature as Mahalanobis (or ZCA) whitening.
 
\begin{remark}From the spectral theorem we know that 
\begin{equation}\label{rem:choiceofR:eq2}
R_{(\chi)}\equiv \mathcal{R}(C_\chi) \,=\, Q\Lambda_\chi^{-1/2} Q^\intercal
\end{equation} 
for some $Q=Q(C_\chi)\in\operatorname{O}_d(\R)$ and $\Lambda_\chi\coloneqq\mathrm{ddiag}(\lambda_1^\chi, \ldots, \lambda_d^\chi)$, where $\lambda_1^\chi\geq\ldots\geq\lambda_d^\chi>0$ are the eigenvalues of $C_\chi$ (enumerated in descending order, counting multiplicities). Further note
\begin{equation}\label{rem:choiceofR:eq3}
C_\chi \, = \,  A\big[\tfrac{1}{2}\big([\zeta]_0 + [\zeta]_0^\intercal\big)\big]A^\intercal \, = \, A C_\zeta A^\intercal, 
\end{equation}
so that $C_\chi$ is invertible iff $C_\zeta$ is invertible. The invertibility of $C_\zeta$ is violated only for the degenerate case that the increment $\zeta_{0,1}$ is supported on a hyperplane of $\R^d$ (Lemma \ref{lem:invertcmu}). \hfill $\bdiam$ 
\end{remark}
For usage below, consider the pre-transformed matrices $A_R\coloneqq RA$ and 
\begin{equation}\label{rem:choiceofR:eq4}
B_R\equiv(b_1|\cdots|b_d)^\intercal\coloneqq(A_R)^{-1}, \ \text{ and } \ \bar{B}_R\coloneqq\mathrm{ddiag}(|b_1|, \cdots, |b_d|)^{-1}B_R
\end{equation} 
for the matrix $B_R$ rescaled to unit rows.\\[-0.5em] 
 
\noindent
The central components for our inversion procedure are the normalised statistics
\begin{equation}\label{contrast:eq0.1}
\begin{aligned}
\mathfrak{x}_0(\theta)\coloneqq N_{\theta\cdot\chi}^{-1}\cdot[\theta\cdot\chi]_0\cdot N_{\theta\cdot\chi}^{-1} \quad\text{ and } \quad
\mathfrak{x}_\nu(\theta)\coloneqq N_{\theta\cdot\chi}^{-1}\cdot\frac{[\theta\cdot\chi]_\nu}{\sqrt{\langle\theta\cdot\chi\rangle_{\nu\nu}}}\cdot N_{\theta\cdot\chi}^{-1} \quad (\nu\in[d])
\end{aligned} 
\end{equation}
for $\theta\in\GL\eqqcolon\Theta$ and with $N_\mu =\mathrm{ddiag}\big(\langle\mu\rangle_{11}, \cdots, \langle\mu\rangle_{dd}\big)^{1/2}$ for each $\mu\in\fD$, as before.

\begin{remark}
The statistics in \eqref{contrast:eq0.1} are all well-defined. Indeed, since $C_\chi$ is invertible by assumption and hence $C_\chi\in\mathrm{Sym}_d^+$, we have $C_{\theta\cdot\chi}=\theta C_\chi\theta^\intercal\in\mathrm{Sym}_d^+$ and thus $\langle\theta\cdot\chi\rangle_{ii} = e_i^\intercal C_{\theta\cdot\chi} e_i > 0$ for each $i\in[d]$ and all $\theta\in\Theta$. 
\end{remark}

\begin{lemma}\label{lem:scaleinvariance}
The statistics $\mathfrak{x}_\nu(\theta)$ in \eqref{contrast:eq0.1} are each \emph{scale-invariant} for any $\chi\in\fD$, that is
\begin{equation}\label{lem:scaleinvariance:eq2}
\mathfrak{x}_\nu(\Lambda\cdot\theta) \, = \, \mathfrak{x}_\nu(\theta) \quad \text{for all \ $\theta, \Lambda\in\GL$ \ with $\Lambda$ positive diagonal}. 
\end{equation}
\end{lemma}
\begin{proof}
Immediate by Lemma \ref{lem:linequiv}, but see Appendix \ref{pf:prop1} for completeness.
\end{proof}  

\subsubsection{Blind Inversion via Contrast Optimization}Similar in spirit to classical ICA-approaches, cf.\ \citep{HBS, HKO} and Section \ref{sect:perturbedjdp:relatedwork}, we aim to recover the inverse of the hidden mixing transformation as minimizers of a specially constructed cost function: Quantifying the `off-diagonality' of a matrix $\theta$ via $\|\theta\|_\times\coloneqq\|\theta-\mathrm{ddiag}(\theta)\|$, for $\|\cdot\|$ the Frobenius norm on $\R^{d\times d}$, we combine the statistics \eqref{contrast:eq0.1} to the objective (or `contrast') function
\begin{equation}\label{contrast:eq1}
\Theta\ni\theta \,\longmapsto\,\phi_{\chi}(\theta) \, \coloneqq \, \sum_{\nu=0}^d\big\|\mathfrak{x}_\nu(\theta\cdot R)\big\|_{\times}^2 
\end{equation} 
which is minimal over the (approximate) `joint diagonalisers' of the matrices \eqref{contrast:eq0.1}, see below. 

The objective function $\phi_\chi$ is monomially invariant, which means that
\begin{equation}\label{phi:monomialinvar}
\phi_\chi(M\theta) = \phi_\chi(\theta)  \quad \text{for each } M\in\M, \ \theta\in\Theta.
\end{equation} 
Indeed, since each $M\in\M$ is of the form $M=P\Lambda$ for $P\in\mathrm{P}_d^{\pm}$ and some diagonal $\Lambda>0$, the asserted invariance follows directly from Lemmas \ref{lem:linequiv}, \ref{lem:scaleinvariance} (cf.\ \eqref{sect:linequiv:eq1}) and the definition of the Frobenius norm. Since the desired demixing matrices will be identified as minimisers of \eqref{contrast:eq1}, a suitable restriction (`compactification') of the original domain $\Theta$ will be opportune.\footnote{\ In classical ICA, where the components of the source $\zeta$ are assumed (at least pairwise) independent, such a compactification of the domain is usually reached via `whitening' of $\chi$; indeed, this reduces the original domain $\Theta$ to the compact subdomain $\mathrm{O}_d$ of orthogonal matrices. This whitening step, however, requires the intercomponental covariance structure of $\zeta$ to diagonalise, which we do not assume of our (potentially non-orthogonal) source.}

We propose to consider for this the set      
\begin{equation}\label{lica:unitrows}
\Xi_1^0 \coloneqq \big\{\theta\equiv(\theta_1|\cdots|\theta_d)^\intercal\in\R^{d\times d} \, \big| \, \|\theta_i\|_2=1, \,\forall\, i\in[d] \big\}
\end{equation} 
of all matrices in $\R^{d\times d}$ whose rows constitute a unit basis of $\R^d$, and define as our compactified (see Lemma \ref{lem:xicompact}) subdomain for the optimisation of $\phi_\chi$ the superset of $\mathrm{O}_d$ which is given by 
\begin{equation}\label{lica:xi1_domain}
\Xi_1\,\coloneqq\,\Xi_1^0\cap\big\{\theta\in\GL \, \big| \, \kappa_2(\theta)\leq\kappa_0 \big\}
\end{equation}
for any fixed condition bound $\kappa_0\geq\kappa_2(\bar{B}_R)$ (imposed as an a priori assumption\footnote{\ That is, the assumption that the observer does not underestimate the condition number $\kappa_2(\bar{A}_R)$.}). Since $\Xi_1$ is compact and $\phi_\chi$ is continuous in $\theta$ (even continuously differentiable, cf.\ Lemma \ref{lem:linequiv}), we have
\begin{equation}\label{lica:inverseset}
\mathfrak{I}_\chi\coloneqq\left[\argmin_{\theta\,\in\,\Xi_1}\phi_\chi(\theta)\right]\!\cdot\!R_{(\chi)} \, \in \, \mathfrak{F}
\end{equation}
for the subset $\mathfrak{F}\coloneqq\big\{\mathcal{M}\subseteq\GL \ \big| \ \mathcal{M} \text{ is non-empty and compact}\big\}$ of $2^{\GL}$. Recall here that 
\begin{equation}\label{hausdorffmetric}
(\mathfrak{F}, \mathfrak{d}) \ \text{ is a metric space \quad wrt.\ } \quad \mathfrak{d}(\mathcal{A}, \mathcal{B}) \coloneqq \max\big\{\sup\nolimits_{a\in\mathcal{A}}d(a,\mathcal{B}),\,\sup\nolimits_{b\in\mathcal{B}}d(\mathcal{A},b)\big\},
\end{equation}
the classical Hausdorff distance on $\mathfrak{F}$, where $d(a, \mathcal{B})\coloneqq\inf\nolimits_{b\in\mathcal{B}}\|b-a\|$. We set
\begin{equation}\label{lica:inversemap1}
\hat{\Phi} \,:\ \dot{\fD} \rightarrow \mathfrak{F}, \quad \mu \rightarrow \mathfrak{I}_\mu, 
\end{equation} 
and claim that $\hat{\Phi}$ is not only an ICA-inversion \eqref{def:icainversion:eq1} on a yet to be specified IA $\sI_\star$, but in fact an informatively robust such inversion in the sense of \eqref{BSS:robust_prelim1} and Definition \ref{def:robustICA}. 

\hfill The following two sections establish this claim. 

\begin{remark}[On the Optimization \eqref{lica:inverseset}]As further detailed in Section \ref{sect:perturbedjdp:relatedwork}, the optimization of contrasts of the form \eqref{contrast:eq1} over constrained subsets of $\Theta$, such as \eqref{lica:xi1_domain}, is an established topic in its own right, with strong and ongoing research activity. As for row constraints \eqref{lica:unitrows}, e.g.\ \citep{shi2015}, the imposition of additional condition constraints \eqref{lica:xi1_domain} in the optimization of \eqref{contrast:eq1} has been studied and does not introduce any new mathematical subtleties, see e.g.\ \cite{zhou2009nonorthogonal}. \hfill $\bdiam$           
\end{remark}
\begin{remark}\label{rem:robustness}
\label{rem:robustness2:it2}Since the identifiability map \eqref{lica:inversemap1} returns invertible transformations for its input observables, it may also seem natural to use as an alternative topology on $\mathfrak{F}$ the `residual' semimetric given by \eqref{hausdorffmetric} but for the alternative point-set distances 
\begin{equation}\label{rem:robustness2:it2:eq1}
\tilde{d}(a, \mathcal{B})\coloneqq\inf\nolimits_{b\in\mathcal{B}}\tilde{d}(a,b) \quad\text{with}\quad \tilde{d}(a,b)\coloneqq\big\|a\circ b^{-1} - \mathrm{id}\big\|\,\vee\,\big\|b\circ a^{-1} - \mathrm{id}\big\|
\end{equation}
for $\mathrm{id}$ the identity matrix in $\GL$. Continuity wrt.\ the multiplicative setting \eqref{rem:robustness2:it2:eq1} is implied by [continuity wrt.] \eqref{hausdorffmetric}, however, since the metric topology \eqref{hausdorffmetric} is finer than the \eqref{rem:robustness2:it2:eq1}-induced topology, as can be seen by recalling that matrix inversion is Lipschitz at each point.    
\end{remark}    

\subsection{Blind Inversion for (Non-)Orthogonal Sources}\label{sect:blindinversionthm1}
The main goal of this section is to prove that the blind inversion of signals can be robust -- or stable -- under general perturbations of the source, including violations of its orthogonality (`IC') assumption. Recalling that infringements of the latter type are quantified by the defect $\delta_{\independent}\!(\zeta)$ from \eqref{rem:icdefect-nonstationary:eq1}, the following stability analysis of the relation between $\hat{\Phi}(\chi)$ and the true inverse $A^{-1}$ (up to prefactors in $\mathrm{M}_d$) is an important step towards a quantitative understanding of the desired robustness \eqref{def:robustICA:eq1}. This step involves some constants and auxiliary functions, which we introduce next.\\[-0.5em] 

The following considerations are about source signals in $\mathcal{M}_1$ that lie within the set 
\begin{equation}\label{robustica:s0}
\mathscr{S}_0\coloneqq\big\{\mu\in\dot{\fD}\ \big| \ C_{\mu}\!\in\GL \text{ and } \, \sharp\{i\in[d]\,|\,\langle\mu\rangle_{iii}=0\}\leq 1\big\}
\end{equation}  
of measures $\mu$ such that $C_{\mu}$ is invertible (Lem.\ \ref{lem:invertcmu}) and $\langle\mu\rangle_{iii}=0$ for at most one $i\in[d]$.\footnote{\ Note that $\langle\mu\rangle_{iii}=\E[(\mu_{0,1}^i)^3]$. Further, we can (wlog -- upon re-enumeration of components) assume for convenience that it is $\langle\mu\rangle_{111}$ which may vanish; note that this assumption underlies the definition of $\varsigma_1$.}
 
\begin{notation}For a BSS-triple $(\chi, \zeta, A)_{\mathscr{S}_0\times\GL}$, let $\gamma\coloneqq 1 + \sqrt{5}$ and $k_d\coloneqq{\scriptstyle\sqrt{\tfrac{1}{6}(d-1)d(2d-1)}}$ and $R\equiv R_{(\chi)}$ as in \eqref{rem:choiceofR:eq2}, and consider the $(\zeta, A)$-dependent $\displaystyle\varepsilon_0\coloneqq q_0/(1+q_0)>0$ defined via    
\begin{align}\label{chap:robustICA:sect:signature_identifiability:constants:1}
q_0 \, \coloneqq\, (\gamma k_d r_0)^{-1} \ \text{ and } \ r_0\,\coloneqq\,\kappa_0\varsigma_1\!\Bigg[\tfrac{\big\|B_R\big\|}{\sqrt{d}}\big(1 + \kappa_0 + (1+\xi d)\kappa_0\kappa_2(B_R)\big) + \kappa_2(B_R)\varsigma\Bigg].
\end{align}
Let further $c_1\coloneqq 2dk_d r_0$ and $c_2\coloneqq \sqrt{d}\kappa_2(B_R)c_1$, and note that \eqref{chap:robustICA:sect:signature_identifiability:constants:1} involves the terms
\begin{equation}
\varsigma_1\coloneqq k_d^{-1}\sqrt{\sum_{i=1}^d\frac{(i-1)^2\langle\zeta\rangle_{ii}^3}{\langle\zeta\rangle_{iii}^2}}, \quad\xi\,\coloneqq\,\sqrt{\sum_{\nu=1}^d\big\|[A_R\cdot\zeta]_\nu\big\|^2}, \quad\text{and} \quad \varsigma\coloneqq\sqrt{\sum_{i=1}^d\langle\zeta\rangle_{iii}^2/\langle\zeta\rangle_{ii}^3}, 
\end{equation}  
cf.\ \eqref{def:coordinates:eq1}, \eqref{rem:choiceofR:eq1}. In the following theorems, the above constants should not be taken too seriously in terms of optimality, as we have not attempted to optimize the sharpness of our inequalities.     
\end{notation}

The following theorem is a cornerstone for our general robustness result (Theorem \ref{thm:robustness}). 
\begin{theorem}[Blind Inversion]\label{thm:robust_ica}
Let $(\chi,\zeta,A)$ be a BSS-triple on $\mathscr{S}_0\times\GL$ and $\hat{\Phi}$ as in \eqref{lica:inversemap1}. If $\delta\coloneqq\delta_{\independent}(\zeta) \leq \varepsilon_0$, then for each $\theta_\star\in\hat{\Phi}(\chi)$ there is an $M\in\M$ and $E\in\R^{d\times d}$ such that 
\begin{equation}\label{thm:robust_ica:eq1}
\theta_\star = M(\mathrm{I} + E)A^{-1} \ \ \text{ with } \ \ \|E\|\,\leq\, c_1\tfrac{\delta}{1-\delta} \quad\text{and}\quad \tfrac{\|\theta_\star - M\!A^{-1}\|}{\|M\!A^{-1}\|}\leq c_2\tfrac{\delta}{1-\delta}\,.
\end{equation}
Moreover, $\partial_{\hat{\Phi}}(\zeta, A)\!\leq\! c_2\tfrac{\delta}{1-\delta}$ for the deviance \eqref{def:deviance}, and $(\sI_\star, \hat{\Phi})$ is an ICA-solution for the IA  
\begin{equation}\label{thm:robust_ica:eq2}
\sI_\star\coloneqq\mathscr{S}_\star\times\GL \ \ \text{ with }\ \ \mathscr{S}_\star\coloneqq\big\{\mu\in\mathscr{S}_0\ \big| \ \mu \text{ is orthogonal }\big\}.
\end{equation}    
\end{theorem}
\begin{proof}
The proof of \eqref{thm:robust_ica:eq1} combines standard arguments from matrix analysis with recent perturbation bounds for eigenspaces of (nonsymmetric) matrices; it is rather lengthy and hence deferred to Section \ref{chap:robustICA:subsect:signature_identifiability:proof}. The remaining assertions are direct corollaries of this proof and \eqref{thm:robust_ica:eq1}: 

Statement \eqref{thm:robust_ica:eq1} guarantees that for each $B\in\hat{\Phi}(\chi)$ there is $M\in\M$ and $E\in\R^{d\times d}$ with \begin{equation}\label{thm:robust_ica:eq2.1}
\tfrac{\!|BAu - Mu|}{|Mu|} \,=\, \tfrac{|M\cdot E\cdot M^{-1}(Mu)|}{|Mu|}\,\leq\,\|M\cdot E\cdot M^{-1}\|_2 \,\leq\,c_2\tfrac{\delta}{1-\delta}, \quad \text{for each } \ u\in \R^d\setminus\{0\},
\end{equation} 
where the last inequality holds by the same argument as for \eqref{thm:robust_ica:aux61.2}. This together with the definition \eqref{def:deviance} of the deviance function $\partial_{\hat{\Phi}}$ then implies the asserted $\partial_{\hat{\Phi}}(\zeta,A)$-inequality.  

The $\lceil\,\cdot\,\rceil_{\mathfrak{m}}$-sufficiency of \eqref{thm:robust_ica:eq2} follows from \eqref{thm:robust_ica:eq1} and the fact that $\delta_{\independent}(\mathscr{S}_\star)=0$, cf.\ \eqref{rem:icdefect-nonstationary:eq1}. Indeed: For any BSS-triple $(\chi,\zeta,A)_{\sI_\star}$ we have, in the notation of \eqref{thm:robust_ica:aux2.1}, that $(\chi,\tilde{\zeta},\tilde{A})_{\sI_\star}$ and hence $\mathrm{I}=R C_\chi R^\intercal=(R\tilde{A})C_{\tilde{\zeta}}(R\tilde{A})^\intercal = \tilde{A}_R\tilde{A}_R^\intercal$ and thus $\tilde{B}_R\coloneqq\tilde{A}_R^{-1}\in\mathrm{O}_d\subseteq\Xi_1$ (cf.\ \eqref{lica:xi1_domain}) for each ($\Xi_1$-defining choice of) $\kappa_0\geq 1$. Statement \eqref{thm:robust_ica:eq1} thus applies to [the triple $(\chi,\tilde{\zeta},\tilde{A})_{\sI_\star}$ and] the given map $\hat{\Phi}$ in particular, implying $\hat{\Phi}(\chi)\subseteq\M\cdot\tilde{A}^{-1}=\M\cdot A^{-1}$ as desired.        
\end{proof}

As is to be expected, the deviation \eqref{thm:robust_ica:eq1} between $\mathrm{M}_d\cdot A^{-1}$ and the recovered inverses $\theta_\star\in\hat{\Phi}(\chi)$ recovered from $\chi$ --- and likewise the deviation $\mathrm{dist}_{\tilde{\varrho}}(\hat{\Phi}(X)\cdot X, \lceil S\rceil_{\mathfrak{m}})$, via \eqref{lem:robustness-inversionstability:eq1}, between the exact quasi sources and their estimates --- are expressed on a multiplicative scale [resp.\ in terms of relative error] rather than on an additive scale [resp.\ in absolute error], owing to the multiplicative nature of the group $\GL$ and its action on $\mathcal{M}_1$.\\[-0.75em]

We also note that the proof of \eqref{thm:robust_ica:eq1} (see Section \ref{sect:mainproofs}) allows to quantify the breakdown point of the blind inversion procedure \eqref{lica:inversemap1} as being reached at exactly those causes $(\zeta, A)$ in $\mathfrak{C}$ for which one of the eigenvalues $(\tilde{\lambda}_i)$ in \eqref{thm:robust_ica:aux49} attains an algebraic multiplicity strictly greater than one; we can interpret these causes as the `singularities' of $\hat{\Phi}$ in the causal space $\mathfrak{C}$.        

\begin{remark}\label{rem:lica:proxcomponents}Let us revisit the usual case \eqref{measures_BSStriple} that $\zeta$ and $\chi$ are the laws $S$ and $X$ of some $\R^d$-valued processes $\bm{S}=(\bm{S}_t)_{t\in\I}$ and $\bm{X}=(\bm{X}_t)_{t\in\I}$, respectively, see \eqref{measures_BSStriple}. Denoting $\hat{\bm{S}}\coloneqq\theta_\star\cdot \bm{X}$ for any $\theta_\star\in\hat{\Phi}(X)$, Theorem \ref{thm:robust_ica} and \eqref{lem:robustness-inversionstability:eq1} then yield\footnote{\ Notice that if $|M\cdot \bm{S}_{t}| = 0$ then $\hat{\bm{S}}_t=\theta_\star A \bm{S}_{t}=0$, and for such $t\in\I$ the quotient in \eqref{rem:lica:proxcomponents:eq1} is set to $0/0\coloneqq 0$.} that
\begin{equation}\label{rem:lica:proxcomponents:eq1}
\tilde{\varrho}(\hat{\bm{S}}, M\bm{S})\equiv\sup\nolimits_{t\in\I}\tfrac{|\hat{\bm{S}}_t - M\bm{S}_t|}{|M\bm{S}_t|}\leq c_2\,\hat{\delta} \quad\text{for some } \ M=M(\theta_\star)\in\M 
\end{equation}
and with $\hat{\delta}\coloneqq\tfrac{\delta_{\!\independent\!}(S)}{(1-\delta_{\!\independent\!}(S))}$ and $c_2\equiv c_2(S,A)$ as above, provided that $\delta_{\independent}(S)\leq\varepsilon_0\equiv\varepsilon_0(S,A)$. If in addition to the `global' deviation \eqref{rem:lica:proxcomponents:eq1} one is also interested in the reconstruction error for a single component $\bm{S}^i=(\bm{S}^i_t)_{t\in\I}$ of the source (any $i\in[d]$), then one may consider the proportionality factor $p^i_t\coloneqq|\hat{\bm{S}}_t|/|\hat{\bm{S}}^i_t|\cdot\mathbbm{1}_\times(\hat{\bm{S}}^i_t)$ and note that by \eqref{rem:lica:proxcomponents:eq1}, with $M\eqqcolon(\beta_i\cdot\delta_{\sigma(i)j})_{ij}$, 
\begin{equation}
\tfrac{\big|\hat{\bm{S}}^i_t - \beta_i\bm{S}^{\sigma(i)}_t\big|}{\big|\beta_i\bm{S}^{\sigma(i)}_t\big|}\cdot\imath \,\leq\,\frac{p^i_t\cdot\delta'}{1 - p^i_t\cdot\delta'} \qquad\text{ for each } \ t\in\I \ \text{ with } \ (p^i_t\cdot\delta')\vee\tilde{\delta} < 1,
\end{equation}
where $\imath\coloneqq\mathbbm{1}_\times(\hat{\bm{S}}^i_t)\mathbbm{1}_\times(\bm{S}^{\sigma(i)}_t)$ and $\delta'\coloneqq \tilde{\delta}/(1-\tilde{\delta})$, $\tilde{\delta}\coloneqq c_2\hat{\delta}$; see e.g.\ \cite[Proposition 1.12]{green2003}. \hfill $\bdiam$  
\end{remark}
As mentioned above, the general idea of performing a blind inversion of linearly mixed stochastic sources by jointly diagonalising a suitably chosen family of equivariant tensor statistics of the observed signal -- as done in \eqref{lica:inverseset} -- is a classical one in the current theory of blind source separation. A brief outline of how the results of Theorem \ref{thm:robust_ica} and the arguments behind it fit into existing work for this joint-diagonalisation approach is given in Section \ref{sect:perturbedjdp:relatedwork}.       

\subsection{A Robust ICA-Inversion}\label{sect:robustsolution}In this section, we `globalise' the identifiability result of Section \ref{sect:blindinversionthm1} by embedding it into the robustness framework of Sections \ref{sect:robustica-basic} and \ref{chap:robustICA:sect:premetric}. More specifically, we show that the inversion map $\hat{\Phi}$ from \eqref{lica:inversemap1} can be readily extended to an ICA-solution $(\sI, \hat{\Phi})$ that is robust in the general sense of (both \eqref{BSS:robust_prelim1} and) Definition \ref{def:robustICA}. 

\begin{definition}[$(\sI_\star, \hat{\Phi})$]Let $\sI_\star$ be as in \eqref{thm:robust_ica:eq2} and consider the set-valued inversion map  
\begin{equation}\label{sect:robustsolution:eq1} 
\hat{\Phi} \, : \, \dot{\fD} \longrightarrow \mathfrak{F}, \quad \mu \,\mapsto\, \hat{\Phi}(\mu)\coloneqq\left[\argmin_{\theta\,\in\,\tilde{\Xi}_1}\phi_\mu(\theta)\right]\cdot R_{(\mu)}    
\end{equation} 
defined as in \eqref{lica:inversemap1} and \eqref{lica:inverseset} but with $\Xi_1$ generalised to compact domains $\tilde{\Xi}_1$ of the form \eqref{lica:xi1_domain2}. (Since the difference between \eqref{lica:inversemap1} and \eqref{sect:robustsolution:eq1} is minor, we give both maps the same symbol.)
\end{definition}
\noindent
Recall from Theorem \ref{thm:robust_ica} that the above pair $(\sI_\star, \hat{\Phi})$ is an ICA-solution, that is 
\begin{equation}
\hat{\Phi}(A\cdot\tilde{\mu})\cdot A \,\subseteq\, \M \quad \text{ for each } \ (\tilde{\mu},A)\in\sI_\star, 
\end{equation}
and recall from Lemma \ref{lemma1}, cf.\ also Remark \ref{rem:discrete-case}, that the family $\sI_\star$ of exactly identifiable causes is far from empty and in fact very large.\\[-0.75em]

In line with its blueprint \eqref{def:blindinversion:eq2}, the inversion map \eqref{lem:solmap-matrixvalued:eq1} associated with the M-estimator \eqref{sect:robustsolution:eq1} essentially parameterises (a subset of) the maximal solution \eqref{appendix:sect:BSSformal:intext:eq:maxsol_prodset}. To ensure an informative robustness analysis à la \eqref{BSS:robust_prelim1} of this inversion map,\footnote{\ For formal consistency with Section \ref{sect:robustica-basic}, the following Theorem \ref{thm:robustness} considers a `full-domain' extension $\hat{\Phi} : \mathcal{M}_1\rightarrow\mathfrak{F}$ of \eqref{sect:robustsolution:eq1}, where $\hat{\Phi} : \dot{\fD}\rightarrow\mathfrak{F}$ is as in \eqref{sect:robustsolution:eq1} and $\hat{\Phi} :\mathcal{M}_1\!\setminus\dot{\fD}\rightarrow\mathfrak{F}$ is defined arbitrarily. More meaningful $\mathcal{M}_1$-extensions of \eqref{sect:robustsolution:eq1} can be defined, e.g.\ as per Remark \ref{rem:roughpath-extensions}.} we follow \eqref{lica:xi1_domain} and exhaust the set $\sI_\star$ of identifiable causes by regular sublevel sets \eqref{rem:xi1_domain2:eq1} related to the condition numbers of the cause.          

\begin{remark}[Condition Bounds]\label{rem:xi1_domain2}
Recall that the optimisation domains $\Xi_1=\Xi_1(\kappa_0)$ in \eqref{lica:xi1_domain} are defined in dependence of a prior condition bound $\kappa_0 < \infty$. To `globalise' this definition for use in \eqref{sect:robustsolution:eq1}, note that the `identifiability superset' $G\coloneqq\fD\times\GL$ can be exhausted via 
\begin{equation}\label{rem:xi1_domain2:eq1}
G = \bigcup\nolimits_{b\geq 1}G_b \qquad\text{ for }\quad G_b\coloneqq\big\{(\mu,A)\in G \ \big| \ \kappa_2\big(\bar{B}_{(A,\mu)}\big)\leq b \big\}
\end{equation} 
where for any given $(\mu, A)\in G$ the matrix $\bar{B}_{(A,\mu)}$ is defined analogous to \eqref{rem:choiceofR:eq4}, that is: $\bar{B}_{(A,\mu)}\coloneqq\mathrm{ddiag}(|b_1|, \ldots, |b_d|)^{-1}B_{(A,\mu)}\in\Xi^0_1$ with $B_{(A,\mu)}\equiv(b_1|\cdots|b_d)^{\intercal}\coloneqq (R_\mu A)^{-1}$ for $R_\mu\coloneqq R_{(A\mu)}$ the square root of the inverted (half) covariance $C_{A\mu}^{-1}$, cf.\ \eqref{rem:choiceofR:eq1}. For any $b\geq 1$ and $\Delta_\kappa>0$, we set 
\begin{equation}\label{lica:xi1_domain2}
\tilde{\Xi}_1 \coloneqq \Xi_1^0\cap\big\{\theta\in\GL \, \big| \, \kappa_2(\theta)\leq b + \Delta_\kappa\eqqcolon\tilde{\kappa}_0\big\}.  
\end{equation}
For any priorly chosen regularity parameter $b$ and tolerance $\Delta_\kappa>0$, the compact (see Lem.\ \ref{lem:xicompact}) sets \eqref{lica:xi1_domain2} are the domains we use in \eqref{sect:robustsolution:eq1}. Note that since for each $(\mu, A)\in G_b$ the condition number $\kappa_2(\bar{B}_{(A,\mu)})$ is continuous\footnote{\ More precisely, we have $\kappa_2(\bar{B}_{(A,\mu)}) = \tilde{\varphi}_A([\mu]_0)$ for the continuous (at the point $[\mu]_0\in\GL$; see Lemma \ref{lem:choiceofR}) function $\tilde{\varphi}_A(\mathfrak{a})\coloneqq\kappa_2(\mathcal{R}(\mathcal{C}_{\mathfrak{a}})A\cdot\mathrm{ddiag}_i(|(A^{-1}\mathcal{R}(\mathcal{C}_{\mathfrak{a}})^{-1})^\intercal\cdot e_i|))\in\bar{\R}$; here we set $\kappa_2(\mathfrak{b})=+\infty$ if $\mathfrak{b}\notin\GL$.} in $[\mu]_0$, there is an (explicitly computable) $r_\ast = r_\ast((\mu,A),\Delta_\kappa)>0$ such that $\sup\{\kappa_2(\bar{B}_{(A,\tilde{\mu})})\mid \tilde{\mu}\in\fD\,:\, \|[\tilde{\mu}]_0 - [\mu]_0\|\leq r_\ast\}\leq b + \Delta_\kappa$.  \hfill $\bdiam$    
\end{remark}  

\noindent
Let us fix any $E$ as in \eqref{subsect:topology_identspace:eq1} and denote by $\tau_{\mathfrak{C}}$ its associated causal topology \eqref{def:causaltopology:eq1}. Let further $\tilde{\sI}\coloneqq\sI_\star\cap G_b$ for any fixed $b\geq 1$, which then defines $\tilde{\Xi}_1$ via \eqref{lica:xi1_domain2}, and set $\sI\coloneqq\tilde{\sI}\cap E$.\\[-0.75em]

The next result asserts that the ICA-solution $(\sI, \hat{\Phi})$ is robust in the sense of \eqref{BSS:robust_prelim1} and \eqref{def:robustICA:eq1}. It also quantifies this robustness by providing explicit moduli for the underlying continuities.           

\begin{theorem}[Robustness]\label{thm:robustness}
The ICA-solution $(\sI, \hat{\Phi})$ is robust in the sense of \eqref{BSS:robust_prelim1}$\,\&\,$\eqref{def:robustICA:eq1}. Specifically: Both the map $\mathfrak{I} : (\mathfrak{C}, \tau_{\mathfrak{C}})\rightarrow(\mathfrak{F}, \mathfrak{d})$ given by $(\mu,f)\mapsto\hat{\Phi}(f_\ast\mu)$, and the deviance $\partial_{\hat{\Phi}} : (\mathfrak{C}, \tau_{\mathfrak{C}})\rightarrow\R_+$ of $\hat{\Phi}$, are continuous on $\sI$. In fact, for each cause $(\mu_\star, A)\in\sI$ there are explicit constants $\tilde{c}_i=\tilde{c}_i(\mu_\star, A, c_\fS, \tilde{\kappa}_0)>0$, $i=0,1,2,3,4$, with the following property: Given $\varepsilon>0$ let $\tilde{\delta}\coloneqq \tilde{c}_0\wedge \big[\tilde{c}_1^{2/\alpha}\big(\varepsilon/(\tilde{c}_2 + \varepsilon)\big)^{\!2/\alpha}\big]$, then for each cause $\mathfrak{c}\equiv(\mu, f)\in \mathbb{B}_{\tilde{\delta}}(\mu_\star, A)$ we have:  
\begin{equation}\label{cor:robustness:eq1}  
\begin{gathered}
\text{for every } \ \theta\in\fI(\mathfrak{c}) \ \text{ there is } \ M=M(\theta)\in\M \text{ with } \ \|M\|\leq \tilde{c}_3 \, \text{ and}\\ 
E = E(\theta)\in\R^{d\times d} \text{ with } \ \|E\|\leq\varepsilon\ \text{ such that }\ \theta = M(\mathrm{I} + E)A^{-1};
\end{gathered}
\end{equation}
\begin{equation}\label{cor:robustness:eq2}
\text{further, } \ \ \ \quad \partial_{\hat{\Phi}}(\mathfrak{c})\,\leq\, \tilde{c}_4\big(\varepsilon + \bar{d}(A,f)(1+\varepsilon)\big).
\end{equation} 
If in fact $\mu\in\fD$ and $f\equiv A$, then the above holds as stated but for $\alpha=1$ and with simpler explicit constants $\tilde{c}_0, \ldots, \tilde{c}_4>0$ that can each be chosen independently of $c_\fS$.    
\end{theorem}     
\begin{proof} 
The proof of the $\sI$-continuity of $\fI$ is rather technical and hence delegated to Section \ref{sect:pf:thm:robustness}, whose notation we adopt here. As we will see next, the asserted constants $\tilde{c}_0, \ldots, \tilde{c}_4$ can be distilled from said proof. Indeed: Going through Section \ref{sect:pf:thm:robustness} for any fixed $\fp_\star\equiv(\mu_\star, A)\in\sI$, let first $\tilde{c}_0\coloneqq r_3$ for $r_3=r_3(\p_\star, c_\fS, \tilde{\kappa}_0)>0$ as defined in \eqref{thm:robustness:aux12}. Set further $\tilde{c}_1\coloneqq 1/K'$ and $\tilde{c}_2\coloneqq \tilde{K}_1$, for $K'\,(\equiv\left.K'\right|_{r=r_3})=K_{r_3}(\mu_\star, K_\fS)$ and $\tilde{K}_1=\tilde{K}_1(r_3)$ both as defined right after display \eqref{thm:robustness:aux12}. The $\theta$-related (first) assertion in \eqref{cor:robustness:eq1} then follows from \eqref{thm:robustness:aux14}. Indeed, \eqref{thm:robustness:aux14} gives the $\theta$-identity in \eqref{cor:robustness:eq1} with $\|E\|\leq\epsilon_r$, while the monotonicity of $[0, C_0\vee \hat{r}_\varepsilon]\ni \tilde{r}\mapsto \epsilon_{\tilde{r}}$ implies that $\|E\|\leq\epsilon_{\tilde{\delta}}\leq\epsilon_{\hat{r}_\varepsilon}=\varepsilon$ for $\hat{r}_\varepsilon\coloneqq \tilde{c}_1^{\alpha/2}(\varepsilon/(\tilde{c}_2 + \varepsilon))^{2/\alpha}$. The bound $\|M\|\leq \sqrt{d}\|A_{R_{f_\ast\mu}}\|_2\leq\sqrt{d}L_{\p_\star}\|A\|_2 \eqqcolon \tilde{c}_3$ holds by (\eqref{thm:robust_ica:aux59} and) \eqref{thm:robust_ica:aux56} and \eqref{thm:robustness:aux6.2}. 

We now prove \eqref{cor:robustness:eq2}, from which the $\sI$-continuity of $\partial_{\hat{\Phi}}$ follows immediately. For this, fix any $\mathfrak{c}\equiv(\mu,f)\in\mathbb{B}_{\tilde{\delta}}(\p_\star)$ as before and let $B\in\hat{\Phi}(f_\ast\mu)$ be arbitrary. Then by \eqref{cor:robustness:eq1} there is $M$ in $\M$, in fact with $\kappa_2(M)\leq\sqrt{d}\hat{\sigma}^\star_{\tilde{c}_0}L_{\p_\star}\kappa_2(A)\eqqcolon\tilde{c}_4$ (see \eqref{thm:robust_ica:aux57}, followed by \eqref{rem:choiceofR:eq2},\eqref{thm:robustness:aux6.2}), s.t.\
\begin{equation}\label{cor:robustness:aux1}
\eta\coloneqq\sup\nolimits_{u\in\R^d\setminus\{0\}}\tfrac{|B\!Au - Mu|}{|Mu|}\,\leq\, \tilde{c}_4\cdot\varepsilon\,,  
\end{equation}              
which follows by the same argument as in \eqref{thm:robust_ica:eq2.1}. Recalling \eqref{lem:robustness1:eq1}, \eqref{subsect:topology_identspace:eq2}, \eqref{intext:trafopremetric}, we have $f = A(\mathrm{I} + R)$ for $R\in C^1(\R^d;\R^d)$ with $\|D_R\|_\infty\leq\bar{d}(A,f)\leq\tilde{\delta}$, and since $f$ is a bijection we can choose $v\coloneqq -f^{-1}(0)\in\R^d$. Further $(\mathrm{I} + R)(u -v) = (\mathrm{I} + \tilde{R})(u)$ for $\tilde{R}(u)\coloneqq R(u-v) - v$, so that
\begin{equation}
\begin{aligned}
\tfrac{|B\circ f(u-v) - Mu|}{|Mu|} \,=\, \tfrac{|B\!A(\mathrm{I}+R)(u-v) - Mu|}{|Mu|} \,=\, \tfrac{|B\!Au - Mu + B\!A\tilde{R}(u)|}{|Mu|} \,\leq\, \eta + \tfrac{|BA\tilde{R}(u)|}{|Mu|}  
\end{aligned}
\end{equation}        
for each $u\in\R^d\setminus\{0\}$. As to the last summand, note $BA=M(\mathrm{I}+E)$ (by \eqref{cor:robustness:eq1}) and $\tilde{R}(0) =R(f^{-1}(0)) + f^{-1}(0) = A^{-1}(f(-v)) = 0$, so that for $\hat{R}\coloneqq \tilde{R}\circ M^{-1}$ and each $u\in\R^d$,
\begin{equation}
|B\!A\tilde{R}(u)|\leq\|M(\mathrm{I}+E)\|_2\big|\hat{R}(Mu)\big| \quad\text{and}\quad \big|\hat{R}(Mu)\big|\,\leq\, \|M^{-1}\|\|D_R\|_\infty|Mu|
\end{equation}
where the last inequality uses (the chain rule and) the mean-value theorem. All combined, 
\begin{equation}
\sup\nolimits_{u\in D_{\mu}^{(v)}}\tfrac{|B\circ f(u-v) - Mu|}{|Mu|}\mathbbm{1}_\times(u) \,\leq\, \tilde{c}_4\cdot\varepsilon \, + \, \kappa_2(M)(1+\varepsilon)\bar{d}(A,f) \,\leq\, \tilde{c}_4\big(\varepsilon + \bar{d}(A,f)(1+\varepsilon)\big)  
\end{equation}
from which the asserted $\partial_{\hat{\Phi}}$-inequality \eqref{cor:robustness:eq2} follows via \eqref{def:deviance}. The theorem's last assertion is a direct consequence of Remark \ref{rem:toproof:thm:robustness} \ref{rem:toproof:thm:robustness:it2} applied to \eqref{cor:robustness:eq1} and \eqref{cor:robustness:eq2}.        
\end{proof}  
We emphasize that the proof of Theorem \ref{thm:robustness} is entirely constructive, and that each of the above constants $c_i$ and $\tilde{c}_i$, while not fine-tuned for optimality, can be read off in explicit form as referenced in the proof.    

\section{Applications}\label{sect:applications}
\noindent 
The robustness of an ICA inversion map, as formulated in Definition \ref{def:robustICA} and established by Theorem \ref{thm:robustness}, is clearly of central importance in any situation where the exact structural assumptions \eqref{appendix:sect:BSSformal:def:ICA:eq1} of the ICA-model are violated or other sources of (non)systematic error, including approximations of the involved signals and their statistics, are unavoidable. This naturally includes virtually all practical applications, where the (exact) law of a signal is typically not available and its associated statistics can only be estimated from empirical data, which in turn is often corrupted by noise or subject to other forms of uncertainty. In this section, we present a selection of three essentially independent example applications (cf.\ Example \ref{example:cocktailparty}) to illustrate how our theory helps to conveniently achieve a meaningful quantification of various practice-relevant robustness scenarios.\\[-0.75em]  

Throughout this section, let $\bm{S}=(\bm{S}_t)$ be a stochastic process in $\R^d$ (defined on a probability space $(\Omega,\mathscr{F}, \mathbb{P})$) with an associated BSS-triple $(X,S,A)_{\sI_\ast}$ on \eqref{thm:robust_ica:eq2}, and consider $\hat{\Phi}$ as in \eqref{sect:robustsolution:eq1} but with $\tilde{\Xi}_1$ given by \eqref{lica:xi1_domain2} for $b\coloneqq\kappa_2(\bar{B}_{(A,S)})$ and any fixed $\Delta_\kappa>0$.  

\subsection{Statistical Estimation}\label{sect:applications:estimation}
In practice, we usually have access neither to the law $X\equiv\mathbb{P}_{\bm{X}}$ of the observable (which is needed to compute the statistics \eqref{def:coordinates:eq1}) nor to its continuous-time realisations, but only to time-discretized sample trajectories of $\bm{X}$. The issue of time-discretisation is easily resolved, since a sequence of (random) vectors $\bm{X}_{t_1}, \ldots, \bm{X}_{t_n}$ in $\R^d$ can be immediately identified with a (random) element of $\mathcal{C}_d$ via piecewise-linear interpolation, see Remark \ref{rem:discrete-case}. The inaccessibility of the law, however, is more subtle and must be addressed by estimating $X$ from the available data. While $X$ is generally an infinite-dimensional object, our inversion procedure \eqref{sect:robustsolution:eq1} only requires information about its low-dimensional projection \eqref{def:coordinates:eq2}, so any viable estimate $\hat{\mu}$ of $X$ needs to be accurate enough only at the low-order levels          
\begin{equation}\label{sect:applications:estimation:eq1}
\langle\hat{\mu}\rangle_w \approx\, \langle X\rangle_w = \mathbb{E}\big[\sig_w(\bm{X})\big] \quad \text{ for } \ |w|=2,3.     
\end{equation} 
What all reasonable estimators $\hat{\mu}=(\hat{\mu}_n)\subset\fD$ of $ X$ will have in common is the property of statistical \emph{consistency}, i.e.\ that in the `infinite data limit' $n\rightarrow\infty$ (see e.g.\ \cite{sigNICA2021} for details)
\begin{equation}\label{sect:applications:estimation:eq2}
\langle\hat{\mu}_n\rangle_w = \int_{\mathcal{C}_d}\!\sig_w(x)\,\hat{\mu}_n(\mathrm{d}x) \ \xrightarrow{n\rightarrow\infty} \ \langle X\rangle_w \quad \text{ for each } \ |w|=2,3, 
\end{equation}   
almost surely or in probability; the convergence \eqref{sect:applications:estimation:eq2} is but a topological specification of the informal accuracy requirement \eqref{sect:applications:estimation:eq1}.
\begin{remark}
To understand \eqref{sect:applications:estimation:eq2}, note that strictly speaking any sample-based approximation $(\hat{\mu}_n)$ of $ X$ is in fact naturally a sequence of random measures, i.e.\ a sequence of $\fD$-valued random variables $\hat{\mu}_n \equiv \hat{\mu}(\bm{X},n) : \Omega\ni\omega\mapsto\hat{\mu}_n(\bm{X}(\omega))\in\fD$. This also pertains to the example estimators in Remark \ref{sect:applications:estimation:rem1} below, but we suppress this additional complexity in our notation to avoid cluttering our exposition.
\end{remark} 

\begin{remark}[Signature Estimators]\label{sect:applications:estimation:rem1}An `optimal' choice of $\hat{\mu}$ will generally depend on the data that is given, but for many applications useful estimates can be straightforward. For example, if $\bm{X}$ describes the temporal evolution of some $d$-dimensional health trajectory of a prototypical patient and our data consists of $n$ independently sampled path-valued\footnote{\ Upon piecewise-linear interpolation of discretely-measured data points, say.} observations $x_1, \ldots, x_n\in\mathcal{C}_d$, corresponding to the trajectories of $n$ different real patients, then the empirical measure $\hat{\mu}\equiv\hat{\mu}_n(x_1,\ldots, x_n)=\tfrac{1}{n}\sum_{j=1}^n\delta_{x_j}\in\fD$ can be sufficient when $n$ is large. For other applications, a good choice of $\hat{\mu}$ may be more involved and may require additional statistical assumptions on $\bm{X}$ to allow a feasible approximation of its law from its samples. A very common such case is when the data consists of a single trajectory $x=(x_t)$ of $\bm{X}$ only; in this case, common mixing or ergodicity assumptions on $\bm{X}$ [or, equivalently, on $\bm{S}$] often suffice to qualify the usual ergodic plug-in estimators $\hat{\mu}\equiv\hat{\mu}_n(x)=\tfrac{1}{n}\sum_{j=1}^n\delta_{\left.x\right|_{((j-1)\tau, j\tau]}}$ ($\tau>0$ large enough) as suitable approximators for $ X$; see e.g.\ \citep[Section 8]{sigNICA2021} for details. \hfill $\bdiam$       
\end{remark}

The stability results in Section \ref{chap:robustICA:sect:auxiliaries} provide explicit bounds to relate the convergence \eqref{sect:applications:estimation:eq2} to a controllable approximation of the inverse $A^{-1}$. This allows us to conveniently derive finite-sample results of the following type. 

\begin{notation}[Convergence Rate]\label{notation:estimators}
In line with classical theory, an estimator $\hat{\mu}\equiv\hat{\mu}_n$ of $ X$ is said to have convergence rate $\tilde{\alpha}_n$, for $(\tilde{\alpha}_n)$ some strictly decreasing null sequence, if $\gamma_n\coloneqq\max_{|w|=2,3}|\langle\hat{\mu}_n\rangle_w - \langle X\rangle_w|=O_P(\tilde{\alpha}_n)$, which is to say that for any $\epsilon>0$ there exist numbers $0<m_{\epsilon}, M_{\epsilon}<\infty$ such that $\sup_{n\geq m_\epsilon}\mathbb{P}(\gamma_n\geq M_\epsilon\tilde{\alpha}_n)\leq\epsilon$, see e.g.\ \citep[Section 2.2]{vdv1998}. 
\end{notation}
   
\begin{proposition}\label{prop:estimation-finsample}
Let $\hat{\mu}\equiv(\hat{\mu}_n)\subset\dot{\fD}$ be a consistent estimator of $ X$ with $\max\nolimits_{|w|=2,3}|\langle\hat{\mu}_n\rangle_w - \langle X\rangle_w|= O_P(\tilde{\alpha}_n)$ for some convergence rate $\tilde{\alpha}\equiv(\tilde{\alpha}_n)$. Then for any $\varepsilon>0$ and any $0\leq q<1$, there exists an explicitly computable threshold $n_0=n_0(\varepsilon, q, \tilde{\alpha}, \hat{\mu}, A,  S)\in\N$ such that for each $n\geq n_0$ the following holds with a probability of at least $q$:   
\begin{equation}\label{prop:estimation-finsample:eq1}
\text{for every } \ \hat{\theta}\in\Phi(\hat{\mu}_n) \ \text{ there is } \ M\in\M \quad \text{such that} \quad \sup\nolimits_{t\in\I}\tfrac{|\hat{\theta} \bm{X}_t - M\bm{S}_t|}{|M\bm{S}_t|}\mathbbm{1}_{\!\times}\!{\scriptstyle(\bm{S}_t)}\,\leq\,\varepsilon\,;  
\end{equation} 
in explicit terms, one such threshold $n_0$ is given by the formula \begin{equation}\label{prop:estimation-finsample:eq2}
n_0\coloneqq \min\left\{n\geq m_{1-q} \ \middle|\ \tilde{\alpha}_n < \frac{\eta(\tilde{\delta}(\varepsilon/\tilde{c}_4))}{(1 + \|A^{-1}\|_\infty)^3 M_{1-q}}\right\}
\end{equation}
for any $m_{1-q}$ and $M_{1-q}$ as in Notation \ref{notation:estimators}, $\tilde{c}_4$ and $\tilde{\delta}\equiv\tilde{\delta}(\epsilon\,; \tilde{c}_i)$ as in Theorem \ref{thm:robustness} (for $\alpha\equiv 1$ and independent of $c_{\mathfrak{S}}$), and $\eta\coloneqq \omega_{\varphi}^{-1} : \R_+ \rightarrow \R_+$ the inverse of the modulus of continuity $\omega_{\varphi}(\tilde{\eta})\coloneqq\sup\{\varphi(v)\mid v\in\mathcal{V}, \, \|v - [S]_{\mathfrak{c}}\|_\infty\leq\tilde{\eta}\}$ at $[S]_{\mathfrak{c}}$ of the function $\varphi =\varphi([\tilde{\mu}]_{\mathfrak{c}}): \mathcal{V}\rightarrow\R_+$ defined by the RHS (for $\mu=S$ fixed) of $\left.\eqref{sect:coredinates:eq4}\right|_{\mu=S}$.           
\end{proposition}  
\begin{proof}
Let $B\coloneqq A^{-1}$ and $\upsilon_n\coloneqq B\cdot\hat{\mu}_n$. Lemma \ref{lem:linequiv} then asserts the identities of 2- and 3-tensors $[\upsilon_n]^{(2)}\coloneqq(\langle\upsilon_n\rangle_{ij}) = B^{\otimes 2}(\langle\hat{\mu}_n\rangle_{ij})$ and $[\upsilon_n]^{(3)}\coloneqq(\langle\upsilon_n\rangle_{ijk}) = B^{\otimes 3}(\langle\hat{\mu}_n\rangle_{ijk})$, where $B^{\otimes 2}=B\otimes B$ and $B^{\otimes 3}=B\otimes B\otimes B$ are Kronecker powers of $B$, see also \citep[Proposition 7.53]{FVI}. Since the action of $B^{\otimes m}$ on $(\R^d)^{\otimes m}$ is continuous, we obtain from the [in probability] convergence \eqref{sect:applications:estimation:eq2} by way of the continuous mapping theorem that also $[\upsilon_n]^{(m)} = B^{\otimes m}[\hat{\mu}_n]^{(m)}\longrightarrow_{n\rightarrow\infty} B^{\otimes m}[ X]^{(m)} = [S]^{(m)}$ [in probability] for $m=2,3$. In fact we have $\tilde{\gamma}_n\coloneqq\max_{|w|=2,3}|\langle\upsilon_n\rangle_w - \langle S\rangle_w| = \|[\upsilon_n]_{\mathfrak{c}} - [S]_{\mathfrak{c}}\|_\infty \leq \|\underline{B}\|_{\infty}\|[\hat{\mu}_n]_{\mathfrak{c}} - [X]_{\mathfrak{c}}\|_\infty = C\gamma_n$, for the (direct-sum) linear operator $\underline{B}\coloneqq B^{\otimes 2} + B^{\otimes 3}$ on $(\R^d)^{\otimes 2}\oplus(\R^d)^{\otimes 3}\cong\mathcal{V}$ and its $\ell_\infty$-induced operator norm $C\coloneqq \|\underline{B}\|_\infty \,(\leq \max\{\|B^{\otimes 2}\|_\infty, \|B^{\otimes 3}\|_\infty\})\leq \max\{\|B\|_\infty^2, \|B\|_\infty^3\}$. Take now $\tilde{\varepsilon}>0$, and set $\varepsilon\coloneqq\tilde{\varepsilon}/\tilde{c}_4$ for $\tilde{c}_4\equiv\tilde{c}_4(S, A)>0$ as in Theorem \ref{thm:robustness}. The latter then provides an explicit radius $\tilde{\delta}=\tilde{\delta}(\varepsilon,S, A, \tilde{\kappa}_0)>0$ such that \eqref{cor:robustness:eq2} holds for $\mu_\star=S$, i.e.\ for which
\begin{equation}\label{prop:estimation-finsample:aux1}
\sup\nolimits_{(\upsilon, A)\in\mathbbm{B}_{\tilde{\delta}}(S, A)}\partial_{\hat{\Phi}}(\upsilon,A) \,\leq\, \tilde{\varepsilon}. 
\end{equation}
To connect this with \eqref{sect:applications:estimation:eq2}, note that by definition \eqref{def:delta:eq1} of $\delta(\cdot, \cdot)$ and continuity\footnote{\ More precisely, by the fact that $\delta(S, \upsilon) = \varphi([\upsilon]_{\mathfrak{c}})$ for $\varphi : \mathcal{V}\rightarrow\R_+$ continuous with $\varphi([S]_{\mathfrak{c}})=0$.} there is an $\eta>0$ such that $\delta(S,\upsilon)< \tilde{\delta}$ for each $\upsilon\in\dot{\fD}$ with $\max_{|w|=2,3}|\langle\upsilon\rangle_w - \langle S\rangle_w|<\eta$. Now fix any $0\leq q < 1$ and let $q'\coloneqq 1-q$. Since $\gamma_n=O_P(\tilde{\alpha}_n)$, we can find $m_q, M_q>0$ with $\sup_{n\geq m_q}\mathbb{P}(\gamma_n\geq M_q\cdot\tilde{\alpha}_n)\leq q'$. Due to $\tilde{\alpha}_n\downarrow 0$, we can choose a (smallest) $n_0\geq m_q$ such that $\tilde{\alpha}_n\leq\eta/(CM_q)$ for each $n\geq n_0$. For this $n_0$ we have $\sup_{n\geq n_0}\mathbb{P}(\tilde{\gamma}_n\geq \eta)\leq \sup_{n\geq n_0}\mathbb{P}(C\gamma_n\geq \eta) \leq \sup_{n\geq n_0}\mathbb{P}(\gamma_n\geq M_q\tilde{\alpha}_n)\leq\sup_{n\geq m_q}\mathbb{P}(\gamma_n\geq M_q\tilde{\alpha}_n)\leq q'$ and thus $\inf_{n\geq n_0}\mathbb{P}(\tilde{\gamma}_n<\eta)\geq 1-q' = q$. Hence for each $n\geq n_0$, it holds with probability at least $q$ that $\mathbbm{d}((\upsilon_n, A), (\mu_S, A)) = \delta(S, \upsilon_n) < \tilde{\delta}$, i.e.\ that $(\upsilon_n, A)\in\mathbbm{B}_{\tilde{\delta}}(S, A)$. Assertion \eqref{prop:estimation-finsample:eq1} now follows from \eqref{prop:estimation-finsample:aux1} and Lemma \ref{lem:robustness-inversionstability} (noting that here, by \eqref{cor:robustness:aux1}, the spatial support in \eqref{lem:robustness-inversionstability:aux3} can be replaced by the entirety of $\R^d$). The formula \eqref{prop:estimation-finsample:eq2} is but a summary of this proof.
\end{proof}

\subsection{Asynchronous Observations}
Apart from statistical estimation \eqref{sect:applications:estimation:eq2}, another source of systematic error in blind inversion of time-dependent signals are deviations from the instantaneity assumption \eqref{appendix:sect:BSSformal:eq1} that are caused by an asynchronous time discretisation of the components $\bm{X}^1,$ $\ldots,$ $\bm{X}^d$ of $\bm{X}$, cf.\ e.g.\ \cite{aitsahalia2010} for the classical context of covariance estimation:

In practice, the channel signals $\bm{X}^i$ are often each observed by different sensors that take their measurements at some discrete sets of time-points which may differ from sensor to sensor. Empirical observations (e.g.\ \citep{li2014asynchronism, signalstack2021}) demonstrate that this asynchronicity affects BSS-performance to the point of dysfuntionality. As the following proposition shows, the causal topology \eqref{def:causaltopology:eq1} is coarse enough to capture this `perturbation by asynchronicity', so that its effect can then be quantified and analysed via Theorem \ref{thm:robustness}.\\[-0.75em]   

\noindent
To do this, we can formalise the synchronicity infringement of \eqref{appendix:sect:BSSformal:eq1} as follows: Write $\bm{X}^i_{\mathcal{J}_i}\coloneqq\big(\bm{X}^i_t\ \big| \ t\in\mathcal{J}_i\coloneqq\{0\leq t^{(i)}_0 < t^{(i)}_1 < \ldots < t^{(i)}_{n_i-1}\leq 1\}\big)$ for the sensor data of channel $\bm{X}^i$, denoting $\hat{\bm{X}}^i_{\mathcal{J}_i}$ for the piecewise-linear interpolation of that data. Formally, $\hat{\bm{X}}^i_{\mathcal{J}_i}=\hat{\iota}_{n_i}(\bm{X}^i_{\mathcal{J}_i})$ for $\hat{\iota}_{n_i}$ as in \eqref{rem:discrete-case:eq1} for $d=1$, see Remark \ref{rem:discrete-case}. We then say that the channels $\bm{X}^i$ are observed synchronously if $\mathcal{J}_1 = \mathcal{J}_2 = \ldots = \mathcal{J}_d$, otherwise we say that the $\bm{X}^1, \ldots, \bm{X}^d$ are observed asynchronously. \\[-0.75em]  

Writing $\|\mathcal{J}_i\|\coloneqq\max_{\nu\in[n_i-1]}|t^{(i)}_\nu-t^{(i)}_{\nu-1}|$ for the mesh-size of $\mathcal{J}_i$, we can handle asynchronicity via Theorem \ref{thm:robustness} by noting that the interpolated observations $\hat{\bm{X}}^\ast_{(\mathcal{J}_i)}\coloneqq\big(\hat{\bm{X}}^1_{\mathcal{J}_1}, \cdots, \hat{\bm{X}}^d_{\mathcal{J}_d}\big)$ of $\bm{X}$ over the channel-dependent sample times $(\mathcal{J}_i)\equiv(\mathcal{J}_1, \ldots, \mathcal{J}_d)$ converge to $\bm{X}$ in \eqref{def:causaltopology:eq1} as $\|(\mathcal{J}_i)\|\coloneqq\max_{i}\|\mathcal{J}_{i}\|\rightarrow 0$. The symbol $\tilde{\varrho}$ below denotes the relative error from \eqref{lem:robustness-inversionstability:eq1}.  

\begin{proposition}
For each $i\in[d]$, let $(\mathcal{J}_i^n)_{n\in\N}$ be a sequence of dissections of $[0,1]$ whose mesh-size $\|\mathcal{J}_i^{n}\|$ converges to zero. Suppose further that $\E\|\bm{X}\|^3_{1\text{-}\mathrm{var}} < \infty$ and let $\hat{\mu}_n\coloneqq\mathbb{P}_{\hat{\bm{X}}^\ast_n}\in\dot{\fD}$ be the law of $\hat{\bm{X}}^\ast_n\coloneqq\hat{\bm{X}}^\ast_{(\mathcal{J}_i^n)}$. Then for each $\varepsilon>0$ there is an explicit $n_0=n_0(\varepsilon,A,S, (\hat{\mu}_n))\in\N$ such that the error bound $\sup_{n\geq n_0}\operatorname{dist}_{\tilde{\varrho}}(\hat{\Phi}(\hat{\mu}_n)\cdot\bm{X}, \lceil\bm{S}\rceil_{\mathfrak{m}})\leq\varepsilon$ holds almost surely.            
\end{proposition} 
\begin{proof}
This follows from Theorem \ref{thm:robustness} and a few standard convergence arguments. Indeed: Denoting by $\hat{\pi}_n^i : \mathcal{C}^1_1\rightarrow \mathcal{C}^1_1$ (cf.\ \eqref{notation:absolutelycontinuous}) the linear operator that sends a $\mathcal{C}^1_1$-path $y : [0,1]\rightarrow\R$ to its $\hat{\mathcal{J}}_n^i$-piecewise-linear interpolation $\hat{y}_{\mathcal{J}^i_n}\coloneqq\hat{\iota}_{|\mathcal{J}^i_n|}(\pi_{\mathcal{J}^i_n}(y))$ (cf.\ Remark \ref{rem:discrete-case}), note that $\|\hat{\pi}^i_n(y)\|_{1\text{-}\mathrm{var}}\leq\|y\|_{1\text{-}\mathrm{var}}$ for all $y\in\mathcal{C}^1_1$ and each $n\in\N$, and further that $\lim_{n\rightarrow\infty}\hat{\pi}^i_n = \mathrm{id}$ pointwise [on $\mathcal{C}^1_1$] wrt.\ $\|\cdot\|_\infty$, see \citep[Thm.\ 5.23]{FVI}. Consequently, for each $x\in\mathcal{C}^1_d$ and $n\in\N$,
\begin{equation}
d^{-1/2}\|x - \hat{x}^\ast_n\|_\infty \,\leq\, \max_{i\in[d]}\big\|x^i - \pi_i(\hat{x}^\ast_n)\big\|_\infty = \max_{i\in[d]}\big\|x^i - \hat{\pi}^i_n(x^i)\big\|_\infty \ \longrightarrow \ 0 \ \ (\text{as } n\rightarrow\infty)\vspace{-0.5em}    
\end{equation}   
and $\|\hat{x}^\ast_n\|_{1\text{-}\mathrm{var}}\leq \sum_{i=1}^d\|\hat{\pi}^i_n(x^i)\|_{1\text{-}\mathrm{var}}\leq\sum_{i=1}^d\|x^i\|_{1\text{-}\mathrm{var}}$ (recall \eqref{sect:pweak-convergence:eq2} for the first inequality). By \citep[Lemma 5.27 (i)]{FVI}, the above implies that $\lim_{n\rightarrow\infty}\|x - \hat{x}^\ast_n\|_{p\text{-}\mathrm{var}} = 0$ for each $x\in\mathcal{C}^1_d$, for any $p>1$. Hence by the $p$-variation continuity of $\sig(\cdot)$ (for $p\in[1,2)$) we have
\begin{equation}
\lim_{n\rightarrow\infty}\sig_w(\hat{x}^\ast_n) \,=\, \sig_w(x) \qquad (\text{each } |w|\leq 3)\vspace{-0.5em}
\end{equation}
for each $x\in\mathcal{C}^1_d$. Due to this and the $L_1(\mathbb{P}_{\bm{X}})$-domination $|\sig_w(\hat{x}^\ast_n)|\lesssim \|\hat{x}^\ast_n\|^3_{1\text{-}\mathrm{var}}\leq d^3\|x\|_{1\text{-}\mathrm{var}}^3$ (which holds by the signature estimate underlying \eqref{lem:unifsigintegr:sufficient:aux1.1}), we can apply dominated convergence to obtain \eqref{sect:applications:estimation:eq2}, i.e.\ that $[\hat{\mu}_n]_{\mathfrak{c}} \rightarrow [X]_{\mathfrak{c}}$ and hence $\delta(\hat{\mu}_n, X)\rightarrow 0$ as $n\rightarrow\infty$. The claimed conclusion, which is equivalent to \eqref{prop:estimation-finsample:eq1}, now follows as in the proof of Proposition \ref{prop:estimation-finsample}.  
\end{proof} 

\subsection{Additive and Multiplicative Noise}\label{sect:noise}Despite ever-increasing efforts to develop more sophisticated measuring instruments and data-processing techniques, empirical observations such as experimental measurements or econometric records are never exact. Instead, any kind of observation-based data is prone to unwanted deviations and corruptions, whose effects need to be understood as they may otherwise invalidate the (algorithmic) consistency of a given statistical method. One way to account for this is through mathematical `noise models', the most established of which are \emph{additive} and \emph{multiplicative} noise affecting a signal $\bm{Y}=(\bm{Y}_t)_{t\in\I}$,
\begin{equation}\label{sect:noise:eq1}
\bm{Y}_t \, = \, \bm{Y}_t + \bm{\eta}_t \qquad\text{or}\qquad \bm{Y}_t \, = \, \bm{\eta}_t\cdot \bm{Y}_t \qquad (t\in\I) 
\end{equation}for some stochastic process\footnote{\ The multiplication $\bm{\eta}_t\cdot \bm{Y}_t$ is understood componentwise, i.e.\ $\bm{\eta}_t\cdot \bm{Y}_t \equiv (\bm{\eta}_t^1 \bm{Y}^1_t, \cdots, \bm{\eta}_t^d \bm{Y}_t^d)^\intercal$ for each $t\in\I$.} $\bm{\eta} = (\bm{\eta}_t)_{t\in\I}$ in $\R^d$ referred to as the ``noise''. For noisy BSS models \cite{MOU}, the noise $\bm{\eta}$ is typically subject to additional, often (semi)parametric and usually quite constricting structural assumptions, see e.g.\ \citep{belkin2013, HBS, kawanabe2005,kervazo2021robust} and the references therein.\footnote{\ To avoid dealing with any identifiability issues at all, some authors impose rather restrictive (`contrasting') structural assumptions on $\bm{S}$ or $\bm{\eta}$ to ensure that the additive noise model \eqref{sect:noise:eq2} remains exactly invertible. Notable such assumptions include: time-dependence (of $\bm{S}$) vs.\ time-independence (of $\bm{\eta}$) \cite{choi2000blind}, non-Gaussianity ($\bm{S}$) vs.\ Gaussianity ($\bm{\eta}$) \cite{hyvarinen1999fast}, or time-variability of [auto]covariances ($\bm{S}$) vs.\ group-wise stationary confounding ($\bm{\eta}$) \cite{pfister2019}. Of course, these cases are all hand-picked exceptions, and the treatment of the general, model-free case \eqref{sect:noise:eq2} requires a more comprehensive analytical framework, see Proposition \ref{prop:noise}.} The generality of our topological framework allows us to keep such assumptions to a minimum.\\[-0.5em]   

In the classical BSS-context \eqref{appendix:sect:BSSformal:eq1} where $f$ is linear, it is essentially irrelevant whether the additive noise \eqref{sect:noise:eq1} applies to $\bm{Y}=\bm{X}$ or $\bm{Y}=\bm{S}$ (or both) because 
\begin{equation}\label{sect:noise:eq2}
\bm{X} + \bm{\eta} \,=\, A(\bm{S}+\bm{\eta}') \qquad\text{for}\quad \bm{\eta}'\coloneqq A^{-1}\bm{\eta}\,;
\end{equation} 
for the corruption with multiplicative noise, we may distinguish\footnote{\ Since $A$ is a homoeomorphism, these cases, too, are essentially equivalent, but see the proof of Prop.\ \ref{prop:noise}.} the cases 
\begin{equation}\label{sect:noise:eq3}
\bm{X}^{\bm{\eta}}_t\coloneqq \bm{\eta}_t\cdot \bm{X}_t \quad\text{ or }\quad \bm{S}^{\bm{\eta}}_t\coloneqq \bm{\eta}_t\cdot\bm{S}_t \qquad(t\in\I).
\end{equation} 
Theorem \ref{thm:robustness} allows to quantify how these corruptions affect the accuracy of blind inversion.  

\begin{proposition}\label{prop:noise}
Let $\bm{\eta}$ be a noise process with law $\eta\in\fD$, and consider the corruptions 
\begin{equation}\label{prop:noise:eq1}
\tilde{\bm{X}}^{(1)}\coloneqq A(\bm{S} + \bm{\eta}) \quad \text{ and } \quad \tilde{\bm{X}}^{(2)}\coloneqq A \bm{S}^{\bm{\eta}} \quad \text{ and } \quad \tilde{\bm{X}}^{(3)}\coloneqq \bm{X}^{\bm{\eta}},
\end{equation}
with $\tilde{X}^{(i)}\coloneqq\mathbb{P}_{\tilde{\bm{X}}^{(i)}}$ their respective laws. There are then explicit constants $c_j = c_j(S,A,\Delta_\kappa)>0$ $(j=0,\ldots,3)$ with the following property: Given $\varepsilon>0$ let $\beta_\varepsilon\coloneqq c_0\wedge\!\big(c_1(\varepsilon/(c_2 + \varepsilon))^2\big)$, then
\begin{gather}\label{prop:noise:eq2}
\text{for each } \ \hat{\theta}\in\hat{\Phi}(\tilde{X}^{(i)}) \ \text{ there is } \ M\in\M \ \text{ with } \, \ \tfrac{\|\hat{\theta} - MA^{-1}\|}{\|MA^{-1}\|}\leq\varepsilon \quad\text{and}\\\label{prop:noise:eq3}
\max\big\{\partial_{\hat{\Phi}}(\mathbb{P}_{\bm{S}+\bm{\eta}}, A), \partial_{\hat{\Phi}}(S^\eta, A), \partial_{\hat{\Phi}}(S, \bm{\eta} A)^{\ast}\big\}\,\leq\,\varepsilon,
\end{gather}\let\thefootnote\relax\footnotetext{\label{footnote:noisepartial}\hspace{-1em}${\ }^{\ast}$\ \, Denoting $\bm{\eta} A : [0,1]\times\R^d\rightarrow\R^d,\ (t,u)\mapsto\bm{\eta}_t Au\coloneqq\mathrm{ddiag}(\bm{\eta}_t^1, \cdots, \bm{\eta}^d_t)\cdot Au$ (a random time-dependent transformation), the last inequality in \eqref{prop:noise:eq3} (`$\partial_{\hat{\Phi}}(S,\bm{\eta} A)\leq\varepsilon$') is to be read as: $\partial_{\hat{\Phi}}(S, \bm{\eta} A)\equiv \sup_{B\in\hat{\Phi}(\mathbb{P}_{(\bm{\eta}_t AS_t\,|\,t\in[0,1])})}\inf_{(M,v)\in\M\times\R^d}\sup_{t\in[0,1]}\sup_{u\in D^{(v)}_{S_t}}\!\!\tfrac{|B\circ (\bm{\eta}_tA)(u-v) - Mu|}{|Mu|}\mathbbm{1}_{\!\times}\!{\scriptstyle(Mu)}\leq\varepsilon$ with prob.\ one.}provided that $\bm{S}$ and the noise process $\bm{\eta}$ meet the following case-dependent assumptions: 
\begin{itemize}
\item for the case $i=1$, suppose that $\bm{\eta}$ and $\bm{S}$ are independent and both mean-stationary and that $\bm{\eta}$ satisfies the growth condition $\sum_{\nu=0}^d\langle S\rangle_{\nu\nu}^{-1}\big\|N_S^{-1}[\eta]_\nu N_S^{-1}\big\|^2 \leq \beta_\varepsilon^2$\,;
\item for the case $i=2$, suppose that $\bm{S}$ is mean-stationary and centered and that $\bm{\eta}$ is independent of $\bm{S}$ and satisfies the growth condition $\beta_2(S,\eta)\leq\beta_\varepsilon^2$;
\item for the case $i=3$, suppose that $\bm{\eta}$ is independent from $\bm{S}$, mean-stationary with $\E[\bm{\eta}]\equiv\mathrm{I}$ and IC and satisfies the growth condition $\beta_3(X,\eta)\leq\gamma_\varepsilon$ with $\beta_3$ from \eqref{prop:noise:aux14} and for $\gamma_\varepsilon=\gamma_\varepsilon(\beta_\varepsilon, A)>0$ as introduced directly thereafter;
\end{itemize} 
here, $\langle S\rangle_{00}\coloneqq 1$ and $\beta_2(S,\eta)\coloneqq\sum_{i=1}^d \langle S\rangle_{ii}^{-2}(\tilde{\alpha}_i^2 + \langle S\rangle_{ii}^{-1}\check{\alpha}_i^2)$ with $\tilde{\alpha}_i\coloneqq \eqref{prop:noise:aux8}$, $\check{\alpha}_i\coloneqq \eqref{prop:noise:aux9}$. 
\end{proposition}  
\begin{proof}
Thm.\ \ref{thm:robustness} provides constants $\tilde{c}_j = \tilde{c}_j(S,A)>0$ $(j=0,1,2)$ and $\tilde{c}_4=\tilde{c}_4(S,A,\Delta_\kappa)>0$ such that for any $\varepsilon>0$ and $r_{\tilde{\varepsilon}}\coloneqq \tilde{c}_0\wedge \big[\tilde{c}_1^2\big(\tilde{\varepsilon}/(\tilde{c}_2 + \tilde{\varepsilon})\big)^{\!2}\big]$ with $\tilde{\varepsilon}\coloneqq\varepsilon/\tilde{c}_4$, we have that
\begin{equation}
\text{if } \ \delta(S, \tilde{\mu})\leq r_{\tilde{\varepsilon}} \quad \text{then:} \quad \text{for every }\, \hat{\theta}\in\Phi(\!A\tilde{\mu}) \ \text{ there is } \  M\in\M \ \text{ with } \ \tfrac{\|\hat{\theta} - M\!A^{-1}\|}{\|M\!A^{-1}\|}\leq \varepsilon,
\end{equation} 
where the above implication applies to all $\tilde{\mu}\in\dot{\fD}$. To see that this holds, simply combine \eqref{cor:robustness:eq1} with the argument behind $\eqref{thm:robust_ica:aux61}\Rightarrow\eqref{thm:robust_ica:aux61.2}$ and the definition of $\tilde{c}_4$. Setting $c_j\coloneqq\tilde{c}_j$ for $j=0,1$ and $c_2\coloneqq\tilde{c}_2 \tilde{c}_4$ yields $r_{\tilde{\varepsilon}}=\beta_\varepsilon$, for $\beta_\varepsilon$ as defined above. Hence \eqref{prop:noise:eq2} follow if we have
\begin{equation}\label{prop:noise:aux1.1.3}
\delta(S,A^{-1}\tilde{X}^{(i)}) \,\leq\,\beta_\varepsilon \quad \text{ for each of the corruptions in } \ \eqref{prop:noise:eq1}.
\end{equation}
Moreover, from Theorem \ref{thm:robustness} \eqref{cor:robustness:eq2} we further obtain that \eqref{prop:noise:aux1.1.3} also implies $\partial_{\hat{\Phi}}(\upsilon^{(i)}, A)\leq \varepsilon$ for $\upsilon^{(i)}\coloneqq A^{-1}\tilde{X}^{(i)}$ and each $i=1,2,3$, which proves \eqref{prop:noise:eq3} for the cases $i=1,2$. To see that it also implies the case $\partial_{\hat{\Phi}}(S,\eta A)\leq\varepsilon$ and hence \eqref{prop:noise:eq3} altogether, notice that (cf.\ Footnote $\ast$)
\begin{equation}
\sup\nolimits_{t\in[0,1]}\sup\nolimits_{u\in D^{(v)}_{S_t}}\!\!\tfrac{|B\circ (\eta_tA)(u-v) - Mu|}{|Mu|}\mathbbm{1}_{\!\times} \leq \, \sup\nolimits_{u\in {D_{\!\upsilon^{(3)}}^{(v)}}}\!\!\tfrac{|BA(u-v) - Mu|}{|Mu|}\mathbbm{1}_{\!\times} \quad\text{ almost surely }  
\end{equation}
(for any $B\in\hat{\Phi}(\tilde{X}^{(3)})$ and each $(M,v)\in\M\times\R^d$) and hence $\partial_{\hat{\Phi}}(S,\bm{\eta} A) \leq \partial_{\hat{\Phi}}(\upsilon^{(3)}, A)$ a.s. 

\noindent
Let us first establish \eqref{prop:noise:aux1.1.3} for the additive-noise case $i=1$, which the triple $(\tilde{X}^{(1)}, A, \tilde{S})$, with $\tilde{\bm{S}}\coloneqq \bm{S} + \bm{\eta}$, represents in full generality (cf.\ \eqref{sect:noise:eq2}). For any $ijk\in[d]^\star_3$ and by multilinearity,   
\begin{align}\label{prop:noise:aux1.1}
\langle\tilde{S}\rangle_{ij} \, &= \, \mathbb{E}\!\left[\int_0^1\!\!\!\int_0^t\!\!\mathrm{d}\tilde{\bm{S}}^i_s\,\mathrm{d}\tilde{\bm{S}}^j_t\right] \,=\, \langle S\rangle_{ij} + \langle S^i, \eta^j\rangle + \langle\eta^i,S^j\rangle + \langle \eta\rangle_{ij}, \ \text{ and similarly}\\\label{prop:noise:aux1.2}
\langle\tilde{S}\rangle_{ijk} \,&= \, \langle S\rangle_{ijk} + \langle S^i,S^j,\eta^k\rangle + \langle S^i,\eta^j,S^k\rangle + \langle S^i,\eta^j,\eta^k\rangle + \langle\eta^i,S^j,S^k\rangle\\
 &\hspace{1em}+ \langle \eta^i,S^j,\eta^k\rangle + \langle\eta^i,\eta^j,S^k\rangle + \langle\eta\rangle_{ijk},
\end{align}
with $\langle Y^i, \tilde{Y}^j\rangle\coloneqq\mathbb{E}\big[\!\int_0^1\!\!\int_0^t\mathrm{d}\bm{Y}^i_s\mathrm{d}\tilde{\bm{Y}}^j_t\big]$ and $\langle Y^i, \tilde{Y}^j, \hat{Y}^k\rangle\coloneqq\mathbb{E}\big[\!\int_0^1\!\!\int_0^t\!\!\int_0^s\mathrm{d}\bm{Y}^i_r\mathrm{d}\tilde{\bm{Y}}^j_s\mathrm{d}\hat{\bm{Y}}^k_t\big]$ for $Y, \tilde{Y}, \hat{Y}\in\{S, \eta\}$. Since $\bm{\eta}, \bm{S}\in\mathcal{C}^1_d$, we can use Fubini to evaluate the above statistics \eqref{prop:noise:aux1.1} and \eqref{prop:noise:aux1.2} to  
\begin{equation}\label{prop:noise:aux1.3}
\langle S^i, \eta^j\rangle = \int_0^1\!\!\!\int_0^t\!\!\mathbb{E}\big[\dot{\bm{S}}^i_s\dot{\bm{\eta}}^j_t\big]\mathrm{d}s\mathrm{d}t  \quad\text{ and }\quad \langle\eta^i,S^j,\eta^k\rangle = \int_0^1\!\!\!\int_0^t\!\!\!\int_0^s\!\!\mathbb{E}\big[\dot{\bm{\eta}}^i_r\dot{\bm{S}}^j_s\dot{\bm{\eta}}^k_t\big]\mathrm{d}r\mathrm{d}s\mathrm{d}t \quad\text{etc.}
\end{equation}
From the assumptions that $\bm{\eta}$ and $\bm{S}$ are independent and both mean-stationary, we obtain $\mathbb{E}[\dot{\bm{S}}^i_s\dot{\bm{\eta}}^j_t]=\E[\dot{\bm{S}}^i_s]\E[\dot{\bm{\eta}}^j_t] = 0$ and $\E[\dot{\bm{\eta}}^i_r\dot{\bm{S}}^j_s\dot{\bm{\eta}}^k_t] = \E[\dot{\bm{S}}^j_s]\E[\dot{\bm{\eta}}^i_r\dot{\bm{\eta}}^k_t]=0$ etc. For this we used that $\E[\dot{\bm{S}}^i_t]=0=\E[\dot{\bm{\eta}}^j_t]$ for each $i,j\in[d]$ and $t\in[0,1]$, which follows by interchanging expectation and differentiation (as is permitted by the assumption that the derivative $\dot{\bm{S}}$ and $\dot{\bm{\eta}}$ are $L^1(\mathbb{P})$-dominated). This implies that the mixed statistics \eqref{prop:noise:aux1.3} vanish, so that we are left with
\begin{equation}
[\tilde{S}]_\nu \,=\, [S]_\nu + [\eta]_\nu \qquad\text{for each } \quad \nu=0,1,\ldots, d. 
\end{equation} 
Hence and from \eqref{sect:coredinates:eq4}, we obtain that
\begin{equation}
\delta(S,\tilde{S})^2 \,=\, \sum_{\nu=0}^d\langle S\rangle_{\nu\nu}^{-1}\big\|N_S^{-1}[\eta]_\nu N_S^{-1}\big\|^2 \,\leq\, \beta_\varepsilon^2,
\end{equation}
which proves \eqref{prop:noise:aux1.1.3} for the additive-noise case $i=1$ as desired. The proof of \eqref{prop:noise:aux1.1.3} for the multiplicative-noise cases $i=2,3$ is similar but slightly more technical, see Appendix \ref{pf:prop:noise}.  
\end{proof}

Notice how the deviance bounds \eqref{prop:noise:eq3} provide a concise quantification of the inversion stability of $\hat{\Phi}$ under the model deviations \eqref{prop:noise:eq1}. The proposition in fact shows that both the additive and the multiplicative noise case are continuous generalisations of the `noiseless' idealisations $\eta\equiv 0$ and $\eta\equiv 1$, respectively.

\subsection{Extensions and Outlook}\label{chap:robustICA:sect:addrems}The robustness analysis of this paper may be extended and applied in several further directions, including the derivation of stability guarantees for the blind inversion of strongly nonlinear mixtures and the study of stochastic dynamics for iterated function systems. As an example, let us briefly consider how our approach can be adapted to accommodate the alternative base assumption that the hidden mixing relationship \eqref{BSS-Robust-BlindInversion} between source and observable is (strongly) nonlinear. The central update for this would concern generalisations of the algebraic property \eqref{lem:linequiv:eq1} of linear equivariance, which was one of the core properties of the signature on which our blind inversion algorithm was based, see Section \ref{chap:robustICA:subsect:signature_identifiability:algorithm}: Is there a similarly structured way in which even non-linear functions of a signal are encoded in the signature representation \eqref{def:coordinates:eq1} of the untransformed signal? For general polynomial transformations the answer is yes, and is in fact a direct generalisation of \eqref{lem:linequiv:eq1}, see e.g.\ \cite{colmenarejo2018}. This `nonlinear equivariance' of the signature allows for a straightforward extension of the low-order robustness topology from Section \ref{chap:robustICA:sect:product_topology} to higher-order signature moments, where the relevant order of these higher moments grows controllably with the degree of any polynomial (mixing) relationship between source and observable. Using the fact that the source identities \eqref{lem1:pf:aux1} hold up to any order (at least for mean-stationary IC sources), one can derive corresponding families of appropriately diagonalisable tensor contractions (similar to \eqref{def:coordinates:eq2}) from these higher-order signature moments of the source, see e.g.\ \citep[Section 5.9.4]{schell2022phd}, which, in generalisation of Proposition \ref{prop1} and Section \ref{sect:blindinversionthm1}, can then be arranged in such a way as to explicitly characterise (up to permutation and scale, cf.\ \citep[Def.\ 5]{sigNICA2021} in generalisation of \eqref{def:icainversion:eq1}) exactly invertible nonlinear retransformations of the observable via the aforementioned relations of polynomial equivariance that these transformations inflict on the multivariate signature representation \eqref{def:coordinates:eq1} of the signal. This would lead to an explicitly computable procedure for the exact blind inversion of polynomial mixtures that comes with explicit robustness guarantees similar to those in Theorems \ref{thm:robust_ica} and \ref{thm:robustness}. Potential applications could include the stability analysis for random dynamical systems induced by polynomial function iterations \cite{grassberger1989symbolic}, among others. Further explorations of these and other extensions are left to future research. 

\subsection*{Acknowledgements}The author is grateful to Harald Oberhauser for his supervision of this work and for many helpful discussions and comments, as well as to Aapo Hyvärinen, Zhongmin Qian, and Markus Reiß for helpful discussions and suggestions. During the main development of this project, the author was financially supported by the Oxford-Cocker Graduate Scholarship and a Mathematical Institute Scholarship, both provided by the University of Oxford. 

\bibliographystyle{plain}
\bibliography{References}

\newpage

\appendix

\setcounter{equation}{0}
\renewcommand{\theequation}{\thesection.\arabic{equation}}

\etocsettocstyle{}{}
\addcontentsline{toc}{section}{Appendix}
\part*{\large \centerline{Appendix}}
\localtableofcontents

\vspace{1em}
\section{Notation and Preliminaries}\label{sect:preliminariesnotation}
\vspace{0.5em}
\subsection{List of Symbols}\label{sect:notation}
Below are some of the symbols and terminology that we use.
\begin{scriptsize}
\begin{center}\label{tab:notation}
\renewcommand{\arraystretch}{1.5}
  \begin{longtable}{cp{10cm}r}
\toprule
Symbol & Meaning & Page \\
\toprule 
$\mathcal{C}_d$ & $\coloneqq\{x : [0,1]\rightarrow\R^d\mid x \text{ is continuous}\}$; space of continuous paths in $\R^d$, often endowed with the uniform norm $\|x\|_\infty\coloneqq\sup_{t\in[0,1]}|x_t|$. & \pageref{convention:signals-distributions}\\

$\mathcal{M}_1(E)$ & $\coloneqq\{\mu : \mathcal{B}(E)\rightarrow[0,1]\}$; set of all Borel probab.\ measures on a (topol.) space $E$. & \pageref{convention:signals-distributions}\\

$\mathcal{M}_1$ & $\coloneqq\mathcal{M}_1(\mathcal{C}_d)$; set of Borel probability measures on $(\mathcal{C}_d,\|\cdot\|_\infty)$. & \pageref{convention:signals-distributions}\\

$D_Y$ & see $\eqref{def:spatial_support:eq1}$; the spatial support of a signal $Y\in\mathcal{M}_1$.& \pageref{appendix:sect:BSSformal:eq1}\\

$C^{0,0}(D)$ & $\coloneqq\{\tilde{f} : D \rightarrow \R^d\mid \text{$\tilde{f}$ homeomorph.\ onto $\tilde{f}(D)$}\}$; set of all homeomorphisms on $D$. & \pageref{causalspace}\\

$\operatorname{ev}_X(g)$ & $\equiv \operatorname{ev}(X,g)\coloneqq g(X)$ for $g : D_X \rightarrow \R^d$; the evaluation map at the process $X$. & \pageref{appendix:sect:BSSformal:intext:eq:identconds}\\

$\mathrm{M}_d$ & $\coloneqq \{M \in \operatorname{GL}_d \ | \ M=D\cdot P \ \text{ for } \ D \in\operatorname{GL}_d \text{ diagonal }\text{ and } P \in \mathrm{P}_d\}$; the group of (real) monomial matrices of degree $d$. & \pageref{fig:BSStriple}\\

$\tilde{\Delta}_d$ & $\coloneqq \{\Lambda=(\lambda_i\cdot\delta_{ij}) \in \operatorname{GL}_d \ | \  \lambda_1, \ldots, \lambda_d \in\mathbb{R}\setminus\{0\}\}$; the group of invertible diagonals. &  \pageref{sect:robustica-basic}\\

$C^{1,1}$ & $\equiv C^{1,1}(\R^d)\coloneqq C^{1,1}(\R^d;\R^d)$; group of all $C^1$-diffeomorphisms on $\R^d$. & \pageref{sect:robustica-basic}\\

$\restr{\varphi}{\tilde{A}}$ & the restriction of a map $\varphi : A\rightarrow B$ to a subdomain $\tilde{A}\subseteq A$.& \pageref{lem:robustness-inversionstability}\\

$\operatorname{GL}_d$& $\coloneqq\{A \in \mathbb{R}^{d\times d} \ | \
                       \operatorname{det}(A)\neq 0\};$ the general linear group of degree $d$ over $\mathbb{R}$. & \pageref{appendix:sect:BSSformal:def:ICA:eq1}\\
                       
$|\cdot|$ & $: \R^n\rightarrow\R_+$, Euclidean norm on $\R^n$; in partic., $|\cdot|$ is the absolute value on $\R=\R^1$.& \pageref{def:deviance}\\

$\mathfrak{D}$ $[\dot{\fD}]$ & see \eqref{notation:absolutelycontinuous2}/\eqref{sect:convergence_sigmoments:subspace}; space of $\sig$-coordinatisable [-premetrizable] signals. & \pageref{notation:absolutelycontinuous}/\pageref{sect:convergence_sigmoments:subspace}\\ 

$[k]$ & $\coloneqq\{1, \ldots, k\}$, \ and \ $[k]_0\coloneqq[k]\cup\{0\}$ \ ($k\in\mathbb{N}$). & \pageref{subsect:sigmoments:eq1.1}\\

$[d]^\star$ & $\coloneqq\bigcup_{m\geq 0}[d]^{\times m}=\{\emptyset, 1, 2, \ldots, 11, 12, 112, \ldots\}$; the set of all multi-indices with entries in $[d]$, where $[d]^0\coloneqq\{\emptyset\}$ and $i_1\cdots i_m\equiv (i_1, \ldots, i_m)$. & \pageref{subsect:sigmoments:eq1.1}\\

$\bar{\R}$ & $\coloneqq\R\cup\{-\infty, \infty\}$; the affine closure of the real numbers. & \pageref{def:coredinates}\\

$[d]^\star_{\leq m}$ & $\coloneqq\bigcup_{0\leq j\leq m}[d]^{\times j}$; the set of all multi-indices of order up to $m$ ($m\in\N$). & \pageref{lem:unifsigintegr}\\

$\mathrm{ddiag}_{i\in[d]}[a_i]$ & $\coloneqq (a_i\cdot\delta_{ij})_{ij}$; the diagonal matrix with main diagonal $(a_1, \ldots, a_d)$; for $\theta=(\theta_{ij})\in\R^{d\times d}$ we write $\mathrm{ddiag}(\theta) = \mathrm{ddiag}_{i\in[d]}[\theta_{ii}]$. & \pageref{sect:coredinates:eq4}\\

$D_\varphi$ & $\coloneqq \big(\frac{\partial}{\partial x_j}\varphi_i\big)_{ij}$; the Jacobian of a map $\varphi\equiv(\varphi_1,\cdots,\varphi_d)^\intercal\in C^1(G;\R^d)$.& \pageref{lem:premetric_facts2}\\ 

$\Delta_m(\mathbb{I})$ & $\coloneqq\{(t_1, \ldots, t_m)\in\mathbb{I}^{\times m}\mid t_1 < \ldots < t_m\}$; the (relatively) open $m$-simplex on $\I^{\times m}$; we denote $\Delta_m\equiv\Delta_m([0,1])$. & \pageref{lem1:pf:aux6}\\ 

$\langle u, v\rangle_2$ & $\coloneqq u_1v_1+\ldots+u_dv_d$; the dot product of two vectors $u=(u_i), v=(v_i)$ in $\R^d.$ & \pageref{prop1:eq4}\\

$\mathrm{P}_d^{\pm}$ & $\coloneqq \{(\epsilon\cdot\delta_{\sigma(i),j})_{i,j\in[d]} \in \operatorname{GL}_d \ | \ \epsilon\in\{0,1\}, \ \sigma\in S_d \}$; signed permutation matrices.  & \pageref{phi:monomialinvar}\\ 

$S_d$ & $\coloneqq \{\sigma : [d]\rightarrow[d]\mid \text{$\sigma$ is bijective}\}$; the group of all permutations of $[d]$.  & \pageref{tab:notation}\\

$f_1\times f_2$ & $:\,U_1\times U_2\rightarrow V_1\times V_2, \ (u_1, u_2)\mapsto (f_1(u_1), f_2(u_2));$ the Cartesian product of two maps $f_i : U_i\rightarrow V_i$ $(i=1,2)$.& \pageref{transformation}\\ 

wlog/wrt./s.t. & ``without loss of generality''/ ``with respect to'' / ``such that''. & \\

i.o.w. & ``in other words'' & \\

\bottomrule
\end{longtable} 
\vspace{-1em}
\end{center} 
\end{scriptsize}

\subsection{Discrete-Time and Continuous-Time Signals}\label{sect:signals}Throughout this paper, the term \emph{signal} refers to a stochastic processes in $\R^d$, that is to an ordered family (printed in bold)
\begin{equation}\label{prelim:signals:eq1}
\bm{Y}=(\bm{Y}_t \mid t\in\I) \ \ \text{ of (Borel) random vectors } \ \ \bm{Y}_t=(\bm{Y}^1_t, \cdots, \bm{Y}^d_t) \ \text{in} \ \R^d,
\end{equation}
for $\I$ an ordered subset of $\R$, typically either $\I=\N$ (`discrete-time') or $\I=[0,1]$ (`continuous-time'). The scalar processes $\bm{Y}^i=(\bm{Y}^i_t \mid t\in\I)\equiv (\bm{Y}^i_t)$, $i\in[d]$, are called the components of $\bm{Y}$.\\[-0.5em]

\noindent
By rescaling $\I$ and thanks to Rem.\ \ref{rem:pwl-interpolation}, we may assume without loss of generality that $\I=[0,1]$ and\footnote{\ See also Remark \ref{rem:discrete-case}.} that the vectors of any signal \eqref{prelim:signals:eq1} are ordered continuously, i.e.\ that with probab.\ one 
\begin{equation}\label{subsect:sigmoments:eq1_prelim}
\I \,\ni\, t \ \longmapsto \ \bm{Y}_t \ \ \text{ is continuous}.
\end{equation} 
The latter is to say that (with probability one) the samples $\bm{Y}(\omega)=(\bm{Y}_t(\omega))_{t\in\I}$ of $\bm{Y}$ are in $\mathcal{C}_d$, the space of continuous paths from $[0,1]$ to $\R^d$; the space $\mathcal{C}_d$ carries the uniform norm topology per default. The continuity assumption \eqref{subsect:sigmoments:eq1_prelim} allows us to 
\begin{itemize}
\item use the term `stochastic process' synonymously to `continuous-time stochastic process', as we will do unless stated otherwise, and to
\item view a signal \eqref{prelim:signals:eq1} as a random path, that is as a $\mathcal{C}_d$-valued (Borel) random variable over some probability space $(\Omega, \mathscr{F}, \mathbb{P})$. 
\end{itemize} 
\begin{remark}\label{rem:pwl-interpolation}Any signal $\bm{Y}=(\bm{Y}_t)_{t\in\I}$ with $\I$ discrete can be viewed as a continuous-time process $\hat{\bm{Y}}$ upon piecewise-linear interpolation of its points $(\bm{Y}_t\mid t\in\I)\equiv(\bm{Y}_{t_j}\mid j\in\N)$, namely
\begin{equation}
\hat{\bm{Y}}=\big(\hat{\bm{Y}}_t\coloneqq \bm{Y}_{t_j} + \tfrac{t - t_j}{t_{j+1}-t_j}(\bm{Y}_{t_{j+1}} - \bm{Y}_{t_j}) \text{ for } t\in[t_j, t_{j+1})\big).
\end{equation} 
(Note that this interpolation commutes with linear transformations: $A\cdot\hat{\bm{Y}}\equiv(A\hat{\bm{Y}}_t)=\widehat{(A\bm{Y}_{t_j})}$ for each $A\in\R^{d\times d}$.) Upon scaling $\I$, we may further assume $\hat{\bm{Y}}$ to be $\mathcal{C}_d$-valued. \hfill $\bdiam$
\end{remark}

\noindent
By identifying a signal with its law (Convention \ref{convention:signals-distributions}) we can embed\footnote{\ \ldots up to equality in distribution, which is entirely sufficient for our purposes here \ldots} the class of all signals into $\mathcal{M}_1\equiv\mathcal{M}_1(\mathcal{C}_d)$, the space of all (Borel) probability measures on $\mathcal{C}_d$. This is of significant mathematical convenience as it allows us to describe the Problem of Blind Source Separation in terms of maps on a single pre-structured domain ($\mathcal{M}_1$), from where the notion of robustness can then be flexibly formulated and analysed from the angle of general topology (see Sect.\ \ref{appendix:sect:BSSformal}). 

\convA*

{\ }\hfill Unless otherwise mentioned, Convention \ref{convention:signals-distributions} is in continuous use throughout this paper.\hfill{\ }\\[-0.75em]

\noindent
We will at times make use of the \emph{spatial support} of a signal \eqref{prelim:signals:eq1}, which is defined as the set 
\begin{equation}\label{def:spatial_support:eq1}
D_{\bm{Y}} \, = \, \overline{\bigcup\nolimits_{t\in\I}\supp(\bm{Y}_t)}, \quad \text{ where \ \ $\supp(\bm{Y}_t)\equiv\supp(\mathbb{P}_{\bm{Y}_t})$}
\end{equation}
is the support of the law $\mathbb{P}_{\bm{Y}_t}$ of $\bm{Y}_t$ and the closure is taken wrt.\ the Euclidean topology on $\R^d$. The set $D_{\bm{Y}}$ is the smallest closed subset of $\R^d$ which contains (the trace of) $\mathbb{P}$-almost each sample path of $\bm{Y}$, see \citep[Section 4.2]{sigNICA2021} for details. We call a signal $\bm{Y}$ `degenerate' if its spatial support $D_{\bm{Y}}$ is contained in a hyperplane of $\R^d$.\\[-0.75em] 

\hfill(Readers uncomfortable with \eqref{def:spatial_support:eq1} may for simplicity assume that $D_{\bm{Y}}=\R^d$ throughout.)\\[-0.75em]

The support \eqref{def:spatial_support:eq1} and other notions and operations on signals \eqref{prelim:signals:eq1} can be readily generalised to measures in $\mathcal{M}_1$, as described in the following subsection. 

\subsection{Operations on Measures}\label{sect:operationsonmeasures}The Problem of Blind Source Separation naturally involves a range of (non)linear transformations $\varphi$ applied to a signal. Since the space $\mathcal{M}_1(\mathcal{C}_d)\equiv\mathcal{M}_1$ of (Borel) probability measures on $\mathcal{C}_d$ is convex but not a vector space, these transformations need to be `induced' on $\mathcal{M}_1(\mathcal{C}_d)$ in order to be well-defined.\\[-0.75em] 

\noindent
Given some target vector space $V$, a canonical way to do this is via the pushforward 
\begin{equation}\label{sect:notation:pushforward}
\varphi_\ast \, : \, \mathcal{M}_1\ni\mu \, \longmapsto\,\varphi_\ast(\mu)\equiv \mu\circ\varphi^{-1}\in\mathcal{M}_1(V) \quad \text{ of a (Borel) map } \ \varphi : \mathcal{C}_d \rightarrow V. 
\end{equation}     
Owing to their ubiquity, a few special signal actions \eqref{sect:notation:pushforward} will be given their own name:\\[-0.75em] 

\noindent
For $x=(x^1_t, \cdots, x^d_t)_{t\in[0,1]}\in\mathcal{C}_d$ any fixed path, consider the maps 
\begin{equation}\label{sect:notation:pushforward:eq2}
\pi_i \,: \, x \mapsto x^i \equiv (x^i_t)_{t\in[0,1]} \quad \text{and} \quad \pi_t \, : \, x \mapsto x_t \quad\text{and} \quad \pi^i_t\,:\, x\mapsto x^i_t \quad (i\in[d], \ t\in[0,1])
\end{equation}   
and denote their respectively induced pushforward measures from a signal $\mu\in\mathcal{M}_1$ by $\mu^i\coloneqq\mu\circ\pi_i^{-1}$ and $\mu_t\coloneqq\mu\circ\pi_t^{-1}$ and $\mu^i_t\coloneqq\mu\circ(\pi_t^i)^{-1}$; we may then also write $\mu=(\mu^1, \cdots, \mu^d)$ and $\mu=(\mu_t)$ to denote that $\mu$ has the spatial marginals $\mu^i$ and the fixed-time marginals $\mu_t$.\\[-0.5em] 

In accordance with \eqref{def:spatial_support:eq1}, the spatial support of $\mu\in\mathcal{M}_1$ is defined as 
\begin{equation}\label{spatsup:law}
D_\mu\coloneqq\overline{\bigcup\nolimits_{t\in\I}\mathrm{supp}(\mu_t)}
\end{equation} 

\noindent
Any $g : \R^d\rightarrow \R^d$ induces a map $\tilde{g} : \mathcal{C}_d \rightarrow \mathcal{C}_d$ via $\tilde{g}(x)\coloneqq(g(x_t))_{t\in[0,1]}$, and we set 
\begin{equation}\label{transformation}
g(\mu) \equiv g_\ast\mu \,\coloneqq\, \mu\circ\tilde{g}^{-1}. 
\end{equation}  
Extending \eqref{sect:notation:pushforward:eq2}, we may further declare in-time increments $\mu_{s,t}\equiv\mu_t - \mu_s\coloneqq(\pi_t - \pi_s)_\ast(\mu)$ and products $\mu^i_{s,t}\cdot\mu^j_{s,t}\coloneqq\big[\mathfrak{m}\circ(\pi^i_{s,t}\times\pi^j_{s,t})\big]_\ast(\mu)$ for the maps $\pi^k_{s,t}\coloneqq\pi^k_t - \pi^k_s$ and $\mathfrak{m}(a,b)\coloneqq a\cdot b$ and $s,t\in\I$. Further, $\E\,\nu\coloneqq\int_{\R^n}\!u\,\nu(\mathrm{d}u)$ is the expectation of a (possibly signed) Borel measure $\nu$ on $\R^n$, and a measure $\mu=(\mu_t)\in\mathcal{M}_1$ is called mean-stationary if $\E[\mu_t]=\E[\mu_0]$ for each $t\in\I$.\\[-0.5em] 

Finally, a signal $\mu\in\mathcal{M}_1$ is said to have (mutually) \emph{independent components}, or to be \emph{IC}, if $\mu=\mu^1\otimes\cdots\otimes\mu^d$ (up to the identification $\mathcal{M}_1(\mathcal{C}_d)\cong\mathcal{M}_1(\mathcal{C}_1^{\times d})$, see e.g.\ \citep[Section A.1]{sigNICA2021}). The set of all IC signals in $\mathcal{M}_1$ is denoted $\mathscr{S}_\bot$.   

\begin{example}To illustrate, suppose that $\mu=\mathbb{P}_{\bm{Y}}\in\mathcal{M}_1$ for a some signal $\bm{Y}\equiv(\bm{Y}^i)=(\bm{Y}_t)$ in $\R^d$. Then $\mu^i = Y^i$ and $\mu_t = Y_t$ and $g(\mu) = g(Y) \equiv \mathbb{P}_{g(\bm{Y})}$ for any measurable $g : \R^d\rightarrow\R^d$. Further $\E\,\mu_{s,t} = \E \bm{Y}_t - \E \bm{Y}_s$ and $\E[\mu_{s,t}^i\mu_{s,t}^j] = \mathbb{E}[\bm{Y}^i_{s,t}\bm{Y}^j_{s,t}]$ for any $i,j\in[d]$ and $s,t\in\I$, and $\mu$ is mean-stationary iff the process $\bm{Y}$ is mean-stationary, and $\mu$ is IC iff the components $\bm{Y}^1, \ldots, \bm{Y}^d$ of $\bm{Y}$ are mutually statistically independent (as stochastic processes). 
\end{example}       

\section{Complementary Motivation for Section \ref{appendix:sect:BSSformal}}\label{sect:bssformal:add_motivation}
\noindent
The following builds on Example \ref{example:cocktailparty} to both illustrate and motivate the definitions introduced in Section \ref{appendix:sect:BSSformal}. While the below perspectives and ideas on the robustness of BSS are relevant and intuitive, the current literature does not seem to offer any flexible and informative theoretical framework to formalise and exploit them. This paper is a first proposal in this direction.

\begin{example}[Cocktail Party Revisited]\label{example:cocktailparty_contd}The original speech signals belonging to the speakers in Example \ref{example:cocktailparty} can be modelled as [drawn from] the components $S^i$ of a stochastic process $S=(S^1, \ldots, S^d)\in\mathcal{M}_1$. The microphone recordings of these signals are then mixtures of the form $X^i = f_i(S^1,\ldots, S^d)$ which combine to give an observable $X=(X^1,\ldots, X^d) = f(S)$ for some invertible map $f=(f_1, \ldots, f_d) : D_S \rightarrow \R^d$. The inverse problem from Example \ref{example:cocktailparty} thus translates into the BSS-problem \eqref{appendix:sect:BSSformal:BSS_reformulation} for the triple $(X,S,f)$, and is to be solved to an accuracy $\lceil\,\cdot\,\rceil$ not explicitly given in the example. The example states, however, that the goal is to ``recover from $X$ what each of the $d$ speakers has said''. So the goal is \emph{not} an exact recovery\footnote{\ $\ldots$ i.e., to obtain from $X$ not only the content of \emph{what was said}, but also \emph{which speaker} (identified by their number $i$) it was who said it and \emph{at what original sound volume} they said it; this is usually impossible.} of $S$ from $X$, but to recover from $X$ only the messages -- that is, the informational content of their respective speeches -- that the speakers communicated: which speaker it was that said something, or how loudly they originally said it, is irrelevant.\\[-0.75em] 

\noindent
This abstracts from the original source $S$ the order of its components [i.e., the attribution (via enumeration) of a speech signal to its original speaker] and the scale of its components [i.e., how loud such speech was originally communicated]. Consequently, any two signals $(b_1 S^{j_1}, \ldots, b_d S^{j_d})$ and $(\hat{b}_1 S^{k_1}, \ldots, \hat{b}_d S^{k_d})$ with $b_i, \hat{b}_i\neq 0$ and $\{j_i\}=\{k_i\}=[d]$ are both equally desirable outcomes of the blind inversion. I.o.w., the task in Example \ref{example:cocktailparty} is `accurately' accomplished if from $X$ we arrive at (not necessarily $S$ but instead) any signal of the form
\begin{equation}\label{example:cocktailparty_contd:eq1}
\tilde{S}\,\in\,\big\{\big(\alpha_1 S^{\tau(1)}, \ldots, \alpha_d S^{\tau(d)}\big) \, \big| \, \alpha_1, \ldots, \alpha_d\in\R_\times, \, \tau\in S_d \big\} \, = \, \mathrm{M}_d\cdot S,
\end{equation}
where $\mathrm{M}_d\coloneqq\{\Lambda P\mid \Lambda\in\tilde{\Delta}_d, P\in\mathrm{P}_d\}$ is the subgroup [of $\GL$] of monomial matrices. This establishes the monomial orbit $\lceil S\rceil\coloneqq\M\cdot S$ as a suitable task-specific accuracy \eqref{def:blindinversion:eq1} for Ex.\ \ref{example:cocktailparty}.\\[-0.75em] 

\noindent
Since our prior knowledge of the speakers, $S$, and of `the physics' of their recording situation, $f$, is very limited, we can generally locate the original constellation $(S,f)$ only as an element of some subset $\sI$ in $\mathfrak{C}$, and not as a singleton in said space -- this is our `blindness'. Consequently, at this prior level, every signal that conforms to $\sI$ (that is, any $\hat{S}\coloneqq g(X)$ s.t.\ $(\hat{S}, g^{-1})\in\sI$) appears to us indistinguishable from the original speeches we are seeking. We therefore call such signals `quasi sources' and subsume them in the set \eqref{appendix:sect:BSSformal:intext:eq:maxsol2}.\\[-0.75em] 

\noindent
The main question for us at this point is whether our limited knowledge $\sI$ of the speakers and their mixing is sufficient, at least in principle, to recover the signal $S$ from $X$ up to the desired accuracy \eqref{example:cocktailparty_contd:eq1}. This translates into the requirement that $\langle
X\rangle_{\!\sI}\subseteq\lceil S\rceil$ on $\sI$, and if this requirement is met, we want to perform the inversion $X\mapsto S$ accordingly. The latter asks for an explicitly computable inversion procedure $\Phi : \mathcal{M}_1\rightarrow 2^{\mathcal{M}_1}$ which,  since it should make optimal use of the data $(X,\sI)$, we naturally expect to satisfy $\Phi(X)\neq\emptyset$ and $\Phi(X)\subseteq\langle X\rangle_{\!\sI}$.\\[-0.75em]
   
\noindent
Now due to complex interdependencies between the speakers $S^1, \ldots, S^d$, or because the recording situation $f$ is somewhat complicated, or both, we may not be able to formulate an assumption $\sI$ on $(S,f)$ for which the accuracy condition \eqref{def:blindinversion:eq1} is exactly met. However, we may often find that $(S,f)$ is still more or less `well approximated' by an idealised constellation $(S_\star, f_\star)$ which conforms to an assumption $\sI_\ast$ in $\mathfrak{C}$ for which the accuracy condition is satisfied. In our example, a viable such approximation of $(S,f)$ might be to choose $S_\star$ to have the same component signals as $S$ but with no statistical dependence between them, and to take $f_\star$ as the linearisation of $f$ at some point $x_0\in D_S$, in symbols: $(S_\star, f_\star)\equiv(\mathbb{P}_{S^1}\otimes\cdots\otimes\mathbb{P}_{S^d}, D_{\!f}(x_0))$ (cf.\ Section \ref{sect:robustica-basic}).\\[-0.75em] 

\noindent
This idealisation then pays off as follows. The $\lceil\,\cdot\,\rceil$-sufficiency of $\sI_\ast$ allows for a map $\Phi\equiv\Phi_{\!\sI_\ast}$ of the form \eqref{def:blindinversion:eq2} which achieves an exact blind inversion on $\sI_\ast$-caused observables. Provided that the inversion property $\Phi(\tilde{f}(\tilde{S}))\subseteq\lceil\tilde{S}\rceil$ of this map is \emph{stable} around the idealisation $(\tilde{S}, \tilde{f})\equiv(S_\star, f_\star)\in\sI_\ast$ as per \eqref{BSS:robust_prelim1}, the proximity $(S,f)\approx(S_\star, f_\star)$ will then guarantee that $\Phi(X)$ is approximately a subset of $\lceil S\rceil$, as desired. More specifically, given an estimate on the `deviation' $\Delta$ between $(S,f)$ (the original cause) and  $(S_\star, f_\star)$ (the $\sI_\ast$-idealisation of that cause), we can provide results that assert, in terms of precise and informative error bounds with fully explicit constants, that every element in $\Phi(X)$ is equal to an element of $\lceil S\rceil$ up to a relative error of the order $\mathcal{O}(\Delta)$. The merits this has for our example application are clear: While the actual cause $(S,f)$ of $X$ may be too complex to blindly recover $S$ exactly, we may find a `structurally simpler approximation' to $(S,f)$, say $\sI_\ast\ni(S_\star, f_\star)$ with the inversion $\Phi$ on $\sI_\ast$ as above, and find that the estimated speeches $\tilde{S}^\Delta\equiv\Phi(X)$ will still be $\lceil\,\cdot\,\rceil$-accurate up to a controlled relative distance from \eqref{example:cocktailparty_contd:eq1}. Provided that the idealisation error $\Delta$ is small enough,\footnote{\ which one can try to achieve, e.g., by adjusting the microphone positions or attenuating external noise.} the resulting signal $\tilde{S}^\Delta$ may show acoustic distortions, but it will still be close enough to the optimum \eqref{example:cocktailparty_contd:eq1} (the clearly audible speeches) to allow the original information conveyed by the speakers to be discerned. \hfill $\bdiam$ 
\end{example} 

\subsection{Supplement to Definition \ref{appendix:sect:BSSformal:def:ICA}}\label{sect:def:lica:appendix:smoothnessset}For its use in \eqref{appendix:sect:BSSformal:def:ICA:eq1}, define the set (`smoothness class')
\begin{equation}\label{def:lica:appendix:smoothnessset}
\hat{\mathscr{S}}_\bot\coloneqq\big\{\mu\in\mathscr{S}_\bot\ \big| \ \sharp\{i\in[d]\mid\E\big[(\mu^i_{0,1})^3\big]= 0\}\leq 1 \, \text{ and } \, \sharp\sigma(\mathrm{Cov}(\mu_0, \mu_1))=d\big\}
\end{equation}   
of all IC signals whose incremental covariance has pairwise distinct eigenvalues and for which at most one of their third incremental moments vanishes. These signals are of `conventional regularity' in the sense that they are in particular non-Gaussian, hence (linearly) identifiable by the classical paradigm \citep{COM}, and also of a sufficiently rich time-structure to be recoverable via the methods in \cite{nordh2012}. I.o.w., \eqref{def:lica:appendix:smoothnessset} is a class of sources which are so regular (`smooth') that a blind inversion procedure should be able to recover them in order to be called an ICA-method. 

\subsection{Extensions to Rougher Paths}\label{rem:roughpath-extensions}While the sample space \eqref{notation:absolutelycontinuous} is fairly general already and, in particular, covers all signals that one encounters in statistical practice (cf.\ Rem.\ \ref{rem:discrete-case}, Prop.\ \ref{prop:boundeddomains}), note that the considerations of this paper can be readily extended to spaces of even more irregular paths by using the theory of rough paths \citep{FVI,FLO,LYQ}. As this is mainly a technical exercise, we refrain from implementing these extensions here and instead refer the interested reader to the above references to not exceed the scope of this paper.       

\section{Technical Proofs and Lemmas (Sections \ref{appendix:sect:BSSformal} to \ref{chap:robustICA:sect:auxiliaries})}\label{sect:bss-appendixA}
\subsection{Proof of Lemma \ref{lem:robustness-inversionstability}}\label{pf:lem:robustness-inversionstability}
\begin{proof}
That $\partial_{\hat{\Phi}}(\sI)\equiv 0$ is clear from \eqref{def:icainversion:eq1}. Write $\varphi(\bm{S},f)$ for the left-hand side of the inequality \eqref{lem:robustness-inversionstability:eq1} and notice our slight abuse of notation, cf.\ Convention \ref{convention:signals-distributions}: here, $\varphi(\bm{S},f)$ is a function of the process $\bm{S}$ in $\R^d$, while $\partial_{\hat{\Phi}}(S, f)$ is a function of the law, $S\equiv\mathbb{P}_S\in\mathcal{M}_1$, of $\bm{S}$. By \eqref{BSS:robust_prelim1:eq1},
\begin{equation}\label{lem:robustness-inversionstability:aux1}
\varphi(\bm{S},f) = \sup\nolimits_{\tilde{\bm{X}}\in\Phi(\bm{X})}\inf\nolimits_{h\in\tilde{\mathrm{M}}_d}\tilde{\varrho}(\tilde{\bm{X}}, h(\bm{S})) \quad \text{for } \ \bm{X}=f(\bm{S}),
\end{equation}
where $\Phi(\bm{X})=(\hat{\Phi}(X)\circ f)(\bm{S})\equiv\{(B\circ f)(\bm{S})\mid B\in\hat{\Phi}(X)\}$ by definition. Hence for any $\tilde{\bm{X}}\in\Phi(\bm{X})$ and all $h\equiv M + v\in\tilde{\mathrm{M}}_d$ (some $M\in\M$, $v\in\R^d$) there is $B\equiv B(\tilde{\bm{X}})\in\hat{\Phi}(X)$ with 
\begin{equation}
\tilde{\varrho}(\tilde{\bm{X}}, h(\bm{S}))=\sup\nolimits_{t\in\I}\tfrac{|(B\circ f)(\bm{S}_t) - (M\bm{S}_t+v)|}{|M\bm{S}_t+v|}\mathbbm{1}_{\!\times}{\scriptstyle\!(h(\bm{S}_t))} = \sup\nolimits_{t\in\I}\tfrac{|(B\circ f)(\tilde{\bm{S}}_t - \tilde{v}) - M\tilde{\bm{S}}_t|}{|M\tilde{\bm{S}}_t|}\mathbbm{1}_{\!\times}{\scriptstyle\!(M\tilde{\bm{S}}_t)}    
\end{equation}
where $\tilde{\bm{S}}=(\tilde{\bm{S}}_t)$ is given by $\tilde{\bm{S}}_t\coloneqq \bm{S}_t + \tilde{v}$ with $\tilde{v}\coloneqq M^{-1}v$. Hence and since the trace of a process is almost surely contained in its spatial support, we have that, with probability one,
\begin{equation}\label{lem:robustness-inversionstability:aux3}
\tilde{\varrho}(\tilde{\bm{X}}, h(\bm{S})) \,\leq\, \sup\nolimits_{u\in D_{\tilde{S}}}\!\tfrac{|(B\circ f)(u - \tilde{v}) - Mu|}{|Mu|}\mathbbm{1}_{\!\times}{\scriptstyle\!(Mu)}=\varrho_B((S,f), (M,\tilde{v}))
\end{equation} 
where $D_{\tilde{S}} (= D_{\tilde{\bm{S}}}) = D_S^{(\tilde{v})}$. Statement \eqref{lem:robustness-inversionstability:eq1} now follows by \eqref{lem:robustness-inversionstability:aux3} and def.s \eqref{lem:robustness-inversionstability:aux1} and \eqref{def:deviance}.     
\end{proof}

\subsection{Proof of Lemma \ref{lem:sigma-algebras}}\label{pf:lem:sigma-algebras}
\begin{proof}
Note first that since $\onevar{\,\cdot\,}\geq\|\cdot\|_\infty$ (which is easy to see), we find that the 1-variation topology on $\mathcal{C}^1$ is finer than the uniform topology on $\mathcal{C}^1$, which of course implies that $\mathcal{B}_{1\text{-}\mathrm{var}}\coloneqq\sigma(\mathcal{C}^1, \onevar{\,\cdot\,})\supseteq\sigma(\mathcal{C}^1, \|\cdot\|_\infty)\eqqcolon\mathcal{B}_\infty$. Since the separability of $(\mathcal{C}^1,\onevar{\,\cdot\,})$ guarantees that the $\sigma$-algebra $\mathcal{B}_{1\text{-}\mathrm{var}}$ is generated by the closed $\onevar{\,\cdot\,}$-balls, the converse inclusion $\mathcal{B}_{1\text{-}\mathrm{var}}\subseteq\mathcal{B}_\infty$ follows if we can show that 
\begin{equation}\label{lem:sigma-algebras:eq1}
B_r^1(x)\coloneqq \{y\in\mathcal{C}^1\mid \onevar{y-x}\leq r\} \,\in\,\mathcal{B}_\infty \ \text{ for every } \ x\in\mathcal{C}^1 \text{ and any } r\geq 0. 
\end{equation}
To see that this holds, fix any $x\in\mathcal{C}^1$ and $r\geq 0$ and recall that, by definition (cf.\ \eqref{sect:pweak-convergence:eq2}),
\begin{equation}
\onevar{z} = \sup_{\mathcal{I}\in\mathfrak{I}}V_{\mathcal{I}}(z) \ \text{ with } \ V_{(t_\nu)}(z)\coloneqq |z_0| + \sum_{\nu}\big|z_{t_{\nu+1}} - z_{t_\nu}\big|  
\end{equation}  
and where $\mathfrak{I}\coloneqq\{\mathcal{I}=(t_\nu)\mid \mathcal{I} \text{ is a (finite) dissection of } [0,1]\}$. Given any $\mathcal{I}\in\mathfrak{I}$ it is clear that the function $Q_{\mathcal{I}} : \mathcal{C}^1\ni y\mapsto V_{\mathcal{I}}(y-x)$ is continuous wrt.\ $\|\cdot\|_\infty$, whence the level set $C_{\mathcal{I}}\coloneqq\{y\in\mathcal{C}^1\mid Q_{\mathcal{I}}(y)\leq r\}$ is $\|\cdot\|_\infty$-closed. Combined with this, the immediate identity 
\begin{equation}
B_r^1(x) = \bigcap\nolimits_{\mathcal{I}\in\mathfrak{I}} C_{\mathcal{I}} \quad\text{ implies that } \ \ B^1_r(x) \text{ is closed wrt.\ } \|\cdot\|_\infty\,, 
\end{equation} 
which shows that \eqref{lem:sigma-algebras:eq1} holds as desired.
\end{proof} 

\subsection{Proof of Lemma \ref{lem:linequiv}}\label{pf:lem:linequiv}
\begin{proof} 
Fix any multiindex $(i_1, \ldots, i_m)\in[d]^{\times m}$. Then by the multilinearity of iterated (Stieltjes-)integration, we find for each $x\in\mathcal{BV}$ that
\begin{equation}
\begin{aligned}
\mathfrak{sig}_{i_1\cdots i_m}\!(f(x))\,&=\, \int_{\Delta_m}\!\,\mathrm{d}\big[f(x)\big]^{i_1}\!\cdots\mathrm{d}\big[f(x)\big]^{i_m} \\
&=\, \int_{\Delta_m}\!\,\mathrm{d}\big[\textstyle\sum_{j_1=1}^d a_{i_1j_1}x^{j_1}\big]\!\cdots\mathrm{d}\big[\textstyle\sum_{j_m=1}^d a_{i_mj_m}x^{j_m}\big] \\
&=\, \sum_{j_1, \ldots, j_m=1}^d a_{i_1j_1}\!\cdots a_{i_mj_m}\sig_{j_1\cdots j_m}(x)  
\end{aligned} 
\end{equation}
for the $m$-simplex $\Delta_m\coloneqq\{t\in[0,1]^m\mid t_1\leq\cdots\leq t_m\}$. Hence by definitions \eqref{def:coordinates:eq1} and \eqref{transformation}, 
\begin{equation}
\begin{aligned}
\big\langle f\cdot\mu\big\rangle_{i_1\ldots i_m} \, &= \, \int_{\mathcal{C}_d}\!\mathfrak{sig}_{i_1\cdots i_m}(x)\,f_\star\mu(\mathrm{d}x) \,=\, \int_{\mathcal{BV}}\!\mathfrak{sig}_{i_1\cdots i_m}\!(f(x))\,\mu(\mathrm{d}x)\\
&=\,\sum_{j_1, \ldots, j_m=1}^d a_{i_1 j_1}\!\cdots\, a_{i_m j_m}\!\cdot\langle\mu\rangle_{j_1\ldots j_m},   
\end{aligned} 
\end{equation}
where the second identity is the change-of-variables formula for pushforward measures. 
\end{proof}    

\subsection{Proof of Proposition \ref{prop:boundeddomains}}\label{pf:prop:boundeddomains}
\begin{proof}
Note first that the subspaces $\fD^p_K$ and $\widehat{\fD}_\mathcal{I}$ are well-defined, since $\mathcal{C}^1_{p,K}$ and $\hat{\mathcal{C}}_\mathcal{I}$ are both Borel-measurable wrt.\ $\|\cdot\|_\infty$. Indeed: $\mathcal{C}^1_{p,K}$ is $\onevar{\,\cdot\,}$-closed [as a sublevel set] by \eqref{sect:pweak-convergence:eq3} (these inclusions of course remain valid also if the $b$'s are dropped) and $\mathcal{B}(\mathcal{C}^1;\onevar{\,\cdot\,})=\mathcal{B}(\mathcal{C}^1; \|\cdot\|_\infty)$ by Lemma \ref{lem:sigma-algebras}, and the $\mathcal{C}^1$-subspace $\hat{\mathcal{C}}_{\mathcal{I}}$ is $\|\cdot\|_\infty$-closed because it is finite-dimensional.\\[-1.25em] 

Using next that \eqref{rem:pweak-convergence:eq1} can be generalised to $\|\cdot\|_\infty\leq\|\cdot\|_{q\text{-}\mathrm{var}}\leq 2\|\cdot\|^{\frac{p}{q}}_{p\text{-}\mathrm{var}}\|\cdot\|_\infty^{1 - \frac{p}{q}}$ for any finite $q>p$, see \citep[Prop.\ 5.5 (i)]{FVI}, we obtain that the norms $\|\cdot\|_\infty$ and $\|\cdot\|_{q\text{-}\mathrm{var}}$ [have the same convergent sequences on $\mathcal{C}^1_{p,K}$ and thus] are equivalent on $\mathcal{C}^1_{p,K}$. In particular $C_b(\mathcal{C}^1_{p,K};\|\cdot\|_\infty) = C_b(\mathcal{C}^1_{p,K};\|\cdot\|_{q\text{-}\mathrm{var}})$, whence for any net $(\mu_\alpha)\subset\fD^p_K\cong\mathcal{M}_1(\mathcal{C}^1_{p,K})$ we have: $\mu_\alpha\rightarrow\mu\in\fD^p_K$ weakly iff $\mu_\alpha\xrightharpoonup{q\text{-}\mathrm{var}}\mu$ (cf.\ \eqref{sect:pweak-convergence:eq1} and \eqref{sect:pweak-convergence:eq4}, with $\mathcal{C}^1$ replaced by $\mathcal{C}^1_{p,K}$). The lemma's assertions now follow by the fact that two topologies coincide iff they have the same convergent nets, and upon noting that $C_b(\hat{\mathcal{C}}_{\mathcal{I}};\|\cdot\|_\infty) = C_b(\hat{\mathcal{C}}_{\mathcal{I}};\|\cdot\|_{1\text{-}\mathrm{var}})$, which holds by the equivalence of $\|\cdot\|_\infty$ and $\|\cdot\|_{1\text{-}\mathrm{var}}$ on $\hat{\mathcal{C}}_\mathcal{I}$. With $\|\cdot\|_{1\text{-}\mathrm{var}}\geq\|\cdot\|_\infty$ known, the latter follows from $\|x\|_{1\text{-}\mathrm{var}} = |x_0| + \sum_{\nu=1}^{n-1}|x_{t_\nu} - x_{t_{\nu-1}}| \leq (2|\mathcal{I}|-1)\|x\|_\infty$  (any $x\in\hat{\mathcal{C}}_\mathcal{I}$).
\end{proof}

\subsection{Proof of Proposition \ref{lem:unifsigintegr:sufficient}}\label{pf:lem:unifsigintegr:sufficient}
\begin{proof}
Let $\mathfrak{U}\subseteq\mathcal{M}_1(\mathcal{C}^1)$ satisfy \eqref{lem:unifsigintegr:sufficient:eq1} for some $q>1$. Then for any $\mu\in\mathfrak{U}$ and with $r$ the conjugate index of $q$,
\begin{align}\label{lem:unifsigintegr:sufficient:aux1.1}
\int_{\{|\sig_w|>a\}}\!|\sig_w|\,\mathrm{d}\mu \ &\leq \ \frac{c}{m!}\int_{\mathcal{C}_d}\!\|x\|^m_{1\text{-}\mathrm{var}}\cdot\mathbbm{1}_{\{|\sig_w|>a\}}\,\mu(\mathrm{d}x)\\\label{lem:unifsigintegr:sufficient:aux1.2}
&\leq \ \frac{c}{m!}\left(\int_{\mathcal{C}_d}\|x\|^{mq}_{1\text{-}\mathrm{var}}\,\mu(\mathrm{d}x)\right)^{\!\!\frac{1}{q}}\!\!\cdot\Big(\mu\big(\{|\sig_w|>a\}\big)\Big)^{\!\frac{1}{r}} \ \, \stackrel{a\rightarrow\infty}{\longrightarrow} \, \ 0   
\end{align} 
for all $w\in[d]^{\leq m}$, where the first inequality is due to the classical signature bound \citep[Theorem 3.7]{FLO} (the constant $c$ is specified there), the second inequality is due to Hölder, and the final convergence to zero follows by Chebyshev's inequality. The latter in fact implies that 
\begin{equation}\label{lem:unifsigintegr:sufficient:aux2}
\sup\nolimits_{\mu\in\mathfrak{U}}\mu\big(\{|\sig_w|>a\}\big) \, \leq \, \frac{1}{a}\sup\nolimits_{\mu\in\mathfrak{U}}\int_{\mathcal{C}_d}\!|\sig_w(x)|\,\mu(\mathrm{d}x) \, \lesssim \, \frac{1}{a}\sup\nolimits_{\mu\in\mathfrak{U}}\int_{\mathcal{C}_d}\!\|x\|^m_{1\text{-}\mathrm{var}}\,\mu(\mathrm{d}x) \ \rightarrow \ 0   
\end{equation} 
as $a\rightarrow\infty$, where the last inequality is again due to the cited signature bound and the last supremum is finite by assumption \eqref{lem:unifsigintegr:sufficient:eq1} (recalling $L^q(\mu)\subseteq L^1(\mu)$, as the measures $\mu$ are all finite). Both \eqref{lem:unifsigintegr:sufficient:aux1.2} and the above $\mathfrak{U}$-uniformity of this convergence combine to \eqref{def:unifsigintegr:eq1}. 
\end{proof} 

\begin{remark}\label{rem:unifsigintegr}
The sets $\mathfrak{U}$ in $\fD$ that satisfy \eqref{lem:unifsigintegr:sufficient:eq1} are precisely the bounded subsets of the $mq$-Wasserstein space on $(\mathcal{C}^1, \|\cdot\|_{1\text{-}\mathrm{var}})$; for a discussion of those, see e.g.\ \citep[Sects.\ 2.1, 2.2]{panaretos2020}. Other than that, signals satisfying \eqref{lem:unifsigintegr:sufficient:eq1} are also known as Radon measures of order $mq$, and they are central to the theory of Radonifying mappings developed by Laurent Schwartz \citep{schwartz1969,schwartz1971} and others. Such signals and their uniformly bounded families are described in many sources, see e.g.\ \citep{bogachev2000, garling1973, linde1982, panaretos2020, takahashi1985} and associated references. For the special case $\mu\equiv\hat{\nu}\in\widehat{\fD}_1$ of random vectors in $\R^d$ (cf.\ Remark \ref{rem:discrete-case}), the integrals in \eqref{def:unifsigintegr:eq1} and \eqref{notation:unifsigintegr:eq1} reduce to the classical statistics  
\begin{equation}
\frac{1}{|w|!}\int_{\{|u^w|>|w|! a\}}\!|u^w|\,\nu(\mathrm{d}u) \quad\text{ and }\quad \int_{\hat{\mathcal{C}}_0}\!\|x\|_{1\text{-}\mathrm{var}}^{\tilde{q}}\,\hat{\nu}(\mathrm{d}x) = \int_{\R^d}\!|u|^{\tilde{q}}\,\nu(\mathrm{d}u), 
\end{equation}
respectively, with $u^{i_1\cdots i_m}\coloneqq u_1^{\sharp\{j\mid i_j=1\}}\cdots u_d^{\sharp\{j\mid i_j=d\}}$ for $u=(u_i)\in\R^d$. Both  \eqref{def:unifsigintegr:eq1} and \eqref{notation:unifsigintegr:eq1} thus appear as natural generalisations (from $\R^d$ to $\mathcal{C}_d$) of classical integrability conditions for multivariate moments, cf.\ e.g.\ \citep[pp.\ 30]{BIL}. \hfill $\bdiam$    
\end{remark} 

\subsection{Remark on Definition \ref{def:delta}}
\begin{remark}\label{rem:pretometric}The premetric \eqref{sect:coredinates:eq4} does not satisfy the identity of indiscernibles (``$\delta(\mu, \tilde{\mu})=0$ iff $\mu=\tilde{\mu}$''), as would be required of (a quasisemimetric or) a statistical divergence on $\fD$. This `missing' identity can be restored, however (see for instance \citep[Section 5.4]{CHO} for a proof), if instead of considering only its \nth{2} and \nth{3}-order moments one takes into account all the signature moments \eqref{def:coordinates:eq1} of a signal and computes their $\ell_2$-distance in the associated (Hilbert) coordinate space $\mathcal{V}^\infty\coloneqq\bigoplus\nolimits_{m\geq 1}(\R^d)^{\otimes m}$, which is an isometric extension of the above coordinate space $\mathcal{V}\cong(\R^d)^{\otimes 2}\oplus(\R^d)^{\otimes 3}$ that houses the coredinates \eqref{def:coordinates:eq2}. While higher-order extensions of this and other kinds are possible, cf.\ Section \ref{chap:robustICA:sect:addrems}, the following sections show that for the present case of classical [i.e.\ linear] BSS/ICA, the premetric topology $\tau_\delta$ associated [via \eqref{def:delta:eq1}] with the third-order cap \eqref{def:coordinates:eq2} is already sufficient. \hfill $\bdiam$  
\end{remark}           

\subsection{Some Technical Estimates}
\begin{appendixlemma}\label{lem:premetric_facts}
Let $p\in(1,2)$ and $\eta>0$ be arbitrary and denote $\alpha\coloneqq 1-1/p$ and $\gamma\coloneqq1/p-\alpha$.  
\begin{enumerate}[font=\upshape, label=(\roman*)]
\item\label{lem:premetric_facts:it1} Let $\{Y, \tilde{Y}\}\subset\mathcal{C}_{d|\mathbb{P}}\cap\mathfrak{D}$ be core-integrable and such that $\|\tilde{Y} - Y\|_\infty \leq \eta$ with probability one. Then for each $w\in[d]^\star$ with $2\leq |w|\leq 3$ we have
\begin{equation}\label{lem:premetric_facts:it1:eq1}
\big|\langle\tilde{Y}\rangle_w - \langle Y\rangle_w\big| \,\leq\, \kappa_{p,|w|}\cdot\eta^\alpha
\end{equation} 
with $\kappa_{p,k}\coloneqq (2+k)c_p(2^{\beta_p}K_p+1)$, for the constants $\beta_p\coloneqq 2+1/p$ and $c_p\coloneqq (1-2^{1-2/p})^{-1}$ and with $K_p\coloneqq\max_{V\in\{\tilde{Y}, Y\}}\mathbb{E}\big[\|V\|^{\beta_p}_{1\text{-}\mathrm{var}}\big]$. 
\item\label{lem:premetric_facts:it1.2} For any $\tilde{\mu}, \mu\in\fD$ such that $\max_{|w|=2,3}\big|\langle\tilde{\mu}\rangle_w - \langle\mu\rangle_w\big|\leq\kappa\eta$ for some $\kappa\geq 0$, we have
\begin{equation}\label{lem:premetric_facts:it1:eq2}
\big|\delta_{\independent}(\tilde{\mu}) - \delta_{\independent}(\mu)\big| \,\leq\, \varphi_\kappa(\mu,\tilde{\mu})\sqrt{\eta + \eta^2}
\end{equation}for the auxiliary function $\varphi_\kappa(\mu,\tilde{\mu})\coloneqq \sqrt{(3\vartheta_{\mu,\tilde{\mu}} + \kappa)\kappa\sum_{k=0}^d\sum_{i,j=1}^d\varphi_{ijk}(\mu,\tilde{\mu})}$ with 
\begin{equation}\label{lem:premetric_facts:it1:eq2.2}
\varphi_{ijk}(\mu,\tilde{\mu})\coloneqq\frac{2\big(\langle\mu\rangle_{ijk}\vee\sqrt{\langle\tilde{\mu}\rangle_{ii}\langle\tilde{\mu}\rangle_{jj}\langle\mu\rangle_{kk}}\big)^2}{\big(\!\sqrt{\langle\mu\rangle_{ii}\langle\mu\rangle_{jj}\langle\tilde{\mu}\rangle_{kk}}\,\wedge\sqrt{\langle\tilde{\mu}\rangle_{ii}\langle\tilde{\mu}\rangle_{jj}\langle\mu\rangle_{kk}}\big)^4}
\end{equation}
and $\vartheta_{\mu,\tilde{\mu}}\coloneqq\max\{\langle\mu\rangle_{ii}\langle\mu\rangle_{jj}, \langle\mu\rangle_{ii}\langle\tilde{\mu}\rangle_{jj}\mid i,j\in[d]\}$.       
\item\label{lem:premetric_facts:it2} For any $\mu, \tilde{\mu}\in\mathfrak{D}$ with $\delta(\mu,\tilde{\mu})\leq\eta$, we have
\begin{equation}\label{lem:premetric_facts:it2:eq1}
\max\nolimits_{|w|= 2,3}\big|\langle\tilde{\mu}\rangle_w - \langle\mu\rangle_w\big|\,\leq\, \mathfrak{m}_\mu\eta
\end{equation} 
for $\mathfrak{m}_\mu\coloneqq\max\{\sqrt{\langle\mu\rangle_{ii}\langle\mu\rangle_{jj}\langle\mu\rangle_{kk}} \mid i,j\in[d],k\in[d]_0\}$; if further $\mu$ is orthogonal, then $\delta_{\independent}(\tilde{\mu}) \,\leq\, \hat{\varphi}(\mu, \tilde{\mu})\sqrt{\eta + \eta^2}$ for $\hat{\varphi}(\mu,\tilde{\mu})^2\coloneqq\max\{3\vartheta_{\mu,\tilde{\mu}}\mathfrak{m}_\mu, \mathfrak{m}_\mu^2\}\sum_{k=0}^d\sum_{i,j=1}^d\varphi_{ijk}(\mu,\tilde{\mu})$. 
\item\label{lem:premetric_facts:it3} Let $R\in C^1(\R^d;\R^d)$ with\footnote{\ Here and in the following, we norm the Jacobian $D_R$ of a map $R\in C^1(\R^d;\R^d)$ via $\|D_R\|_\infty\coloneqq\sup_{x\in\R^d}\|D_R(x)\|_2$ where $\|\cdot\|_2$ is the Euclidean operator norm on $\R^{d\times d}$.} $\|D_R\|_\infty<\infty$, and let $\mu\in\fD$ be core-integrable with $\tilde{\mu}\coloneqq(\mathrm{I} + R)_\ast\mu$. Then, for each $w\in[d]^\star$ with $|w|= 2,3$: 
\begin{align}\label{lem:premetric_facts:it3:eq1}
\big|\langle\tilde{\mu}\rangle_w - \langle\mu\rangle_w\big| \,&\leq\, \tilde{K}_{p,|w|}\cdot\phi_{|w|}(\|D_R\|_\infty)\big(\|D_R\|_\infty\|R\|_\infty\big)^{\!\alpha}, \ \text{ and}\\\label{lem:premetric_facts:it3:eq1.2}
\big|\delta_{\independent}(\tilde{\mu}) - \delta_{\independent}(\mu)\big| \,&\leq\, \varphi_{\tilde{K}_{p,3}}\!(\mu,\tilde{\mu})\,\hat{\phi}_p(R)\cdot\phi_3(\|D_R\|_\infty)^{\tfrac{1}{2}}\big(\|D_R\|_\infty\|R\|_\infty\big)^{\!\tfrac{\alpha}{2}} 
\end{align} 
for $\tilde{K}_{p,k}\coloneqq 2^\alpha c_p(1+k/2)(1 + \mathbb{E}_\mu\!\big[\|x\|^\beta_{1\text{-}\mathrm{var}}\big])$ and $\phi_k(u)\coloneqq(1+u)^{k-1}u^\gamma$ and with $\varphi_{\tilde{K}_{p,3}}\!(\mu,\tilde{\mu})$ as in \emph{\ref{lem:premetric_facts:it1.2}} but for $\kappa\coloneqq\tilde{K}_{p,3}$, and $\hat{\phi}_p(R)\coloneqq\sqrt{1 + \phi_3(\|D_R\|_\infty)(\|D_R\|_\infty\|R\|_\infty)^\alpha}$.           
\end{enumerate}      
\end{appendixlemma} 
\begin{proof} 
\ref{lem:premetric_facts:it1}\,:\, Let $\{Y, \tilde{Y}\}\subset\fD$ be square-integrable with $\|\tilde{Y} - Y\|_\infty\leq\eta$ almost surely, and let $w\in[d]^{\leq m}$ be an arbitrary multiindex of length $|w|\leq m$. Take any $p\in(1,2)$. The classical interpolation inequality \citep[Prop.\ 5.5 (i)]{FVI} then implies that for $Z\coloneqq \tilde{Y} - Y$ we have
\begin{equation}\label{lem:premetric_facts:aux2.1.0}
\|Z\|_{p\text{-}\mathrm{var}} \,\leq\, 2^\alpha\|Z\|^{1/p}_{1\text{-}\mathrm{var}}\|Z\|_\infty^{(1-1/p)} \,\leq\, 2^\alpha\|Z\|^{1/p}_{1\text{-}\mathrm{var}}\eta^\alpha \,\leq\, 2\,\Xi^{1/p}\cdot\eta^\alpha  
\end{equation} 
where $\alpha\coloneqq 1-1/p$ and $\Xi\coloneqq\max\big\{\|Y\big\|_{1\text{-}\mathrm{var}},\big\|\tilde{Y}\big\|_{1\text{-}\mathrm{var}}\big\}$. The local Lipschitz continuity of the signature transform (cf.\ \citep[Proposition 7.63]{FVI}) then implies that, for each $2\leq|w|\leq m$,  
\begin{equation}\label{lem:premetric_facts:aux2}
\big|\langle\tilde{Y}\rangle_w - \langle Y\rangle_w\big| \,\leq\, \E\big|\sig_w(\tilde{Y}) - \sig_w(Y)\big|\,\leq\, (2 + |w|)c_p\,\E\big[\Xi^{|w|-1+1/p}\big]\cdot\eta^\alpha 
\end{equation}
for $c_p\coloneqq (1-2^{1-2/p})^{-1}$, as follows from the proof of \citep[Prop.\ 7.63]{FVI} but with the Young-Loeve estimates $|\tilde{Y}_{s,r} - Y_{s,r}| = |Z_{s,r}| \leq \|Z\|_{p\text{-}\mathrm{var}; [s,t]}$ and $\big|\int_s^t\tilde{Y}_{s,r}\,\mathrm{d}Z_r\big|\leq c_p\|\tilde{Y}\|_{p\text{-}\mathrm{var};[s,t]}\|Z\|_{p\text{-}\mathrm{var};[s,t]}$. 

Indeed: For each $|w|\geq 2$ we have       
\begin{align}
\big|\sig_w(\tilde{Y}) - \sig_w(Y)\big| &= \big|\langle\pi_{|w|}\big[\sig_{0,1}(\tilde{Y}) - \sig_{0,1}(Y)\big], w\rangle\big| \leq \big\|\pi_{|w|}\big[\sig_{0,1}(\tilde{Y}) - \sig_{0,1}(Y)\big]\big\|\\
&\hspace{-4em}= \left\|\int_0^1\!\pi_{|w|-1}\big[\sig_{0,r}(\tilde{Y}) - \sig_{0,r}(Y)\big]\mathrm{d}\tilde{Y}_r + \int_0^1\!\pi_{|w|-1}\big[\sig_{0,r}(Y)\big]\mathrm{d}Z_r\right\|\label{lem:premetric_facts:aux2.1.1} 
\end{align}  
where $\|\cdot\|\coloneqq |\cdot|_1$ is the projective tensor norm on the tensor powers of $\R^d$, cf.\ \citep[Sect.\ 5.6]{LYQ}. Using that for any $x\in\mathcal{C}^1([0,1];\R^d)$ and $y\in\mathcal{C}^1([0,1];\R^{e \times d})$ we have 
\begin{equation}
\begin{gathered}
\int_s^\cdot\!y_{s,r}\,\mathrm{d}x_r\,\in\,\mathcal{C}^1([s,t];\R^e) \quad\text{ and }\quad 
\left|\int_s^t\!y_{s,r}\,\mathrm{d}x_r\right| \,\leq\, c_p\|y\|_{p\text{-}\mathrm{var};[s,t]}\|x\|_{p\text{-}\mathrm{var}; [s,t]}
\end{gathered}
\end{equation}     
for each $0\leq s\leq t\leq 1$, see e.g.\ \citep[Sect.\ 1.3]{FLO} and \citep[Prop.\ 6.4]{FVI}, we find from \eqref{lem:premetric_facts:aux2.1.1} that 
\begin{equation}
\begin{aligned}
\big|\sig_w(\tilde{Y}) - \sig_w(Y)\big| \,&\leq\, c_p\big(\|Y\|_{p\text{-}\mathrm{var}} + \|\tilde{Y}\|_{p\text{-}\mathrm{var}}\big)\|Z\|_{p\text{-}\mathrm{var}}\leq 2c_p\Xi\|Z\|_{p\text{-}\mathrm{var}} \ \text{ if } \ |w|=2, \ \text{ and}\\
\big|\sig_w(\tilde{Y}) - \sig_w(Y)\big|\,&\leq\, \big(2c_p\Xi\|Y\|_{1\text{-}\mathrm{var}} + \frac{c_p}{2!}\|\tilde{Y}\|^2_{1\text{-}\mathrm{var}}\big)\|Z\|_{p\text{-}\mathrm{var}} \,\leq\, \tfrac{5}{2}c_p\Xi^2\|Z\|_{p\text{-}\mathrm{var}} \ \text{ if } \ |w|=3,
\end{aligned} 
\end{equation} 
where for the penultimate inequality we used the standard estimate $\|\pi_2[\sig_{s,t}(\tilde{Y})]\|_{1\text{-}\mathrm{var}}\leq\tfrac{1}{2!}\|\tilde{Y}\|^2_{1\text{-}\mathrm{var}}$ (see e.g.\ \citep[Prop.\ 2.2]{FLO}). Combined with \eqref{lem:premetric_facts:aux2.1.0}, we obtain that for $|w|=2,3$,      
\begin{equation}
\big|\langle\tilde{Y}\rangle_w - \langle Y\rangle_w\big| \,\leq\, (1 + |w|/2)c_p\mathbb{E}\!\left[\Xi^{|w|-1}\|Z\|_{p\text{-}\mathrm{var}}\right] \,\leq\, (2 + |w|)c_p\mathbb{E}\!\left[\Xi^{|w|-1+1/p}\right]\cdot\eta^\alpha,  
\end{equation}
as claimed in \eqref{lem:premetric_facts:aux2}. To finally arrive at \eqref{lem:premetric_facts:it1:eq1}, notice that $\Xi^{|w|-1+1/p} \leq 1 + \Xi^{2+1/p} \leq 1 + \big(\|\tilde{Y}\|_{1\text{-}\mathrm{var}} + \|Y\|_{1\text{-}\mathrm{var}}\big)^\beta$ for each $|w|=2,3$ and $\beta\coloneqq 2+1/p$ and hence, by Minkowski, 
\begin{align}\label{lem:premetric_facts:aux2.1.2}
\mathbb{E}\!\left[\Xi^{|w|-1+1/p}\right] - 1 \,\leq\, \left[\mathbb{E}\big[(\tilde{U} + U)^\beta\big]^{1/\beta}\right]^{\!\beta} \leq\, \left[\mathbb{E}\big[\tilde{U}^\beta\big]^{\!1/\beta} \!\!+ \mathbb{E}\big[U^\beta\big]^{\!1/\beta}\right]^{\!\beta} \leq\, 2^\beta K 
\end{align}  
for $\tilde{U}\coloneqq \|\tilde{Y}\|_{1\text{-}\mathrm{var}}$ and $U\coloneqq\|Y\|_{1\text{-}\mathrm{var}}$, as asserted.   

\ref{lem:premetric_facts:it1.2}\,:\, Let $\mu, \tilde{\mu}\in\fD$ and suppose, as required, that there is $\kappa\geq 0$ and $\eta>0$ such that
\begin{equation}\label{lem:premetric_facts:aux3.0}
\max_{|w|=2,3}\big|\langle\tilde{\mu}\rangle_w - \langle\mu\rangle_w\big|\leq\kappa\eta.
\end{equation}
The bound \eqref{lem:premetric_facts:it1:eq2} then follows from \eqref{lem:premetric_facts:aux3.0} and the definitions \eqref{def:delta:eq1} and \eqref{def:delta:eq2}, as shown next. 

\noindent
We use the representation \eqref{sect:coredinates:eq4} of $\delta$ for convenience. As is conventional for direct sums, we thereby write $C\equiv C_0 \bm{+} \ldots \bm{+} C_d$ for any $C\equiv(C_0, \cdots, C_d)\in \mathcal{V}$ (boldfaced `$+$'), and accordingly set $N_1\cdot C\cdot N_2 \equiv \bm{\sum}_{\nu=0}^d N_1C_\nu N_2 \equiv (N_1C_\nu N_2)_{\nu=0}^d\in\mathcal{V}$ for any $N_1, N_2 \in \mathbb{R}^{d\times d}$. Then 
\begin{equation}\label{lem:premetric_facts:aux3}
\delta(\mu, \tilde{\mu}) \, = \, \left\|N_\mu^{-1}\big(C_0 \bm{+} \textstyle{\bm{\sum}}_{k=1}^d C_k\big)N_\mu^{-1}\right\|_{\mathcal{V}} \quad\text{ with }\quad C_\nu= \left(\frac{\langle\tilde{\mu}\rangle_{ij\nu} - \langle\mu\rangle_{ij\nu}}{\sqrt{\langle\mu\rangle_{\nu\nu}}}\right)_{ij}   
\end{equation}
where $\langle\mu\rangle_{ij0}\coloneqq\langle\mu\rangle_{ij}$, $\langle\mu\rangle_{00}\coloneqq 1$ and $N_\mu$ as in \eqref{sect:coredinates:eq4}; we further set $C_{\mu\tilde{\mu}}\coloneqq C_0 \bm{+} \bm{\sum}_{k=1}^d C_k$. The assumed core-integrability of $\{\mu,\tilde{\mu}\}\subset\fD$ ensures that all of these quantities exist. 
 
The bound \eqref{lem:premetric_facts:it1:eq2} is based on the simple triangle inequality
\begin{equation}\label{lem:premetric_facts:aux4}
\begin{aligned}
\big|\delta_{\independent}(\tilde{\mu}) - \delta_{\independent}(\mu)\big| \, &= \, \Big|\big\|N_{\tilde{\mu}}^{-1}C_{\tilde{\mu}_\star\tilde{\mu}}N_{\tilde{\mu}}^{-1}\big\|_{\mathcal{V}} - \big\|N_{\mu}^{-1}C_{\mu_\star\mu}N_\mu^{-1}\big\|_{\mathcal{V}}\Big| \\
&\leq \, \left\|N_{\tilde{\mu}}^{-1}\Big[C_{\tilde{\mu}_\star\tilde{\mu}}R - R^{-1}C_{\mu_\star\mu}\Big]N_\mu^{-1}\right\|_{\mathcal{V}} \qquad \text{ for } \ R\coloneqq N_{\tilde{\mu}}^{-1}N_\mu, 
\end{aligned}
\end{equation}
where for a signal $\mu=(\mu^i)\in\fD$ we denote by $\mu_\star\coloneqq\mu^1\otimes\cdots\otimes\mu^d$ its associated product signal. By definition the matrices $R\equiv(R_{ij})$ are diagonal with $R_{ii}= \sqrt{\langle\mu\rangle_{ii}/\langle\tilde{\mu}\rangle_{ii}}$. 
Moreover, if $\tilde{\mu}, \mu$ are both mean-stationary, then so are their respective product signals $\tilde{\mu}_\star, \mu_\star$, whence Lemma \ref{lemma1} gives that the (flattened) tensors $C_{\mu_\star \mu}\eqqcolon(C_{ijk})$ and $C_{\tilde{\mu}_\star\tilde{\mu}}\eqqcolon(\tilde{C}_{ijk})$ read   
\begin{equation}\label{lem:premetric_facts:aux6}
C_{ij0} = (1-\delta_{ij})\langle \mu\rangle_{ij}  \quad\text{ and }\quad C_{ijk} = \frac{(1-\delta_{ijk})\langle \mu\rangle_{ijk}}{\sqrt{\langle \mu\rangle_{kk}}} \qquad (i,j,k\in[d]),
\end{equation}   
and likewise for $(\tilde{C}_{ijk})$. If $\tilde{\mu}$ or $\mu$ are not mean-stationary, then all of the above remains valid as stated but with \eqref{lem:premetric_facts:aux6} holding by definition of $\delta_{\independent}(\cdot)$, see Remark \ref{rem:icdefect-nonstationary}. The flattened [along its \nth{3} dimension, i.e.\ embedded in $\mathcal{V}$] tensor $M\equiv (m_{ijk})\coloneqq N_{\tilde{\mu}}^{-1}(C_{\tilde{\mu}_\star\tilde{\mu}}R - R^{-1}C_{\mu_\star \mu})N_\mu^{-1}$ thus reads
\begin{equation}\label{lem:premetric_facts:aux7} 
\begin{aligned}
m_{ij0} &= \frac{C_{ij0} R_{jj} - \tilde{C}_{ij0} R_{ii}^{-1}}{\sqrt{\langle\tilde{\mu}\rangle_{ii}\langle \mu\rangle_{jj}}} = \left[\frac{\langle \mu\rangle_{ij}}{\sqrt{\langle\tilde{\mu}\rangle_{ii}\langle\tilde{\mu}\rangle_{jj}}} - \frac{\langle\tilde{\mu}\rangle_{ij}}{\sqrt{\langle \mu\rangle_{ii}\langle \mu\rangle_{jj}}}\right]\!(1 - \delta_{ij}), \qquad\text{ and }\\
\quad m_{ijk} &= \frac{C_{ijk} R_{jj} - \tilde{C}_{ijk} R_{ii}^{-1}}{\sqrt{\langle \tilde{\mu}\rangle_{ii}\langle \mu\rangle_{jj}}} = \left[\frac{\langle \mu\rangle_{ijk}}{\sqrt{\langle\tilde{\mu}\rangle_{ii}\langle\tilde{\mu}\rangle_{jj}\langle \mu\rangle_{kk}}} - \frac{\langle\tilde{\mu}\rangle_{ijk}}{\sqrt{\langle \mu\rangle_{ii}\langle \mu\rangle_{jj}\langle\tilde{\mu}\rangle_{kk}}}\right]\!(1-\delta_{ijk}) 
\end{aligned} 
\end{equation}   
Writing $\langle\mu\rangle_{00}\coloneqq 1$ and $\langle\mu\rangle_{ij0}\coloneqq\langle\mu\rangle_{ij}$, as well as $\beta_{ijk}\coloneqq\sqrt{\langle \mu\rangle_{ii}\langle \mu\rangle_{jj}\langle\tilde{\mu}\rangle_{kk}}$ and $\tilde{\beta}_{ijk}\coloneqq\sqrt{\langle \tilde{\mu}\rangle_{ii}\langle \tilde{\mu}\rangle_{jj}\langle \mu\rangle_{kk}}$ and $\gamma_{ijk}\coloneqq\mathrm{min}\{\beta_{ijk}, \tilde{\beta}_{ijk}\}$ and $\varsigma_{ijk}\coloneqq\max\{|\langle \mu\rangle_{ijk}|,\,|\tilde{\beta}_{ijk}|\}$, we find that 
\begin{equation}\label{lem:premetric_facts:aux7.1}
\begin{aligned}
|m_{ijk}| = \left|\frac{\langle \mu\rangle_{ijk}\beta_{ijk} - \langle\tilde{\mu}\rangle_{ijk}\tilde{\beta}_{ijk}}{\beta_{ijk}\tilde{\beta}_{ijk}}\right| \,\leq\, \frac{\varsigma_{ijk}}{\gamma_{ijk}^2}\Big(|\beta_{ijk} - \tilde{\beta}_{ijk}| + |\langle \mu\rangle_{ijk} - \langle\tilde{\mu}\rangle_{ijk}|\Big).
\end{aligned}
\end{equation}     
Writing $\nu_j\coloneqq\langle \mu\rangle_{ii}$ and $\tilde{\nu}_j\coloneqq\langle\tilde{\mu}\rangle_{jj}$, for the first of the above bounding summands we find\footnote{\ Recall that $|\sqrt{a} - \sqrt{b}\,| \leq \sqrt{|a-b|}$ for any $a,b\geq 0$.}
\begin{equation}\label{lem:premetric_facts:aux7.2}
\begin{aligned}
\big|\beta_{ijk} - \tilde{\beta}_{ijk}\big|^2 \,&\leq\, \big|\nu_i\nu_j\tilde{\nu}_k - \tilde{\nu}_i\tilde{\nu}_j\nu_k\big|  
\,\leq\, \nu_i\nu_j\big|\tilde{\nu}_k - \nu_k\big| + \nu_k\big|\nu_i\nu_j - \tilde{\nu}_i\tilde{\nu}_j\big|\\
&\leq\, \nu_i\nu_j\big|\tilde{\nu}_k - \nu_k\big| + \nu_k\big(\nu_i|\nu_j-\tilde{\nu}_j| + \tilde{\nu}_j|\nu_i-\tilde{\nu}_i|\big).   
\end{aligned}
\end{equation} 
Hence after denoting $\vartheta\coloneqq\max\{\nu_i\nu_j, \nu_i\tilde{\nu}_j\mid i,j\in[d]\}$ and by \eqref{lem:premetric_facts:aux3.0}, we obtain that
\begin{equation}\label{lem:premetric_facts:aux7.3}
\big|\beta_{ijk} - \tilde{\beta}_{ijk}\big| \,\leq\, \sqrt{3\vartheta\kappa}\cdot\eta^{\tfrac{1}{2}}. 
\end{equation} 
Defining $\xi_{ijk}\coloneqq \sqrt{2}\varsigma_{ijk}/\gamma_{ijk}^2$, we can now combine \eqref{lem:premetric_facts:aux7.3} and \eqref{lem:premetric_facts:aux3.0} with \eqref{lem:premetric_facts:aux7.1} to arrive at
\begin{equation}\label{lem:premetric_facts:aux7.4}
m_{ijk}^2 \,\leq \, \xi_{ijk}^2\big(|\beta_{ijk} - \tilde{\beta}_{ijk}|^2 + |\langle \mu\rangle_{ijk} - \langle\tilde{\mu}\rangle_{ijk}|^2\big) \,\leq\, K_1\xi_{ijk}^2\cdot(\eta+\eta^2) 
\end{equation}
for $K_1\coloneqq (3\vartheta+\kappa)\kappa$. Using \eqref{lem:premetric_facts:aux4}, this allows us to conclude   
\begin{equation}\label{lem:premetric_facts:aux8}
\big|\delta_{\independent}(\tilde{\mu}) - \delta_{\independent}(\mu)\big|^2  \ \leq \ \sum_{k=0}^d\sum_{i,j=1}^d m_{ijk}^2 \, \leq \, K_2\cdot(\eta+\eta^2)
\end{equation}   
for $K_2\coloneqq K_1\sum_{k=0}^d\sum_{i,j=1}^d\xi_{ijk}^2$, which gives \eqref{lem:premetric_facts:it1:eq2} as claimed.  

\ref{lem:premetric_facts:it2}\,:\, Given $\delta(\mu, \tilde{\mu})\leq\eta$, we know from Def.\ \ref{def:delta} that for each $w\equiv ijk\in[d]^\star_{|\cdot|=2,3}$ we have 
\begin{equation}\label{lem:premetric_facts:aux9}
\big|\langle\tilde{\mu}\rangle_w - \langle\mu\rangle_w\big| \leq \sqrt{\langle\mu\rangle_{ii}\langle\mu\rangle_{jj}\langle\mu\rangle_{kk}}\cdot\eta \,\leq\, \mathfrak{m}_\mu\eta
\end{equation}
for $\mathfrak{m}_\mu=\max\{\sqrt{\langle\mu\rangle_{ii}\langle\mu\rangle_{jj}\langle\mu\rangle_{kk}} \mid i,j\in[d],k\in[d]_0\}$. Hence by the very same computations as in point \ref{lem:premetric_facts:it1.2} (upon replacing assumption \eqref{lem:premetric_facts:aux3.0} by \eqref{lem:premetric_facts:aux9}) we find     
\begin{equation}\label{lem:premetric_facts:aux10}
\big|\delta_{\independent}(\tilde{\mu}) - \delta_{\independent}(\mu)\big|^2 \,\leq\, \max\{3\vartheta_{\mu,\tilde{\mu}}\mathfrak{m}_\mu, \mathfrak{m}_\mu^2\}\sum_{k=0}^d\sum_{i,j=1}^d\xi_{ijk}^2\cdot(\eta+\eta^2) 
\end{equation}
with $\vartheta_{\mu,\tilde{\mu}}$ and $\xi_{ijk}$ as defined previously. If $\delta_{\independent}(\mu)=0$ (cf.\ Lemma \ref{lemma1}), this gives the claim. 

\ref{lem:premetric_facts:it3}\,:\, Let $w\in[d]^\star$ be an arbitrary multiindex. Then by the change-of-variables formula,
\begin{equation}\label{lem:premetric_facts:aux11}
\begin{aligned}
\big|\langle\tilde{\mu}\rangle_w - \langle\mu\rangle_w\big| \,&=\, \left|\int_{\mathcal{C}_d}\!\sig_w\big(x + R(x)\big)\,\mu(\mathrm{d}x) - \int_{\mathcal{C}_d}\!\sig_w(x)\,\mu(\mathrm{d}x)\right|\\
&\leq\, \int_{\mathcal{C}_d}\!\big|\sig_w\big(x + R(x)\big) - \sig_w(x)\big|\,\mu(\mathrm{d}x).   
\end{aligned} 
\end{equation}    
Now by definition of $\|\cdot\|_{1\text{-}\mathrm{var}}$ we find (with the sup taken over all dissections $(t_\nu)$ of $[0,1]$)
\begin{equation}\label{lem:premetric_facts:aux12}
\|R(x)\|_{1\text{-}\mathrm{var}}=\sup_{(t_\nu)}\sum_\nu\big|R(x_{t_{\nu+1}}) - R(x_{t_\nu})\big| \,\leq\, \|D_R\|_\infty\|x\|_{1\text{-}\mathrm{var}} 
\end{equation}
and hence -- again by \citep[Proposition 5.5 (i)]{FVI}, just as for \eqref{lem:premetric_facts:aux2.1.0} -- that 
\begin{equation}\label{lem:premetric_facts:aux13}
\|R(x)\|_{p\text{-}\mathrm{var}} \,\leq\, 2^\alpha\|D_R\|_\infty^{1/p}\|x\|_{1\text{-}\mathrm{var}}^{1/p}\|R\|_\infty^\alpha\,=\, 2^\alpha\|x\|_{1\text{-}\mathrm{var}}^{1/p}\|D_R\|_\infty^\gamma\cdot\big(\|D_R\|_\infty\|R\|_\infty\big)^{\!\alpha}. 
\end{equation} 
for $\gamma\coloneqq 1/p-\alpha>0$. From the estimates of point \ref{lem:premetric_facts:it1} -- with $\tilde{Y}(\omega)$ and $Y(\omega)$ replaced by $x+R(x)$ and $x$, respectively -- combined with \eqref{lem:premetric_facts:aux12}, we obtain for $|w|=2,3$ that
\begin{equation}\label{lem:premetric_facts:aux14}
\big|\sig_w\big(x + R(x)\big) - \sig_w(x)\big| \,\leq\, (1+|w|/2)c_p(1+\|D_R\|_\infty)^{|w|-1}\|x\|_{1\text{-}\mathrm{var}}^{|w|-1}\|R(x)\|_{p\text{-}\mathrm{var}}.   
\end{equation}       
Together with \eqref{lem:premetric_facts:aux11} and \eqref{lem:premetric_facts:aux13}, the above combines to 
\begin{equation}
\big|\langle\tilde{\mu}\rangle_w - \langle\mu\rangle_w\big| \,\leq\, \tilde{c}_{p,|w|}\mathbb{E}_\mu\big[\|x\|_{1\text{-}\mathrm{var}}^{|w|-1+1/p}\big]\cdot\phi_{|w|}(\|D_R\|_\infty)\big(\|D_R\|_\infty\|R\|_\infty\big)^{\!\alpha}   
\end{equation}
for $\tilde{c}_{p,|w|}\coloneqq 2^\alpha c_p(1+|w|/2)$ and $\phi_k(u)\coloneqq(1+u)^{k-1}u^\gamma$. Since $\mathbb{E}_\mu\big[\|x\|_{1\text{-}\mathrm{var}}^{|w|-1+1/p}\big]\leq 1 + \mathbb{E}_\mu\big[\|x\|^\beta_{1\text{-}\mathrm{var}}\big]\eqqcolon 1 + \tilde{K}_p$ for $\beta=2+1/p$, we thus find for each $w\in[d]^\star$ with $|w|=2,3$ that   
\begin{equation}\label{lem:premetric_facts:aux15}
\big|\langle\tilde{\mu}\rangle_w - \langle\mu\rangle_w\big| \,\leq\, \tilde{c}_{p,|w|}(1 + \tilde{K}_p)\cdot\phi_{|w|}(\|D_R\|_\infty)\big(\|D_R\|_\infty\|R\|_\infty\big)^{\!\alpha} 
\end{equation} 
as claimed. Given \eqref{lem:premetric_facts:aux15}, the last inequality \eqref{lem:premetric_facts:it3:eq1.2} then holds by \ref{lem:premetric_facts:it1.2} for $\kappa\coloneqq \tilde{c}_{p,3}(1+\tilde{K}_p)$ and $\eta\coloneqq\phi_3(\|D_R\|_\infty)(\|D_R\|_\infty\|R\|_\infty)^\alpha$.    
\end{proof}    
  
\subsection{Proof of Lemma \ref{lem:premetric_facts2}}\label{pf:lem:premetric_facts2}
\begin{proof}
Note that any functional $\varphi$ which acts on signals in $\dot{\fD}$ via their coredinates \eqref{def:coordinates:eq2} (such as $\varphi_\kappa$, $\mathfrak{m}_\mu$ etc.\ in Lemma \ref{lem:premetric_facts}), factorizes as $\varphi = \phi_\varphi\circ[\,\cdot\,]^{\times k}$ with $[\,\cdot\,]^{\times k} : \dot{\fD}^{\times k}\rightarrow\mathcal{V}^{\times k}$ and $\phi_\varphi: \mathcal{V}^{\times k}\rightarrow\R$ (for $(\mathcal{V}, \|\cdot\|_{\mathcal{V}})$ as in \eqref{sect:coredinates:eq4} and $k$ the number of arguments of $\varphi$), where $[\,\cdot\,] : \mu \mapsto ([\mu]_\nu)\in\mathcal{V}$ is the coredinate map (see \eqref{def:coordinates:eq1}) and $\phi_\varphi$ is the coredinate action of $\varphi$. 

Note further that by definition \eqref{lem:premetric_topofacts:eq1} of the $\delta$-balls and the restriction $r<\hat{r}$, the image set 
\begin{equation}\label{lem:premetric_facts2:aux1}
[B_r]\equiv[\,\cdot\,](B_r) \quad \text{is contained in} \quad \hat{\mathfrak{B}}\coloneqq\bar{B}_{q\rho_0}([\hat{\mu}]_0)\times\bar{B}_{q\rho_0\sqrt{\rho_1}}([\hat{\mu}]_{\nu\geq 1})\equiv\hat{\mathfrak{B}}_0\times\hat{\mathfrak{B}}_1  
\end{equation} 
for some $q\equiv q_r\in(0,1)$, where we defined $\hat{\mathfrak{B}}_0\equiv\bar{B}_{q\rho_0}([\hat{\mu}]_0)\coloneqq\{C\in\R^{d\times d}\mid \|C-[\hat{\mu}]_0\| \leq q\rho_0\}$ and $\hat{\mathfrak{B}}_1\equiv\bar{B}_{q\rho_0\sqrt{\rho_1}}([\hat{\mu}]_{\nu\geq 1})\coloneqq\{C'\in(\R^{d\times d})^{\oplus d}\mid \|(0,C') - (0,[\hat{\mu}]_1,\ldots,[\hat{\mu}]_d)\|_{\mathcal{V}} \leq q\rho_0\sqrt{\rho_1}\}$. 

Clearly $\hat{\mathfrak{B}}\subset\mathcal{V}$ compact, so that (by the extreme value theorem) we have
\begin{equation}\label{lem:premetric_facts2:aux2}
K_{\varphi|r}\coloneqq\sup_{\mu, \tilde{\mu}\in B_r} |\varphi(\mu,\tilde{\mu})| \,\leq\, \sup_{(\mathfrak{q}, \tilde{\mathfrak{q}})\in\hat{\mathfrak{B}}^{\times 2}} |\phi_\varphi(\mathfrak{q}, \tilde{\mathfrak{q}})| < \infty  \quad\text{ if }\quad \restr{\phi_\varphi}{\hat{\mathfrak{B}}^{\times 2}} \text{ is continuous.} 
\end{equation} 
This allows us to derive the asserted inequalities from Lemma \ref{lem:premetric_facts}, as done below.
 
\ref{lem:premetric_facts2:it1}\,:\, Take any $\tilde{\mu}\in B_r$, for which then \eqref{lem:premetric_facts:aux9} and \eqref{lem:premetric_facts:aux10} (for $\mu\coloneqq\hat{\mu}$ and $\eta\coloneqq r$). Hence
\begin{equation}\label{lem:premetric_facts2:aux3}
\sup_{\mu\in B_r}\!\Big[\max_{|w|=2,3}\big|\langle\mu\rangle_w - \langle\hat{\mu}\rangle_w\big|\Big]\leq \mathfrak{m}_{\hat{\mu}|r}r \quad\text{and}\quad \sup_{\mu\in B_r}\!\Big[\big|\delta_{\independent}(\mu) - \delta_{\independent}(\hat{\mu})\big|\Big]\leq K_{\hat{\mu}|r}\sqrt{r + r^2} 
\end{equation} 
for $\mathfrak{m}_{\hat{\mu}|r}\coloneqq\sup_{\mu\in B_r}\mathfrak{m}_\mu$ and $K_{\hat{\mu}|r}\equiv K_{\hat{\varphi}|r}\coloneqq\sup_{\mu\in B_r}|\hat{\varphi}(\mu,\hat{\mu})|$, with $\mathfrak{m}_\mu$ and $\hat{\varphi}$ as in Lemma \ref{lem:premetric_facts} \ref{lem:premetric_facts:it2}. Now since -- as noted at the beginning of this proof -- both $\mathfrak{m}_\mu$ and  $\hat{\varphi}$ factor through the map $[\,\cdot\,]$ resp.\ $[\,\cdot\,]^{\times 2}$ with coredinate actions $\phi_{\mathfrak{m}}$ resp.\ $\phi_{\hat{\varphi}}$, the finiteness of $\mathfrak{m}_{\hat{\mu}|r}$ and $K_{\hat{\mu}|r}$ holds by \eqref{lem:premetric_facts2:aux2}, which is applicable as the restricted actions $\restr{\phi_{\hat{\varphi}}}{\hat{\mathfrak{B}}^{\times 2}}$ and $\restr{\phi_{\mathfrak{m}}}{\hat{\mathfrak{B}}^{\times 2}}$ are both continuous; the latter follows from\footnote{\ Indeed: Clearly $\phi_{\hat{\varphi}}=\sqrt{\max\{3\phi_{\vartheta}, \phi_{\mathfrak{m}}, \phi_{\mathfrak{m}}^2\}\sum_{k,i,j}\phi_{\varphi_{ijk}}}$ with $\phi_{\varphi_{ijk}}\circ[\,\cdot\,]^{\times 2} = \varphi_{ijk}$, (cf.\ \eqref{lem:premetric_facts:it1:eq2.2}), and the inclusion \eqref{lem:premetric_facts2:aux1}, specifically $\{[\mu]_0\mid \mu\in B_r\}\subseteq\hat{\mathfrak{B}}_0$, implies that the functions $\restr{\phi_{\varphi_{ijk}}}{{\hat{\mathfrak{B}}^{\times 2}}}$ are each continuous since their denominators are bounded away from zero. (To be explicit, the latter is seen from
\begin{equation}\label{lem:premetric_facts2:aux4}
\langle\mu\rangle_{ii} \,\geq\, \langle\hat{\mu}\rangle_{ii} - |\langle\hat{\mu}\rangle_{ii} - \langle\mu\rangle_{ii}| \,\geq\, \langle\hat{\mu}\rangle_{ii} - \|[\hat{\mu}]_0 - [\mu]_0\| \,\geq\, \langle\hat{\mu}\rangle_{ii} - q\rho_0 \,\geq\, (1-q)\rho_0 \,>\, 0\,
\end{equation}which holds for each $\mu\in B_r$ and all $i\in[d]$.) The continuity of the remaining constituents of $\phi_{\hat{\varphi}}$, i.e.\ $\phi_\vartheta$ and $\phi_{\mathfrak{m}}$, is clear by definition of $\vartheta$ and $\mathfrak{m}$, respectively (see Lemma \ref{lem:premetric_facts}, where said definitions are given). } \eqref{lem:premetric_facts2:aux1} resp.\ the definition of $\mathfrak{m}_\mu$.     

\ref{lem:premetric_facts2:it2}\,:\, This assertion follows similarly from Lemma \ref{lem:premetric_facts} \ref{lem:premetric_facts:it3}. Indeed: Let $\mu\in B_r$ be arbitrary. Then $\|[\mu]_0-[\hat{\mu}]_0\| \leq \rho_1 r$, thus $\langle\mu\rangle_{ii}\geq\langle\hat{\mu}\rangle_{ii} - |\langle\hat{\mu}\rangle_{ii} - \langle\mu\rangle_{ii}| \geq \rho_0 - \rho_1 r$ and, hence,  
\begin{equation}\label{lem:premetric_facts2:aux5}
\langle\tilde{\mu}\rangle_{ii} \,\geq\, \langle\mu\rangle_{ii} - |\langle\tilde{\mu}\rangle_{ii} - \langle\mu\rangle_{ii}| \,\geq\, \rho_0 - \rho_1 r - \rho_R \,>\, 0 \quad \text{ for }\quad \tilde{\mu}\coloneqq(I+R)_\star\mu
\end{equation}  
and any $i\in[d]$ (recall that the penultimate inequality holds by \eqref{lem:premetric_facts:it3:eq1}). The lower bounds on the $\langle\mu\rangle_{ii}$ and $\langle\tilde{\mu}\rangle_{ii}$ prevents the functional $\varphi_R: \mu \mapsto \varphi_{\tilde{\kappa}}(\mu, (I+R)_\star\mu)\hat{\phi}_p(R)$ (with $\hat{\phi}_p$ and $\varphi_{\tilde{\kappa}}$ as in \eqref{lem:premetric_facts:it3:eq1.2} but for $\tilde{\kappa}\coloneqq2^{\alpha+1}c_p(1+K_0)\geq\tilde{K}_{p,3}$), resp.\ its coredinate action $\phi_{\varphi_R}$, from blowing up on the domain $B_r$, resp.\ on $\tilde{\mathfrak{B}}\equiv\bar{B}_{\rho_1r}([\hat{\mu}]_0)\times\hat{\mathfrak{B}}_1\supseteq[B_r]$. As for point \ref{lem:premetric_facts2:it1} this ensures that the action $\restr{\phi_{\varphi_R}}{\tilde{\mathfrak{B}}}$ is continuous, whence from \eqref{lem:premetric_facts:it3:eq1.2} we obtain as desired that
\begin{equation}
\sup\nolimits_{\mu\in B_r}\big|\delta_{\independent}((I+R)_\star\mu) - \delta_{\independent}(\mu)\big| \,\leq\, \tilde{K}_{\hat{\mu}|r}\cdot\phi_3(\|D_R\|_\infty)^{\tfrac{1}{2}}\big(\|D_R\|_\infty\|R\|_\infty\big)^{\!\tfrac{\alpha}{2}}
\end{equation}  
for the finite -- by \eqref{lem:premetric_facts2:aux2} -- constant $\tilde{K}_{\hat{\mu}|r}\equiv\tilde{K}_{\hat{\mu}|r, R}\coloneqq\sup_{\mu\in B_r}|\varphi_R(\mu)|<\infty$.  

For any $C, c >0$ with $c < \rho_0$ and $\mathcal{R}^C_c$ as in \eqref{lem:premetric_facts2:eq3}, the lemma's last point follows by replacing the above $\tilde{K}_{\hat{\mu}|r}$ with $\tilde{K}_{\hat{\mu}|c,C}\coloneqq\sup_{(\mu,R)\in\mathcal{R}^C_c}|\varphi_R(\mu)|$, which is clearly finite as is seen from the definition of $\varphi_R$, the (lower and upper) bounds \eqref{lem:premetric_facts2:aux5} [and its preceding line] and \eqref{lem:premetric_facts:it3:eq1}.   
\end{proof}  

\subsection{Proof of Proposition \ref{prop1} and Lemma \ref{lem:scaleinvariance}}\label{pf:prop1}
\begin{proof}[Proof of Proposition \ref{prop1}]
Applied to the identity $\chi = A\cdot\zeta$ and written out at order $m=2,3$, the above equivariance of the signature moments reads
\begin{equation}\label{sect:linequiv:eq0}
\begin{aligned}
\langle\chi\rangle_{ij} &= \textstyle\sum_{\alpha,\beta=1}^d a_{i\alpha}a_{j\beta}\cdot\langle\zeta\rangle_{\alpha\beta} &&\big((i,j)\in[d]^2\big) \qquad \text{ and}\\ 
\langle\chi\rangle_{ijk} &= \textstyle\sum_{\alpha,\beta,\gamma=1}^d a_{i\alpha}a_{j\beta}a_{k\gamma}\cdot\langle\zeta\rangle_{\alpha\beta\gamma} &&\big((i,j,k)\in[d]^3\big)
\end{aligned}
\end{equation} 
which, recalling \eqref{def:coordinates:eq2}, is a system of equations equivalent to the matrix congruences
\begin{equation}\label{sect:linequiv:eq1}
[\chi]_{0} = A[\zeta]_0A^{\intercal} \quad\text{and}\quad [\chi]_k=A\Big[\textstyle\sum_{\gamma=1}^d a_{k\gamma}[\zeta]_\gamma\Big]A^{\intercal}  
\end{equation} 
for each $k\in[d]$. The identities \eqref{prop1:eq2} follow. As to the conjugacy \eqref{prop1:eq3}, this holds by \eqref{sect:linequiv:eq1} upon noting that $[\chi]_\odot^c = c_0[\chi]_0^{-1}\cdot[\chi]_{\underline{c}}$ for the linear combination
\begin{equation}
[\chi]_{\underline{c}}\ \equiv\ \textstyle\sum_{\nu=1}^d \tilde{c}_\nu[\chi]_\nu \ = \ A\Big[\textstyle\sum_{\gamma=1}^d\!\big(\sum_{\nu=1}^d a_{\nu\gamma}c_\nu\big)[\zeta]_\gamma\Big]A^\intercal \ = \ A\Big[\textstyle\sum_{\gamma=1}^d (a_\gamma\cdot\underline{c})[\zeta]_\gamma\Big]A^\intercal
\end{equation}       
as this implies that 
\begin{equation}
\begin{aligned}
[\chi]_\odot^c \ = \ A^{-\intercal}c_0[\zeta]_0^{-1}A^{-1}\cdot A\Big[\textstyle\sum_{\gamma=1}^d (a_\gamma\cdot\underline{c})[\zeta]_\gamma\Big]A^\intercal \ = \ A^{-\intercal}\cdot[\zeta]_\odot^{A|c}\cdot A^\intercal. 
\end{aligned}
\end{equation} 
The asserted diagonality of $[\zeta]^{A|c}_\odot$ is clear by Lemma \ref{lemma1}. 
\end{proof}

\begin{proof}[Proof of Lemma \ref{lem:scaleinvariance}]
As $\chi$ is arbitrary and the $\GL$-action \eqref{transformation} leaves $\mathcal{D}$ invariant, it suffices to show \eqref{lem:scaleinvariance:eq2} for $\theta\coloneqq\mathrm{id}$ and any $\Lambda\equiv\mathrm{ddiag}[\lambda_1, \ldots, \lambda_d]$ with $\lambda_i>0$. Then by Lemma \ref{lem:linequiv}, cf.\ \eqref{sect:linequiv:eq0},
\begin{equation}
\begin{aligned}
\big\langle\Lambda\cdot\chi\big\rangle_{ij} &= \textstyle\lambda_i\lambda_j\sum_{\alpha,\beta=1}^d\delta_{i\alpha}\delta_{j\beta}\langle\chi\rangle_{\alpha\beta} = \lambda_i\lambda_j\langle\chi\rangle_{ij}\\
\big\langle\Lambda\cdot\chi\big\rangle_{ijk} &= \lambda_i\lambda_j\lambda_k\textstyle\sum_{\alpha,\beta,\gamma=1}^d\delta_{\alpha\beta\gamma}^{ijk}\langle\chi\rangle_{\alpha\beta\gamma} = \lambda_i\lambda_j\lambda_k\langle\chi\rangle_{ijk} 
\end{aligned}  
\end{equation}for $i,j,k\in[d]$. In particular, $N_{\Lambda\cdot\chi} = \Lambda\cdot N_{\chi}$ as well as 
\begin{equation}
\big[\Lambda\cdot\chi\big]_0 = \Lambda[\chi]_0\Lambda \quad\text{and}\quad \big[\Lambda\cdot\chi\big]_\nu = \lambda_\nu\Lambda[\chi]_\nu\Lambda \quad\text{for each $\nu\in[d]$,}
\end{equation}
which implies \eqref{lem:scaleinvariance:eq2} by way of the definitions \eqref{contrast:eq0.1}.   
\end{proof}

\subsection{Compactness of $\Xi_1$ and \texorpdfstring{$\tilde{\Xi}_1$}{Xi_1}} 
\begin{lemma}\label{lem:xicompact}
For $\Xi^0_1$ as in \eqref{lica:unitrows} and $\gamma\geq 1$, the set 
\begin{equation}
\Xi_1^\gamma\coloneqq\Xi^0_1\cap\{\theta\in\Theta\mid\kappa_2(\theta)\leq\gamma\} \ \text{ is compact in \ } \R^{d\times d}.
\end{equation}  
\end{lemma}
\begin{proof}
The set $\Xi_1^\gamma\subset\R^{d\times d}$ is bounded (by definition of $\Xi^0_1$), so it remains to show that it is also closed. For this, let $(\theta_n)\subset\Xi_1^\gamma$ and $\theta\in\R^{d\times d}$ be such that $\theta_n\rightarrow\theta$ in $\R^{d\times d}$. Then necessarily also $\theta\in\Xi^0_1$ and $\lim_{n\rightarrow\infty}\|\theta_n\|_2 = \|\theta\|_2\neq 0$. The latter implies that $\|\theta_n\|_2\geq c$ for some $c>0$ and almost all $n\in\N$ (say for all $n\geq n_0$, with some $n_0\in\N$). Suppose now that $\theta$ is singular, i.e.\ $\theta\notin\Theta$. Then and by the fact that $\|\cdot\|_2$ is submultiplicative (from which we obtain the last inequality below, see e.g.\ \citep[Problem 5.6.P47 (p.\ 369)]{horn2012matrix}), we must have 
\begin{equation}
\gamma\,\geq\,\kappa_2(\theta_n)=\|\theta_n\|_2\|\theta_n^{-1}\|_2 \,\geq\, c\|\theta_n^{-1}\|_2 \,\geq\, c/\|\theta_n - \theta\|_2 \quad \text{for all } n\geq n_0,
\end{equation} 
which is clearly a contradiction. This implies $\theta\in\Theta$ and hence, since $\kappa_2(\cdot)$ is continuous on $\Theta$, also $\kappa_2(\theta)=\lim_{n\rightarrow\infty}\kappa_2(\theta_n)\leq\gamma$ and thus $\theta\in\Xi_1^\gamma$, as desired.    
\end{proof}

\section{Technical Lemmas for Sections \ref{sect:blindinversionthm1} and \ref{sect:robustsolution}}\label{sect:bss-appendixB}
\subsection{The Matrix $C_\mu$ is Generically Invertible}
\begin{lemma}\label{lem:invertcmu}
Let $C_\mu\coloneqq\tfrac{1}{2}([\mu]_0 + [\mu]_0^\intercal)$ for a signal $\mu\in\fD$ with \nth{0}-coredinate $[\mu]_0$. Then $C_\mu\notin\GL$ only if $\mu_{0,1}(H)=1$ for some hyperplane $H$ in $\R^d$.   
\end{lemma}
\begin{proof}
Suppose that $C_\zeta\notin\GL$, which implies that there is $v=(v_i)\in\R^d\setminus\{0\}$ with $C_\zeta v = 0$. By recalling the definition \eqref{def:coordinates:eq2} of $[\mu]_0$ and from a quick computation of the $C_\zeta$-defining sums of integrals \eqref{def:coordinates:eq1}, we see that $C_\zeta = \tfrac{1}{2}\big(\mathbb{E}[\zeta_{0,1}^i\cdot\zeta_{0,1}^j]\big)_{ij}$. Therefore $0 = v^\intercal C_\zeta v = \sum_{i,j=1}^d v_i C_\zeta^{ij} v_j = \mathbb{E}[Z_v^2]$ for $Z_v\coloneqq \sum_{i=1}^d\zeta_{0,1}^i v_i = \zeta_{0,1}^\intercal\cdot v$ ($\coloneqq\varphi_\ast(\zeta)$ for $\varphi\coloneqq\langle\pi_{0,1}(\cdot), v\rangle_2$, see \eqref{sect:notation:pushforward} and \eqref{sect:notation:pushforward:eq2}). This implies $Z_v = \delta_0$ and hence $\zeta_{0,1}(H)=1$ for the hyperplane $H\coloneqq\{u\in\R^d\mid u^\intercal v = 0\}$. 
\end{proof}

\subsection{Hyperplanes, Inverses, and Contracted Tensors}
The following auxiliary observations facilitate the analysis that underlies the proof of Theorem \ref{thm:robust_ica}.  

\begin{appendixlemma}\label{lem:hyperplanes}
Let $\vartheta\in\R^{d\times d}$ be invertible and $(e_i)_{i\in[d]}$ the standard basis of $\R^d$. Then the set 
\begin{equation}\label{lem:hyperplanes:eq1}
  \big\{u\in\mathbb{R}^d\ \big|\ \exists\, i,j\in[d], \,  i\neq j\, : \langle \vartheta^\intercal u, e_i \rangle = \kappa_{ij}\langle \vartheta^\intercal u, e_j \rangle \big\}
\end{equation}has Lebesgue measure zero for any $\kappa\equiv(\kappa_{ij})\in\mathbb{R}^{d\times d}$. 
\end{appendixlemma}
\begin{proof}
For any fixed $\kappa\equiv(\kappa_{ij})\in\mathbb{R}^{d\times d}$, consider the sets 
\begin{equation}
H_{ij} \coloneqq  \big\{u=(u_\nu)\in\mathbb{R}^d \ \big| \ u_i = \kappa_{ij}\cdot u_j \big\}
\end{equation}
for $i,j\in[d]$ with $i\neq j$. As $H_{ij} = \ker(\varphi_{ij})$ for the non-zero linear map $\varphi_{ij} : u\,\mapsto\,u_i - \kappa_{ij}\cdot u_j$, each set $H_{ij}$ has codimension one (by the rank-nullity theorem) and hence is a Lebesgue nullset in $\mathbb{R}^d$. The union $H\coloneqq \bigcup_{i,j\in[d], i\neq j}H_{ij}$ is thus also a Lebesgue nullset in $\mathbb{R}^d$, and hence so is its diffeomorphic image $(\vartheta^\intercal)^{-1}(H)$. The latter equals \eqref{lem:hyperplanes:eq1}, which proves the claim.
\end{proof}  

The following two lemmas apply for $\|\cdot\|$ any submultiplicative matrix norm.
\begin{appendixlemma}\label{lem:neumann_inv} Let $\Lambda\in\GL(\R)$ and $\Delta\in\R^{d\times d}$ with $\epsilon\,\coloneqq\,\|\Lambda^{-1}\Delta\| < 1$. Then $\Lambda + \Delta$ is invertible with 
$\big\|(\Lambda + \Delta)^{-1} - \Lambda^{-1}\big\|\,\leq\, \big\|\Lambda^{-1}\big\|\frac{\epsilon}{1-\epsilon}$. 
\end{appendixlemma}
\begin{proof}
Abbreviating $\rho\coloneqq-\Lambda^{-1}\Delta$, notice that the invertibility of $\Lambda + \Delta$ is due to $\Lambda + \Delta = \Lambda(\mathrm{I} - \rho)$ and the convergence of the Neumann series (e.g.\ \citep[Ex.\ 5.6.P26]{horn2012matrix}). In fact  
\begin{equation}
\begin{aligned}
(\Lambda + \Delta)^{-1} = (\mathrm{I} - \rho)^{-1}\Lambda^{-1} = \big(\mathrm{I} + \textstyle\sum_{n=1}^\infty\rho^n\big)\Lambda^{-1} 
\end{aligned}
\end{equation}so that, by $\epsilon<1$ and for $\eta\coloneqq\|\Lambda^{-1}\|$, we have $
\big\|(\Lambda + \Delta)^{-1} - \Lambda^{-1}\big\|\leq\eta\sum_{n=1}^\infty\|\rho\|^n\leq\frac{\eta\epsilon}{1-\epsilon}$.    
\end{proof}    

\begin{appendixlemma}\label{lem:contracted_struct}
Let $C\equiv(C_0, \ldots, C_d)\in\mathcal{V}$ and denote $\Lambda_\nu\coloneqq\mathrm{ddiag}[C_\nu]$ and $\underline{C}\coloneqq(0, C_1, \ldots, C_d)$ as well as $\underline{\Lambda}\coloneqq (0, \Lambda_1,\ldots, \Lambda_d)$. If $\Lambda_0\in\GL(\R)$ and $\delta_1\,\coloneqq\,\big\|\Lambda_0^{-1}\!\cdot(C_0 - \Lambda_0)\big\| < 1$, then $C_0$ is invertible and the contraction 
\begin{equation}
C^c_{\odot}\,\coloneqq\, c_0 C_0^{-1}\cdot\textstyle\sum_{\nu=1}^d c_vC_\nu, 
\end{equation}
with $c\equiv(c_0,\underline{c})\in\R^{1+d}$ arbitrary, is of the form   
\begin{equation}
C^c_\odot \, = \, c_0\Lambda_0^{-1}\Lambda_c + \, \Delta_c \!\quad\text{with}\quad \|\Delta_c\|\leq \beta_c(\delta_1,\delta_2)
\end{equation}
for the diagonal matrix $\Lambda_c\coloneqq\sum_{\nu=1}^d c_\nu \Lambda_\nu$, the number $\delta_2\coloneqq\|\underline{C}-\underline{\Lambda}\|_{\mathcal{V}}$, and the bounding function $\beta_c(a,b)\,\coloneqq\, |c_0|\eta\left[\|\Lambda_c\| + 1 + \gamma^{\underline{c}}_{a,b}\right]\gamma^{\underline{c}}_{a,b}$ with $\eta\coloneqq\|\Lambda_0^{-1}\|$ and $\gamma^{\underline{c}}_{a,b}\coloneqq\max\big\{\tfrac{a}{1-a}, b|\underline{c}|\big\}$.     
\end{appendixlemma} 
\begin{proof}
Fix any $c\equiv(c_0,\underline{c})\in\R^{1+d}$. Then by definition, $
C^c_\odot=c_0\tilde{D}_0\cdot D_1$ for the factors $\tilde{D}_0\coloneqq C_0^{-1}$ and $D_1\coloneqq \sum_{\nu=1}^d c_\nu C_\nu$, where $D_1 = \Lambda_c + \Delta_1$ with $\|\Delta_1\| \leq |\underline{c}|\delta_2$ by construction. Since $\delta_1<1$, Lemma \ref{lem:neumann_inv} implies that $\tilde{D}_0 = \Lambda^{-1}_0 + \tilde{\Delta}_0$ with $\big\|\tilde{\Delta}_0\big\|\,\leq\,\eta\tilde{\delta}_1$ for $\tilde{\delta}_1\coloneqq\tfrac{\delta_1}{1-\delta_1}$. From this we obtain that the `off-diagonal' $\Delta_c \coloneqq C^c_\odot - c_0\Lambda^{-1}_0\Lambda_c =c_0\big( \tilde{\Delta}_0\Lambda_c + \Lambda^{-1}_0\Delta_1 + \tilde{\Delta}_0\Delta_1\big)$ can be bounded by  
\begin{equation}
\begin{aligned}
|c_0|^{-1}\|\Delta_c\| \ &\leq \ \|\tilde{\Delta}_0\|\|\Lambda_c\| + \eta\|\Delta_1\| + \|\tilde{\Delta}_0\|\|\Delta_1\|\\
&\leq \ \eta\tilde{\delta}_1\|\Lambda_c\| + \eta\tilde{\delta}_2 + \eta\tilde{\delta}_1\tilde{\delta}_2 \ \leq \ \eta\Big[\|\Lambda_c\| + 1 + \bar{\delta}\Big]\cdot\bar{\delta}
\end{aligned} 
\end{equation}   
for $\tilde{\delta}_2\coloneqq |\underline{c}|\delta_2$ and $\bar{\delta}\coloneqq \max\{\tilde{\delta}_1, \tilde{\delta}_2\}$, as desired.
\end{proof}

\subsection{Eigengaps} Given a vector $w\equiv(w_1, \ldots, w_d)\in\C^d$, we call the \emph{gap of $w$} the number
\begin{equation}
\gamma[w] \,\coloneqq\, \min_{i,j\in[d]\,:\,i\neq j}|w_i - w_j| \qquad (d\geq 2)     
\end{equation}and refer to $\ell(w)\coloneqq \|w - (w_{i_0}, \ldots, w_{i_0})\|_2$, with $i_0=\min\!\big[\!\operatorname{arg\,min}_{i\in[d]}|w_i|]$, as the \emph{spreadlength of $w$}.      

\begin{appendixlemma}\label{lem:spread}
Over all vectors in $\R^d$ of fixed spreadlength $\ell>0$, the gap is maximised by any spreadlength-$\ell$ vector $w^\ast\equiv(w^\ast_i)\in\R^d$ such that $|w_{i+1}^\ast - w_i^\ast| = |w_{j+1}^\ast - w_j^\ast|$ for $i,j\in[d-1]$.\! In particular: for $k_d\coloneqq\sqrt{\tfrac{1}{6}(d-1)d(2d-1)}$,  
\begin{equation}\label{lem:spread:eq1}
\max_{w\in\R^d\,:\,\ell(w)=\ell}\gamma[w] \ = \ \ell/k_d. 
\end{equation} 
\end{appendixlemma}
\begin{proof} 
Let $w=(w_1,\ldots,w_d)\in\R^d$ be any fixed vector of spreadlength $\ell(w)\eqqcolon\ell>0$. Since the gap $\gamma$ is invariant under permutations and constant offsets (i.e., $\gamma[w]=\gamma[\tau(w)+(c, \ldots, c)]$ for any $c\in\R$ and $\tau\in S_d$), we may assume that $w_1=0$ and $w_\nu\leq w_{\nu+1}$ for all $\nu\in[d-1]$ wlog. Suppose now that $|w_{i+1} - w_i| = |w_{j+1} - w_j|\eqqcolon\delta$ for each $i,j\in[d-1]$. Then $w_2 = |w_2 - w_1| = \delta$ and likewise $w_\nu = \sum_{j=1}^{\nu-1}(w_{j+1} - w_j) = (\nu-1)\delta$ for each $\nu\in[d]$, implying that 
\begin{equation}
\ell^2 = \ell(w)^2 = \|w\|_2^2 = \delta^2\textstyle{\sum}_{j=0}^{d-1}j^2 = \delta^2 k_d^2 
\end{equation} 
and thus $\delta = \frac{\ell}{k_d}$. Also, since by the above assumption $|w_i - w_j| \geq |w_{j+1} - w_j| = \delta$ for any $i,j\in[d]$ with $i\neq j$, we find that $\gamma[w] = \delta$. Assuming now that there is $\tilde{w}\equiv(\tilde{w}_i)\in\R^d$ (with $0=\tilde{w}_1\leq \tilde{w}_2\leq\ldots\leq \tilde{w}_d$ (wlog) as above) of spreadlength $\ell(\tilde{w})=\ell$ such that $\gamma[\tilde{w}] > \delta\equiv\ell/k_d$, we deduce that
\begin{equation}
\begin{aligned}
\ell^2 \, = \, \ell(\tilde{w})^2 = \sum_{i=1}^d\tilde{w}_i^2 = \sum_{i=1}^d\!\left[\sum_{j=1}^{i-1}\big|\tilde{w}_{j+1}-\tilde{w}_j\big|\right]^2 \! > \sum_{i=1}^d\big[(i-1)\delta\big]^2 = \delta^2k_d^2 \, = \, \ell^2.  
\end{aligned}   
\end{equation}
As this is a contradiction,  \eqref{lem:spread:eq1} holds as claimed.                        
\end{proof}

In the next two subsections, all eigenvalues are counted with their respective multiplicity, and $\kappa(B)\coloneqq\|B\|\|B^{-1}\|$ denotes the condition number (for $\|\cdot\|$ a given norm) of a matrix $B$. 
\subsection{Perturbation Bounds on Non-Degenerate Eigenspaces}
\noindent
As an important ingredient to our robustness analysis, we cite the following bounds for simple eigenvalues and -vectors of perturbed (not necessarily symmetric) matrices. 
\begin{appendixlemma}\label{lem:perturbationresults}
Let $C, E \in\C^{d\times d}$ and $\tilde{C}\coloneqq C + E$. Suppose that the eigenvalues $\lambda_1, \ldots, \lambda_d$ of $C$ are pairwise distinct, let $v_1, \ldots, v_d$ be their associated unit-norm eigenvectors \emph{(}i.e., $\|v_i\|_2=1$ for all $i\in[d]$\emph{)} and $(u_1, \lambda_1), \ldots, (u_d, \lambda_d)$ be the associated `relatively-normed' left-eigenpairs of $C$ \emph{(}i.e., $u_i^\intercal\cdot C = \lambda_i u_i^\intercal$ with $u_i^\intercal v_i = 1$, $i\in[d]$\emph{)}.\footnote{\ In the context of Lemma \ref{lem:perturbationresults}, $v^\intercal\coloneqq\bar{v}^\intercal$ denotes the complex conjugate of a vector $v\in\C^d$.} Provided 
\begin{equation}\label{lem:perturbationresults:eq2}
\begin{gathered}
\|E\|_{\mathrm{F}} \ < \ \tfrac{1}{2}\min_{i\in[d]}\frac{s_i}{\kappa_i} \qquad\text{with}\\
s_i\coloneqq \sigma_{d-1}\!\left(\left[\mathrm{I}_d - v_i v_i^\intercal\right]\left(\lambda_i\mathrm{I}_d - C\right)\right) \ \text{ and }\ \kappa_i\coloneqq \|u_i\|_2 + \sqrt{\|u_i\|_2^2 - 1}, 
\end{gathered}
\end{equation}where $\sigma_{d-1}(\cdot)$ denotes the \nth{2}-smallest singular value of its argument, it holds that  
\begin{equation}\label{lem:perturbationresults:eq3}
\begin{gathered} 
\forall\, i\in[d] \ : \ \exists\, \text{$\mathrm{EV}$ \ $\tilde{v}_{\sigma(i)}$ of \ $\tilde{C}$  \ \ such that} \\
u_i^\intercal\tilde{v}_{\sigma(i)} = 1 \ \text{ and }\  \|\tilde{v}_{\sigma(i)} - v_i\|_2 \ \leq \ \frac{2\kappa_i}{s_i}\|E\|_{\mathrm{F}} 
\end{gathered}
\end{equation}for $\sigma:[d]\rightarrow[d]$ some assignment, where each associated eigenpair $(\tilde{\lambda}_j, \tilde{v}_j)$ of $\tilde{C}$ satisfies  
\begin{equation}\label{lem:perturbationresults:eq4}
\begin{aligned}
|\tilde{\lambda}_{\sigma(i)} - \lambda_i| \ \leq \ |u_i^\intercal Ev_i + \ell^i_E| \quad\text{ with } \quad |\ell^i_E| \ \leq \  \frac{2\|u_i\|_2\kappa_i}{s_i}\|E\|_{\mathrm{F}}^2. 
\end{aligned} 
\end{equation}    
\end{appendixlemma} 
\begin{proof}
This is immediate from -- and in fact basically a quote of -- \citep[Theorem 4.1 (ii)]{karowkressner}.
\end{proof} 

\subsection{Continuity and Bounds of Some Matrix Functions}\label{sect:lem:choiceofR}
\begin{lemma}\label{lem:choiceofR}
The map $\lambda : \mathrm{Sym}_d\rightarrow\R^d$, $\mathfrak{a}\mapsto(\lambda_1(\mathfrak{a}), \cdots, \lambda_d(\mathfrak{a}))$ which assigns a symmetric $d\times d$-matrix to the vector of its eigenvalues (listed in descending order, counting multiplicities) and the `inverse square root' function $\mathcal{R} : \mathrm{Sym}_d^+ \rightarrow \mathrm{Sym}_d^+$, $\mathfrak{a} \mapsto \mathfrak{a}^{-1/2}$, are both continuous. Furthermore, for $\|\cdot\|$ any unitarily-invariant and submultiplicative matrix norm,
\begin{equation}\label{lem:choiceofR:eq1}
\Big\|\mathcal{R}(\mathfrak{a}_1)^{-1} - \mathcal{R}(\mathfrak{a}_2)^{-1}\Big\| \leq \left[\lambda_d^{1/2}(\mathfrak{a}_1) + \lambda_d^{1/2}\big(\mathfrak{a}_2)\right]^{-1}\!\!\!\!\!\cdot\|\mathfrak{a}_1 - \mathfrak{a}_2\|
\end{equation}  
for each $\mathfrak{a}_1, \mathfrak{a}_2\in\mathrm{Sym}_d^+$. In particular, for any fixed $A\in\GL$ and $\mathcal{S}(\mathfrak{a})\coloneqq\tfrac{1}{2}(\mathfrak{a} + \mathfrak{a}^\intercal)$, both of
\begin{align}\label{lem:choiceofR:eq1.1}
\text{the maps}\quad \mathbb{R}^{d\times d}\,\ni\, \mathfrak{a} \ &\longmapsto \ \lambda\big(A\cdot\mathcal{S}(\mathfrak{a})\cdot A^\intercal\big)\,\in\mathbb{R}^d \qquad \text{and}\\\label{lem:choiceofR:eq1.2}
\mathrm{Pos}_d\,\ni\, \mathfrak{a} \ &\longmapsto \ \mathcal{R}\big(A\cdot\mathcal{S}(\mathfrak{a})\cdot A^\intercal\big)^{-1}\,\in\mathbb{R}^{d\times d} \quad\text{are continuous.}
\end{align} 
Moreover, for every $\mathfrak{a}\equiv(\tilde{a}_{ij})\in\R^{d\times d}$ with $\min_{i\in[d]}|\tilde{a}_{ii}|>0$ and $\mathcal{C}_{\mathfrak{a}}\coloneqq A\mathcal{S}(\mathfrak{a}) A^\intercal\in\mathrm{Sym}_d^+$ and $N_{\mathfrak{a}}\coloneqq\mathrm{ddiag}(|\tilde{a}_{11}|,\cdots, |\tilde{a}_{dd}|)^{1/2}$, and $A_{\mathcal{R}(\mathcal{C}_{\mathfrak{a}})}$ and $\bar{A}_{\mathcal{R}(\mathcal{C}_{\mathfrak{a}})}$ defined analogous to above, we have 
\begin{align}\label{lem:choiceofR:eq2}
\big\|A^{-1}_{\mathcal{R}(C_{\mathfrak{a}})}\big\|_{\mathrm{F}} \,&\leq\, \|A^{-1}\|_{\mathrm{F}}\max\nolimits_{i\in[d]}\lambda_i^{1/2}(C_{\mathfrak{a}})\,\eqqcolon\,\beta_1(\mathfrak{a};A),\\\label{lem:choiceofR:eq3}
\kappa_2\big(A_{\mathcal{R}(C_{\mathfrak{a}})}\big)\,&\leq\, \frac{2\det(\mathrm{ddiag}[\lambda(C_{\mathfrak{a}})]^{1/2})\|A\|_{\mathrm{F}}^d}{\det(A)\big(\sqrt{d}\min\nolimits_{i\in[d]}\lambda_i^{1/2}(C_{\mathfrak{a}})\big)^d}\,\eqqcolon\,\beta_2(\mathfrak{a};A). 
\end{align} 
\end{lemma} 
\begin{proof} 
These are simple consequences of classical facts and inequalities from matrix analysis. 

Indeed: The continuity of the map $\lambda$ (which is well-defined as the elements of its (closed) domain $\mathrm{Sym}_d$ are all diagonalisable with real eigenvalues) is simply a [specialised] restatement of the well-known fact that the spectrum of a matrix depends continuously on its coefficients, see e.g.\ \citep[Appendix D]{horn2012matrix}. The (continuity-implying; recall that matrix inversion is (locally Lipschitz) continuous) inequality \eqref{lem:choiceofR:eq1} is precisely the Ando-Hemmen inequality \citep[Theorem 6.2]{higham2008}. As to the remaining assertions: Note that \eqref{lem:choiceofR:eq2} is immediate from the definition of $A_R$, the representation \eqref{rem:choiceofR:eq2} and the standard inequality: $\|\mathfrak{a}\mathfrak{b}\|_{\mathrm{F}}\leq\|\mathfrak{a}\|_{\mathrm{F}}\|\mathfrak{b}\|_2$ for each $\mathfrak{a}, \mathfrak{b}\in\R^{d\times d}$. The latter combined with the bound in \cite{guggenheimer1995} gives \eqref{lem:choiceofR:eq3}.   
\end{proof}       

\subsection{Continuity of Extremal Functions}

\begin{appendixlemma}\label{lem:minfuncont}
Let $\|\cdot\|$ be any norm on $\R^n$, and $\varphi : \R^n\rightarrow\R$ be a continuous function. Let further $\bar{B}_r\coloneqq\{x\in\R^n \mid \|x\|\leq r\}$ be the zero-centered closed $\|\cdot\|$-ball of radius $r\geq 0$. Then the functions $\phi, \psi : [0,\infty)\rightarrow\R$ given by $\phi(r)\coloneqq\min_{y\in\bar{B}_r}\varphi(y)$ and $\psi(r)\coloneqq\max_{y\in\bar{B}_r}\varphi(y)$, respectively, are both continuous. 
\end{appendixlemma}
\begin{proof}
Since $\psi(r) = -\min_{y\in\bar{B}_r}(-\varphi(y))$ for each $r\geq 0$, it suffices to show the continuity of $\phi$. 

Because the open rays are a subbase of the Euclidean (i.e., order) topology on $\R$, the continuity of $\phi$ follows if, for each $a\in\mathbb{R}$, the sets $\ell_a:=\{r\mid \phi(r)<a\}$ and $\rho_a:=\{r\mid \phi(r)>a\}$ are both open in $[0,\infty)$. To see this for $\ell_a$, fix any $s\in\ell_a$. Then $m_s := \min_{y\in \bar{B}_s}\varphi(y) < a$ and, since $\varphi$ is continuous, $\varphi(x_s) = m_s$ for some $x_s\in \bar{B}_s$. Denoting by $B_u(z)$ the open $z$-centered $\|\cdot\|$-ball of radius $u$, the continuity of $\varphi$ further implies that there is $\delta_s>0$ such that $\varphi(y) < a$ for each $y \in B_{\delta_s}(x_s),$ and hence $\phi(s') < a$ for all $s'\in (\|x_s\| - \delta_s/2, s] \cap [0,\infty)$. Since trivially $\phi(s') < a$ for all $s'\geq s$, we thus found $\left.\phi\right|_{J_s} < a$ for the open (in $[0,\infty)$) set $J_s:=(\|x_s\| - \delta_s/2, \infty)\cap [0,\infty) \ni s$. Hence $J_s\subset \ell_a$, which shows that $\ell_a$ is open.

The openness of $\rho_a$ follows similarly: Fix any $s\in\rho_a$. Then $\phi(s) > a$ and, hence, clearly also $\phi(s') > a$ for each $s'\leq s$. Moreover: $\min_{y\in\partial B_s} \varphi(y) \geq \phi(s) > a$, from which the continuity of $\varphi$ implies that for each $z\in\partial B_s$ there is $\delta_z>0$ such that: $\varphi(y) > a$ for each $y\in B_{\delta_z}(z)$. Hence $\restr{\varphi}{G_s}>a$ for the open superset $G_s\coloneqq \bigcup_{z\in\partial B_s}B_{\delta_z}(z) \cup B_s$ of $\bar{B}_s$. Now since $\bar{B}_s$ is compact and $G_s^{\mathrm{c}}$ (i.e.\ the complement of $G_s$) is closed, the distance $\hat{\delta}_s\coloneqq d(\bar{B}_s, G_s^{\mathrm{c}})= \inf_{u\in \bar{B}_s, v\in G_s^{\mathrm{c}}}\|u-v\|$ is positive, which implies that $\bar{B}_{s + \delta}\subset G_s$ for any $0\leq \delta < \hat{\delta}_s$. Hence $\restr{\phi}{J_s'}> a$ for the open (in $[0,\infty)$) superset $J_s'\coloneqq(-\infty, s+\hat{\delta}_s)\cap[0,\infty)$ of $s$, which shows that $\rho_a$ is open, as desired.        
\end{proof}   

With the above, we can finally prove the main results of Section \ref{chap:robustICA:sect:auxiliaries}.

\section{Proof of Theorems \ref{thm:robust_ica} and \ref{thm:robustness}}\label{sect:mainproofs} 
\subsection{Proof of Theorem \ref{thm:robust_ica}}\label{chap:robustICA:subsect:signature_identifiability:proof}
\noindent
For $A\equiv(a_1|\cdots|a_d)\in\GL$ and $\zeta\equiv(\zeta^1, \ldots, \zeta^d)\in\fD$, suppose
\begin{equation}\label{thm:robust_ica:aux1}
\chi \, = \, A\cdot \zeta \quad\text{ and }\quad\varepsilon\equiv\delta_{\independent}(\zeta)<\varepsilon_0
\end{equation} 
for $\varepsilon_0$ as in \eqref{chap:robustICA:sect:signature_identifiability:constants:1}, as required. Upon rescaling $\zeta$ by $N_\zeta^{-1}$ and noting that for $\tilde{\zeta}\coloneqq N_\zeta^{-1}\cdot\zeta$ we have $N_{\tilde{\zeta}}=\mathrm{I}$ and $\delta_{\independent}(\tilde{\zeta})=\delta_{\independent}(\zeta)$, whence \eqref{thm:robust_ica:aux1} still holds for $(\zeta,A)$ and $(\tilde{\zeta}, \tilde{A}\coloneqq AN_\zeta)$ replaced, we obtain for each $\nu\in[d]$, see Lemma \ref{lem:scaleinvariance} and \eqref{contrast:eq0.1}, that
\begin{equation}\label{thm:robust_ica:aux2.1}
\begin{aligned}
\mathfrak{s}_0 \, &\coloneqq \, \mathfrak{x}_0(A^{-1}) = \mathfrak{x}_0\big(\tilde{A}^{-1}\big) = N_{\tilde{\zeta}}^{-1}\big[\tilde{\zeta}\big]_0 N_{\tilde{\zeta}}^{-1} = \big[\tilde{\zeta}\big]_0 \qquad\text{and}\\
\mathfrak{s}_\nu \, &\coloneqq \, \mathfrak{x}_\nu(A^{-1}) = \mathfrak{x}_\nu\big(\tilde{A}^{-1}\big) = \big[\tilde{\zeta}\big]_\nu/\sqrt{\langle\tilde{\zeta}\rangle_{\nu\nu}} \, = \, \big[\tilde{\zeta}\big]_\nu. 
\end{aligned}
\end{equation}    
Moreover, for the whitening matrix $R\equiv R_{(\chi)}$ from \eqref{rem:choiceofR:eq1} and the preprocessed observable $\bar{\chi}\coloneqq R\cdot \chi$, we find that for each $\nu\in[d]_0$ and every $\theta\in\Xi_1$, 
\begin{equation}\label{thm:robust_ica:aux2.3}
N_{\theta\cdot\bar{\chi}} \, = \, \mathrm{I} \quad \text{ and hence } \quad \mathfrak{x}_\nu(\theta R) = \big[\theta R\cdot\chi\big]_\nu.
\end{equation}
Indeed: A direct computation of the iterated integrals \eqref{def:coordinates:eq1} yields that for each $\mu\in\fD$, the matrix $C_\mu\equiv\tfrac{1}{2}([\mu]_0 + [\mu]_0^\intercal)$ is identical to the classical moment matrix $\big(\tfrac{1}{2}\E[\mu^i_{0,1}\mu^j_{0,1}]\big)$. Hence and since $C_{\bar{\chi}}=\mathrm{I}$ by choice of $R_{(\chi)}$, we have for any $\theta\equiv(\theta_1|\cdots|\theta_d)^\intercal\in\Xi_1$ that
\begin{equation}
\begin{aligned}
\langle\theta\cdot\bar{\chi}\rangle_{ii} \,&=\, \tfrac{1}{2}\mathbb{E}\!\left[\!\big(\textstyle\sum_{j=1}^d\theta_{ij}\bar{\chi}^j_{0,1}\big)^{\!2}\right] \,=\, \tfrac{1}{2}\sum_{j=1}^d\theta_{ij}^2\mathbb{E}\big[(\bar{\chi}_{0,1}^j)^2\big] + \sum_{j<k}\theta_{ij}\theta_{ik}\mathbb{E}[\bar{\chi}_{0,1}^j \bar{\chi}_{0,1}^k]\\[-1em]
&=\, \sum_{j=1}^d\theta_{ij}^2 \, = \, \|\theta_i\|^2_2 \, = \, 1 \qquad \text{for each } i\in[d],
\end{aligned}   
\end{equation} 
whence $N_{\theta\cdot\bar{\chi}}=\mathrm{I}$ and hence $\mathfrak{x}_\nu(\theta R)_\nu = [\theta\cdot\bar{\chi}]_\nu$ for all $\nu\in[d]_0$, as claimed. From \eqref{thm:robust_ica:aux1} we find  
\begin{equation}\label{thm:robust_ica:aux3}
\mathfrak{s}_\nu = \Lambda_\nu^\zeta + \Delta_\nu^\zeta \quad\text{with}\quad \Lambda_\nu^\zeta \, \text{ diagonal } \ \text{and } \ \big\|\Delta_\nu^\zeta\big\|\leq\varepsilon 
\end{equation}
for each $\nu\in[d]_0$. Indeed: By definition \eqref{sect:coredinates:eq4} and Lemma \ref{lemma1} resp.\ Remark \ref{rem:icdefect-nonstationary}, we have    
\begin{equation}\label{thm:robust_ica:aux4}
\begin{aligned}
\varepsilon_0^2 \ &> \ \varepsilon^2=\delta_{\independent}(\zeta)^2 = \delta_{\independent}(\tilde{\zeta})^2 \ = \ \big\|\mathfrak{s}_0 - \Lambda_0^\zeta\big\|^2 + \sum_{\nu=1}^d\big\|\mathfrak{s}_\nu - \Lambda^\zeta_\nu\big\|^2 
\end{aligned} 
\end{equation} 
for the diagonal matrices  
\begin{equation}\label{thm:robust_ica:aux5}
\begin{aligned}
\Lambda_0^\zeta &\,\coloneqq\,\mathrm{ddiag}\big([\tilde{\zeta}]_0\big) \, = \, N_{\tilde{\zeta}}^2 \ = \, \mathrm{I} \quad\text{ and }\quad
\Lambda_\nu^\zeta &\,\coloneqq\, \big(\langle\tilde{\zeta}\rangle_{\nu\nu\nu}\cdot\delta_{ij\nu}\big)_{\!ij}\,;   
\end{aligned}
\end{equation}  
setting $\Delta_\nu^\zeta\coloneqq\mathfrak{s}_\nu-\Lambda^\zeta_\nu$ then gives \eqref{thm:robust_ica:aux3} as asserted. Now from \eqref{thm:robust_ica:aux1}, the relations $\mathfrak{x}_\nu(A^{-1}) = \mathfrak{s}_\nu$ from \eqref{thm:robust_ica:aux2.1} and the definition \eqref{contrast:eq1} of the contrast $\phi_\chi$ and its scale invariance \eqref{phi:monomialinvar}, we obtain  
\begin{equation}\label{thm:robust_ica:aux6}
\begin{aligned}
\min_{\theta\,\in\,\Xi_1}\phi_\chi(\theta) \, &\leq \, \phi_\chi(\bar{B}_R) = \phi_\chi(B_R) \,=\, \sum_{\nu=0}^d\|\mathfrak{s}_\nu\|_\times^2 \,\leq\, \sum_{\nu=0}^d \big\|\Delta_\nu^\zeta\big\|^2 \, \leq \, \varepsilon^2\,, 
\end{aligned}   
\end{equation} 
where $B_R\coloneqq (RA)^{-1}\eqqcolon(b_1|\cdots|b_d)^\intercal$ and $\bar{B}_R\coloneqq\mathrm{ddiag}\big(|b_1|, \ldots, |b_d|\big)^{-1}\!\cdot B_R\in \Xi_1$ scales the rows of $B_R$ to unit length. Let now $\theta_\star\!\in\,\Xi_1$be arbitrary with
\begin{equation}\label{thm:robust_ica:aux7}
\phi_\chi(\theta_\star)\,\leq\,\varepsilon^2\,, \quad\text{ and set }\quad \tilde{\theta}_\star\coloneqq \theta_\star R
\end{equation}
(note that $\theta_\star$ exists by \eqref{thm:robust_ica:aux6} and the compactness of $\Xi_1$). By construction, cf.\ \eqref{thm:robust_ica:aux2.3}, we have 
\begin{equation}\label{thm:robust_ica:aux8}
\mathfrak{x}_\nu(\tilde{\theta}_\star) = \big[\tilde{\theta}_\star\cdot\chi\big]_\nu \quad\text{for all} \ \nu\in[d]_0,
\end{equation} 
and from \eqref{thm:robust_ica:aux7} and the definition of $\phi_\chi$ we obtain for each $\nu\in[d]_0$ that  
\begin{equation}
\|\Delta_\nu^\star\|\,\leq\,\varepsilon \quad\text{ with }\quad \Delta_\nu^\star\,\coloneqq\,\mathfrak{x}_\nu(\tilde{\theta}_\star) - \Lambda_\nu^\star \ \text{ and } \ \Lambda_\nu^\star\coloneqq\mathrm{ddiag}(\mathfrak{x}_\nu(\tilde{\theta}_\star)).  
\end{equation}           
Hence by \eqref{thm:robust_ica:aux8} and the affine equivariance \eqref{sect:linequiv:eq1} of the signature coredinates, we have that   
\begin{equation}\label{thm:robust_ica:aux10}
\begin{aligned}
\tilde{\theta}_\star\cdot[\chi]_0\cdot\tilde{\theta}_\star^\intercal \, &=\, \Lambda^\star_0 + \Delta_0^\star \quad \text{ and}\\
\tilde{\theta}_\star\!\cdot\left[\textstyle\sum_{j=1}^d\tilde{\theta}_{\nu j}[\chi]_j\right]\cdot\tilde{\theta}_\star^\intercal \,&=\, \Lambda_\nu^\star + \Delta_\nu^\star \quad \text{ for all $\nu\in[d]$},\\
\text{as well as} \quad \big\|(\Lambda_0^\star)^{-1}\Delta_0^\star\big\| \, &= \, \|\Delta_0^\star\| \quad (\text{since} \ \Lambda_0^\star= N_{\tilde{\theta}_\star\cdot\chi}^2 = \mathrm{I}).
\end{aligned} 
\end{equation}  
Next we connect the relations \eqref{thm:robust_ica:aux10} to the source statistics \eqref{thm:robust_ica:aux2.1} via the coredinate identities \eqref{sect:linequiv:eq1} which hold by \eqref{thm:robust_ica:aux1}: Using \eqref{sect:linequiv:eq1} to substitute the matrices $[\chi]_\nu$ in \eqref{thm:robust_ica:aux10}, we obtain
\begin{equation}\label{thm:robust_ica:aux11}
\vartheta\cdot\diamondsuit_\nu^\zeta\cdot\vartheta^\intercal \, = \, \diamondsuit_\nu^\star \quad (\nu\in[d]_0) \quad\text{ with }\quad \vartheta\coloneqq\tilde{\theta}_\star A 
\end{equation} 
and where the $\diamondsuit_\nu^\zeta$, $\ds^\star_\nu$ are defined as the coredinate-based statistics
\begin{equation}
\diamondsuit^\zeta_0\coloneqq [\zeta]_0 \quad\text{and}\quad \diamondsuit_\nu^\zeta\coloneqq \textstyle\sum_{\gamma=1}^d \vartheta_{\nu\gamma}[\zeta]_\gamma \quad (\nu\in[d]), \quad\text{and} \quad \ds^\star_\nu\coloneqq \Lambda^\star_\nu + \Delta^\star_\nu \quad (\nu\in[d]_0). 
\end{equation} 
For an arbitrary vector $c\equiv(c_0, \underline{c})\in\R_{\times}\times\R^d$, denoting $\underline{c}\equiv(\tilde{c}_1, \ldots, \tilde{c}_d)$, the congruences \eqref{thm:robust_ica:aux11} can then be contracted to the `condensed' relations 
\begin{equation}\label{thm:robust_ica:aux13}
\vartheta\cdot c_0\ds^\zeta_0\cdot\vartheta^{\intercal}\,=\,c_0\ds^\star_0 \quad\text{ and }\quad \vartheta\cdot\ds^\zeta_{\underline{c}}\cdot\vartheta^\intercal\,=\,\ds^\star_{\underline{c}}
\end{equation}     
which involve the $\underline{c}$-weighted linear combinations
\begin{equation}
\ds^\zeta_{\underline{c}}\coloneqq\textstyle\sum_{\nu=1}^d \tilde{c}_\nu\ds^\zeta_\nu \quad\text{ and }\quad \ds^\star_{\underline{c}}\coloneqq\textstyle\sum_{\nu=1}^d \tilde{c}_\nu\ds^\star_\nu.
\end{equation}
As a consequence (upon `left-multiplying the inverse of the first identity in \eqref{thm:robust_ica:aux13} with the second identity in \eqref{thm:robust_ica:aux13}') we obtain the $\underline{c}$-weighted \emph{conjugate} relation  
\begin{equation}\label{thm:robust_ica:aux15}
\vartheta^{-\intercal}\cdot\hat{\ds}^\zeta_c\cdot\vartheta^\intercal \, = \, \hat{\ds}^\star_c \,, \qquad \forall\, c\in\R_{\times}\!\times\R^d,   
\end{equation} 
for the composite statistics $\hat{\ds}^\zeta_c\coloneqq\tilde{c}_0\big[\ds^\zeta_0\big]^{-1}\!\!\cdot\ds^\zeta_{\underline{c}}$ and 
$\hat{\ds}^\star_c\coloneqq\tilde{c}_0\big[\ds^\star_0\big]^{-1}\!\!\cdot\ds^\star_{\underline{c}}$, with $\tilde{c}_0\coloneqq c_0^{-1}$. By Lemma \ref{lem:contracted_struct} (cf.\ \eqref{thm:robust_ica:aux3}, \eqref{thm:robust_ica:aux10} and that $\varepsilon_0<1$), these composite statistics take the form
\begin{equation}
\hat{\ds}^\zeta_c \, = \, \hat{\Lambda}^\zeta_c \, + \, \hat{\Delta}^\zeta_c  \quad\text{ and }\quad \hat{\ds}^\star_c \, = \, \hat{\Lambda}^\star_c \, + \, \hat{\Delta}^\star_c 
\end{equation}
for the diagonal matrices, notationally distinguished via $\eta\in\{\zeta, \star\}$,
\begin{equation}\label{thm:robust_ica:aux20.2}
\begin{gathered}
\hat{\Lambda}^{\eta}_c\,\coloneqq\,\tilde{c}_0\big[\Lambda_0^\eta\big]^{-1}\textstyle\sum_{\nu=1}^d \tilde{c}_\nu\bar{\Lambda}_\nu^\eta \qquad \text{with} \qquad \bar{\Lambda}_\nu^\star\coloneqq\Lambda_\nu^\star \quad\text{ and }\quad \bar{\Lambda}_\nu^\zeta\coloneqq\textstyle\sum_{\gamma=1}^d \vartheta_{\nu\gamma}\Lambda^\zeta_\gamma,    
\end{gathered}
\end{equation}  
and off-diagonals $\hat{\Delta}_c^\eta\equiv\hat{\ds}^\eta_c - \hat{\Lambda}^\eta_c$ which are bounded by
\begin{equation}\label{thm:robust_ica:aux20.3}
\begin{aligned}
\big\|\hat{\Delta}_c^\zeta\big\| \, &\leq \, |\tilde{c}_0|\left[\big\|\Lambda_c^\zeta\big\|_2 + 1 + \gamma^{\underline{c}}_{\epsilon^\zeta_0,\epsilon^\zeta_1}\right]\cdot\gamma^{\underline{c}}_{\epsilon^\zeta_0,\epsilon^\zeta_1} \quad\text{and}\\  
\big\|\hat{\Delta}_c^\star\big\| \, &\leq \, |\tilde{c}_0|\left[\big\|\Lambda_c^\star\big\|_2 + 1 + \gamma^{\underline{c}}_{\epsilon_0^\star,\epsilon_1^\star}\right]\cdot\gamma^{\underline{c}}_{\epsilon^\star_0,\epsilon_1^\star}
\end{aligned}
\end{equation}  
for $\Lambda^\zeta_c\coloneqq\sum_{\nu=1}^dc_\nu\bar{\Lambda}_\nu^\zeta \, \big(\!\!= c_0\hat{\Lambda}_c^\zeta\big)$ and $\Lambda_c^\star\coloneqq c_0\hat{\Lambda}^\star_c$ and $\gamma^{\underline{c}}_{a,b}\coloneqq\max\{\tfrac{a}{1-a}, |\underline{c}|\cdot b\}$, and the scalars
\begin{equation}
\epsilon_0^\eta\coloneqq\big\|\Delta_0^\eta\big\| \quad\text{ and }\quad \epsilon_1^\eta\coloneqq \big\|(0,\Delta^\eta_1, \ldots, \Delta^\eta_d)\big\|_{\mathcal{V}}.
\end{equation}  
Note that by \eqref{thm:robust_ica:aux4} and \eqref{thm:robust_ica:aux7} we have
\begin{equation}\label{thm:robust_ica:aux20.5}
(\epsilon_0^\eta)^2 \, + \, (\epsilon_1^\eta)^2 \, \leq \, \varepsilon^2 \quad \text{ for both \ $\eta=\zeta$ \ and \ $\eta=\star$}. 
\end{equation}
Let us now distinguish two cases.\\[-0.5em]

\noindent
First suppose that $\varepsilon=0$, i.e.\ that $\delta_{\independent}(\zeta)=0$, meaning that the coredinates $([\zeta]_\nu)_{\nu\in[d]_0}$ of the source $\zeta\equiv(\zeta^1, \ldots, \zeta^d)$ coincide with the coredinates of the product signal $\zeta^1\otimes\cdots\otimes\zeta^d$. In this case (cf.\ \eqref{thm:robust_ica:aux20.2},\eqref{thm:robust_ica:aux20.3},\eqref{thm:robust_ica:aux20.5}) the matrices $\hat{\ds}^\star_c$ and $\hat{\ds}^\zeta_{c}$ are both diagonal, the latter reading 
\begin{equation}\label{thm:robust_ica:aux21}
\hat{\ds}^\zeta_c  = \hat{\Lambda}^\zeta_c = \tilde{c}_0\!\!\sum_{\gamma, \nu=1}^d \vartheta_{\nu\gamma}\tilde{c}_\nu\Lambda^\zeta_\gamma = \mathrm{ddiag}\big(\lambda^c_1, \lambda^c_2, \cdots, \lambda^c_d\big)   
\end{equation} 
for $c$-dependent eigenvalues $\lambda^c_1, \ldots, \lambda^c_d$ given by 
\begin{equation}\label{thm:robust_ica:aux22}
\lambda_i^c = \alpha_i(c)\cdot\big(\langle\zeta\rangle_{iii}/\langle\zeta\rangle_{ii}^{3/2}\big) \quad\text{for}\quad \alpha_i(c)\coloneqq\tilde{c}_0 \vartheta_i^\intercal\cdot\underline{c}     
\end{equation}   
with $\vartheta\eqqcolon(\vartheta_1|\cdots|\vartheta_d)$. Now since $\vartheta$ is invertible, Lemma \ref{lem:hyperplanes} implies that the set
\begin{equation}\label{thm:robust_ica:aux23}
\mathcal{N} \,\coloneqq\, \big\{c\in\R^{1+d} \ \big| \ \exists\, i,j\in[d], \, i\neq j \, : \, \lambda^c_i=\lambda^c_j \,\big\} 
\end{equation}    
has Lebesgue measure zero, as follows from applying said lemma for the coefficient matrix
\begin{equation}\label{thm:robust_ica:aux24}
(\kappa_{ij})\,\coloneqq\,\Big(\tfrac{\langle\zeta\rangle_{jjj}\langle\zeta\rangle_{ii}^{3/2}}{\langle\zeta\rangle_{jj}^{3/2}\langle\zeta\rangle_{iii}}\cdot(1-\delta_{ii_0})\Big)  
\end{equation}  
with $i_0\in[d]$ being the unique index at which $\zeta_{i_0i_0i_0}=0$. (By assumption on $\zeta$, the index $i_0$ is unique if it exists; otherwise, i.e.\ if none of the $\langle\zeta\rangle_{iii}$ vanish, set $(\kappa_{ij})\coloneqq\Big(\tfrac{\langle\zeta\rangle_{jjj}\langle\zeta\rangle_{ii}^{3/2}}{\langle\zeta\rangle_{jj}^{3/2}\langle\zeta\rangle_{iii}}\Big)$.) 

\noindent
Since the complement of \eqref{thm:robust_ica:aux23} is thus certainly non-empty, we can find a $c_\star\in\R_{\times}\!\times\R^d$ such that the eigenvalues $\lambda^{c_\star}_1, \ldots, \lambda^{c_\star}_d$ of $\hat{\ds}^\zeta_c$ in \eqref{thm:robust_ica:aux21} --- which, by matrix conjugation \eqref{thm:robust_ica:aux15}, coincide with those of $\hat{\ds}^\star_{c_\star}$ ---  are pairwise distinct and hence all simple. As $\hat{\ds}^\zeta_c$ and $\hat{\ds}^\star_{c_\star}$ are both diagonal matrices, this allows us to conclude $\vartheta^\intercal\in\M$, whence
\begin{equation}
\vartheta\,\in\,\M \quad\text{ and hence }\quad \tilde{\theta}_\star\,\in\,\M\cdot A^{-1}
\end{equation}
as claimed in the theorem.\\[-0.5em]

\noindent
Next we consider the case $\varepsilon>0$, i.e.\ the case where the components $\zeta^i$ of the source signal $\zeta=(\zeta^1, \ldots, \zeta^d)$ are allowed to be statistically dependent. The identity \eqref{thm:robust_ica:aux15} then reads
\begin{equation}\label{thm:robust_ica:aux26}
\tilde{\vartheta}^{-1}\cdot\left[\hat{\Lambda}^\zeta_c + E_c\right]\cdot\tilde{\vartheta} \, = \, \hat{\Lambda}^\star_c \quad \text{ for } \quad \tilde{\vartheta}\coloneqq\vartheta^\intercal,
\end{equation}    
with $\hat{\Lambda}^\zeta_c$ and $\hat{\Lambda}^\star_c$ as in \eqref{thm:robust_ica:aux20.2} and for the off-diagonal matrix  
\begin{equation}\label{thm:robust_ica:aux27}
E_c \,\coloneqq\, \hat{\Delta}^\zeta_c - \tilde{\vartheta}\hat{\Delta}^\star_c\tilde{\vartheta}^{-1}.   
\end{equation} 
Using \eqref{thm:robust_ica:aux20.3}, we observe that the `perturbation' $E_c$ can be controlled by 
\begin{equation}\label{thm:robust_ica:aux28}
\|E_c\| \, \leq \, \big\|\hat{\Delta}^\zeta_c\big\| \, + \, \kappa_2(\tilde{\vartheta})\big\|\hat{\Delta}^\star_c\big\| \ \leq \ |\tilde{c}_0|\left[\kappa_c\cdot\hat{\varphi}_{\underline{c}}(\varepsilon) + (1+\kappa_0)\hat{\varphi}_{\underline{c}}(\varepsilon)^2\right] 
\end{equation}  
for the bounding function 
\begin{equation}\label{thm:robust_ica:aux29}
\hat{\varphi}_{\underline{c}}(a)\,\coloneqq\, \max\big\{\tfrac{a}{1-a}, |\underline{c}|\cdot a\big\}
\end{equation}
and the $c$-dependent constant
\begin{align}\label{thm:robust_ica:aux30}
\kappa_c &\coloneqq \big\|\tilde{\vartheta}\big\|\varsigma|\underline{c}| + 1 + \kappa_2(\tilde{\vartheta})\big[\xi d |\underline{c}| + 1\big]\leq\|A_R\|\varsigma\sqrt{d}|\underline{c}| + 1 + \kappa_0\kappa_2(A_R)\big[\xi d|\underline{c}| + 1\big] 
\end{align}
for $\xi\coloneqq\big\|(0, [\bar{\chi}]_1, \ldots, [\bar{\chi}]_d)\big\|_{\mathcal{V}}$ and $\varsigma\coloneqq\sqrt{\sum_{i=1}^d\langle\zeta\rangle_{iii}^2/\langle\zeta\rangle_{ii}^3}$ and $A_R\coloneqq RA$. Indeed: Writing $\vartheta_i$ for the $i^{\mathrm{th}}$ column of $\vartheta$, Cauchy-Schwarz applied to \eqref{thm:robust_ica:aux20.2}, combined with \eqref{thm:robust_ica:aux5}, yields that
\begin{equation}
\big\|\Lambda^\zeta_c\big\| \, \leq \, |\underline{c}| \big\|\big(0, \big(\bar{\Lambda}^\zeta_\nu\big)_{\nu=1}^d\big)\big\|_{\mathcal{V}}\,=\, |\underline{c}|\sqrt{\textstyle\sum_{i=1}^d\|\vartheta_i\|_2^2\langle\tilde{\zeta}\rangle_{iii}^2}\,\leq\,|\underline{c}|\big\|\tilde{\vartheta}\big\|\varsigma,
\end{equation} 
while the inclusion $\theta_\star\equiv(\theta^\star_{ij})\in\Xi_1$ (unit rows) implies $\sum_{\alpha,\beta,\gamma=1}^d\big(\theta_{i\alpha}^\star\theta_{i\beta}^\star\theta_{\nu\gamma}^\star\big)^2=1$ and hence 
\begin{equation}
\begin{aligned}
\big\|\Lambda^\star_c\big\| \,&\leq\, |\underline{c}|\big\|\big(0, \big(\mathrm{ddiag}[\mathfrak{x}_\nu(\theta_\star R)]\big)_{\nu=1}^d\big)\big\|_{\mathcal{V}}\,=\, |\underline{c}|\sqrt{\textstyle\sum_{\nu, i=1}^d\!\big[\sum_{\alpha,\beta,\gamma=1}^d\theta^\star_{i\alpha}\theta^\star_{i\beta}\theta^\star_{\nu\gamma}\langle\bar{\chi}\rangle_{\alpha\beta\gamma}\big]^2}\\
&\leq\, |\underline{c}|\sqrt{d^2\textstyle\sum_{\alpha,\beta,\gamma=1}^d\langle\bar{\chi}\rangle_{\alpha\beta\gamma}^2} = d|\underline{c}| \xi. 
\end{aligned}
\end{equation}   
The (second) inequality in \eqref{thm:robust_ica:aux28} hence follows from \eqref{thm:robust_ica:aux20.3} and $\max_{\eta\in\{\zeta, \star\}}\gamma^{\underline{c}}_{\epsilon_0^\eta, \epsilon_1^\eta}\leq \hat{\varphi}_{\underline{c}}(\varepsilon)$, where the latter holds by \eqref{thm:robust_ica:aux20.5} and monotonicity. With the perturbation \eqref{thm:robust_ica:aux27} of the matrix $\hat{\Lambda}_c^\zeta$ (cf.\ \eqref{thm:robust_ica:aux21}) under control, we can now apply Lemma \ref{lem:perturbationresults} to find contraction parameters $c=(c_0, \underline{c})$ which via \eqref{thm:robust_ica:aux26} imply that $\tilde{\vartheta}$ is close to an element of $\M$. To this end, recall from \eqref{thm:robust_ica:aux21} and \eqref{thm:robust_ica:aux22} (and $N_{\tilde{\zeta}}=\mathrm{I}$) that the diagonal entries $\underline{\lambda}_c^\vartheta\coloneqq(\lambda^c_1, \ldots, \lambda^c_d)$ of $\hat{\Lambda}^\zeta_c$ read 
\begin{equation}
\underline{\lambda}_c^\vartheta \, = \, \tfrac{1}{c_0}H_{\vartheta}\cdot\underline{c} \quad\text{ for }\quad H_\vartheta\coloneqq \mathrm{ddiag}\big(\langle\tilde{\zeta}\rangle_{111}, \ldots, \langle\tilde{\zeta}\rangle_{ddd}\big)\cdot\tilde{\vartheta},
\end{equation}     
where, upon permuting $\tilde{\zeta}^1, \ldots, \tilde{\zeta}^d$ if necessary, we may assume that the `vanishing' index $i_0$ from \eqref{thm:robust_ica:aux24} reads $i_0=1$. Now for a given $\alpha>0$, which will be specified below, define 
\begin{equation}\label{thm:robust_ica:aux34}
\underline{\lambda}_\alpha \, \coloneqq \, \big(0, \alpha/k_d, 2\alpha/k_d, \ldots, (d-1)\alpha/k_d\big) \equiv \big(\lambda^\alpha_i\big)_{i=1}^d 
\end{equation}for the constant $k_d$ defined in Lemma \ref{lem:spread}, and set further $c_\alpha\coloneqq(1, \underline{c}_\alpha)$ with
\begin{equation}\label{thm:robust_ica:aux35}
\underline{c}_\alpha \,\coloneqq\, \tilde{\vartheta}^{-1}\mathrm{ddiag}\big(0, \langle\tilde{\zeta}\rangle_{222}^{-1}, \ldots, \langle\tilde{\zeta}\rangle_{ddd}^{-1}\big)\cdot\underline{\lambda}_\alpha. 
\end{equation}
Clearly $|\underline{\lambda}_\alpha|=\alpha$ and $\underline{\lambda}_{\alpha} = \underline{\lambda}^\vartheta_{c_\alpha}$, and by Lemma \ref{lem:spread} the matrix $\hat{\Lambda}^\zeta_{c_\alpha}$ is such that the minimal distance between any two of its eigenvalues (the latter are the entries of $\underline{\lambda}_\alpha$) is maximal among any matrix in $\{\hat{\Lambda}^\zeta_{c}\mid c \in\R_\times\!\times\R^d \,:\, |\mathrm{diag}(\hat{\Lambda}^\zeta_{c})|=\alpha\}$. To simplify the following estimates, we choose $\alpha$ such that $\hat{\varphi}_{\underline{c}_\alpha}\!(\varepsilon)$ from \eqref{thm:robust_ica:aux29} is dominated by its second argument, i.e.\        such that $\hat{\varphi}_{\underline{c}_\alpha}\!(\varepsilon) = |\underline{c}_\alpha|\varepsilon$. Clearly this holds for   
\begin{equation}\label{thm:robust_ica:aux37}
\alpha>0 \quad \text{ such that } \quad |\underline{c}_\alpha| \,=\, \frac{1}{1-\varepsilon}.
\end{equation}
Note that an according choice of $\alpha$ is certainly possible since, cf.\ \eqref{thm:robust_ica:aux35}, 
\begin{equation}\label{thm:robust_ica:aux38}
|\underline{c}_\alpha| = \alpha\cdot \kappa^\vartheta_{\zeta} \quad\text{for}\quad \kappa^\vartheta_{\zeta}\coloneqq\big|\tilde{\vartheta}^{-1}[D_\zeta\underline{\lambda}_1]\big| > 0,  
\end{equation}
where $D_\zeta\coloneqq \mathrm{ddiag}\big[\big(0, \langle\tilde{\zeta}\rangle_{222}^{-1}, \ldots, \langle\tilde{\zeta}\rangle_{ddd}^{-1}\big)\big]$ and the positivity of $\kappa^\vartheta_{\zeta}$ follows from the invertibility of $\tilde{\vartheta}^{-1}$ (as thus $\ker(\tilde{\vartheta}^{-1}) = \{0\} \niton D_\zeta\cdot\underline{\lambda}_1$). In respect of \eqref{thm:robust_ica:aux28} and \eqref{thm:robust_ica:aux37}, we may bound $E_\alpha\coloneqq E_{c_\alpha}$ by
\begin{align}
\|E_\alpha\| \, &\leq \, \kappa_{c_\alpha}|\underline{c}_\alpha|\cdot\varepsilon + (1 + \kappa_0)|\underline{c}_\alpha|^2\cdot\varepsilon^2 \\
&=\, \frac{\varepsilon}{(1-\varepsilon)^2}\left[1 + \big\|\tilde{\vartheta}\big\|\varsigma + \kappa_2(\tilde{\vartheta})\big[\xi d + (1-\varepsilon)\big] + \kappa_0\varepsilon\right]\\
&<\, \left[1 + \kappa_0 + \big\|\tilde{\vartheta}\big\|\varsigma + (1 + \xi d)\kappa_2(\tilde{\vartheta})\right]\cdot\frac{\varepsilon}{(1-\varepsilon)^2}.\label{thm:robust_ica:aux39}
\end{align}
Recalling $A_R\equiv RA$ and $\|\tilde{\vartheta}\| = \|\vartheta\|\leq\sqrt{d}\|A_R\|$ (as $\vartheta=\theta_\star A_R$ with $\theta_\star\in\Xi_1$), we find that
\begin{align}\label{thm:robust_ica:aux39.1}
K_\vartheta\,&\coloneqq\,1 + \kappa_0 + \big\|\tilde{\vartheta}\big\|\varsigma + (1 + \xi d)\kappa_2(\tilde{\vartheta})\\
&\phantom{:}\leq\, 1  + \sqrt{d}\|A_R\|\varsigma + \big(1 + (1+\xi d)\kappa_2(A_R)\big)\kappa_0 \, \eqqcolon \, \bar{K}_0 \label{thm:robust_ica:aux39.2} 
\end{align}     
since also $\kappa_2(\tilde{\vartheta})\leq\kappa_0\kappa_2(A_R)$ by the assumed condition number bound on $\theta_\star$. Next we tie in Lemma \ref{lem:perturbationresults}. Using the terminology of said lemma, note that by the above choice \eqref{thm:robust_ica:aux35} of $c_\alpha$, the eigenvalues $\lambda_i\coloneqq\lambda^{c_\alpha}_i$, $i\in[d]$, of the diagonal matrix $C\coloneqq\hat{\Lambda}_{c_\alpha}^\zeta$ are pairwise distinct with associated unit-norm eigenvectors given by $(v_i)_{i\in[d]}\coloneqq(e_i)_{i\in[d]}$, the standard basis of $\R^d$. Accordingly we have (in the notation of Lem.\ \ref{lem:perturbationresults}) $(u_i)_{i\in[d]}=(e_i^\intercal)_{i\in[d]}$ and, thus, $\kappa_i=1$ and
\begin{equation}\label{thm:robust_ica:aux40}
s_i \, = \, \sigma_{d-1}\big((\delta_{\mu\nu}-\delta_{i\mu\nu})_{\mu\nu}\cdot(C-\lambda_i\mathrm{I})\big) \,= \, \min_{j\in[d]\setminus\{i\}}\big|\lambda_j - \lambda_i\big| = \alpha/k_d
\end{equation} 
for each $i\in[d]$ (cf.\ \eqref{thm:robust_ica:aux34}). We observe that $E_\alpha$ $(\equiv E)$ thus meets the hypothesis \eqref{lem:perturbationresults:eq2} of Lemma \ref{lem:perturbationresults}. Indeed: Notice first that the stronger (than \eqref{lem:perturbationresults:eq2}) inequality  
\begin{equation}\label{thm:robust_ica:aux41}
\|E_\alpha\| \, < \, \gamma^{-1}\min_{i\in[d]}\frac{s_i}{\kappa_i} \, = \, \frac{\alpha}{\gamma k_d}\, , \quad \gamma\coloneqq 1 + \sqrt{5},    
\end{equation}  
follows -- due to \eqref{thm:robust_ica:aux38}, \eqref{thm:robust_ica:aux39} and \eqref{thm:robust_ica:aux39.1} -- from the (yet to be proved) inequality 
\begin{equation}\label{thm:robust_ica:aux42}
\frac{\varepsilon}{1-\varepsilon}\, < \,\big(\gamma k_d\kappa^\vartheta_\zeta K_\vartheta\big)^{-1}\eqqcolon q_\vartheta\,, 
\end{equation}  
which in turn is equivalent to $\varepsilon < q_\vartheta/(1+q_\vartheta)$. But since by definition, cf.\ \eqref{thm:robust_ica:aux38} and \eqref{thm:robust_ica:aux39.2}, 
\begin{equation}
\kappa^\vartheta_\zeta\cdot K_\vartheta \,\leq\, \kappa^\vartheta_\zeta(1 + \kappa_0) + \kappa_2(\tilde{\vartheta})\varsigma\varsigma_1 + (1 + \xi  d)\kappa^\vartheta_\zeta\kappa_2(\tilde{\vartheta})
\end{equation}    
for the source-dependent constant $\varsigma_1\,\coloneqq\,\big\|D_\zeta\cdot\underline{\lambda}_1\big\| = k_d^{-1}\sqrt{\sum_{i=1}^d\frac{(i-1)^2}{\langle\tilde{\zeta}\rangle_{iii}^2}}$ and with 
\begin{equation}
\kappa(\tilde{\vartheta})\leq\kappa_0\kappa_2(A_R) \quad \text{as well as} \quad \kappa^\vartheta_\zeta \,\leq\, \big\|\tilde{\vartheta}^{-1}\big\|\varsigma_1 \, \leq \, \frac{\kappa_0\big\|A_R^{-1}\big\|}{\sqrt{d}}\varsigma_1
\end{equation}  
(since $\|\tilde{\vartheta}^{-1}\| = \|\vartheta^{-1}\| = \|A_R^{-1}\|\|\theta_\star^{-1}\|$ and $\|\theta_\star^{-1}\|\leq\kappa_0/\|\theta_\star\| = \kappa_0/\sqrt{d}$), we find that   
\begin{equation}\label{thm:robust_ica:aux46}
\kappa^\vartheta_\zeta K_\vartheta \,\leq\, r_0\coloneqq\kappa_0\varsigma_1\left[\frac{\big\|A_R^{-1}\big\|}{\sqrt{d}}\big(1 + \kappa_0 + (1+\xi d)\kappa_0\kappa_2(A_R)\big) + \kappa_2(A_R)\varsigma\right]  
\end{equation} 
and hence $q_\vartheta \, \geq \, \big(\gamma k_d r_0\big)^{-1}\eqqcolon q_0$ and thus
\begin{equation}\label{thm:robust_ica:aux47}
\frac{q_\vartheta}{1 + q_\vartheta} \,\geq\, \frac{q_0}{1 + q_0}.  
\end{equation}  
But since $\varepsilon<\varepsilon_0 = q_0/(1+q_0)$ by assumption, we conclude that \eqref{thm:robust_ica:aux47}, and hence \eqref{thm:robust_ica:aux42} and thus \eqref{thm:robust_ica:aux41} and with it \eqref{lem:perturbationresults:eq2} in particular, hold as desired. This guarantees that Lemma \ref{lem:perturbationresults} applies to the above matrix $C\equiv\hat{\Lambda}^\zeta_{c_\alpha}$ and its perturbation $\tilde{C}\coloneqq C + E_\alpha$ (cf.\ \eqref{thm:robust_ica:aux26}). Denote by $\tilde{\lambda}_1, \ldots, \tilde{\lambda}_d$ the eigenvalues of $\tilde{C}$ enumerated in the order (from left to right) in which they appear on the diagonal matrix $\hat{\Lambda}^\star_{c_\alpha}$ in \eqref{thm:robust_ica:aux26}. Since, in further consequence of \eqref{thm:robust_ica:aux40} and \eqref{thm:robust_ica:aux41},
\begin{equation}
\|E_\alpha\| + \frac{2\|u_i\|_2\kappa_i}{s_i}\|E_\alpha\|^2 \ < \ \frac{(2 + \gamma)\alpha}{\gamma^2 k_d} = \frac{\alpha}{2 k_d}\,,     
\end{equation}               
Lemma \ref{lem:perturbationresults} (cf.\ \eqref{lem:perturbationresults:eq4}) then implies that there is an assignment $\sigma: [d] \rightarrow [d]$ such that 
\begin{equation}\label{thm:robust_ica:aux49}
\big|\tilde{\lambda}_{\sigma(i)} - \lambda_i\big| \, < \, \frac{\alpha}{2k_d} \quad \text{ for each } i\in[d]. 
\end{equation}   
It follows that $\sigma$ is in fact a permutation and the eigenvalues $\tilde{\lambda}_1, \ldots, \tilde{\lambda}_d$ of $\tilde{C}$ (all real) are pairwise distinct. Indeed: assuming otherwise that there are $i,j\in[d]$ with $i\neq j$ such that $\tilde{\lambda}_{\sigma(i)} = \tilde{\lambda}_{\sigma(j)}$ implies, due to \eqref{thm:robust_ica:aux49}, that
\begin{equation}
|\lambda_i - \lambda_j| \, \leq \, \big|\lambda_i - \tilde{\lambda}_{\sigma(i)}\big| + \big|\tilde{\lambda}_{\sigma(j)} - \lambda_j\big| \, < \, \alpha/k_d\,,
\end{equation}   
contradicting $\min_{i,j\in[d],\, i\neq j}|\lambda_i - \lambda_j|=\alpha/k_d$ (recall \eqref{thm:robust_ica:aux34}). Hence for each $i\in[d]$, the eigenspace $V_i$ of $\tilde{C}$ associated to $\tilde{\lambda}_i$ is of real dimension one, namely $V_i = \langle \tilde{\vartheta}_i\rangle_\R$ for $\tilde{\vartheta}_i$ the $i^{\mathrm{th}}$ column of $\tilde{\vartheta}$. Thus by Lemma \ref{lem:perturbationresults} \eqref{lem:perturbationresults:eq3} we obtain that,\footnote{\ Since for every eigenpair $(\tilde{\lambda}_j, \tilde{w}_j)$ of $\tilde{C}$ we must have $\tilde{w}_j = q\cdot\tilde{\vartheta}_j$, $q\in\R_\times$, as $\tilde{\lambda}_j$ is simple, the eigenvector $\tilde{v}_j\equiv(\tilde{v}_j^1, \ldots, \tilde{v}_j^d)$ to $\tilde{\lambda}_j$ which has been singled out in \eqref{lem:perturbationresults:eq3} via $(\tilde{v}_j^i=)\,u_i^\intercal\tilde{v}_j = 1$ reads $\tilde{v}_j = \tilde{\vartheta}_{ij}^{-1}\cdot\tilde{\vartheta}_j$.} for each $i\in[d]$ and $\sigma\in S_d$ as above,
\begin{equation}\label{thm:robust_ica:aux51}
\big\|\beta_i^{-1}\cdot\tilde{\vartheta}_{\sigma(i)} - e_i\big\|_2 \ \leq \ \frac{2 k_d}{\alpha}\|E_\alpha\| \quad \text{ for } \ \beta_i\coloneqq\tilde{\vartheta}_{\sigma(i)i}.     
\end{equation}             
Using \eqref{thm:robust_ica:aux38} and \eqref{thm:robust_ica:aux39}, the right-hand side in \eqref{thm:robust_ica:aux51} can be estimated to
\begin{equation}\label{thm:robust_ica:aux52}
\frac{2k_d}{\alpha}\|E_\alpha\| \ \leq \ r_1\frac{\varepsilon}{1 - \varepsilon} \quad\text{ for }\quad r_1\coloneqq 2k_d\kappa^\vartheta_\zeta K_\vartheta\,, 
\end{equation}    
whence in partic.\ $\tfrac{2k_d}{\alpha}\|E_\alpha\| < 2/\gamma < 1$ by \eqref{thm:robust_ica:aux42}. As for an estimate of the $\beta_i$, notice first that 
\begin{equation}
\beta_i\,=\,\max_{\nu\in[d]}\big|\tilde{\vartheta}_{\nu i}\big| \,\equiv\, \|\tilde{\vartheta}_i\|_\infty \qquad \text{for all } i\in[d], 
\end{equation}
as otherwise $\max_{\nu\neq \sigma(i)}\big|\beta_i^{-1}\cdot\tilde{\vartheta}_{\nu i}\big|\geq 1$ and thus $\big\|\beta_i^{-1}\cdot\tilde{\vartheta}_{\sigma(i)} - e_i\big\|_2 \geq 1$ in contradiction to $\tfrac{2k_d}{\alpha}\|E_\alpha\|<1$. Hence for each $i\in[d]$, 
\begin{equation}\label{thm:robust_ica:aux54}
\beta_i \ \geq \ d^{-1/2}\|\tilde{\vartheta}_i\|_2 \ \geq \ \frac{\|A_R\|_2}{\kappa_2(A_R)\sqrt{d}}   
\end{equation}      
where the first inequality is due to the equivalence of maximum and Euclidean norm and the last inequality follows from $\|\tilde{\vartheta}_i\|_2 = \|A_R^\intercal\cdot\theta_i\|_2 \ \geq \ \sigma_d(A_R) = \|A_R\|_2/\kappa_2(A_R),$ where we used that $\theta_i$, which denotes the $i^{\mathrm{th}}$ row of $\theta_\star$, has Euclidean norm $1$. Further,
\begin{equation}\label{thm:robust_ica:aux56}
\beta_i \, \leq \, \|\tilde{\vartheta}_i\|_2 = \|A_R^\intercal\cdot\theta_i\|_2 \, \leq \, \|A_R\|_2 
\end{equation}     
for each $i\in[d]$. Thus, the condition number of $\Lambda\coloneqq\mathrm{ddiag}[\beta_1^{-1}, \ldots, \beta_d^{-1}]$ is bounded by
\begin{equation}\label{thm:robust_ica:aux57}
\kappa_2(\Lambda) \, \leq \, \sqrt{d}\kappa_2(A_R).
\end{equation}
\noindent
Combining \eqref{thm:robust_ica:aux51} and \eqref{thm:robust_ica:aux52}, we obtain 
\begin{equation}\label{thm:robust_ica:aux58}
\left\|\tilde{\vartheta}\cdot \tilde{M} - \mathrm{I}\right\| \ \leq \ r_2\frac{\varepsilon}{1-\varepsilon} \quad\text{ for }\quad r_2\coloneqq d\cdot r_1,  
\end{equation} 
where $\tilde{M}\coloneqq P_\sigma^\intercal\cdot\Lambda\in\M$ for $\Lambda$ as above and $P_\sigma=(\delta_{\sigma(i)j})$ the permutation matrix associated to $\sigma$. Hence for $M\coloneqq \tilde{M}^{-\intercal}\in\M$ and $E\coloneqq (\tilde{\vartheta}\tilde{M} - \mathrm{I})^\intercal$ we get 
\begin{equation}\label{thm:robust_ica:aux59}
\vartheta \,=\, M(\mathrm{I} + E) \quad \text{ \ with \ } \quad \|E\|\leq \tfrac{r_2\varepsilon}{1-\varepsilon}  
\end{equation}         
and with $r_2 \, \leq \, 2dk_dr_0$, for $r_0$ as in \eqref{thm:robust_ica:aux46}; this proves \eqref{thm:robust_ica:eq1}. To see the last inequality in \eqref{thm:robust_ica:eq1}, we note first that since $\vartheta=\tilde{\theta}_\star A$ by definition, equation \eqref{thm:robust_ica:aux59} yields
\begin{equation}\label{thm:robust_ica:aux61}
\tilde{\theta}_\star \, = \, M\cdot A^{-1} + ME\cdot A^{-1}.
\end{equation}    
In particular, we thus have obtained the relative-error bound 
\begin{equation}\label{thm:robust_ica:aux61.2}
\frac{\|\tilde{\theta}_\star - MA^{-1}\|}{\|MA^{-1}\|}\,\leq\, \frac{\|MEM^{-1}\|_2\|MA^{-1}\|}{\|MA^{-1}\|}\,\stackrel{\eqref{thm:robust_ica:aux57}}{\leq}\, r_3\frac{\varepsilon}{1-\varepsilon}\,\eqqcolon\,\varepsilon'
\end{equation}where $r_3\coloneqq\sqrt{d}\kappa_2(A_R)r_2$ (recall \eqref{thm:robust_ica:aux57} and \eqref{thm:robust_ica:aux58}). This concludes proof.\hfill $\square$     

\subsection{Related Work}\label{sect:perturbedjdp:relatedwork}
As mentioned in Sect.\ \ref{sect:blindinversionthm1}, the general strategy of performing a blind inversion of linearly mixed stochastic sources via a joint diagonalisation of suitably chosen equivariant tensor statistics of the observed signal -- as done in \eqref{lica:inverseset} -- is a well-established part of the current theory of blind source separation, see e.g.\ \citep{cai2017algebraic, chabriel2014, usevich2020} or \citep{HBS, MNT, nordhausen2022, pfister2019, theis2006, usevich2020} for an overview. 
Its relevance in the BSS-context has brought this joint diagonalisation problem (JDP), which underlies the minimisation \eqref{lica:inverseset} of the contrast \eqref{contrast:eq1} over the manifold $\Theta$ (and certain subsets such as \eqref{lica:xi1_domain}), considerable popularity and much research attention in numerical analysis and optimization communities, cf.\ the cited literature and references therein.\\[-0.75em] 

Yet despite the manifest relevance of this actively researched problem, the perturbation analysis of JDP, as well as related sensitivity analyses for tensor decompositions in general, has been notoriously underexplored, cf.\ \citep[Sect.\ 1.4]{cai2019perturbation}. Among the very few perturbation results available, the approaches that anticipate our own bound \eqref{thm:robust_ica:eq1} most closely are \cite{afsari2008}, who obtained first-order perturbation bounds for joint diagonalizers of symmetrically perturbed symmetric matrices, and \cite{shi2015}, who under the same row-normalisation constraints as ours, but once again in a fully symmetric setting only, obtained bounds between symmetrically perturbed and exact joint diagonalizers, and finally \cite{cai2019perturbation}, who obtained upper bounds related to ours in the more general context of perturbed joint block diagonalization. For completeness, we also mention the recent but not directly related work \cite{li2022vector}, which, again under the symmetry constraint, investigates some numerical stability properties of joint approximate diagonalizers.\\[-0.75em] 

We emphasize that, although similar in nature, none of the above works render our bound \eqref{thm:robust_ica:eq1} or its underlying perturbation analysis (Section \ref{chap:robustICA:subsect:signature_identifiability:proof}) redundant: The perturbation results of Theorem \ref{thm:robust_ica} are a cornerstone for our proof of Theorem \ref{thm:robustness}, where the robustness of \eqref{lica:inversemap1} is established via the continuity of maps between premetric spaces, and the latter requires an understanding of how general and thus also asymmetric perturbations of the statistics \eqref{contrast:eq0.1} affect the minimizers \eqref{lica:inverseset} (leaving only \cite{cai2019perturbation} applicable). Based on recently obtained second-order bounds \citep{karowkressner} for the eigenspaces of generically perturbed matrices, the tailor-made perturbation analysis in Section \ref{chap:robustICA:subsect:signature_identifiability:proof} is methodologically independent of all prior approaches and provides this understanding in an informative, concise and essentially self-contained way. As remarked above, however, our resulting bound \eqref{thm:robust_ica:eq1} has not been optimized, so a combination of Section \ref{chap:robustICA:subsect:signature_identifiability:proof} with some of the referenced approaches is likely to yield improvements.

\subsection{Proof of Theorem \ref{thm:robustness}: The Continuity of \texorpdfstring{$\fI$}{I}}\label{sect:pf:thm:robustness} 
\begin{proof} 
By Remark \ref{rem:causaltopology} and because $(E, \mathbbm{d})$ is a premetric space and $(F, \mathfrak{d})$ is a metric space, the map $\fI$ is continuous on $\sI$ if (and only if; see Remark \ref{rem:premetric_topofacts:continuity}) for each $\fp_\star\in\sI$ we have that: 
\begin{equation}\label{thm:robustness:aux1}
\forall\,\varepsilon>0 \,:\, \exists\,r>0\,:\,\forall\,\fp\in E\,:\, \text{ if } \quad \mathbbm{d}(\fp_\star, \fp)<r \quad \text{ then } \quad \mathfrak{d}(\fI(\fp_\star),\fI(\fp))<\varepsilon.
\end{equation} 
To prove \eqref{thm:robustness:aux1}, fix any $\fp_\star\equiv(\mu_\star, A)\in\sI$ and let $\varepsilon>0$ be arbitrary. 

Let further $0< r< r_0$, for $r_0$ as in Lemma \ref{lem:robustness1}, and take any $\fp\equiv(\mu,f)$ in $E$ with $\mathbbm{d}(\fp_\star, \fp) < r$. By Lemma \ref{lem:robustness1}, this $\p$ can then be assigned a unique signal $\mu_\zeta$ in $\fD$ such that 
\begin{equation}\label{thm:robustness:aux2}
\mu_\chi\equiv f_\ast\mu = A\mu_\zeta \quad\text{ with }\quad \delta_{\independent}(\mu_\zeta) \leq K'r^{\alpha/2} \quad\text{ and }\quad \big\|[\mu_\zeta]_{\mathfrak{c}} - [\mu_\star]_{\mathfrak{c}}\big\|_{\mathcal{V}} \,\leq\, \hat{K}r^\alpha
\end{equation} 
for (explicit) constants $K'\equiv K_{r}\coloneqq K_{r}(\mu_\star, K_{\fS})>0$, increasing in $r$, and $\hat{K}=\hat{K}(\mu_\star, K_{\fS})>0$.  

By Section \ref{chap:robustICA:subsect:signature_identifiability:algorithm}, the triple $(\mu_\chi,\mu_\zeta, A)$ can be associated the identifiability-related parameters
\begin{equation}\label{thm:robustness:aux3}
\varepsilon(\mu_\zeta) \coloneqq \frac{q(\mu_\zeta)}{1+q(\mu_\zeta)} \quad \text{ with } \quad q(\mu_\zeta)\coloneqq \big(\gamma k_d p(\mu_\zeta)\big)^{-1}  
\end{equation}            
for $\gamma = 1+\sqrt{5}$,  $k_d= \sqrt{\tfrac{d}{6}(d-1)(2d-1)}$ and the function $p(\nu)=p(\nu; \mathfrak{p}_\star)$ ($\nu\in\fD$) given by 
\begin{equation}\label{thm:robustness:aux4}
p(\nu) = \tilde{\kappa}_0\varsigma_1(\nu)\Bigg[\tfrac{\big\|A_{R_\nu}^{-1}\big\|}{\sqrt{d}}\big(1 + \tilde{\kappa}_0 + (1+\xi(\nu) d)\tilde{\kappa}_0\kappa_2(A_{R_\nu})\big) + \kappa_2(A_{R_\nu})\varsigma(\nu)\Bigg]
\end{equation} 
for $A_{R_\nu}\coloneqq R_\nu A$ and the whitening matrix $R_\nu\coloneqq\mathcal{R}(\mathcal{C}_{[\nu]_0})$ as in Lemma \ref{lem:choiceofR}, writing $\mathcal{C}_{\mathfrak{a}}\coloneqq A\mathcal{S}(\mathfrak{a})A^\intercal$ for the $A$ fixed above, and with the auxiliary functions 
\begin{equation}\label{thm:robustness:aux5}
\varsigma_1(\nu)\,\coloneqq\,k_d^{-1}\sqrt{\textstyle\sum_{i=1}^d\!\Big(\frac{(i-1)\langle\nu\rangle_{ii}^{3/2}}{\langle\nu\rangle_{iii}}\Big)^{\!2}} \quad\text{ and }\quad\varsigma(\nu)\,\coloneqq\,\sqrt{\textstyle\sum_{i=1}^d\langle\nu\rangle_{iii}^2/\langle\nu\rangle_{ii}^3}
\end{equation} 
and $\xi(\nu)\,\coloneqq\,\sqrt{\textstyle\sum_{j=1}^d\big\|[A_{R_\nu }\cdot\nu]_j\big\|^2}$ and $\tilde{\kappa}_0 \coloneqq b + \Delta_\kappa$, cf.\ \eqref{lica:xi1_domain2}. Recall further from Remark \ref{rem:xi1_domain2} that there is a radius $r_\ast \equiv r_\ast(\p_\star, \Delta_\kappa)>0$ such that, for $\bar{B}_{(A,\nu)}$ as defined in that remark,
\begin{equation}\label{thm:robustness:aux5.1}
\bar{B}_{(A,\nu)}\in\,\tilde{\Xi}_1 \quad \text{ if } \ \ \big\|[\nu]_0 - [\mu_\star]_0\big\| \leq r_\ast.
\end{equation}
Using the bound in \eqref{thm:robustness:aux2} on the distance between $[\mu_\zeta]_{\mathfrak{c}}$ and $[\mu_\star]_{\mathfrak{c}}$, we see that the coefficients $\left.\eqref{thm:robustness:aux5}\right|_{\nu=\mu_\zeta}$ are well-defined -- i.e.\ that $\langle\mu_\zeta\rangle_{ii}>0$ and $\langle\mu_\zeta\rangle_{iii}>0$ for all $i\in[d]$ -- if 
\begin{equation}
r \,<\, \big[(\rho_0\wedge \varrho_0)/\hat{K}\big]^{1/\alpha}\eqqcolon r_1, \quad\text{ where }\quad \varrho_0\coloneqq \min\nolimits_{i\geq 2}\,\langle\mu_\star\rangle_{iii}\,;  
\end{equation}  
note that $\varrho_0=\varrho_0(\mu_\star)>0$ by definition of $\sI$. As the denominators in \eqref{thm:robustness:aux5} are bounded away from zero on $\{\nu\in\fD\mid \|[\nu]_{\mathfrak{c}} - [\mu_\star]_{\mathfrak{c}}\|_{\mathcal{V}}\leq \mathfrak{q}\}\eqqcolon\tilde{B}_\mathfrak{q}$ for any $\mathfrak{q}<\rho_0\wedge\varrho_0$, we see that the coefficient bounds $\mathfrak{K}_1(\mathfrak{q})\coloneqq\sup_{\nu\in\tilde{B}_\mathfrak{q}}\varsigma_1(\nu)$ and $\mathfrak{K}_2(\mathfrak{q})\coloneqq\sup_{\nu\in\tilde{B}_\mathfrak{q}}\varsigma(\nu)$ are both finite. Further, since the eigenvalues $(\tilde{\lambda}_i(\mathfrak{a}; A))\coloneqq\lambda(A\mathcal{S}(\mathfrak{a})A^\intercal)$ of the symmetrized matrix $\mathcal{C}_{\mathfrak{a}}=A\mathcal{S}(\mathfrak{a})A^\intercal$ depend continuously on $\mathfrak{a}\in\R^{d\times d}$ (Lemma \ref{lem:choiceofR}), so does their minimum $\underline{\tilde{\sigma}}(\mathfrak{a})\coloneqq\min_{i\in[d]}\tilde{\lambda}_i(\mathfrak{a}; A)$. Hence and since $\sigma_\star\equiv\underline{\tilde{\sigma}}([\mu_\star]_0) > 0$ [recall that $\mathcal{C}_{[\mu_\star]_0}$ is congruent to $\mathcal{S}([\mu_\star]_0)\in\mathrm{Pos}_d$ (the $\mathrm{Pos}_d$-inclusion holds as $C_{\mu_\star}$ is invertible), cf.\ \eqref{rem:choiceofR:eq3}], we have $\tilde{\rho}_0\coloneqq\sup\{s\geq0\mid \min_{\{\mathfrak{a}\in\R^{d\times d}\,:\, \|\mathfrak{a}-[\mu_\star]_0\|\leq s\}}\underline{\tilde{\sigma}}(\mathfrak{a})\geq\sigma_\star/2\} > 0$. This implies, via \eqref{rem:choiceofR:eq2}, that 
\begin{equation}\label{thm:robustness:aux6.1}
\sup_{\mathfrak{a}\in\mathfrak{B}_{\tilde{\rho}_0}}\!\big\|\mathcal{R}(\mathcal{C}_{\mathfrak{a}})\big\| \leq \sqrt{2d/\sigma_\star}\eqqcolon L_{\p_\star} \ \text{ for } \ \mathfrak{B}_{\tilde{\rho}_0}\equiv\mathfrak{B}_{\tilde{\rho}_0}([\mu_\star]_0)\coloneqq\big\{\mathfrak{a}\in\R^{d\times d} \, \big| \, \|\mathfrak{a}-[\mu_\star]_0\|\leq \tilde{\rho}_0\big\}  
\end{equation}   
(notice that $\{\mathcal{C}_{\mathfrak{a}}\mid \mathfrak{a}\in\mathfrak{B}_{\tilde{\rho}_0}\}\subset\mathrm{Sym}_d^+$ by definition of $\tilde{\rho}_0$). Hence it is clear that $\mathfrak{K}_3\coloneqq\sup\{\xi(\nu) \mid \nu\in\fD \,:\, \|[\nu]_{\mathfrak{c}}- [\mu_\star]_{\mathfrak{c}}\|_{\mathcal{V}}\leq \tilde{\rho}_0\}$ is finite, and combined with \eqref{lem:robustness1:eq2} we find 
\begin{equation}\label{thm:robustness:aux6.2}
\sup\nolimits_{(\tilde{\mu}, \tilde{f})\in\mathbb{B}_r^\star}\big\|\mathcal{R}(\mathcal{C}_{[\tilde{\mu}_\zeta]_0})\big\| \,\leq\, L_{\p_\star} \ \text{ provided that } \ r \,<\,r_0\wedge r_1 \wedge (\tilde{\rho}_0/\hat{K})^{1/\alpha}\eqqcolon r_2'.   
\end{equation}    
Here, the functional assignment [as used in \eqref{thm:robustness:aux2} and \eqref{thm:robustness:aux6.2}] 
\begin{equation}\label{thm:robustness:aux6.3}
\mathbb{B}^\star_r\ni(\tilde{\mu}, \tilde{f})\,\longmapsto\,\tilde{\mu}_\zeta\in\tilde{B}_{\hat{K}r^\alpha}
\end{equation}
is the one asserted by Lemma \ref{lem:robustness1}, cf.\ \eqref{thm:robustness:aux2}. Assuming from now on that 
\begin{equation}
r \,<\, r_2\,\coloneqq\, r_2' \wedge(r_\ast/\hat{K})^{1/\alpha},
\end{equation} we just saw that, for each $\nu\in\tilde{B}_{\mathfrak{q}_r}$ with $\mathfrak{q}_r\coloneqq\hat{K}r^\alpha$, we have both \eqref{thm:robustness:aux5.1} and the domination    
\begin{equation}\label{thm:robustness:aux8}
|p(\nu)| \,\leq\, \tilde{\kappa}_0\mathfrak{K}_1(\mathfrak{q}_r)\!\Bigg[\frac{\beta_1([\nu]_0)}{\sqrt{d}}\big(1 + \tilde{\kappa}_0 + (1+\mathfrak{K}_3d)\tilde{\kappa}_0\beta_2([\nu]_0)\big) + \beta_2([\nu]_0)\mathfrak{K}_2(\mathfrak{q}_r)\Bigg] 
\end{equation}      
where $\beta_j\equiv\beta_j(\,\cdot\,;A)$ $(j=1,2)$ are the bounding functions from Lemma \ref{lem:choiceofR} \eqref{lem:choiceofR:eq2}--\eqref{lem:choiceofR:eq3}. The above is equivalent (cf.\ the proof of Lemma \ref{lem:premetric_facts2}) to the inequality 
\begin{equation}\label{thm:robustness:aux8.1}
p(\cdot)\leq\vartheta_r\circ[\,\cdot\,]_0 \quad\text{ for }\quad \vartheta_r\coloneqq\tilde{\kappa}_0\mathfrak{K}_1(\mathfrak{q}_r)\!\bigg[\frac{\beta_1}{\sqrt{d}}\big(1 + \tilde{\kappa}_0 + (1+\mathfrak{K}_3d)\tilde{\kappa}_0\beta_2\big) + \beta_2\mathfrak{K}_2(\mathfrak{q}_r)\bigg]
\end{equation}
which holds pointwise on $\tilde{B}_{\mathfrak{q}_r}$ and where $[\,\cdot\,]_0 : \fD\rightarrow \R^{d\times d}$ is the map that sends a signal to its ground-level coredinate (cf.\ \eqref{def:coordinates:eq2}). Note that while $\beta_1$ is continuous, the function $\beta_2$ is a fraction of the form $\beta_2 = \hat{\beta}/z$, with $\hat{\beta}$ the enumerator and $z$ the denominator in its definition \eqref{lem:choiceofR:eq3}, where $\hat{\beta}$ and $z$ are each continuous on $\R^{d\times d}$ (Lemma \ref{lem:choiceofR}). Note further that by definition of $r_2$, see \eqref{thm:robustness:aux6.2} [and recall \eqref{rem:choiceofR:eq2}], the denominator $z$ is bounded away from zero on $[\tilde{B}_{\mathfrak{q}_r}]_0\equiv\{[\nu]_0\mid\nu\in\tilde{B}_{\mathfrak{q}_r}\}\eqqcolon\mathfrak{A}_r$, that is $\mathfrak{z}(r)\coloneqq\inf\{z(\mathfrak{a})\mid \mathfrak{a}\in\mathfrak{A}_r\}>0$. Hence $\restr{\vartheta_r}{\mathfrak{A}_r}\!\leq\hat{\vartheta}_r\big|_{\mathfrak{A}_r}$, and for $\hat{\beta}_r\coloneqq\mathfrak{z}(r)^{-1}\hat{\beta}$ we consequently have the $\tilde{B}_{\mathfrak{q}_r}$-pointwise domination
\begin{equation}\label{thm:robustness:aux8.2}
p(\cdot)\leq\hat{\vartheta}_r\circ[\,\cdot\,]_0 \ \text{ for } \ \hat{\vartheta}_r\coloneqq\tilde{\kappa}_0\mathfrak{K}_1(\mathfrak{q}_r)\!\bigg[\frac{\beta_1(r)}{\sqrt{d}}\big(1 + \tilde{\kappa}_0 + (1+\mathfrak{K}_3d)\tilde{\kappa}_0\hat{\beta}_r\big) + \hat{\beta}_r\mathfrak{K}_2(\mathfrak{q}_r)\bigg]. 
\end{equation}        
The majorant $\hat{\vartheta}_r$ is continuous and strictly positive on the ball $\mathfrak{B}_{\tilde{\rho}_0}\supseteq\mathfrak{A}_r$ defined above. 

In particular [recall $\mathfrak{q}_r\leq \tilde{\rho}_0$], for the function $\psi_r(\mathfrak{a}, \mathfrak{a}_1, \ldots, \mathfrak{a}_d)\coloneqq\hat{\vartheta}_r(\mathfrak{a})$ we have    
\begin{equation}
\psi_r\in C\big(\bar{B}_{\mathfrak{q}_r}^{\mathcal{V}};\R_{>0}\big) \quad\text{on}\quad \bar{B}_{\mathfrak{q}_r}^{\mathcal{V}}\equiv\bar{B}_{\mathfrak{q}_r}^{\mathcal{V}}([\mu_\star]_{\mathfrak{c}})\coloneqq\{\tilde{\mathfrak{a}}\in\mathcal{V}\mid \|\tilde{\mathfrak{a}} - [\mu_\star]_{\mathfrak{c}}\|_{\mathcal{V}} \leq r\}, 
\end{equation}  
and by the above $p(\cdot)$-domination \eqref{thm:robustness:aux8.2}, recalling \eqref{thm:robustness:aux6.3} (see also \eqref{thm:robustness:aux2}), we obtain
\begin{equation}\label{thm:robustness:aux8.4}
p(\tilde{\mu}_\zeta) \,\leq\, \psi_r([\tilde{\mu}_\zeta]_{\mathfrak{c}}) \quad\text{ for each } \ (\tilde{\mu}, \tilde{f})\in\mathbb{B}^\star_r. 
\end{equation}   
Setting $\psi\coloneqq\psi_{\tilde{r}_2}$ for $\tilde{r}_2\coloneqq 0.95\cdot r_2$ and introducing the function $\tau : [0, \tilde{r}_2] \rightarrow (0, \infty)$ defined by 
\begin{equation}\label{thm:robustness:aux10}
\tau(r)\coloneqq\min\Big\{Q\big((\gamma k_d\psi(\tilde{\mathfrak{a}}))^{-1}\big) \ \big| \ \tilde{\mathfrak{a}}\in\bar{B}_{\mathfrak{q}_r}^\mathcal{V}([\mu_\star]_{\mathfrak{c}})\Big\} \quad\text{ for }\quad Q(u)\coloneqq\frac{u}{1+u},
\end{equation}     
we find for any $\mu_\zeta\in\tilde{B}_{\mathfrak{q}_r}$ as in \eqref{thm:robustness:aux8.4} [recall \eqref{thm:robustness:aux6.3}] that, in consequence of \eqref{thm:robustness:aux3} and \eqref{thm:robustness:aux8.4}, 
\begin{equation}\label{thm:robustness:aux11}
\varepsilon(\mu_\zeta) = Q\big(q(\mu_\zeta)\big) \, \geq \, Q\big((\gamma k_d\psi([\mu_\zeta]_{\mathfrak{c}}))^{-1}\big) \,\geq\, \tau(r), \quad \text{ for each } \ r\in[0, \tilde{r}_2].      
\end{equation} 
Now since $\tau$ is continuous (Lemma \ref{lem:minfuncont}) with $\tau(0) = Q\big((\gamma k_d\psi([\mu_\star]_{\mathfrak{c}}))^{-1}\big) > 0$, there will be a positive radius up to which $\tau$ is greater than the continuous majorant $\upsilon : [0,r_0]\ni r \mapsto K'r^{\alpha/2}$ of $\delta_{\independent}(\mu_\zeta)$ (cf.\ \eqref{thm:robustness:aux2}). This implies, cf.\ again \eqref{thm:robustness:aux2}, that we must have
\begin{align}\label{thm:robustness:aux12}
r_3&\coloneqq\inf\big\{0\leq r\leq \tilde{r}_2 \mid \tau(r) - \sup\nolimits_{(\mu, f)\in \mathbb{B}^\star_r}\delta_{\independent}(\mu_\zeta) \leq 0\big\} \\
&\phantom{:}\geq\inf\big\{0\leq r\leq \tilde{r}_2 \mid \tau(r) - \upsilon(r) \leq 0\big\} \, > \, 0. 
\end{align} 
Let now $0\leq r\leq r_3$, and set $\tilde{K}_1\coloneqq 2dk_d\max\{\psi(\tilde{\mathfrak{a}})\mid \tilde{\mathfrak{a}}\in\bar{B}_{\mathfrak{q}_{r_3}}^\mathcal{V}\}$ and fix $K'\equiv K_{r_3}\,(=\left.K'\right|_{r=r_3})$. Then for any $(\mu,f)\in\mathbb{B}_r(\mu_\star, A)$ with \eqref{thm:robustness:aux6.3}-associated signal $\mu_\zeta\in\bar{B}_{\mathfrak{q}_{r_3}}^\mathcal{V}$, we have $\delta_{\independent}(\mu_\zeta) \leq \tau(r) \leq \varepsilon(\mu_\zeta) < 1$ [by \eqref{thm:robustness:aux12} and \eqref{thm:robustness:aux11}] and hence, by Theorem \ref{thm:robust_ica} \eqref{thm:robust_ica:eq1} [which by \eqref{thm:robustness:aux5.1} is fully applicable], obtain that  
\begin{equation}\label{thm:robustness:aux14}      
\begin{gathered}
\forall\, \theta_\p\in\Phi(f_\ast\mu) \ : \ \text{there is } \  M=M(\theta_\p)\in\M \text{ and } E = E(\theta_\p)\in\R^{d\times d} \quad \text{such that}\\ 
\theta_\p = M(\mathrm{I} + E)A^{-1} \ \text{ with }\ \|E\|\leq\epsilon_r \quad \text{ for } \ \epsilon_r\coloneqq \tilde{K}_1\tfrac{\upsilon(r)}{1-\upsilon(r)}
\end{gathered}
\end{equation} 
(see also \eqref{thm:robust_ica:aux59} in the proof of Theorem \ref{thm:robust_ica} for reference). Clearly, the error bound $\epsilon_r$ is strictly increasing in $r$ (recall that $\upsilon(r)<1$) with $\lim_{r\rightarrow 0+}\epsilon_r = 0$. 

Next, denote the constants $K_2\coloneqq \kappa_2(A)L_{\p_\star}$ and $\tilde{\gamma}_\star\coloneqq \min_{i\in[d]}\big|e_i^\intercal\cdot A^{-1}\mathcal{C}^{1/2}_{[\mu_\star]_0}\big|$ and consider the auxiliary function $\eta:[0,r_3]\rightarrow\mathbb{R}_+$ given by 
\begin{equation}\label{thm:robustness:aux15}
\eta(r)\coloneqq \max\big\{\tilde{\gamma}_\star^{-1}\big\|A^{-1}\big(\mathcal{R}(\mathcal{C}_{\mathfrak{a}})^{-1} - \mathcal{R}_\star^{-1}\big)\big\| \ \big| \ (\mathfrak{a}, \mathfrak{a}_1,\cdots,\mathfrak{a}_d)\in\bar{B}_{\mathfrak{q}_r}^\mathcal{V}\big\}
\end{equation}  
for $\mathcal{R}_\star\coloneqq\mathcal{R}(\mathcal{C}_{[\mu_\star]_0})$. By \eqref{lem:choiceofR:eq1.2} and Lemma \ref{lem:minfuncont}, the function $\eta$ is continuous with $\eta_i(0) = 0$. Let further $\Lambda_{\mathfrak{a}}\coloneqq\mathrm{ddiag}_{i=1,\ldots,d}\big(|e_i^\intercal\cdot A^{-1}\mathcal{R}(\mathcal{C}_{\mathfrak{a}})^{-1}|\big)$ and set $\Lambda_{(\mu,f)}\coloneqq\Lambda_{[\mu_\zeta]_0}$ via \eqref{thm:robustness:aux6.3} (cf.\ \eqref{thm:robustness:aux2}).     

For the (arbitrary) $\varepsilon>0$ fixed in the beginning of this proof, cf.\ \eqref{thm:robustness:aux1}, choose the largest 
\begin{equation}\label{thm:robustness:aux16}
0 < r_4\leq r_3 \quad \text{ such that } \quad \max\nolimits_{r\in[0,r_4)}\!\big(\sqrt{2}\eta(r) + (\|A^{-1}\|_2\hat{\sigma}_r + 1)\epsilon_r\big) < \varepsilon/K_2,
\end{equation}  
where $\hat{\sigma}^\star : [0, r_3] \ni r \mapsto \hat{\sigma}^\star_r\coloneqq\max_{\mathfrak{a}\in\mathfrak{B}_{\mathfrak{q}_r}}\!\!\sqrt{\max\lambda(\mathcal{C}_\mathfrak{a})}\in\R_+$ (which is [monotonically increasing and] continuous by \eqref{lem:choiceofR:eq1.1} and Lemma \ref{lem:minfuncont}). The proof concludes (recall \eqref{thm:robustness:aux1}) if we can show
\begin{equation}\label{thm:robustness:aux17}
\mathfrak{d}\big(\fI(\p_\star), \fI(\p)\big) \leq \varepsilon, \quad \text{ for each } \ \p\in\mathbb{B}_{r_4}(\p_\star).  
\end{equation}
And indeed: Writing $\tilde{\Xi}_{\mathcal{R}}\coloneqq\tilde{\Xi}_1\cdot\mathcal{R}$, we know from Theorem \ref{thm:robust_ica} (applicable by def.\ of $\sI$) that
\begin{equation}
\fI(\p_\star) = \M\cdot A^{-1}\cap\tilde{\Xi}_{\mathcal{R}_\star} = \big\{P\Lambda_{\mathfrak{p}_\star}^{-1} A^{-1} \, \big| \, P\in \mathrm{P}_d^\pm\big\}\eqqcolon\mathcal{A}_\star, 
\end{equation}   
where the second identity can be verified easily from the above definitions (recalling \eqref{lica:xi1_domain} and the unit-row conditions \eqref{lica:unitrows}). To see that \eqref{thm:robustness:aux17} holds, take any $\p\equiv(\tilde{\mu},\tilde{f})\in\mathbb{B}_{r_4}(\p_\star)$ and let us first fix any $\mathfrak{a}_\star\equiv P_{\mathfrak{a}_\star}\Lambda_{\mathfrak{p}_\star}^{-1} A^{-1}\in\mathcal{A}_\star$, denoting $\mathcal{R}_{(\tilde{\mu}, \tilde{f})}\coloneqq\mathcal{R}(\mathcal{C}_{[\tilde{\mu}_\zeta]_0})$ (cf.\ \eqref{thm:robustness:aux6.3}) for convenience. Since $\fI(\p)$ is non-empty \eqref{lica:inverseset}, there exists some $\theta_\p\in\tilde{\Xi}_{\mathcal{R}_\p}$ such that $\theta_\p = M_\p(\mathrm{I} + E_\p)A^{-1}$ for $M_\p\equiv P_{\p}\tilde{\Lambda}_{(\p)}\in\mathrm{M}_d$ (with $P_\p\in\mathrm{P}_d^{\pm}$ and $\tilde{\Lambda}_{(\p)}\equiv\mathrm{ddiag}(m_\p^1, \ldots, m_\p^d)>0$) and $\|E_\p\|\leq\epsilon_{r_4}$, see \eqref{thm:robustness:aux14}. By the $\mathrm{P}_d^\pm$-invariance of $\phi_{\tilde{f}_\star\tilde{\mu}}$ and $\tilde{\Xi}_{\mathcal{R}_\p}$ (cf.\ \eqref{phi:monomialinvar} and \eqref{lica:xi1_domain2}) we find that also $\theta^{\star}_\p\coloneqq (P_{\mathfrak{a}_\star}\!P_\p^{-1})\theta_\p\in\hat{\fI}(\p)$. Hence (and since $\mathrm{P}_d^\pm\subset\mathrm{O}_d$ and $\|A_1 A_2\|\leq \|A_1\|\|A_2\|_2,$ $\forall\, A_1,A_2\in\R^{d\times d}$)
\begin{equation}\label{thm:robustness:aux18.1}
d\big(\mathfrak{a}_\star, \fI(\p)\big) \,\leq\, \big\|\mathfrak{a}_\star - \theta^{\star}_\p\big\| \,\leq\, \|A^{-1}\|_2\big\|\Lambda_{\p_\star}^{-1} - \tilde{\Lambda}_{(\p)}(\mathrm{I} + E_\p)\big\|. 
\end{equation}    
Since $\theta^\star_\p\in\tilde{\Xi}_{\mathcal{R}_\p}$ and hence $\theta^\star_\p\mathcal{R}_p^{-1} = P_{\mathfrak{a}_\star}\!\tilde{\Lambda}_{(\p)}(\mathrm{I} + E_\p)B_\p\in\tilde{\Xi}_1$ for $B_\p\coloneqq(\mathcal{R}_\p A)^{-1}$, we have  
\begin{equation}\label{thm:robustness:aux22}
1 = \big|e_i^\intercal\cdot\theta^\star_\p\mathcal{R}_\p^{-1}\big| = m_{\p}^{\sigma(i)}\left|e_{\sigma(i)}^\intercal\!\!\cdot\!\big(B_\p + E_\p B_\p\big)\right| \quad \text{ for each } \ i\in[d]
\end{equation}
by the unit-rows requirement \eqref{lica:unitrows}; here, $\sigma$ is the permutation underlying $P_{\mathfrak{a}_\star}\!$. Consequently,
\begin{equation}
\begin{aligned}
\big\|\Lambda_{\p_\star}^{-1} &- \tilde{\Lambda}_{(\p)}(\mathrm{I} + E_\p)\big\| \,\leq\, \sqrt{\sum_{i=1}^d\left|\frac{1}{|e_i^\intercal\cdot B_{\p_\star}|} - m^i_\p\right|^2} + \|A\|_2 L_{\p_\star}\!\epsilon_{r_4}\\
&\stackrel{\mathclap{\eqref{thm:robustness:aux22}}}{\leq}\,\sqrt{2}\sqrt{\sum_{i=1}^d\left[\frac{m_{\p}^i\big|e_i^\intercal(B_\p - B_{\p_\star})\big|}{|e_i^\intercal\cdot B_{\p_\star}|}\right]^2 + \big(m_\p^i|e_i^\intercal E_\p B_\p|\big)^2} + \|A\|_2 L_{\p_\star}\!\epsilon_{r_4}\\
&\leq\, \frac{\sqrt{2}\|A\|_2 L_{\p_\star}\big\|B_\p - B_{\p_\star}\big\|}{\min_{i\in[d]}|e_i^\intercal\cdot B_{\p_\star}|} + \|A\|_2\!\left(\|A^{-1}\|_2\,\hat{\sigma}_{r_4}^\star + 1\right)\!L_{\p_\star}\!\epsilon_{r_4}\,;  
\end{aligned}
\end{equation} 
the first inequality is due to $\|\tilde{\Lambda}_{(\p)}E_\p\|\leq \|\tilde{\Lambda}_{(\p)}\|_2\epsilon_r$ and $\|\tilde{\Lambda}_{(\p)}\|_2\leq \|A\|_2\|\mathcal{R}_\p\|$ (recall \eqref{thm:robust_ica:aux56}), and for the last inequality we further used $\|B_\p\|_2\leq \|A^{-1}\|_2\|\mathcal{R}_\p^{-1}\|_2\leq \|A^{-1}\|_2\hat{\sigma}_r^\star$ (cf.\ \eqref{rem:choiceofR:eq2}).  
Combined with \eqref{thm:robustness:aux18.1} and \eqref{thm:robustness:aux16} (also recalling \eqref{thm:robustness:aux6.3}), the above implies that
\begin{equation}
d\big(\mathfrak{a}_\star, \fI(\p)\big) \,\leq\, K_2\!\left(\sqrt{2}\eta(r_4) + \big(\|A^{-1}\|_2\,\hat{\sigma}_{r_4}^\star + 1\big)\epsilon_{r_4}\right) \,\leq\, \varepsilon
\end{equation}      
and hence, since $\mathfrak{a}_\star\in\mathcal{A}_\star$ was arbitrary, $\sup_{\mathfrak{a}\in\hat{\fI}(\p_\star)}\!d(\mathfrak{a}, \fI(\p))\leq\varepsilon$. That, for any $\p\in\mathbb{B}_{r_4}(\p_\star)$, also $\sup_{\mathfrak{b}\in\hat{\fI}(\p)}\!d(\fI(\p_\star), \mathfrak{b})\leq\varepsilon$ follows similarly. Indeed: Take an arbitrary $\mathfrak{b}\in\fI(\p)$, say $\mathfrak{b}\equiv m_{\mathfrak{b}}(\mathrm{I}+E_{\mathfrak{b}})A^{-1}\in\tilde{\Xi}_{\mathcal{R}_\p}$ (cf.\ \eqref{thm:robustness:aux14}) with $\|E_{\mathfrak{b}}\|\leq\epsilon_{r_4}$ and $m_{\mathfrak{b}}\equiv P_\sigma\tilde{\Lambda}_{(\mathfrak{b})}\in\M$ for $P_\sigma\in\mathrm{P}_d^{\pm}$ and $\tilde{\Lambda}_{(\mathfrak{b})}\equiv\mathrm{ddiag}[m_\mathfrak{b}^{1}, \ldots, m_\mathfrak{b}^{d}]>0$. As above, unit-rows again implies \eqref{thm:robustness:aux22} (but with $(\theta^\star_\p, E_p)$ replaced by $(\mathfrak{b}, E_\mathfrak{b})$) so that for $\tilde{\mathfrak{a}}\coloneqq P_\sigma\Lambda_{\p_\star}^{-1}A^{-1}\in\fI(\p_\star)$ we obtain as above that 
\begin{equation}
d(\fI(\p_\star), \mathfrak{b})\,\leq\, \|\tilde{\mathfrak{a}} - \mathfrak{b}\| \,\leq\, \|A^{-1}\|_2\big\|\Lambda_{\p_\star}^{-1} - \tilde{\Lambda}_{(\mathfrak{b})}(\mathrm{I} + E_{\mathfrak{b}})\big\| \,\leq\, \varepsilon. 
\end{equation} 
Hence $\sup_{\mathfrak{b}\in\fI(\p)}d(\fI(\p_\star), \mathfrak{b})\leq\varepsilon$ (as $\mathfrak{b}\in\hat{\fI}(\p)$ was arbitrary), implying \eqref{thm:robustness:aux17} as desired.              
\end{proof} 

\begin{remark}\label{rem:toproof:thm:robustness}
\begin{enumerate}[font=\upshape, label=(\roman*)]
\item\label{rem:toproof:thm:robustness:it1} For simplicity of constants, we can use \eqref{lem:choiceofR:eq1} to majorize \eqref{thm:robustness:aux15} via
\begin{equation}
\eta(r) \,\leq\, \kappa_2^2(A)L_{\p_\star}\frac{\|A\|_2}{\tilde{\gamma}_\star\sigma_\star}\sup\nolimits_{\mathfrak{a}\in\mathfrak{B}_{\mathfrak{q}_r}}\!\big\|\mathfrak{a} - [\mu_\star]_0\big\| \,\leq\, \check{K}_1 r^\alpha 
\end{equation}
for the ($\p_\star$-dependent) constant $\check{K}_1\coloneqq \kappa_2^2(A)L_{\p_\star}\frac{\|A\|_2}{\tilde{\gamma}_\star\sigma_\star}\hat{K}$; let also $\check{K}_2\coloneqq\|A^{-1}\|_2\hat{\sigma}_{r_3}+1$. Accordingly, for any given $\varepsilon>0$ we may then instead of \eqref{thm:robustness:aux16} choose $\hat{r}=\hat{r}(\varepsilon;\p_\star)$ as 
\begin{equation}
\hat{r}\,\coloneqq\,\sup\big\{0\leq r\leq r_3 \, \big| \, \max\nolimits_{s\in[0,r)}(\sqrt{2}\check{K}_1 s^\alpha + \check{K}_2\epsilon_s) \leq \varepsilon/(\kappa_2(A)L_{\p_\star})\big\}\,\in(0, r_4]\,  
\end{equation} 
to ensure that the continuity-relation $\left.\eqref{thm:robustness:aux1}\right|_{r=\hat{r}}$ holds. 
\item\label{rem:toproof:thm:robustness:it2} If for any fixed $A\in\GL$ one is interested in the continuity at $\sI$ of the map
\begin{equation}\label{rem:toproof:thm:robustness:eq3}
\fI_A\,:\, \big(\fD, \delta\big)\ni\mu \ \longmapsto \ \hat{\Phi}(A\cdot\mu)\in (F,\mathfrak{d}),
\end{equation}
i.e.\ the (partial) continuity of the solution map $\fI(\,\cdot\,, A)$ on the sectionalised identifiability conditions $\sI_A\coloneqq\tilde{\mathscr{I}}\cap(\fD\times\{A\})$ of $E'=\fD\times\GL$, then the above proof of Theorem \ref{thm:robustness} shows that this continuity holds without imposing any additional integrability assumptions on $\fD$ (cf.\ \eqref{lem:unifsigintegr:sufficient:eq1}). More precisely, the continuity of \eqref{rem:toproof:thm:robustness:eq3} admits a simplified quantification via Remark \ref{rem:partialcontinuity}: Operating on the identifiability domain $E_A\coloneqq\fD\times\{A\}$ allows us to replace the above $r_0$ by $\rho_0/\rho_1$ and then copy the argumentation for the proof of Theorem \ref{thm:robustness} but from the adapted premise   
\begin{equation}\tag{\ref*{thm:robustness:aux2}'}\label{rem:toproof:thm:robustness:eq4}
\mu_\chi\equiv A_\ast\mu \quad\text{ with }\quad \delta_{\independent}(\mu) \leq L_{r}\sqrt{r} \quad\text{ and }\quad \big\|[\mu]_{\mathfrak{c}} - [\mu_\star]_{\mathfrak{c}}\big\|_{\mathcal{V}} \,\leq\, \hat{L}r,
\end{equation}
which for any $(\mu, A)\in E_A$ with $\delta(\mu_\star,\mu)< r \leq \rho_0/\rho_1$ holds by Remark \ref{rem:partialcontinuity} \eqref{rem:partialcontinuity:eq2}. Working from \eqref{rem:toproof:thm:robustness:eq4}, the above proof of Thm.\ \ref{thm:robustness} goes through verbatim but with 
\begin{equation}
\text{the (global\footnotemark) replacements:} \ \ K'\leftrightarrow L_r, \ \hat{K}\leftrightarrow\hat{L}, \ \alpha\leftrightarrow 1,
\end{equation}\footnotetext{\ So that $r_1=r_1(\rho_0, \varrho_0,\hat{K},\alpha)$ is replaced by $r_1(\rho_0, \varrho_0, \hat{L}, 1)$ etc.}which results in eased moduli of continuity build from \eqref{notation:unifsigintegr:eq1}-independent constants. \hfill $\bdiam$     
\end{enumerate}  
\end{remark} 

\section{Proof of Proposition \ref{prop:noise} (Cases $i=2,3$)}\label{pf:prop:noise} 
\begin{proof}
Let us now prove \eqref{prop:noise:aux1.1.3} for the cases $i=2,3$ of multiplicative noise. For this we first consider the case $i=2$, that is the triple $(\tilde{X}^{(2)}, A, S^\eta)$ for $\tilde{X}^{(2)}$ as in \eqref{prop:noise:eq1} and $S^\eta$ as in \eqref{sect:noise:eq3}.

Then, since Lebesgue-Stieltjes integration obeys the product rule $\int_a^b\!\varphi\,\mathrm{d}(g_1g_2) = \int_a^b\!\varphi g_1\,\mathrm{d}g_2 + \int_a^b\!\varphi g_2\,\mathrm{d}g_1$ for any $\varphi, g_1, g_2\in\mathcal{C}^1_1$ and every $0\leq a \leq b \leq 1$, we have 
\begin{equation}\label{prop:noise:aux2}
\begin{aligned}
\langle S^\eta\rangle_{ij} \,&=\, \langle\eta^i,\eta^j\,\big|\,S^i\!, S^j\rangle + \langle S^i,\eta^j\,\big|\,\eta^i, S^j\rangle + \langle\eta^i,S^j\,\big|\,S^i, \eta^j\rangle + \langle S^i,S^j\,\big|\,\eta^i, \eta^j\rangle \quad\text{and}\\
\langle S^\eta\rangle_{ijk} \,&=\, \langle\eta^i,\eta^j,\eta^k\,\big|\,S^i,S^j,S^k\rangle + \langle\eta^i,\eta^j,S^k\,\big|\,S^i,S^j,\eta^k\rangle + \langle S^i,\eta^j,\eta^k\,\big|\,\eta^i,S^j,S^k\rangle\\ 
&\hspace{1em}+ \langle S^i,\eta^j,S^k\,\big|\,\eta^i,S^j,\eta^k\rangle + \langle\eta^i,S^j,\eta^k\,\big|\,S^i,\eta^j,S^k\rangle + \langle\eta^i,S^j,S^k\,\big|\,S^i,\eta^j,\eta^k\rangle\\
&\hspace{1em}+ \langle S^i,S^j,\eta^k\,\big|\,\eta^i,\eta^j,S^k\rangle + \langle S^i,S^j,S^k\,\big|\,\eta^i,\eta^j,\eta^k\rangle 
\end{aligned}
\end{equation}  
where for $Y,\tilde{Y}, \hat{Y}, \check{Y}, Z, \tilde{Z}\in\{S,\eta\}$ we denote $\langle Y^i, \tilde{Y}^j\,\big|\, \hat{Y}^i, \check{Y}^j\rangle\coloneqq \mathbb{E}\big[\!\int_0^1\!\!\int_0^t \bm{Y}^i_s \tilde{\bm{Y}}^j_t\mathrm{d}\hat{\bm{Y}}_s^i\mathrm{d}\check{\bm{Y}}_t^j\big]$ and 
\begin{equation}\label{prop:noise:aux3}
\langle Y^i,\tilde{Y}^j,Z^k\,\big|\,\hat{Y}^i,\check{Y}^j,\tilde{Z}^k\rangle\coloneqq\mathbb{E}\!\left[\!\int_0^1\!\!\!\int_0^t\!\!\!\int_0^s\!\! \bm{Y}^i_r\tilde{\bm{Y}}^j_s \bm{Z}^k_t\,\mathrm{d}\hat{\bm{Y}}^i_r\mathrm{d}\check{\bm{Y}}^j_s\mathrm{d}\tilde{\bm{Z}}_t^k\right].
\end{equation} 
Evaluating the statistics \eqref{prop:noise:aux2} sumand-wise in a similar way as before, we find       
\begin{equation}
\langle\eta^i,\eta^j\,\big|\,S^i\!, S^j\rangle = \!\int_{\Delta_2}\!\!\!\mathbb{E}[\bm{\eta}_s^i\bm{\eta}_t^j]\mathbb{E}[\dot{\bm{S}}^i_s\dot{\bm{S}}^j_t]\,\mathrm{d}^2\bm{t}  \quad\text{and}\quad \langle S^i,\eta^j\,\big|\,\eta^i\!, S^j\rangle= \!\int_{\Delta_2}\!\!\!\mathbb{E}[\dot{\bm{\eta}}^i_s\bm{\eta}_t^j]\mathbb{E}[\bm{S}_s^i\dot{\bm{S}}^j_t]\,\mathrm{d}^2\bm{t} \quad\text{etc.,}
\end{equation}
whence and by the fact that $\bm{S}$ is a mean-stationarity IC-source with $\E[\bm{S}] = 0$, which implies that $\E[\bm{S}_s^i\dot{\bm{S}}^j_t] = \E[\bm{S}_s^i]\E[\dot{\bm{S}}^j_t] = 0 = \E[\bm{S}^i_s\bm{S}^j_t]$ whenever $i\neq j$, we obtain 
\begin{equation}\label{prop:noise:aux5}
\langle S^\eta\rangle_{ij} = \left[\!\int_{\Delta_2}\!\!\!\phi^i_{(S,\eta)}(\bm{t})\,\mathrm{d}^2\bm{t}\right]\cdot\delta_{ij} \qquad\text{for each } \ ij\in[d]^\star_2, 
\end{equation}    
where $\phi^i_{(S,\eta)}(s,t)\coloneqq \E\big[\bm{\eta}^i_s\bm{\eta}^i_t\big]\E\big[\dot{\bm{S}}^i_s\dot{\bm{S}}^i_t\big] + \E\big[\dot{\bm{\eta}}^i_s\bm{\eta}^i_t\big]\E\big[\bm{S}^i_s\dot{\bm{S}}^i_t\big] + \E\big[\bm{\eta}^i_s\dot{\bm{\eta}}^i_t\big]\E\big[\dot{\bm{S}}^i_s \bm{S}^i_t\big] + \E\big[\dot{\bm{\eta}}^i_s\dot{\bm{\eta}}^i_t\big]\E\big[\bm{S}^i_s \bm{S}^i_t\big]$. The third-order statistics \eqref{prop:noise:aux3} behave analogously, which leads to
\begin{equation}\label{prop:noise:aux6}
\langle S^\eta\rangle_{ijk} = \left[\!\int_{\Delta_3}\!\!\!\psi^i_{(S,\eta)}(\bm{t})\,\mathrm{d}^3\bm{t}\right]\cdot\delta_{ijk} \qquad\text{for each } \ ijk\in[d]^\star_3,
\end{equation}  
where this time the integrand is the function  
\begin{equation}
\begin{aligned}
\psi^i_{(S,\eta)}(r,s,t)&\coloneqq \mathbb{E}[\bm{\eta}^i_r\bm{\eta}^i_s\bm{\eta}^i_t]\mathbb{E}[\dot{\bm{S}}^i_r\dot{\bm{S}}^i_s\dot{\bm{S}}^i_t] + \mathbb{E}[\bm{\eta}^i_r\bm{\eta}^i_s\dot{\bm{\eta}}^i_t]\mathbb{E}[\dot{\bm{S}}^i_r\dot{\bm{S}}^i_s \bm{S}^i_t] + \mathbb{E}[\dot{\bm{\eta}}^i_s\bm{\eta}^i_s\bm{\eta}^i_t]\mathbb{E}[\bm{S}^i_r\dot{\bm{S}}^i_s\dot{\bm{S}}^i_t]\\
&\hspace{1em}+ \mathbb{E}[\dot{\bm{\eta}}^i_r\eta^i_s\dot{\bm{\eta}}^i_t]\mathbb{E}[\bm{S}^i_r\dot{\bm{S}}^i_s \bm{S}^i_t] + \mathbb{E}[\bm{\eta}^i_r\dot{\bm{\eta}}^i_s\bm{\eta}^i_t]\mathbb{E}[\dot{\bm{S}}^i_r \bm{S}^i_s\dot{\bm{S}}^i_t] + \mathbb{E}[\bm{\eta}^i_r\dot{\bm{\eta}}^i_s\dot{\bm{\eta}}^i_t]\mathbb{E}[\dot{\bm{S}}^i_r \bm{S}^i_s \bm{S}^i_t]\\
&\hspace{1em}+ \mathbb{E}[\dot{\bm{\eta}}^i_r\dot{\bm{\eta}}^i_s\bm{\eta}^i_t]\mathbb{E}[\bm{S}^i_r \bm{S}^i_s\dot{\bm{S}}^i_t] + \mathbb{E}[\dot{\bm{\eta}}^i_r\dot{\bm{\eta}}^i_s\dot{\bm{\eta}}^i_t]\mathbb{E}[\bm{S}^i_r \bm{S}^i_s \bm{S}^i_t].
\end{aligned}
\end{equation} 
Enumerating the sumands $\phi^i_j\equiv\phi^i_{\ell,j}\cdot\phi^i_{r,j}$ which constitute $\phi^i_{(S,\eta)}$ --- resp.\ the sumands $\psi^i_j\equiv\psi^i_{\ell,j}\cdot\psi^i_{r,j}$ which constitute $\psi^i_{(S,\eta)}$ --- in the order in which they are displayed above (so that e.g.\ $\phi^i_2 = \E\big[\dot{\bm{\eta}}^i_s\bm{\eta}^i_t\big]\E\big[\bm{S}^i_s\dot{\bm{S}}^i_t\big]$ and $(\phi^i_{\ell,2}, \phi^i_{r,2}) = \big(\E\big[\dot{\bm{\eta}}^i_s\bm{\eta}^i_t\big], \E\big[\bm{S}^i_s\dot{\bm{S}}^i_t\big]\big)$, etc.), we from \eqref{prop:noise:aux5} and by Hölder's inequality obtain that, for each $i\in[d]$,
\begin{align}
\big|\langle S^\eta\rangle_{ii} - \langle S\rangle_{ii}\big| \,&=\, \left|\int_{\Delta_2}\!\!\!\big(\phi^i_{\ell,1}(\bm{t}) - 1\big)\phi^i_{r,1}(\bm{t}) + {\textstyle\sum}_{j=2}^4\phi^i_j(\bm{t})\,\mathrm{d}^2\bm{t}\right|\\\label{prop:noise:aux8}
&\leq\, \left[\int_{\Delta_2}\!\!\!\big(\phi^i_{\ell,1}(\bm{t}) - 1\big)^2\mathrm{d}^2\bm{t}\right]^{\!1/2}\!\!\!\!\!\!\!\cdot\tilde{c}_{i,1} + \sum_{j=2}^4\left[\int_{\Delta_2}\!\!\!\big(\phi^i_{\ell,j}(\bm{t})\big)^2\mathrm{d}^2\bm{t}\right]^{\!1/2}\!\!\!\!\!\!\!\cdot\tilde{c}_{i,j}
\end{align} 
for the source-dependent constants $\tilde{c}_{i,j}\coloneqq\big(\!\int_{\Delta_2}\phi^i_{r,j}(\bm{t})^2\,\mathrm{d}^2\bm{t}\big)^{1/2}$. Likewise, for each $i\in[d]$,
\begin{equation}\label{prop:noise:aux9}
\big|\langle S^\eta\rangle_{iii} - \langle S\rangle_{iii}\big| \,\leq\, \left[\int_{\Delta_3}\!\!\!\big(\psi^i_{\ell,1}(\bm{t}) - 1\big)^2\mathrm{d}^3\bm{t}\right]^{\!1/2}\!\!\!\!\!\!\!\cdot\check{c}_{i,1} + \sum_{j=2}^8\left[\int_{\Delta_3}\!\!\!\big(\psi^i_{\ell,j}(\bm{t})\big)^2\mathrm{d}^3\bm{t}\right]^{\!1/2}\!\!\!\!\!\!\!\cdot\check{c}_{i,j} 
\end{equation} 
for the source-dependent constants $\check{c}_{i,j}\coloneqq\big(\!\int_{\Delta_3}\psi^i_{r,j}(\bm{t})^2\,\mathrm{d}^3\bm{t}\big)^{1/2}$. For convenience, let us abbreviate the expression in line \eqref{prop:noise:aux8} by $\langle\tilde{b}_i, \tilde{c}_i\rangle\equiv\sum_{j=1}^4\tilde{b}_{i,j}\cdot\tilde{c}_{i,j}$ for $\tilde{b}_i\coloneqq(\tilde{b}_{i,1}, \tilde{b}_{i,2},\tilde{b}_{i,3}, \tilde{b}_{i,4})$ and $\tilde{c}_i\coloneqq(\tilde{c}_{i,1},\ldots,\tilde{c}_{i,4})$ and $\tilde{b}_{i,j}\coloneqq\big(\!\int_{\Delta_2}(\phi^i_{\ell,j}(\bm{t}) - \delta_{1j})^2\,\mathrm{d}^2\bm{t}\big)^{1/2}$, and likewise abbreviate the RHS in \eqref{prop:noise:aux9} by $\langle\check{b}_i, \check{c}_i\rangle\equiv\sum_{j=1}^8\check{b}_{i,j}\cdot\check{c}_{i,j}$ for $\check{b}_i\coloneqq(\check{b}_{i,1}, \ldots, \check{b}_{i,8})$ and $\check{c}_i\coloneqq(\check{c}_{i,1},\ldots, \check{c}_{i,8})$ and $\check{b}_{i,j}\coloneqq\big(\!\int_{\Delta_3}(\psi^i_{\ell,j}(\bm{t}) - \delta_{1j})^2\,\mathrm{d}^3\bm{t}\big)^{1/2}$. Then by \eqref{prop:noise:aux5}, \eqref{prop:noise:aux6}, \eqref{prop:noise:aux8}, \eqref{prop:noise:aux9} and with \eqref{sect:coredinates:eq4},
\begin{equation}
\begin{aligned}
\delta(S,S^\eta)^2 \,&=\, \big\|N_S^{-1}\big([S^\eta]_0 - [S]_0\big)N_S^{-1}\big\|^2 + \sum_{\nu=1}^d\langle S\rangle_{\nu\nu}^{-1}\big\|N_S^{-1}\big([S^\eta]_\nu - [S]_\nu\big)N_S^{-1}\big\|^2\\
&=\, \sum_{i=1}^d\!\big(\langle S\rangle_{ii}^{-1}\big|\langle S^\eta\rangle_{ii} - \langle S\rangle_{ii}\big|\big)^{\!2} + \sum_{\nu=1}^d\langle S\rangle_{\nu\nu}^{-1}\big(\langle S\rangle_{\nu\nu}^{-1}\big|\langle S^\eta\rangle_{\nu\nu\nu} - \langle S\rangle_{\nu\nu\nu}\big|\big)^2\\
&\leq\, \sum_{i=1}^d \gamma_i^{-2}\big(\langle\tilde{b}_i, \tilde{c}_i\rangle^2 + \gamma_i^{-1}\langle\check{b}_i, \check{c}_i\rangle^2\big)\eqqcolon\beta_2(S,\eta), \qquad\text{for } \ \ \gamma_i\coloneqq\langle S\rangle_{ii}. 
\end{aligned} 
\end{equation}  
Since $\beta(S,\eta)\leq\beta_\varepsilon^2$ by assumption on $\eta$, we obtain \eqref{prop:noise:aux1.1.3} for the noise case $i=2$ as desired.  

The case $i=3$ follows (almost) analogously: As a consequence of the product rule, the statistics $\langle X^\eta\rangle_{ij}$ and $\langle X^\eta\rangle_{ijk}$ are also of the form \eqref{prop:noise:aux2} but with $S$ replaced by $X$. Hence by the same arguments as above we can conclude that for each $ijk\in[d]^\star_3$, 
\begin{align}\label{prop:noise:aux11}
\langle X^\eta\rangle_{ij} \,&=\, \langle\eta^i,\eta^j\,\big|\,X^i\!, X^j\rangle + \langle X^i,X^j\,\big|\,\eta^i, \eta^j\rangle \quad\text{ and}\\
\langle X^\eta\rangle_{ijk} \,&=\, \langle\eta^i,\eta^j,\eta^k\,\big|\,X^i,X^j,X^k\rangle + \langle X^i,\eta^j,X^k\,\big|\,\eta^i,X^j,\eta^k\rangle + \langle\eta^i,X^j,X^k\,\big|\,X^i,\eta^j,\eta^k\rangle\\
&\hspace{1em}+\langle X^i,X^j,\eta^k\,\big|\,\eta^i,\eta^j,X^k\rangle + \langle X^i,X^j,X^k\,\big|\,\eta^i,\eta^j,\eta^k\rangle  
\end{align}    
where we used that $\eta$ is IC, independent from $X$ and mean-stationary. Further, using $\E[\eta]\equiv\mathrm{I}$, 
\begin{equation}
\begin{aligned}
\big\|[X^\eta]_0 - [X]_0\big\|^2 \,&=\, \sum_{i=1}^d \big|\langle\eta^i,\eta^i\,\big|\,X^i\!, X^i\rangle - \langle X\rangle_{ii} + \langle X^i,X^i\,\big|\,\eta^i, \eta^i\rangle\big|^2 \,=\, \sum_{i=1}^d\left|\int_{\Delta_2}\!\!\!\!\xi_i(\bm{t})\,\mathrm{d}^2\bm{t}\,\right|^2,\\
\big\|[X^\eta]_\nu - [X]_\nu\big\|^2 \,&=\, \tilde{\Delta}_{\nu\nu\nu}^2 + \sum\nolimits_{i\in[d]\setminus\{\nu\}}\!\!\left(\tilde{\Delta}_{i\nu\nu}^2 + \tilde{\Delta}_{\nu i\nu}^2\right) \,=\, \left|\int_{\Delta_3}\!\!\!\!\Xi_\nu(\bm{t})\,\mathrm{d}^3\bm{t}\,\right|^2 + \sum_{i\neq\nu}\!\big(\tilde{\Xi}^{(1)}_{i,\nu} + \tilde{\Xi}^{(2)}_{i,\nu}\big)\\ 
\end{aligned}
\end{equation}  
for $\nu=1,\ldots, d$ and $\tilde{\Delta}_{ijk}\coloneqq\big|\langle X^\eta\rangle_{ijk} - \langle X\rangle_{ijk}\big|$ and with the \eqref{prop:noise:aux11}-derived quantities
\begin{equation}
\begin{aligned}
\xi_i(s,t)&\coloneqq \big(\E[\bm{\eta}^i_s\bm{\eta}^i_t] - 1\big)\E[\dot{\bm{X}}^i_s \dot{\bm{X}}^i_t] + \E[\dot{\bm{\eta}}^i_s\dot{\bm{\eta}}^i_t]\E[\bm{X}^i_s \bm{X}^i_t] \quad\text{ and}\\
\Xi_\nu(r,s,t) &\coloneqq \big(\E[\bm{\eta}^\nu_r\bm{\eta}^\nu_s\bm{\eta}^\nu_t] - 1\big)\E[\dot{\bm{X}}^\nu_r\dot{\bm{X}}^\nu_s \dot{\bm{X}}^\nu_t] + \E[\dot{\bm{\eta}}^\nu_r\dot{\bm{\eta}}^\nu_s\dot{\bm{\eta}}^\nu_t]\E[\bm{X}^\nu_r \bm{X}^\nu_s \bm{X}^\nu_t] + \mathcal{R}_\nu
\end{aligned}
\end{equation} 
for $R_\nu\coloneqq \E[\dot{\bm{\eta}}^\nu_r\bm{\eta}^\nu_s\dot{\bm{\eta}}^\nu_t]\E[\bm{X}^\nu_r\dot{\bm{X}}^\nu_s \bm{X}^\nu_t] + \E[\bm{\eta}^\nu_r\dot{\bm{\eta}}^\nu_s\dot{\bm{\eta}}^\nu_t]\E[\dot{\bm{X}}^\nu_r \bm{X}^\nu_s \bm{X}^\nu_t] + \E[\dot{\bm{\eta}}^\nu_r\dot{\bm{\eta}}^\nu_s\bm{\eta}^\nu_t]\E[\bm{X}^\nu_r \bm{X}^\nu_s \dot{\bm{X}}^\nu_t]$ and $\tilde{\Xi}^{(1)}_{i,\nu} \coloneqq \left|\int_{\Delta_3}\!(\E[\bm{\eta}^\nu_s\bm{\eta}^\nu_t]-1)\E[\dot{\bm{X}}^i_r\dot{\bm{X}}^\nu_s\dot{\bm{X}}^\nu_t] + \E[\dot{\bm{\eta}}^\nu_s\dot{\bm{\eta}}^\nu_t]\E[\dot{\bm{X}}^i_r \bm{X}^\nu_s \bm{X}^\nu_t]\,\mathrm{d}^3\bm{t}\,\right|^2$, with $\tilde{\Xi}^{(2)}_{i,\nu}$ defined likewise. 

Abbreviating $\hat{\xi}_i\coloneqq \left|\int_{\Delta_2}\!\!\xi_i(\bm{t})\,\mathrm{d}^2\bm{t}\,\right|^2$ and $\hat{\Xi}_\nu\coloneqq\left|\int_{\Delta_3}\!\!\Xi_\nu(\bm{t})\,\mathrm{d}^3\bm{t}\,\right|^2$, we thus have the control 
\begin{equation}\label{prop:noise:aux14}
\|[X^\eta]_{\mathfrak{c}} - [X]_{\mathfrak{c}}\|^2\,=\, \sum_{i=1}^d\!\big(\hat{\xi}_i + \hat{\Xi}_i\big)^2 + \sum_{\nu=1}^d\sum_{i\neq\nu}\!\big(\tilde{\Xi}^{(1)}_{i,\nu} + \tilde{\Xi}^{(2)}_{i,\nu}\big) \eqqcolon \beta_3(X,\eta).
\end{equation}       
Now by continuity of the $A^{-1}$-induced tensor action on $\mathcal{V}$, cf.\ the proof of Prop.\ \ref{prop:estimation-finsample} [first half], there is an explicitly computable threshold $\gamma_\varepsilon\equiv\gamma_\varepsilon(\beta_\varepsilon, A)>0$ small enough such that
\begin{equation}
\beta_3(X,\eta)\,\leq\, \gamma_\varepsilon \quad\text{ implies }\quad \delta(S, A^{-1}X^\eta) \,\leq\,\beta_\varepsilon.
\end{equation}
This shows \eqref{prop:noise:aux1.1.3} for the case $i=3$, as desired. 
\end{proof}

\end{document}